\documentclass[11pt]{amsart}
\usepackage{amsthm}
\usepackage{amsmath}
\usepackage{amsxtra}
\usepackage{amscd}
\usepackage{amssymb}
\usepackage{xypic}
\usepackage{color}
\usepackage{mathtools}
\usepackage{xspace}
\usepackage{amsfonts}
\usepackage{mathrsfs}
\usepackage{todonotes}
\usepackage{enumerate}
\usepackage{multirow}
\usepackage{tikz}
\usepackage{hyperref}
\usepackage{cancel}
\usepackage{tikz-cd}
\usepackage[all]{xy}
\usepackage{amscd}
\usepackage{caption}
\usepackage{epsfig}
\usepackage{graphicx}
\usepackage{here}
\usepackage{hyperref}
\usepackage{psfrag}
\usepackage{setspace}
\usepackage{wasysym}
\usepackage{xparse}
\usepackage{manfnt}
\usepackage[normalem]{ulem}
\usepackage{musicography}

\usetikzlibrary{calc}
\usetikzlibrary{decorations.pathreplacing,decorations.markings,decorations.pathmorphing}
\usetikzlibrary{positioning,arrows,patterns}
\usetikzlibrary{cd}
\usetikzlibrary{intersections}
\usetikzlibrary{arrows}

\DeclareRobustCommand{\SkipTocEntry}[9]{}

\usepackage[style=alphabetic,
            isbn=false,
            doi=false,
            url=false,
            backend=bibtex, maxnames = 6,
           ]{biblatex}
\defbibheading{bibliography}[\bibname]{%
  \section*{References}%
}

\def\changed#1{\textcolor{black}{#1}}

\DeclareFieldFormat[article]{title}{{\it #1}}
\DeclareFieldFormat{journaltitle}{{\rm #1}}
\renewbibmacro{in:}{%
  \ifentrytype{article}{}{\printtext{\bibstring{in}\intitlepunct}}}
\DeclareFieldFormat[incollection]{title}{{\it #1}}
\DeclareFieldFormat{journaltitle}{{\rm #1}}

\input{preamble.tex}
\bibliography{plumb_biblio}

\title[The moduli space of multi-scale differentials]{The moduli space of multi-scale differentials}

\author[Bainbridge]{Matt Bainbridge}
\address{Department of Mathematics, Indiana University, Bloomington, IN 47405, USA}
\email{mabainbr@indiana.edu}
\thanks{Research of the first author is supported in part by the Simons Foundation grant \#713192.}
\author[Chen]{Dawei Chen}
\address{Department of Mathematics, Boston College, Chestnut Hill, MA 02467, USA}
\email{dawei.chen@bc.edu}
\thanks{Research of the second author was supported in part by National Science Foundation grants DMS-23-01030, DMS-20-01040, the CAREER award DMS-13-50396, the Simons Foundation Travel Support for Mathematicians, a von Neumann Fellowship, and a Simons Fellowship.}
\author[Gendron]{Quentin Gendron}
\address{Centro de Ciencias Matem\'aticas-UNAM, Antigua Car. a P\'atzcuaro 8701,
Col. Ex Hacienda San Jos\'e de la Huerta,
Morelia, Mich., M\'exico}
\curraddr{Instituto de Matem\'{a}ticas de la UNAM
Ciudad Universitaria, CDMX, 04510,
M\'{e}xico}

\email{quentin.gendron@im.unam.mx}
\thanks{Research of the third author was supported in part by the project CONACyT A1-S-9029 ``Moduli de curvas y Curvatura en $A_{g}$" of Abel Castorena.}
\author[Grushevsky]{Samuel Grushevsky}
\address{Mathematics Department, and Simons Center for Geometry and Physics, Stony Brook University,
Stony Brook, NY 11794-3651, USA}
\email{sam@math.stonybrook.edu}
\thanks{Research of the fourth author was supported in part by the National Science Foundation under the grant DMS-21-01631}
\author[M\"oller]{Martin M\"oller}
\address{Institut f\"ur Mathematik, Goethe-Universit\"at Frankfurt, Robert-Mayer-Str. 6-8,
60325 Frankfurt am Main, Germany}
\email{moeller@math.uni-frankfurt.de}
\thanks{Research of the fifth author is partially supported
  by the DFG-project MO 1884/2-1 and by the LOEWE-Schwerpunkt
  ``Uniformisierte Strukturen in Arithmetik und Geometrie''}

\begin{document}

\begin{abstract}
We construct a compactification $\PP\LMS$ of the moduli spaces of
abelian differentials on Riemann surfaces with prescribed zeroes and poles.
This compactification, called the moduli space of \msds,  is a complex orbifold
with normal crossing boundary. Locally, $\PP\LMS$ can be described as the normalization of an explicit blowup of the incidence variety compactification, which was defined in~\cite{strata} as the closure of the stratum of abelian differentials in the closure of the Hodge bundle. We also
define families of projectivized \msds, which gives a proper smooth
Deligne-Mumford stack, and $\PP\LMS$ is the orbifold corresponding to it. Moreover, we perform a real oriented blowup of the unprojectivized space~$\LMS$ such that the $\GLtwoRplus$-action in the interior of the moduli space extends
continuously to the boundary.

A \msd on a pointed stable curve is the data of an enhanced level structure on the dual graph, prescribing the orders of poles and zeroes at the nodes, together with a collection of meromorphic differentials on the irreducible components satisfying certain conditions. Additionally, the \msd encodes the data of a prong-matching at the nodes, matching the incoming and outgoing horizontal trajectories in the flat structure. The construction of $\PP\LMS$ furthermore requires defining families of
\msds, where the underlying curve can degenerate, and understanding the notion of equivalence of \msds under various rescalings.

Our construction of the compactification proceeds via first constructing an augmented Teichm\"uller space of flat surfaces, and then taking its suitable quotient. Along the way, we give a complete proof of the fact that the conformal and quasiconformal topologies on the (usual) augmented Teichm\"uller space agree.
\end{abstract}

\date{\today}

\maketitle
\newpage
\tableofcontents

%%%%%%%%%%%%%%%%%%%%%%%%%%%%%%%%%%%%%%%%%%%%%%%

%%%%%%%%%%%%%%%%%%%%%%%%%%%%%%%%%%%%
\section{Introduction} \label{sec:intro}
%%%%%%%%%%%%%%%%%%%%%%%%%%%%%%%%%%%%%

The goal of this paper is to construct a compactification of the (projectivized)
moduli spaces of abelian differentials $\PP\Omega\calM_{g,n}(\mu)$ of type
$\mu=(m_1,\dots,m_n)$ with zeros and poles of order~$m_i$ at the marked points.
Our compactification shares almost all of the useful properties of the Deligne-Mumford
compactification~$\barmoduli$ of the moduli space of curves~$\moduli$. These properties
include a normal crossing boundary divisor, natural coordinates near the boundary, and
representing a natural moduli functor. Applications of the compactification
include justification for intersection theory computations, a notion of the tautological
ring, an algorithm to compute Euler characteristics of $\PP\Omega\calM_{g,n}(\mu)$,
and potentially contributions to the classification of $\SL_2(\RR)$-orbit closures.
Throughout this paper the zeroes and poles are labeled. The reader may
quotient by a symmetric group action as discussed in Section~\ref{sec:notation} to obtain
the (unmarked) strata of abelian differentials.
\par
The description of our compactification as a moduli space of what we call
\msds should be compared with the objects characterizing the naive compactification, the
{\em incidence variety compactification (IVC)} we studied in \cite{strata}. The IVC
is defined as the closure of the moduli space $\Omega\calM_{g,n}(\mu)$ in the extension
of the Hodge bundle $\obarmoduli[g,n]$ over $\overline\calM_{g,n}$ in the
holomorphic case, and as the closure in a suitable twist in the meromorphic case. The IVC
can have bad singularities near the boundary, e.g.\ they can fail to be $\QQ$-factorial
(see Section~\ref{subsec:blowup} and Example~\ref{ex:q-factorial}), and we are not
aware of a good coordinate system near the boundary. Points in the IVC can be described
by twisted differentials, whose definition we now briefly recall. \
\par
The {\em dual graph} of a stable curve~$X$ has vertices~$v \in V(\Gamma)$
corresponding to irreducible components~$X_v$ of the stable curve, and
edges~$e \in E(\Gamma)$ corresponding to
nodes~$q_e$.  A {\em level graph} endows $\Gamma$ with a level function
$\ell\colon V(\Gamma)\to\RR$, and we may assume that its image, called the set of levels
$L^\bullet(\Gamma)$, is the set $\lbrace 0,-1,\dots,-N\rbrace$ for some~$N \in \ZZ_{\ge 0}$.
We write $X_{(i)}$ for the union of all irreducible components of~$X$ that
are at level~$i$. A {\em twisted differential
of type~$\mu$} compatible with a level graph is a collection $(\eta_{(i)})_{i \in
L^\bullet(\Gamma)}$ of non-zero meromorphic differentials on the subcurves~$X_{(i)}$,
having order prescribed by~$\mu$
at the marked points and satisfying the matching order  condition, the matching residue
condition, and the global residue condition (GRC), which we restate in detail in
Section~\ref{sec:deftwd}.
\par
The top level~$X_{(0)}$ is the subcurve on which, in a one-parameter family
over a complex disc with parameter~$t$, the limit of differentials~$\omega_t$
is a non-zero differential~$\eta_{(0)}$, while this limit is zero on all
lower levels. By rescaling with appropriate powers of~$t$, we obtain
the non-zero limits on the lower levels. The order of the levels here reflects the exponents of~$t$.
Note that a point in the IVC determines a twisted differential
only up to rescaling individually on each irreducible component of the limiting curve.
\par
\medskip
The notion of a \msd refines the notion of a twisted differential in three ways.
First, the equivalence relation is a rescaling level-by-level, by the level rotation torus
(defined below, see also Section~\ref{sec:levrottori}) instead of component-by-component.
Second, the graph records besides the level structure an enhancement prescribing the vanishing order at the nodes,
see Section~\ref{sec:enh}. Third, we additionally record in a prong-matching (defined
below, see also Section~\ref{sec:prmatch}) a finite amount of extra data at every node,
a matching of horizontal directions for the flat structure at the two preimages of
the node. \changed{Here we give a brief explanation of
  these new concepts and an informal definition.  We give a more
  precise definition in Section~\ref{sec:AugTeich} after the relevant concepts
  have all been introduced.}  
\par
\begin{df}
  \label{def:one_msd}
A {\em \msd of type $\mu$} on a stable pointed curve $(X,\bfz)$ consists of
\begin{itemize}
  \item[(i)] an enhanced level structure on the dual graph~$\Gamma$ of~$(X,\bfz)$,
  \item[(ii)] a twisted differential of type~$\mu$ compatible with the
enhanced level structure,
\item[(iii)]  and a \prma for each node of~$X$ joining components of non-equal level.
\end{itemize}
Two \msds are considered equivalent if they differ by the action of the level rotation torus.
\end{df}
\par
The notion of a family of \msds requires to deal with the subtleties of
the enhanced level graph varying, with vanishing rescaling parameters, and also
with the presence of nilpotent functions on the base space. The complete definition
of a family of \msds, the corresponding functor $\MSfun$ on the category
of complex spaces, and the groupoid $\MSgrp$ will be given in Section~\ref{sec:famnew}.
They come with projectivized versions, denoted by $\PP\MSfun$ and $\PP\MSgrp$.
\par
\begin{thm}[Main theorem] \label{intro:main}
There is a complex orbifold $\LMS$, the {\em moduli space of \msds}, with the following
properties:
\begin{enumerate}
\item The moduli space $\Omega\calM_{g,n}(\mu)$ is open and dense within $\LMS$.
\item The boundary $\LMS \setminus \Omega\calM_{g,n}(\mu)$ is a normal crossing divisor.
\item $\LMS$ admits a $\CC^*$-action, and the projectivization $\PP\LMS$ is compact.
\item The complex space underlying $\LMS$ is a coarse moduli space for $\MSfun$.
\item The complex space underlying $\LMS$ admits a forgetful map to the normalization
  of the IVC.
\end{enumerate}
\end{thm}
\par
In fact, the codimension of a boundary stratum of \msds compatible with an
enhanced level graph~$\Gamma$ is equal to the number of levels below zero
plus the number of horizontal nodes, that is, nodes joining components on the
same level.
\par
Our proof of algebraicity requires us to recast this theorem in the language
of stacks. For the next theorem note that we may view an orbifold such as
$\PP\LMS$ as a smooth stack glued from quotient stacks.
\par
\begin{thm}[Functorial viewpoint] \label{intro:funct}
The groupoid $\PP\MSgrp$ of projectivized \msds is a proper
Deligne-Mumford stack. Moreover, there is a morphism of proper algebraic
Deligne-Mumford stacks $\PP\LMS \to \PP\MSgrp$,
which is an isomorphism over the open substack~$\PP\omoduli[g,n](\mu)$.
\end{thm}
\par
The groupoid $\PP\MSgrp$ is thus a hybrid object, smooth with
non-trivial isomorphism groups (due to automorphisms) at some places,
and with finite quotient singularities \changed{(due to a non-trivial
  quotient of twist groups, see Section~\ref{sec:twrot})} at other
places. The description as orderly blowup $\PP\MSgrp$ in
Theorem~\ref{intro:orderly} makes this functorial viewpoint even more
natural.  In fact the map in Theorem~\ref{intro:funct} is an
isomorphism over the substack where the local groups~$K_\Lambda$
introduced in Section~\ref{sec:coverGH} are trivial. With a similar
construction one can obtain a compactification of the space of
$k$-differentials for all $k \geq 1$ with the same good properties,
see \cite{CoMoZa} for details.
\par
\medskip
\paragraph{\bf Other compactifications} We briefly mention the relation
with other compactifications in the literature. The space constructed in \cite{fapa}
can have extra components, hence in general it is a reducible space that
contains the IVC only as one of its components. As emphasized in that paper,
the moduli spaces of meromorphic $k$-differentials can be viewed as
generalizations of the double ramification cycles. There are several
(partial) compactifications of the ($k$-twisted version of the) double
ramification cycle, see for example  \cite{HoKaPa} and \cite{HolmesSchmitt},
mostly with focus on extending the Abel-Jacobi maps.
\par
Mirzakhani-Wright~\cite{mwinv} considered the compactification of holomorphic strata
that simply forgets all irreducible components of the stable curve
on which the limit differential is identically zero.
This is called the WYSIWYG (``what you see is what you get") compactification.
Since this compactification reflects much of the tangent space of an $\SL_2(\RR)$-orbit
closure, it has proven useful to their classification. This compactification is
however not even a complex analytic space, see \cite{CW}.
\par
\medskip
\paragraph{\bf Applications} Many applications of our compactification are
based on the normal crossing boundary divisor and a good coordinate system,
given by the perturbed period coordinates (see Section~\ref{sec:perturbed}) near
the boundary. The first application in \cite{CoMoZa} shows that the area form
is a good enough metric on the tautological bundle. This is required in
\cite{SauvagetMinimal, CMSPrincQuad, SauvagetFlat} for direct computations
of Masur-Veech volumes, and in \cite{CMSZ} to justify the
volume formula for the spin components.
\par
A second application in \cite{CoMoZaEU} is the construction of an analogue of
the Euler
sequence for projective spaces on $\LMS$. This allows to recursively compute
all Chern classes of the (logarithmic) cotangent bundle to $\LMS$. In particular
this gives a recursive way to compute the orbifold Euler characteristic of the
moduli spaces $\PP\omoduli[g,n](\mu)$. Moreover, it gives a formula for the
canonical bundle. As in the case of the moduli space of curves, this opens the
gate towards determining the Kodaira dimension of $\PP\omoduli[g,n](\mu)$.
Previously the Kodaira dimension was known only for some series of
special cases, see~\cite{farkasverraoddspin, gendron, barros}. \changed{Building on the compactification developed in this paper, the recent work of \cite{CCMKod} determined the Kodaira dimension for a large class of $\PP\omoduli[g,n](\mu)$.} 
\par
A third application is towards the classification of connected components of
the strata of $k$-differentials in~\cite{ChenGendronComponents}. Crucially,
the viewpoint in that paper relies on the smoothness of a $k$-differential
version of $\LMS$ developed later in~\cite{CoMoZa}. In particular,
the concept of \prmas helps to construct certain multi-scale
$k$-differentials in the boundary that can provide information for generalized
spin and hyperelliptic structures after smoothing into the interior of the strata.
\par
A fourth large circle of applications concerns the dynamics of the action
of $\GLtwoRplus$ on $\Omega\calM_{g,n}(\mu)$, in particular in the case when the
type $\mu$ corresponds to holomorphic differentials. In this case the results of
Eskin-Mirzakhani-Mohammadi~\cite{eskinmirzakhani, esmimo} and Filip~\cite{filip}
show that the closure of every orbit is an algebraic variety defined by linear
equations in period coordinates \changed{that have real algebraic numbers as coefficients}. The classification
of these orbit closures is an important goal towards which significant progress
has been made recently, see for example constraints found by Eskin-Filip-Wright~\cite{esfiwr}
and the constructions of special orbit closures by Eskin-McMullen-Mukamel-Wright
in~\cite{emmw}. Using our moduli space $\LMS$ as a black box, in~\cite{CW} the Mirzakhani-Wright
formula for the tangent space to the boundary of an orbit closure in the WYSIWYG
space is generalized to the case of multi-component surfaces. Using the details
of the construction of $\LMS$, it is further shown by Benirschke~\cite{benirschke}
that the boundary of any orbit closure in $\LMS$ is again given by linear equations
in generalized period coordinates of the boundary. Besides $\SLtwoR$-orbits, taking the closures of other moduli spaces
in $\LMS$ such as certain Hurwitz spaces (e.g., the double ramification loci) can provide nice boundary structures, see~\cite{bdg, BeniDRC}.
\par
Moreover, it is also important yet challenging to explore dynamical invariants
associated to orbit closures, such as saddle connections and related counting problems.
The space $\LMS$ is also used as a key tool by Dozier~\cite{DozierSaddle} to prove
regularity for the $\SLtwoR$-invariant measure of the set of translation surfaces
with multiple short saddle connections in the strata; his results for example
immediately imply the finiteness of volume of strata of $k$-differentials first
proven by Nguyen~\cite{Nguyen}.
\par
As a further evidence of potential applications towards the study of orbit closures,
we show that our compactification provides a natural bordification of
$\Omega\calM_{g,n}(\mu)$ on which the action of $\GLtwoRplus$ extends.
\par
\begin{thm}
There exists an orbifold with corners ${\RBLMS}$ containing $\omoduli[g,n](\mu)$
as open and dense subspace with the following properties.
\begin{enumerate}
\item There is a continuous map ${\RBLMS} \to \LMS$ whose fiber over a \msd with
$N$~levels below zero is isomorphic to the real torus $(S^1)^{N}$.
\item ${\RBLMS}$ admits an $\RR_{>0}$-action, and the quotient
${\RBLMS}/\RR_{>0}$ is compact.
\item The action of $\GLtwoRplus$ extends continuously to ${\RBLMS}$.
\item Points in ${\RBLMS}$ are in bijection with {\em real \msds}.
\end{enumerate}
\end{thm}
\par
These real \msds are similar to \msds, with a coarser equivalence relation,
see Definition~\ref{def:one_rsd}. This bordification will be constructed in Section~\ref{sec:SL2R}
as a special case of our construction of \lw real blowup. This blowup is an instance of the
classical real oriented blowup construction,
where the blowup is triggered by the level structure underlying a family of \msds,
see Section~\ref{sec:rob}.
\par
We hope that the study of orbit closures in ${\RBLMS}$ will provide new
insights on the classification problem.
\par
\medskip
\paragraph{\bf New notions and techniques}  We next give intuitive explanations of
the new objects and techniques used to construct the moduli space of \msds.
\smallskip
\par
{\em \Prmas.} This is simply a choice of how to match the horizontal directions at the pole of the
differential at one preimage of a node to the horizontal directions at the zero of the differential
at the other preimage of the same node. To motivate that recording this data is necessary to
construct a space dominating the normalization of the IVC, consider two differentials in normal
form, locally given by $\eta_1=u^\kappa (du/u)$ and $\eta_2=C\cdot v^{-\kappa}(dv/v)$ in local
coordinates $u,v$ around two preimages of a node given by~$uv=0$ of~$X$, where $C\in\CC^*$
is some constant. Then plumbing these differentials on the plumbing fixture~$uv=t$ is possible
if and only if $\eta_1 =\eta_2$ after the change of coordinates $v=t/u$, which is equivalent
to $t^\kappa=-C$.  Thus for a given $C$ the different choices of $t$ differ by multiplication
by $\kappa$-th~roots of unity, and the \prma is used to record this ambiguity in the limit
of a degenerating family. The notion of a \prma  will be introduced formally in Section~\ref{sec:BSProng}.
\par
As the above motivation already indicates, this requires locally choosing coordinates such
that the differential has the normal form in these coordinates. Pointwise, this is a classical
result of Strebel. These normal forms for a family of differentials are a technical underpinning
of much of the current paper. The relevant analytic results are proven in Section~\ref{sec:NF},
by solving the suitable differential equations and applying the Implicit Function Theorem in the
suitable Banach space.
\smallskip
\par
{\em Level rotation torus.} This algebraic torus has one~$\CC^*$
factor for each level below zero.
Its action makes the intuition of rescaling level by level precise.
As indicated above, differentials on lower level in degenerating families are obtained
by rescaling by a power of~$t$.  As such a family can also be reparameterized by
multiplying~$t$ by a constant, such a scaled limit on a given component is only
well-defined up to multiplication by a non-zero complex number. Suppose now that while
keeping \changed{the} differential at one side of the node fixed, we start multiplying the differential
on the other side by~$e^{i\theta}$. If we start
with a given prong-matching, which is just some fixed choice of $(-C)^{1/\kappa}$,
this choice of the root is then being multiplied by $e^{i\theta/\kappa}$.
Consequently, varying $\theta$ from~$0$ to~$2\pi$ ends up with the same differential,
but with a different prong-matching.
\par
Thus the equivalence relation among \msds that we consider records simultaneously
all possible rescalings of the differentials on the levels of the stable curve and
the action on the \prmas. This leads to the notion of the level rotation
torus~$\Tprong[\Gamma]$, which will be defined as a finite cover of $(\cx^*)^N$
in Section~\ref{sec:twrot}. See in particular~\eqref{eq:actionLRT} for its action on \msds.
\smallskip
\par
{\em Twist groups and the singularities at the boundary.}   The twist group
$\Tw[\Gamma]$ can be considered as the subgroup of $\Tprong[\Gamma]$ fixing
all prongs under the rotation action. The rank of this group equals the number~$N$
of levels below zero, but the decomposition into levels does in general not
induce an isomorphism of the twist group with $\ZZ^N$. Instead, there is a
subgroup $\sTw[\Gamma]\subset\Tw[\Gamma]$ of finite index that is generated by rotations of
one level at a time. We comment below in connection with the model domain why
this subgroup naturally appears from the toroidal aspects of our compactification.
The quotient group $K_\Gamma = \Tw[\Gamma]/\sTw[\Gamma]$ is responsible for
the orbifold structure at the boundary of our compactification. These groups
are of course always abelian. Our running example in Section~\ref{sec:runningex},
a triangle graph, provides a simple instance where this group~$K_\Gamma$ is
non-trivial (see Section~\ref{sec:coverGH}).
\par
\medskip
\paragraph{\bf From the \Teichmuller space down to the moduli space}
Even though the result of our construction is an algebraic moduli space, our construction
of $\LMS$ starts via \Teichmuller theory and produces intermediate results relevant
for the geometry of moduli spaces of marked meromorphic differentials.
\par
To give the context, recall that recently Hubbard-Koch~\cite{HubKoch} completed a program
of Bers to provide the quotient of Abikoff's augmented \Teichmuller space~$\oldaugteich$
by the mapping class group with a complex structure such that this quotient is
isomorphic to the Deligne-Mumford compactification $\overline\calM_{g,n}$. As an
intermediate step they also provided, for any multicurve~$\Lambda$, the
classical Dehn space $D_\Lambda$ (which Bers in~\cite{bers} called ``deformation space"),
the quotient of $\oldaugteich$ by Dehn twists along~$\Lambda$, with a complex structure.
Our proof proceeds along similar lines, taking care at each step of the extra challenges
due to the degenerating differential.
\par
As a first step, recall that there are several natural topologies on $\barmoduli[g,n]$.
One can define the {\em conformal topology} where roughly a sequence~$X_n$ of pointed
curves converges to~$X$ if there exist diffeomorphisms~$g_n \colon X \to X_n$ that
are conformal on compact subsets that exhaust the complement of nodes and punctures. In the
{\em quasiconformal topology} one relaxes form conformal
to quasiconformal, but requires that the quasiconformal dilatation tends to zero.
Conformal maps are convenient, since they pull back holomorphic differentials to holomorphic
differentials. On the other hand, quasiconformal maps are easier to glue when a surface is
constructed from several subsurfaces. We therefore need both topologies, see
Section~\ref{sec:classaugteich} for precise definitions. The following is an abridged
version of Theorem~\ref{thm:topologiesarethesame}, which was announced
in \cite{MardenString} and \cite{earlemarden}.
\par
\begin{thm}
If $n\geq 1$, then the conformal and quasiconformal topologies on the augmented
\Teichmuller space $\oldaugteich$ are equivalent.
\end{thm}
\par
We upgrade this result in Section~\ref{sec:Space1} to provide also the universal bundle
of one-forms with conformal and quasiconformal topologies that coincides with the usual vector bundle
topology.
\par
\medskip
An outline for the construction of $\LMS$ is then the following. We start with a construction
of the \emph{augmented \Teichmuller space} $\Oaugteich$ of {\em flat surfaces of
type~$\mu$}. As a set, this is the union over all multicurves~$\Lambda$ of the moduli
spaces $\ptwT$ of marked \ptwds as defined in Section~\ref{sec:ptwds}. This mimics the
classical case,
with marked \ptwds taking the role of $\Lambda$-marked  stable curves. We then provide
$\Oaugteich$ with a topology that makes it a Hausdorff topological space
in Theorem~\ref{thm:PaugHausdorff}. For each
multicurve~$\Lambda$, the subspace of~$\Oaugteich$ of strata less degenerate than~$\Lambda$
admits an action of the twist group $\Tw[\Lambda]$ and the quotient is the
\emph{Dehn space~$\ODehn$}. Providing~$\ODehn$ with a complex structure is the goal
of the lengthy plumbing construction in Section~\ref{sec:Dehn}. As a topological space,
$\LMS$ is the quotient of the augmented \Teichmuller space $\Oaugteich$ by the action of
the mapping class group, and its structure as complex orbifold stems from its covering by
the images of Dehn spaces for all $\Lambda$.
\par
We next elaborate on two technical concepts in this construction.
\smallskip
\par
{\em Welded surfaces and asymptotically \tnp maps.}
In order to provide the augmented \Teichmuller space~$\Oaugteich$ with a topology,
we roughly declare a sequence $(X_n,\omega_{n,(i)})$ to be convergent if the curves
converge in the conformal
topology, there exist rescaling parameters $c_{n,(i)}$ such that the rescaled differentials
pull back to nearly the limit differential, and such that the rescaling parameters reflect
the relative sizes determined by the level graph of the limit nodal curve.
\par
To get a Hausdorff topological space one has to rule out unbounded twisting of the
diffeomorphism near the developing node in a degenerating family. The literature contains
formulations in the conformal topology that are not convincing and notions based
on Fenchel-Nielsen coordinates (see e.g.\ \cite[Section 15.8]{acgh}) that do not work
well conformally. Our solution is the following.
\par
Take a nodal curve with a \twd and perform a real blowup of the nodes, i.e.\ replacing each
preimage of each node with an~$S^1$. A prong-matching uniquely determines a way to identify
the boundary circles at the two preimages of each node to form a seam, thereby obtaining
a smooth {\em welded surface}. On such a surface we have a notion of turning number
of arcs, which needs to be
done with care, to consider only a finite collection of arcs; the details are given in Section~\ref{sec:turning}.
\smallskip
\par
{\em Level-wise real oriented blowups.} Usually, in the definition of the classical Dehn
space~$D_\Lambda $, markings are considered isomorphic if they agree up to twists along~$\Lambda$.
Such a definition however loses their interaction with the marking. We are forced to mark
the welded surfaces instead. This, in turn, is not possible over the base~$B$ of a family,
since the welded surface depends on the choice of the \twd in its $T_\Gamma$-orbit. As a
consequence we define a functorial construction of a \lw real oriented blowup
$\widehat{B} \to B$ and define markings using the pullback of the family to~$\widehat{B}$,
see Section~\ref{sec:rob}. Our construction is similar in spirit to several blowup
constructions in the literature, e.g.\ the Kato-Nakayama blowup~\cite{KatoNakayama}.
\par
\medskip
\paragraph{\bf The model domain and toroidal aspects of the compactification}
The augmented \Teichmuller space of flat surfaces parameterizes (marked) \msds,
and in particular admits families in which the underlying Riemann surfaces can
degenerate. In contrast, the model domain only parameterizes equisingular families,
where the topology of the underlying (nodal) Riemann surfaces remains constant, and only
the scaling of the differential on the components varies, while remaining non-zero.
Families of such objects are called model differentials, which serve as auxiliary
objects for our construction. The open model domain~$\MD$ is a finite cover
of a suitable product of (quotients of) \Teichmuller spaces and thus automatically comes
with a complex structure and a universal family.
\par
We define a toroidal compactification $\barMD$ of $\MD$ roughly by allowing the rescalings to
attain the value zero. The actual definition given in Section~\ref{sec:modeldomain} is
not as simple as locally embedding $(\Delta^*)^N \hookrightarrow \Delta^N$, but rather
a quotient of this embedding by the group~$K_\Gamma$ defined above.
As a result, $\barMD$ is a smooth orbifold and the underlying singular space
is a fine moduli space for families of model differentials, called the model domain.
\par
\medskip
\paragraph{\bf The plumbing construction and perturbed period coordinates}
\changed{Recall that the Dehn space $\ODehn$ is the quotient of the subspace of $\Oaugteich$ of strata less degenerate than~$\Lambda$
by the action of the twist group $\Tw[\Lambda]$.} 
We use the model domain to induce a complex structure on~$\ODehn$.
In order to do this, we define a plumbing construction that starts with a family of model
differentials and construct a family of \msds.
The point of this construction is that starting with an equisingular family of curves
with variable scales for differentials, which may in particular be zero, the plumbing
constructs a family of curves of variable topology with a family of non-zero differentials
on the smooth fibers. Whenever the scale is non-zero, the plumbing ``plumbs'' the node,
i.e.~smoothes it in a controlled way. The goal of our elaborate plumbing construction is
to establish the local homeomorphism of the Dehn space with the model domain. As in~\cite{strata},
to be able to plumb one needs to match the residues of the differentials at the two preimages
of every node, and thus in particular one needs to add a small modification differential. We
then argue that the resulting map will still be a homeomorphism of moduli spaces, and to
this end we use the perturbed period coordinates introduced in Section~\ref{sec:modif}.
\par
{\em Perturbed period coordinates} are coordinates at the boundary of our compactification.
They consist of periods of the twisted differential, parameters for the level-wise rescaling,
and a classical additional plumbing parameter for each node joining components on the
same level. These periods are close, but not actually equal, to the periods of the
plumbed differential, whence the name. See~\eqref{eq:defPPer} for the precise amount
of perturbation.
\par
Finally, in Section~\ref{sec:Dehn} we complete this setup and define the plumbing map in
full generality and prove in Theorem~\ref{thm:plumbing} that plumbing is a local
diffeomorphism. This is used in Theorem~\ref{thm:DehnIsOrbi} to show that~$\ODehn$ is a
complex orbifold.
\par
\medskip
\paragraph{\bf Families and the universal properties}
The functorial viewpoint and the proof of Theorem~\ref{intro:main}~(4) rely
on showing first in Theorem~\ref{thm:univnonsimpleMMS} that the Dehn space~$\ODehn$
is a fine moduli space for a functor of marked multi-scale differentials. In order to
provide families of multi-scale differentials with a Teichm\"uller marking, we need
to deal with the problem that the equivalence relation in Definition~\ref{def:one_msd}
will twist the marking around the vertical nodes. To counterbalance
this, the marking will not be defined on the original family, but on the family
of welded surfaces over a real oriented blowup of the base, where the blowup
structure is triggered by the level structure. The existence of the appropriate
\emph{level-wise real blowup} is proven in Theorem~\ref{thm:ORBL}.
\par
The proof of the universal property of~$\ODehn$ uses the (obvious) universal property
of the model domain. To make use of it,  we introduce an unplumbing construction
that is roughly an inverse of the plumbing construction. This  unplumbing construction
is based on the normal form Theorem~\ref{thm:NF} for families of differentials on
plumbing fixtures to pinch off the node and get an equisingular family, after
subtracting differentials that play the inverse role of the modification differentials
above. Finally, Theorem~\ref{thm:coarseMS} completes the program by taking
a family of multi-scale differentials, providing it locally with a marking,
and gluing the moduli maps that universal property of~$\ODehn$ gives.
\par
\medskip
\paragraph{\bf Algebraicity, families, and the orderly blowup construction}
To prove the algebraicity in the main theorem and to prove the precise relation
of $\LMS$ to the IVC we need to encounter some of the details of families of \msds.
First, since $\LMS$ is normal, the forgetful map factors through the normalization
of the IVC. This corresponds to memorizing the extra datum of enhancement of the
dual graph and the prong-matching.
Second, a family of \msds admits level-by-level rescaling, while twisted differentials
a priori do not. While for twisted differentials there exists a rescaling
parameter for each irreducible component, they might be mutually incomparable
or, as we say, disorderly. We thus design in Section~\ref{subsec:blowup} locally a blowup,
the {\em orderly blowup} of the base of a family such that the rescaling parameters
can be put in order, i.e.\ a divisibility relation according to the level structure.
However, the resulting blowup is in general not even normal. The third step is
thus geometrically the normalization of the resulting space. In families of \msds
this is reflected by including the notion of a {\em rescaling ensemble} given in
Definition~\ref{def:RescEns}. It ultimately reflects the normality of the
toroidal compactifications by $\Delta^N/K_\Gamma$ used above.
This procedure, culminating in Theorem~\ref{thm:LMS-blowup}, is summarized as follows.
\par
\begin{thm} \label{intro:orderly}
  The moduli stack of projectivized \msds  $\PP\MSgrp$ is the normalization
of the orderly blowup of the normalization of the IVC.
\end{thm}
\par
Algebraicity and the remaining properties of the main theorems above follow from this
result. The zoo of notations is summarized in a table at the end of the paper.

\subsection*{Acknowledgments}

We are very grateful to the American Institute of Mathematics (AIM)
for supporting our research as SQuaRE meetings in 2017--2019, where we
made much progress on this project. We are also grateful to Institut
f\"ur Algebraische Geometrie of the Leibniz Universit\"at Hannover,
Mathematisches Forschungsinstitut Oberwolfach (MFO),  Max Planck
Institut f\"ur Mathematik (MPIM, Bonn), Casa Matem\'atica Oaxaca
(CMO), and the Mathematical Sciences Research Institute (MSRI,
Berkeley) as well as the organizers of various workshops there, where
various subsets of us met and collaborated on this project. We
thank Sebastian Casalaina-Martin, Qile Chen, Matteo Costantini, David Holmes,
John Smillie, Jakob Stix, Michael Temkin, Ilya Tyomkin, Scott Wolpert, Alex
Wright, Jonathan Zachhuber, and Anton Zorich for inspiring discussions
on differentials, log structures, moduli, algebraicity, and stacks. \changed{Finally we thank the anonymous referee for very carefully reading the paper and copious insightful comments and corrections.}

%%%%%%%%%%%%%%%%%%%%%%%%%%%%%%%%%%%%
\section{Notation and background} \label{sec:notation}
%%%%%%%%%%%%%%%%%%%%%%%%%%%%%%%%%%%%%

The purpose of this section is to recall notation and the main result
from \cite{strata}. Along the way we introduce the notion of
{\em enhanced level graphs} that records the extra data of orders of
zeros and poles that compatible twisted differentials should have.
\par
\medskip
%%%%%%%%%%%%%%%%%%%%
\subsection{Flat surfaces with marked points and their strata.}
\label{sec:flatsurfstrata}
%%%%%%%%%%%%%%%%%%%%
A {\em type} of a (possibly meromorphic) abelian differential
on a Riemann surface is a tuple of integers $\mu=(m_1,\dots,m_n)\in\ZZ^n$
with $m_j\ge m_{j+1}$, such that $\sum_{j=1}^n m_j=2g-2$. We assume that there are
$r$ positive $m$'s, $s$ zeroes, and $l$ negative $m$'s, with $r+s+l=n$,
i.e.,~that we have
$m_1\ge \dots\ge m_r>m_{r+1}=\dots=m_{r+s} = 0>m_{r+s+1}\ge\dots \ge m_n$.
Note that $m_j=0$ is allowed, representing an ordinary marked point.
We use the abbreviation $\on = \{1,\dots,n\}$.
\index[other]{$\on=\{1,\dots,n\}$}
\par
A {\em (pointed) flat surface} or equivalently a {\em (pointed) abelian
differential} is a triple $(X,  \bfz, \omega)$, where~$X$ is
a (smooth and connected) compact  genus~$g$ Riemann surface,~$\omega$ is a non-zero
meromorphic one-form on~$X$, and $\bfz\colon \on \hookrightarrow X$ is an
injective function such that $\ord_{\bfz(j)}\omega = m_j$ for each $j$, and moreover every
zero or pole of $\omega$ is marked by some $\bfz(j)$. We also
denote by $z_j$ the marked point $\bfz(j)$.
\par
The rank $g$ Hodge bundle of holomorphic (stable) differentials on~$n$-pointed
stable genus $g$ curves, denoted by $\obarmoduli[g,n]\to\barmoduli[g,n]$, is the total
space of the relative dualizing sheaf $\pi_* \omega_{\calX/\barmoduli[g,n]}$, where
$\pi\colon \calX \to \barmoduli[g,n]$ is the universal curve. We denote the polar part
of $\mu$ by $\tilde\mu=(m_{r+s+1},\dots,m_n)$. We then define the {\em (pointed) Hodge
bundle twisted by $\tilde{\mu}$} to be the bundle
$$K\barmoduli[g,n](\tilde{\mu}) \= \pi_* \Bigl(\omega_{\calX/\barmoduli[g,n]}
\Bigl(-\sum_{j=r+s+1}^n m_j \calZ_j\Bigr)\Bigr)$$
over $\barmoduli[g,n]$, where we have denoted by $\calZ_j$ the image of the
section~$z_j$ of the universal family $\pi$ given by the~$j$-th marked point.
\index[family]{b002@$\calZ_j$!  Image of the section~$z_j$}
The formal sums
\begin{equation}\label{eq:hordivZ}
\calZ^0 \= \sum_{j=1}^{r+s} m_j \calZ_j  \quad \text{and} \quad \calZ^\infty
\= \sum_{j=r+s+1}^n m_j \calZ_j
\end{equation}
are called the {\em (prescribed) horizontal zero divisor} and {\em (prescribed)
horizontal polar divisor} respectively.
\index[family]{b004@$\calZ^0$! Horizontal zero divisor}
\index[family]{b008@$\calZ^\infty$! Horizontal polar divisor}
\par
The {\em moduli space of abelian differentials of type $\mu$} is denoted (still)
by $\omodulin(\mu)\subseteq K\barmoduli[g,n](\tilde\mu)$, and consists
of those pointed flat surfaces where the divisor of~$\omega$ is
equal to $\sum_{j=1}^n m_j z_j$. We denote by adding $\PP$ to the Hodge bundle
(resp.\ to the strata) the projectivization, i.e., when we want to parameterize differentials
up to scale. The {\em (ordered) incidence variety compactification}
(IVC for short) is then defined to be the closure
$\PP\obarmoduliinc{\mu}$ inside $\PP K\barmoduli[g,n](\tilde{\mu})$ of the
(projectivized) moduli space of abelian differentials of type $\mu$. A point
$(X,\bfz,\omega)\in\PP\obarmodulin(\tilde\mu)$ is called a
{\em pointed stable differential.} The main result of \cite{strata}
is to precisely describe this closure, as we recall below.
\par

%%%%%%%%%%%%%%%%%%%%
\subsection{Removing the labeling by the ${\rm Sym}(\mu)$-action}
\label{sec:symnu}
%%%%%%%%%%%%%%%%%%%%
We emphasize again that throughout the paper and in particular in the moduli space~$\LMS$
in our main theorem the points are labeled. We let ${\rm Sym}(\mu) \subset S_n$ be the
subgroup of permutations that permutes only points with the same prescribed order~$m_i$.
This group acts on the moduli space $\omoduli[g,n](\mu)$ with quotient
the moduli space $\omoduli[g](\mu)$, which gives the usual strata of the Hodge
bundle if $\mu$ is strictly positive. The reader is invited to
check along the whole paper that ${\rm Sym}(\mu)$ acts everywhere and in
particular on $\LMS$. The quotient $\LMS/{\rm Sym}(\mu)$ is a compactification
of $\omoduli[g](\mu)$.
\par

%%%%%%%%%%%%%%%%%%%%
\subsection{Graphs, level graphs and ordered stable curves.}
\label{sec:graphetc}
%%%%%%%%%%%%%%%%%%%%
Throughout this paper $\aG$ will be a graph, which is allowed to
have loops as well as half edges, and connected unless
explicitly stated otherwise. We denote the set of vertices by $V(\aG)$, the set of edges by
$E(\aG)$, and the set of half-edges by $H(\aG)$.
\index[graph]{b013@$V(\aG)$ ! Vertices of $\aG$}
\index[graph]{b014@$E(\aG)$ ! Edges of $\aG$}~We denote by $\val(v)$ the {\em valence} of a vertex $v \in V$,
\index[graph]{b016@$\val(v)$ ! Valence of the vertex $v$} the number of ordinary edges incident to $v$ (a self-loop is counted twice).
\par
A {\em (weak) full order} $\succcurlyeq$ on the graph $\aG$ is an order $\succcurlyeq$
on the set of vertices $V(\aG)$ that is reflexive, transitive,
and such that for any $v_1, v_2 \in V$ at least one of the
statements $v_1 \succcurlyeq v_2$ or $v_2 \succcurlyeq v_1$ holds. The pair $\lGp=(\aG,\succcurlyeq)$ is called a {\em level graph}.
\index[graph]{b010@$\lGp=(\Gamma,\succcurlyeq)$, $\lG$!Level graph with full order $\succcurlyeq$}
In what follows it will be convenient to assume that the full order on $\lGp$ is induced by a {\em level function} $\ell\colon V(\aG)\to\ZZ_{\le 0}$ such that the vertices of top level are elements of the set $\ell^{-1}(0)\ne\emptyset$, and the comparison between vertices is by comparing their $\ell$-values. Any full order can be induced by a level function, but not by a unique one. We thus use the words level graph and a full order on a graph interchangeably. We let
$L^\bullet(\lGp) = \lbrace a\in\ZZ:\ell^{-1}(a)\ne\emptyset\rbrace$ be the
set of levels and let $L(\lGp)$ be the set of all but the top level.
\index[other]{$\uN \= \{0,-1,\dots,-N\}$}
\index[graph]{b020@$L^\bullet(\lGp) $!Set of levels of the level graph $\lGp$}
\index[graph]{b030@$L(\lGp) $!Set of all but the top level of the level graph $\lGp$}
We usually use the {\em normalized level function}
\index[graph]{b045@$\ell \colon \aG \to \uN$!Normalized level function}
\begin{equation}\label{eq:normlev}
\ell\colon \aG \to \uN \= \{0,-1,\dots,-N\}\,,
\end{equation}
where $N = |L^\bullet(\lGp)|-1=|L(\lGp)|\in\ZZ_{\ge 0}$ is the number of levels strictly below $0$.
\index[graph]{b040@$N $! Number of levels strictly below $0$}
\par
For a given level~$i$ we call the subgraph of $\lGp$ that consists of all
vertices $v$ with $\ell(v) > i$, along with edges between them, the
{\em graph above level~$i$} of $\lGp$, and denote it by~$\lGp_{>i}$. We similarly define the graph $\lGp_{\geq i}$
{\em above or at level~$i$}, and use $\lGp_{(i)}$
\index[graph]{b050@$\lGp_{(i)}$, $\lGp_{>i}$!Subgraph at (resp. above) level~$i$ of $\lGp$}
to denote the graph {\em at level~$i$}.  Note that these graphs are usually disconnected.
\par
If $\dG$ is the dual graph of a stable curve $X$ with pointed differential of
type~$\mu$, \changed{we write $X_v$ for the irreducible component of~$X$ associated to a vertex~$v$. 
The dual graph $\dG$ of a pointed stable curve~$(X,\bfz)$ has half-edges, where a half-edge at a vertex $v$ corresponds to a marked point $z_i\in X_v$. We denote by $\mu_v$ for $v \in V(\Gamma)$ the vector recording the orders $m_i$'s at the marked points on the component~$X_v$}.
We also let  $n_v = \val(v) + |\mu_v|$ be the total number of special points (marked points and nodes) of
a component~$X_v$ of a stable curve.
\par
\begin{df}\label{def:graphs}
An edge $e\in E(\lGp)$ of a level graph $\lGp$ is called {\em horizontal} if
it connects two vertices of the same level, and is called {\em vertical}
otherwise. Given a vertical edge~$e$, we denote by $\topvert$ (resp.~$\botvert$)
the vertex that is its endpoint of higher (resp.~lower) order.
\end{df}
\par
We denote the sets of vertical and horizontal edges by~$\vertedge[\lGp]$
and by~$\horedge[\lGp]$ respectively.
\index[graph]{b015@$\vertedge$, $\horedge$! Set of vertical, resp. horizontal, edges of $\lG$}
Implicit in this terminology is our convention
that we draw level graphs so that the map~$\ell$ is given by the projection
to the vertical axis.
\par
We call a stable curve~$X$ equipped with a full order on its dual graph $\dG$
an \emph{ordered stable curve}. We will write 
%$X_v$ for the irreducible component of~$X$ associated to a vertex~$v$, and 
$X_{(i)}$ for the (possibly disconnected) union of the irreducible
components~$X_v$ such that $\ell(v)=i$.
\index[surf]{b070@$X_{(i)}$!Components of~$X$ at level~$i$}
We write~$q_e$ for the node associated to
an edge~$e$.  We call such a node \emph{vertical} or \emph{horizontal} accordingly.
The set of nodes of~$X$ is denoted by $\N$, the set of vertical nodes by $\Nver$ and
the set of horizontal nodes by~$\Nhor$.
\index[surf]{b025@$\N$!Set of nodes of~$X$}
\index[surf]{b027@$\Nver$!Set of vertical nodes of~$X$}
\index[surf]{b028@$\Nhor$!Set of horizontal nodes of~$X$}
\par
For a vertical node $q_{e}$ of~$X$ corresponding to an edge $e$ we write $q_{e}^+\in X_{(\ltop)}$ and $q_{e}^-\in X_{(\lbot)}$ for the
two points lying above~$q_{e}$ in the normalization, and for the irreducible components
in which they lie, ordered so that~$X_{(\lbot)} \prec X_{(\ltop)}$. Moreover we denote the levels of $q_{e}^{\pm}$ by  $\ltopbot$, respectively.
\index[graph]{b060@$\lqtopbot$, $\ltopbot$!Bottom and top levels of the ends of a node}
We use the same notation for horizontal nodes, making an arbitrary choice
of label $\pm$.
\par
\medskip
%%%%%%%%%%%%%%%%%%%%%
\subsection{Twisted differentials and the IVC}
\label{sec:deftwd}
%%%%%%%%%%%%%%%%%%%%%
Recall from  \cite{strata} that a {\em \twd~$\eta$ of type $\mu$}
on a stable~$n$-pointed curve $(X,\bfz)$ is a collection of
\index[surf]{b020@$(X,\bfz)$!Pointed stable curve of genus $g$}
\index[surf]{b025@$X_v$!Irreducible component of~$X$}
(possibly meromorphic) differentials $\eta_v$ on the irreducible components~$X_v$ of~$X$,
such that no $\eta_v$ is identically zero, with the following properties:
\begin{itemize}
\item[(0)] {\bf (Vanishing as prescribed)} Each differential $\eta_v$ is
holomorphic and non-zero outside of the nodes and marked points of~$X_v$.
Moreover, if a marked point $z_j$ lies on~$X_v$, then $\ord_{z_j} \eta_v=m_j$.
\item[(1)] {\bf (Matching orders)} For any node of~$X$ that identifies
$q_1 \in X_{v_1}$ with $q_2 \in X_{v_2}$,
$$\ord_{q_1} \eta_{v_1}+\ord_{q_2} \eta_{v_2}\=-2\,. $$
\item[(2)] {\bf (Matching residues at simple poles)}  If at a node of~$X$
that identifies $q_1 \in X_{v_1}$ with $q_2 \in X_{v_2}$ the condition $\ord_{q_1}\eta_{v_1}=
\ord_{q_2} \eta_{v_2}=-1$ holds, then $\Res_{q_1}\eta_{v_1}+\Res_{q_2}\eta_{v_2}=0$.
\end{itemize}
Let $\lGp=(\dG,\succcurlyeq)$ be a level graph where $\dG$ is the dual graph of~$X$.
A \twd~$\eta$ of type $\mu$ on~$X$
is called {\em compatible with~$\lGp$} if in addition it also
satisfies the following two conditions:
\begin{itemize}
\item[(3)]{\bf (Partial order)} If  a node of~$X$  identifies
$q_1 \in X_{v_1}$ with $q_2 \in X_{v_2}$, then $v_1\succcurlyeq  v_2$ if and
only if $\ord_{q_1} \eta_{v_1}\ge -1$. Moreover,  $v_1\asymp v_2$ if and only if
$\ord_{q_1} \eta_{v_1} = -1$.
\end{itemize}
We remark that this condition only uses the partial order induced by
$\lGp$ on the vertices that are connected by an edge, while the most subtle
condition, which uses the full order, is the following.
\begin{itemize}
\item[(4)] {\bf (Global residue condition)} For every level~$i$
and every connected component~$Y$ of $X_{>i}$ that does not
contain a marked point with a prescribed pole (i.e., there is no $z_i\in Y$
with $m_i<0$) the following condition holds.
\index[surf]{b080@$X_{>i}$!Components of~$X$ at level $>i$}
Let $q_1,\dots,q_b$ denote the set of all nodes where~$Y$ intersects $X_{(i)}$. Then
$$ \sum_{j=1}^b\Res_{q_j^-}\eta \=0\,,$$
where by definition $q_j^-\in X_{(i)}$.
\end{itemize}
\par
For brevity, we write GRC for the global residue condition. We denote
a \twd compatible with a level $\cleq$ by $(X, \bfz, \eta, \cleq)$. Moreover, we will usually
group the restrictions of the twisted differential~$\eta$ according to the
levels of~$\ell$. We will denote the restriction of~$\eta$ to the subsurface
$X_{(i)}$ by~$\eta_{(i)}$.
\index[surf]{b075@$\eta_{(i)}$!Restriction of the twisted differential~$\eta$ on $X_{(i)}$}
\par
We have shown in \cite[Theorem~1.5]{strata}:
\begin{thm} A pointed stable differential $(X,\omega, \bfz)$ is contained in the
incidence variety compactification of $\proj\omodulin(\mu)$ if and only if the
following conditions hold:
\begin{itemize}
\item[(i)] There exists an order $\succcurlyeq$ on the dual graph $\Gamma_X$ of~$X$ such that its maxima are
the irreducible components $X_v$ of~$X$ on which~$\omega$ is not identically zero.
\item[(ii)] There exists a twisted differential~$\eta$ of type $\mu$ on~$X$,
compatible with the level graph $\lGp=(\Gamma_X,\succcurlyeq)$.
\item[(iii)] On every irreducible component $X_v$ where~$\omega$ is not
identically zero, $\eta_{v} = \omega|_{X_v}$.
\end{itemize}
\end{thm}
\par
%%%%%%%%%%%%%%%%%%%%%%%
\subsection{Enhanced level graphs} \label{sec:enh}
%%%%%%%%%%%%%%%%%%%%%%%
Note that a boundary point of the IVC does not necessarily determine a \twd uniquely,
see~\cite[Examples~3.4 and~3.5]{strata}.
The full combinatorics of a \twd is encoded by the following notion.
\par
An \emph{enhanced level graph}~$\eGp$ of type $\mu = (m_1, \dots, m_n)$
\index[graph]{b070@$\eGp$, $\eG$!Enhanced level graph}
is a level graph $\lGp$ together with a numbering of the  half-edges by $\on$ and
with an assignment of a positive number $\kappa_e \in \nats$ for each
vertical edge $e\in\vertedge[\lGp]$. The  \emph{degree} of a vertex $v$ in $\eGp$ is defined to be
  \begin{equation*}
    \deg(v) \= \sum_{j\mapsto v} m_j + \sum_{e \in E^+(v)}  (\kappa_e-1)
- \sum_{e \in E^-(v)} (\kappa_e+1) \,,
  \end{equation*}
where the first sum is over all half-edges incident to~$v$, and
the remaining sums are over the edges~$E^+(v)$ and $E^-(v)$
incident to~$v$ that are going  from $v$ to a respectively lower and upper vertex.
In terms of the notation in Definition~\ref{def:graphs}, the set $E^+(v)$ is the set of edges $\{e\in E(\lGp):\topvert=v\}$.
We require that
  \begin{itemize}
    \item[(i)] {\bf (Admissible degrees)} the degree of each vertex is even
and at least $-2$, and
  \item[(ii)] {\bf (Stability)} the valence of each vertex of degree $-2$ is at least three.
  \end{itemize}
  \par
Our notion of enhancement is equivalent to the notion of twist used e.g.\
in \cite{fapa} or \cite{CMSZ}. The main example is the enhanced level graph~$\eGp_X$ of a
\twd $(X, \bfz, \eta)$, obtained by assigning to each vertical node $q$ the
weight
\begin{equation}\label{eq:kappa}
  \kappa_q = \ord_{q^+} \eta +1\,.
\end{equation}
\index[twist]{b005@$\kappa_q$!Number of prongs, equal to $\ord_{q^+} \eta +1$}In these terms, the above stability condition is equivalent to stability
of $(X, \bfz)$.  The degree of
a vertex~$v$ is the degree of~$\eta_v$. The admissible degrees condition ensures that such
a $\eGp$ can be realized as the enhanced level graph of some \twd. We
also say that a \twd $(X,\eta)$ is compatible with $\eGp$ if it is
compatible with the underlying level graph $\lGp$ and if the markings of $\eGp$
are the weights of~$\eta$ just defined.
\par
In order to keep notation concise, we will denote by $\eG$ the dual graph $\dG$ of
a curve~$X$, a level graph $\lGp$ and write for an enhanced graph $\eGp$,
or simply $\eG$, as appropriate.
\par
%%%%%%%%%%%%%%%%%%%%%%%
\subsection{The running example} \label{sec:runningex}
%%%%%%%%%%%%%%%%%%%%%%%

In order to illustrate the notions that were introduced, we will describe an example.
This example will be used throughout the text to exemplify the different notions
that we will introduce. We will refer to it as the {\em running example}.
\par
The example is for the moduli space $\Omega\calM_{5,4}(4,4,2,-2)$. We fix the curve whose
dual graph is a triangle, with the level function taking three different values $0,-1,-2$
on it, so that the level graph is fixed. The irreducible components are of genus~$3$
(at top level), genus~$1$ (at the intermediate level) and genus~$0$ (at the bottom level).
This level graph admits two different enhanced structures, which we denote~$\eG_{1}$
and~$\eG_{2}$, as pictured in Figure~\ref{cap:running}.
\begin{figure}[ht]
\begin{tikzpicture}[scale=1.9,decoration={
    markings,
    mark=at position 0.5 with {\arrow[very thick]{>}}}]
%Gauche
\fill (0,0) coordinate (x0) circle (1pt); \node [above] at (x0) {$g(X_{(0)})=3$};
\fill (1,-1) coordinate (x1) circle (1pt); \node [below right] at (x1) {$g(X_{(-1)})=1$};
\fill (-1,-2) coordinate (x2) circle (1pt); \node [below right] at (x2) {$g(X_{(-2)})=0$};

 \draw (x0) -- node[right]{$3$} node[left]{$e_{1}$} (x1);
 \draw (x0) --node[left]{$1$} node[right]{$e_{3}$} (x2);
\draw (x1) -- node[below]{$3$} node[above]{$e_{2}$} (x2);

\draw (x0) -- ++(-20:.1) node[right] {$2$};
\draw (x1) -- ++(30:.1) node[right] {$4$};
\draw (x1) -- ++(90:.1) node[above] {$-2$};
\draw (x2) -- ++(180:.1) node[left] {$4$};

%Droite
\fill (4,0) coordinate (y0) circle (1pt); \node [above] at (y0) {$g(X_{(0)})=3$};
\fill (5,-1) coordinate (y1) circle (1pt); \node [below  right] at (y1) {$g(X_{(-1)})=1$};
\fill (3,-2) coordinate (y2) circle (1pt); \node [below right] at (y2) {$g(X_{(-2)})=0$};

 \draw (y0) --node[right]{$1$} node[left]{$e_{1}$}  (y1);
 \draw (y0) -- node[left]{$3$} node[right]{$e_{3}$}(y2);
\draw  (y1) --node[below]{$1$} node[above]{$e_{2}$} (y2);

\draw (y0) -- ++(-20:.1) node[right] {$2$};
\draw (y1) -- ++(30:.1) node[right] {$4$};
\draw (y1) -- ++(90:.1) node[above] {$-2$};
\draw (y2) -- ++(180:.1) node[left] {$4$};
\end{tikzpicture}
\caption{Two different enhanced orders $\eG_1$ and $\eG_2$ on $\lG$.}
\label{cap:running}
\end{figure}
\par
We denote the twisted differentials compatible with the level
graphs~$\eG_i$ by $(X,\bfz,\eta^{i})$. The enhanced structure
$\Gamma_1$ tells us that the differential $\eta_{(-2)}^{1}$ is in
$\omoduli[0,3](4,-2,-4)$, the differential $\eta_{(-1)}^{1}$ is in
$\omoduli[1,4](4,2,-2,-4)$ and $\eta_{(0)}^{1}$ is in
$\omoduli[3,3](2,2,0)$. Similarly $\eta_{(-2)}^{2}$ is in
$\omoduli[0,3](4,-2,-4)$, $\eta_{(-1)}^{2}$ is in
$\omoduli[1,4](4,0,-2,-2)$ and $\eta_{(0)}^{2}$ is in
$\omoduli[3,3](2,2,0)$. The global residue condition in both cases
says that the differential $\eta_{(-1)}^{i}$ has no residue at its
pole at the point~$q_{e_{1}}^{-}$. Note that it follows from
\cite[Theorem 1.1]{geta} that this locus is not empty.
%%% Local Variables:
%%% mode: latex
%%% TeX-master: "PLUMB2"
%%% End:

%%%%%%%%%%%%%%%%%%%%%%%%%%%%%%%%%%%%
\section{The topology on (classical) augmented \Teichmuller space}
\label{sec:classaugteich}
%%%%%%%%%%%%%%%%%%%%%%%%%%%%%%%%%%%%%

The classical augmented \Teichmuller space contains the \Teichmuller space
as a dense subset such that the action of the mapping class group extends
continuously and such that the quotient by the mapping class group is the
Deligne-Mumford compactification of the moduli space of curves.  In this section
we compare various topologies on the augmented \Teichmuller space, and on
the related spaces of one-forms.

%%%%%%%%%%%%%%%%%%
\subsection{Augmented \Teichmuller space}
\label{sec:augm-teichm-space}
%%%%%%%%%%%%%%%%%

To give the precise definition of the augmented \Teichmuller space, we fix a ``base''
compact~$n$-pointed oriented differentiable surface $(\Sigma, \bfs)$ of genus~$g$.
\index[surf]{b010@$(\Sigma, \bfs)$!``Base'' compact~$n$-pointed oriented differentiable surface}
We regard $\bfs$ as a set of $n\geq 0$ distinct
labeled points $\{s_1, \dots, s_n\}\subset \Sigma$, or alternatively as an injective function $\bfs\colon
\on\hookrightarrow \Sigma$.  Let $\oldteich = \teich[\Sigma][\bfs]$ be the \Teichmuller space of $(\Sigma, \bfs)$.
\index[teich]{b010@$\oldteich = \teich[\Sigma][\bfs]$!\Teichmuller space}
Next, recall that a \emph{multicurve}~$\Lambda$ on~$\Sigma\setminus \bfs$ is a collection of
disjoint simple closed curves such that no two curves are isotopic on
$\Sigma\setminus \bfs$  \changed{and no curve in $\Lambda$ bounds a
  disk with at most one puncture}. Two
multicurves are equivalent if the curves they consist of are pairwise
isotopic, \changed{more precisely, there exists an orientation-preserving
   diffeomorphism which is isotopic to the identity,  preserves
  the marked points, and maps one multicurve bijectively to the
  other.} To ease notation, we will speak of curves of a multicurve~$\Lambda$ both when we mean the actual curves or their isotopy equivalence classes, as should be clear from the context.
\par
\begin{df} \label{def:marking}
A \emph{marked pointed stable curve $(X, \bfz, f)$} is a pointed stable
curve $(X, \bfz)$ together with a marking $f\colon(\Sigma, \bfs)\to (X, \bfz)$,
where a \emph{marking} of a pointed stable curve is a continuous map
\index[surf]{b030@$f\colon  \Sigma \to X$!Marking}
$f\colon  \Sigma \to X$ such that
\begin{itemize}
\item[(i)] the inverse image of every node $q\in X$ is a simple closed curve
on $\Sigma\setminus \bfs$,
\item[(ii)] if we denote by $\Lambda\subset\Sigma$ \changed{the preimage} of the set of nodes $\N$ of~$X$,
which is a multicurve on $\Sigma$ that we call the \emph{pinched multicurve}, then the restriction of $f$ to
$\Sigma\setminus \Lambda$ is an orientation-preserving
diffeomorphism $\Sigma\setminus \Lambda \to X \setminus \N$,
\item[(iii)] the map $f$ preserves the marked points, that is, $f\circ \bfs = \bfz$.
\end{itemize}
\changed{Two marked pointed stable curves $(X,\bfz, f)$ and
  $(X',\bfz', f')$ are isomorphic if there exist an isomorphism of
  pointed stable curves $\phi\colon (X, \bfz) \to (X', \bfz')$ and an
  orientation-preserving diffeomorphism
  $\varphi\colon \Sigma \to \Sigma$ which is isotopic to the identity
  (rel $\bfs$) and such that $f' \circ \varphi = \phi \circ f$.}
\end{df}
\par
\changed{In particular two marked pointed stable curves that differ by the
  action of a Dehn twist around a pinched curve are equivalent.}
\par
\medskip
The \emph{augmented
\Teichmuller space $\oldaugteich[g,n]=\augteich[\Sigma][\bfs]$}
\index[teich]{b030@$\oldaugteich[g,n]=\augteich[\Sigma][\bfs]$!
Augmented \Teichmuller space}
is the set of all equivalence classes of pointed stable curves marked
by $(\Sigma, \bfs)$.  We caution the reader that $\oldaugteich$ is not a manifold, and is not even
locally compact at the boundary, in the topology which we define below.
\par
The {\em mapping class group}
$\Mod$ acts properly discontinuously on the classical \Teichmuller space~$\oldteich$
and this action (by pre-composition of the marking) extends to a continuous\index[teich]{b020@$\Mod$!Classical mapping class group}
action on the augmented \Teichmuller space (whose topology is defined below).
\par
The augmented \Teichmuller space is stratified according to the pinched multicurve.  Given a
multicurve $\Lambda\subset \Sigma \setminus \bfs$, we define
$\Bteich[\Lambda] \subset \oldaugteich[g,n]$ to be the stratum consisting of stable curves
where exactly the curves in~$\Lambda$ have been pinched to nodes.
In particular, the empty multicurve recovers the interior $\Bteich[\emptyset] =\oldteich$.
Each $\Bteich[\Lambda]$ is itself a finite unramified cover of the product
of the \Teichmuller spaces of the components of $(\Sigma, \bfs) \setminus \Lambda$ that
takes into account the identification of the branches of the nodes. In
particular each~$\Bteich[\Lambda]$ is smooth.
\par
The topology on the augmented \Teichmuller space can be described in several ways.
For us, the conformal topology (introduced by \cite{MardenString}, see also
Earle-Marden~\cite{earlemarden})   will be most useful.
Abikoff~\cite{Abikoff} described several equivalent topologies on the augmented \Teichmuller
space.  We recall the definition of his \emph{quasiconformal topology} below
(somewhat confusingly, he called this the conformal topology). The equivalence of
the two topologies is claimed in \cite[Theorem~6.1]{earlemarden}.   We include a complete
proof of this equivalence here. We mention \cite{Mon09} for several other viewpoints of the topology,
mainly based on hyperbolic length functions.
\par
We define an \emph{exhaustion} of a (possibly open) Riemann
surface~$X$ to be an increasing sequence of compact subsurfaces with boundary,
$K_m\subset X$, such that each $K_m$ is a deformation retract of~$X$,
and such that the union $\cup_{m=1}^\infty K_m$ is all of~$X$. An
important example of an exhaustion that is used throughout this
article is the following. \changed{Considering $X\setminus \bfz$ with
  its complete hyperbolic metric}, for any sequence~$\epsilon_{m}$ of
positive numbers (smaller than the Margulis constant) converging to
zero, the $\epsilon_{m}$-thick parts of $X\setminus \bfz$, denoted by
$\Thick[(X \setminus \bfz)][\epsilon_{m}]$, form an exhaustion of
$X\setminus \bfz$.  Note that the fact that the~$\epsilon_{m}$ are
smaller than the Margulis constant ensures that the thin part is a
union of annular neighborhoods of short geodesics or cusps.
\index[surf]{b090@ $\Thick[X][\epsilon]$, $\Thick$! $\epsilon$-thick
  part of~$X$, resp. $X\setminus\bfz$}
\par
Let $(X, \bfz,f)$ be a marked pointed stable curve in
$\augteich[\Sigma][\bfs]$, and let $X^s=X\setminus N_X$ denote the
smooth part of~$X$, \index[surf]{b026@$X^{s}=X\setminus N_X$!The
  smooth part of~$X$} that is, the complement of its nodes. A sequence
of marked pointed stable curves $(X_m, \bfz_m, f_m)$ in
$\oldaugteich[g,n]$ \emph{converges quasiconformally} to
$(X, \bfz, f)$ if for some exhaustion~$\{K_m\}$ of~$X^s$, there exists
a sequence of quasiconformal maps $g_m\colon K_m \to X_m$ such that
for each $m$ the maps $f_m\circ f^{-1}$ and $g_m$ are homotopic on
$K_m$, the map~$g_m$ respects the marked points (i.e.\
$g_m\circ \bfz = \bfz_m$), and the quasiconformal dilatations
$\|\delbar g_m/ \del g_m \|_\infty$ tend to $0$ as $m\to\infty$.  The
sequence \emph{converges conformally} if the $K_m$ are instead an
exhaustion of $X^s \setminus \bfz$ and the~$g_m$ can be taken to be
conformal.  We call the topologies on $\augteich[\Sigma][\bfs]$
induced by these notions of convergence the \emph{quasiconformal
  topology} and the \emph{conformal topology}, respectively.
\par
Note that for conformal convergence it no longer makes sense to require that the~$g_m$ respect the
marked points, since they are not in the domain. However, each marked point of~$X$ is contained in a
unique connected component of $X^s \setminus K_m$, and the hypothesis that $f_m\circ f^{-1} \simeq
g_m$ forces $g_m$ to respect these complementary components.
\par
We will sometimes need the conformal maps $g_m$ to respect the marked points.  We say that $(X_m,
\bfz_m, f_m)$ converges \emph{strongly conformally} to $(X, \bfz, f)$ if the conformal maps~$g_m$ can be
defined on an exhaustion $\{K_m\}$ of $X^s$ and the $g_m$ respect the marked points (i.e.\ $g_m\circ \bfz
= \bfz_m$).
\par
The idea of the proof of the equivalence of these topologies is that given a quasiconformal map on
$X$ with small dilatation and an open set~$U$ (generally a neighborhood of a node or marked point),
one can find a nearby quasiconformal map which pushes all of the ``quasiconformality'' into~$U$.  Since strong conformal convergence requires the maps to be conformal near the marked
points, we see that it should only be equivalent to the other types of convergence in the
presence of nodes, as we will need~$U$ to be a neighborhood of the nodes in this case.

\begin{thm}\label{thm:topologiesarethesame}
  If $n\geq 1$, then the conformal and quasiconformal topologies on $\oldaugteich$ are equivalent.
  For any~$n$, if $X\in\oldaugteich$ has any nodes, then quasiconformal, conformal, and strong conformal
  convergence of a sequence to~$X$ are all equivalent.
\end{thm}
\par
Given a measurable subset $E$ of a Riemann surface~$X$, we denote by  $\calM(E)$ the Banach
space of measurable $L^\infty$-Beltrami differentials supported on $E$, and we
denote by $\calM^r(E)\subset\calM(E)$  the open ball of radius $r$.
\par
The proof of Theorem~\ref{thm:topologiesarethesame} is based on the following Lemma.
\par
\begin{lm}
  \label{lm:supercool}
  Let $(X,\bfz)$ be a compact pointed Riemann surface and $\bfz\subset K\subset U\subset X$ subsets such that $K$
  is compact with positive Lebesgue measure and~$U$ is open.  Then there is a constant $0< k < 1$
  such that for every Beltrami differential~$\nu$ on $X \setminus K$ with $\|\nu\|_\infty < k$,
  there exists  a quasiconformal homeomorphism $f_\nu\colon X\to X$, preserving the marked points, such that
  the Beltrami differential of $f_\nu$ restricted to $X\setminus K$  agrees with~$\nu$, and $f(K)\subset
  U$.

  Moreover, the collection of such maps $f_\nu$ may be regarded as a holomorphic map
  $\calM^k(X \setminus K) \to \QC^0(X)$, to the space of quasiconformal homeomorphisms
  of~$X$ isotopic to the identity, equipped with the compact-open topology.
\end{lm}

\begin{proof}
A Beltrami differential $\nu\in\calM^1(X)$ induces a conformal structure on~$X$ which we denote by $X_\nu$.  This defines a holomorphic map
  \begin{equation*}
    \Phi\colon \calM^1(X) \=
    \calM^1(K)\oplus\calM^1(X\setminus K)
    \to\oldteich\,.
  \end{equation*}
  Consider the derivative operator of $\Phi$, 
  \begin{equation*}
    D = D_1 \Phi_{(0,0)}\colon \calM(K) \to
    T_{(X, \bfz)}\oldteich,
  \end{equation*}
  restricted to the tangent space of the first factor of the
  splitting.  We claim that~$D$ is surjective.  This is equivalent to
  show that the dual operator
  $D^* \colon T^*_{(X, \bfz)}\oldteich\to \calM(K)^*$ is injective.
  Under the usual identification of the cotangent space to
  \Teichmuller space at $(X, \bfz)$ with $Q(X,\bfz)$, the space of
  quadratic differentials~$q$ on~$X$ with at worst simple poles
  contained in $\bfz$, the dual $D^*$ is given explicitly by the
  pairing
  \begin{equation*}
    D^*(q)(\nu) \= \int_{K} q\nu\,.
  \end{equation*}
  Taking $\nu_q$ to be the restriction of $\bar{q}/|q|$ to $K$, we obtain
  \begin{equation*}
     D^*(q) (\nu_q) = \int_K |q| >0\,,
   \end{equation*}
   so injectivity follows.

   \changed{Since $T_{(X, \bfz)}\oldteich$ is finite-dimensional, the kernel of $D$ is closed
   and of finite codimension, so it has a complementary closed
   subspace $E$.  This means we have a direct sum decomposition $\calM(K)= E
   \oplus \ker(D)$ such that $D|_E\colon E\to T_{(X, \bfz)}\oldteich$
   is an isomorphism.}

 \changed{By the Implicit Function Theorem, there is for some
   $0< k < 1$ and for some open ball $B\subset \ker(D)$ centered at $0$ a
   holomorphic map $\psi\colon B\oplus\calM^k(X\setminus K)\to E$ such
   that $\Phi(\psi(\nu_1,\nu_2),\nu_1, \nu_2) = (X, \bfz)$ for each
   $(\nu_1, \nu_2)\in B\oplus\calM^k(X\setminus K)$.  It follows that
   for each $\nu\in \calM^k(X\setminus K)$, there is a quasiconformal
   map $f_\nu\colon X\to X$ preserving the marked points~$\bfz$, with Beltrami differential given by
   $\psi(0,\nu) + \nu$.}

   The map $\nu\mapsto f_\nu$ can be regarded as a map $\Psi\colon\calM^k(X\setminus K)\to
   \QC^0(X)$, \changed{such that furthermore the maps in the image preserve the marked points~$\bfz$}.  By holomorphic dependence of solutions to the Beltrami equation on
   parameters (see e.g.\ \cite{HubbardBook}), this map $\Psi$ is holomorphic, and in particular
   continuous, as desired.  Therefore, by the definition of the compact-open topology, by possibly
   decreasing the constant $k$, we can ensure that $f_\nu(K)\subset U$.
\end{proof}
\par
\begin{proof}[Proof of Theorem~\ref{thm:topologiesarethesame}]
  We first show that quasiconformal convergence implies conformal convergence and, if nodes are
  present, also strong conformal convergence.  Suppose a sequence of marked pointed curves
  $(X_m, \bfz_m, f_m)$ converges to $(X, \bfz, f)$ in the quasiconformal topology, so that there is
  an exhaustion of~$X^s$ by compact sets $K_m$ and quasiconformal maps $g_m\colon K_m\to X_m$
  isotopic to $f_m$, whose dilatation tends to~$0$.  Let $U\subset X$ be an (arbitrarily small) open
  neighborhood of the nodes and the marked points.  To show convergence in the conformal topology,
  we must produce, for $m$ sufficiently large, a conformal map $h_m\colon X\setminus U \to X_m$
  isotopic to $f_m$.
\par
  Let $K\subset U$ be compact with positive Lebesgue measure.  By Lemma~\ref{lm:supercool}, for $m$
  sufficiently large, there is a quasiconformal map $k_m\colon X\to X$ sending~$K$ into~$U$ and
  whose Beltrami differential restricted to $X\setminus U$ agrees with the Beltrami differential of~$g_m$.
  The composition $h_m=g_m\circ k_m^{-1}$ is then conformal outside~$U$ as desired.

  If~$X$ has nodes, this argument works just as well to get strong conformal convergence by taking
~$U$ to be a neighborhood of the nodes only.
\smallskip
\par
We now show that conformal convergence implies quasiconformal convergence.   We choose  an
exhaustion $\{K_m\}$ of $X^s\setminus \bfz$ so that the inclusion $K_m\hookrightarrow
X^s \setminus\bfz$ is a homotopy equivalence, and let $g_m\colon K_m\to X_m$ be the conformal
maps which exhibit conformal convergence.  Let $\{K_m^f\}$ be the exhaustion of~$X^s$ obtained
by filling in the disks
  containing the marked points~$z_i$ (in this proof, the superscript $f$ will always mean that we
  fill in the disks around the marked points).  We must show that we can replace the~$g_m$ with
  quasiconformal maps $g_m^f$ on $K_m^f$ in the same isotopy class which respect the marked points
  and whose dilatation tends to $0$.  For concreteness, we fill in the disk containing~$z_1$.
\par
\changed{Let $Y$ be the component of $X^s \setminus \bfz$ whose closure contains $z_1$,}
%Let $Y$ be the component of~$X$ containing $z_1$, 
and let $J_m = K_m\cap Y$ and note that $J_m \hookrightarrow Y$ is a
homotopy equivalence. We first represent~$Y$ as $\half/\Gamma$ for
some Fuchsian group~$\Gamma$.  \changed{The fundamental group of the
  subsurface $L_m^f = g_m(J_m)^f\subset X_m^s$ is a subgroup of the
  fundamental group of the component $Z_m$ of~$X_m^s$ containing
  $L_m^f$, so it determines a cover of $Z_m$, which we represent as
  $\half / \Gamma_m$ for some Fuchsian group $\Gamma_m$. Let
$\tilde{J}_m\subset\tilde{J}_m^f\subset\half$ and
$\tilde{L}_m\subset\tilde{L}_m^f\subset \half$ be the unique connected subsurfaces,
invariant under the Fuchsian groups $\Gamma$ and $\Gamma_m$, whose
quotients are the corresponding subsets of $Y$ and $Z_m$.
The conformal map $g_m$ then lifts to a conformal map
$\tilde{g}_m\colon \tilde{J}_m \to \tilde{L}_m$ which is
equivariant in the sense that
\begin{equation}
  \label{eq:compatible_with_markings}
\rho_m(\gamma)\cdot\tilde{g}_m (z) \= \tilde{g}_m(\gamma\cdot z)
  \end{equation}
for some isomorphism $\rho_m\colon\Gamma\to\Gamma_m$ and for each
$\gamma\in \Gamma$ and $z\in J_m$.}
Note that the Fuchsian groups are really only defined up to conjugacy.  We normalize
the $\Gamma, \Gamma_m$ and all related objects by requiring that $0,1, \infty$ belong to
the limit set of~$\Gamma$ and the extension of $\tilde{g}_m$ to this limit set fixes these
three points.
\par
We claim now that $\tilde{g}_m$ converges uniformly on compact sets to the identity and that the
  Fuchsian groups $\Gamma_m$ converge to~$\Gamma$ algebraically (meaning that for each
  $\gamma\in\Gamma$, the limit of $\rho_m(\gamma)$ is $\gamma$).  By Montel's Theorem,
  any subsequence of
  $\tilde{g}_{m}$ has a further subsequence which converges uniformly on compact sets to
  some $G\colon \half\to \overline{\half}$.  Since each $\tilde{g}_m$ is conformal and
  fixes three points on the boundary of $\half$, in fact $G$ must be the identity map.  Since every
  subsequence of $\tilde{g}_m$ converges to the identity, we see that~$g_m$ converges uniformly on
  compact sets to the identity.  Algebraic convergence of~$\Gamma_m$ to~$\Gamma$ then follows
  immediately from~\eqref{eq:compatible_with_markings}.

  \changed{
Now choose a conformal map $p\colon\Delta\to \tilde{J}_m^f$ whose image is an open disk~$U$ which
  covers a  complementary disk containing some choice of $\tilde{z}_1$
  lying over $z_1$, which sends~$0$ to $\tilde{z}_1$, and which maps
  \index[other]{$\Delta_{r} = \lbrace z \in \CC :\arrowvert  z\arrowvert<r\rbrace  $}
  $\bdry\Delta$ onto a smooth curve $\gamma$ which is eventually contained in~$\tilde{J}_m$.  The
  composition $\tilde{g}_m\circ p$ sends the boundary circle to a smooth curve
  $\gamma_m\subset \tilde{L}_m^f$ which bounds a disk $U_m$ containing
  a point $\tilde{z}_{1,m}$ lying over the marked point $z_{1,m}$.
  Choose two points $a_1,a_2\in\bdry\Delta$, and let $p_m\colon\Delta\to \tilde{L}_m^f$ be the Riemann
  mapping of $\Delta$ onto $U_m$ which is normalized so that $p_m(a_i)
  = \tilde{g}_m\circ p(a_i)$ for each $i$ and $p_m(0) = \tilde{z}_{1,m}$.  Since $\tilde{g}_m$ converges to the identity uniformly on compact sets,
  the sets~$U_m$ converge to~$U$ in the \Caratheodory topology on disks (see \cite[Section~5.1]{mcmullen_renormalization}).
  In fact, they converge uniformly on $\overline{\Delta}$, since the closed sets
  $\half \setminus U_m$ are uniformly locally connected (see \cite[Corollary~2.4]{pommerenke}).  Let
  $\alpha_m\colon \Delta\to \Delta$ be the Douady-Earle extension of
  $p_m^{-1}\circ \tilde{g}_m\circ p|_{\bdry \Delta}$.   The boundary map is uniformly close to
  the identity, so $\alpha_m(0)$ is close to $0$, and we may construct a quasiconformal map
  $\beta_m\colon\Delta\to\Delta$ which is the identity on the boundary, sends $\alpha_m(0)$ back
  to $0$, and has small quasiconformal dilatation.   We define our extension $\tilde{g}_m^f$
  of $\tilde{g}_m$  as before on the complement of~$U$, and we define it to be $p_m\circ
  \beta_m\circ \alpha_m \circ p^{-1}$ on $\overline{U}$.  This is a quasiconformal
  extension of $\tilde{g}_m$ sending~$\tilde{z}_1$
  to~$\tilde{z}_{1,m}$ which lies over the desired quasiconformal
  extension of $g_m$.
  }
\end{proof}

Another reformulation of the same idea allows to assume, for~$X$ smooth and with at least one marked
point, $g_m$ to be conformal on an exhaustion $K_m$ of~$X$ minus a single marked point.

These ideas allow a similar definition of the ``universal curve'' over the augmented \Teichmuller space.
While we are not interested directly in this object as it is not an honest flat family of curves, it
is useful for defining universal curves over other spaces.
\par
Convergence of sequences in the universal curve is defined analogously to
convergence in $\oldaugteich$. Given $(X, \bfz, p, f)$ such that $p$ is not a node, we say that a sequence $(X_m, \bfz_m, p_m, f_m)$ converges to  $(X, \bfz, p, f)$ if
$(X_m, \bfz_m, f_m) \to (X, \bfz, f)$ as pointed stable curves; and moreover, if $g_m\colon K_m \to X_m$
are the (conformal or quasiconformal) maps which exhibit this convergence, then $g_m^{-1}(p_m)$ converges to~$p$.
\par
This definition does not work in the case when $p$ is a node, as then $p\in X\setminus K_m$, and
thus the map~$g_m$ is never defined at~$p$. Instead we require, for $m$ sufficiently large, the
point~$p_m$ to lie in the end of $X_m\setminus g_m(K_m)$ that corresponds
to the end of $X\setminus K_m$ containing $p$. This is well-defined, since $g_m$ eventually
induces a bijection of the components of $X\setminus K_m$ and $X_m\setminus g_m(K_m)$, as
remarked above.
\par
\subsection{The Dehn space and the Deligne-Mumford compactification}
\label{sec:dehn-space-deligne}

We briefly recall the construction of the Deligne-Mumford compactification $\barmoduli[g,n]$
of $\moduli[g,n]$ as well as the closely related Dehn spaces~$\oldDehn$, which give simple
models for $\barmoduli[g,n]$ near its boundary. For more details and proofs of all of these
statements, we refer the reader to \cite{HubKoch}, see also \cite{acgh2} for some of the
statements.
\par
Given a  multicurve $\Lambda \subset \Sigma \setminus \bfs$, the
\emph{full~$\Lambda$-twist group} $\oldtwist\subset \Mod$ is the free
abelian subgroup generated by Dehn twists around the curves of~$\Lambda$.
\index[twist]{d005@$\oldtwist$!Classical~$\Lambda$-twist group}
The
\emph{Dehn space}~$\oldDehn$ is the space obtained by adjoining to $\oldteich$
the stable curves where $f(\Lambda')$ for some subset~$\Lambda'$ of~$\Lambda$ has
been pinched , and then taking the quotient by $\oldtwist$.
\index[teich]{b040@$\oldDehn$!Classical Dehn space}
That is,
\begin{equation*}
  \oldDehn \= \bigcup_{\Lambda' \subset \Lambda} \Bteich[\Lambda'] / \oldtwist\,.
\end{equation*}
(Bers \cite{bers} called this space the ``deformation space''.)  Each $\oldDehn$ is a
contractible complex manifold.  It has a unique complex structure which agrees with the
complex structure induced by~$\oldteich$ in the interior, and such that the boundary is
a normal crossing divisor.
\par
The universal curve $\pi\colon\oldDehnfam\to\oldDehn$ is the quotient
\begin{equation*}
  \oldDehnfam \= \bigcup_{\Lambda' \subset \Lambda} \famaugteich|_{\Bteich[\Lambda'] }/ \oldtwist\,,
\end{equation*}
where \changed{$\famaugteich|_{\Bteich[\Lambda']}$} is the universal family over $\Bteich[\Lambda']$ and where the full
twist group acts \changed{by sending each fiber isomorphicly to the fiber
  which represents the marked stable curve given by composing the marking
  with the given product of twists.}
%trivially \changed{on each fiber of $\oldDehnfam$.}
\changed{This universal curve was constructed in \cite{HubKoch} using a
plumbing construction and shown to be a flat family of stable curves
over $\oldDehn$.}
\par
The \emph{Deligne-Mumford compactification} $\barmoduli[g,n]$ of $\moduli[g,n]$ is the quotient
$\oldaugteich / \Mod[g,n]$.  For each multicurve~$\Lambda$, the natural map $\oldDehn \to
\barmoduli[g,n]$ is a local homeomorphism. The image of $\oldDehn$ is the complement of the
locus of stable curves with a node not arising from pinching~$\Lambda$ \changed{up to a homeomorphism of the reference surface $\Sigma$}. 
These local homeomorphisms provide an atlas of charts for $\barmoduli[g,n]$
which give it the structure of a compact complex orbifold such that the boundary is a normal
crossing divisor.

One may also compactify $\moduli[g,n]$ as a projective variety $\barmoduli[g,n]^{\rm alg}$
(see \cite{demu} or \cite{acgh2}).  Hubbard-Koch \cite{HubKoch} showed
that $\barmoduli[g,n]^{\rm alg}
\isom \barmoduli[g,n]$ as complex orbifolds, so the natural topology of the algebraic variety
$\barmoduli[g,n]^{\rm alg}$ gives yet another equivalent topology on $\barmoduli[g,n]$.

%%%%%%%%%%%%%%%%%%%%%
\subsection{Spaces of one-forms} \label{sec:Space1}
%%%%%%%%%%%%%%%%%%%%%

We now consider topologies on various spaces of surfaces with holomorphic one-forms.  For surfaces
with one-forms, the conformal topology is sometimes more convenient than the quasiconformal
topology, as pullbacks of holomorphic one-forms by quasiconformal maps are in
general only locally $L^2$ as quasiconformal maps have locally $L^2$ derivatives, but we
will use these topologies both. On the
other hand, these spaces already have topologies coming from algebraic geometry, and we will show
that these topologies coincide.
\par
Consider the universal curve over the Dehn space $\pi\colon \oldDehnfam \to \oldDehn$, with its
relative cotangent sheaf $\omega_{\oldDehnfam / \oldDehn}$.  The pushforward $\pi_*
\omega_{\oldDehnfam / \oldDehn}$ is the sheaf of sections of the \emph{Hodge bundle}
$\ooldDehn\to\oldDehn$, a rank $g$ vector bundle whose fiber over a point~$X$ is the space
$\Omega(X)$ of stable forms on~$X$.  \changed{Since $\oldDehn$ is a
  contractible domain of holomorphy (see \cite{bers81}), $\ooldDehn$
  is holomorphically trivial.}  
\index[teich]{b045@$\ooldDehn$!Hodge bundle over the Dehn space}
As $\ooldDehn$ is a vector bundle, it comes with a natural
topology, which we call the \emph{vector bundle topology}.

On the other hand, the conformal topology on $\oldDehn$ gives a second natural topology on
$\ooldDehn$.  A sequence $(X_m, \bfz_m, \omega_m, f_m)$ of marked pointed stable forms converges to
$(X, \bfz, \omega, f)$ \emph{in the conformal topology} if for some exhaustion~$K_m$
of~$X^s\setminus\bfz$, there is a sequence of conformal maps $g_m\colon K_m \to X_m$ such that
$f_m \simeq g_m \circ f$ and $g_m^* \omega_m$ converges to~$\omega$ uniformly on compact
sets. Again, we say that such a sequence converges \emph{strongly conformally} if the $g_m$ are
moreover defined on an exhaustion $K_m$ of $X^s$ and respect the
marked points.

We will occasionally want to allow the maps $f_m$ to be only
quasiconformal.  We say a sequence $(X_m, \bfz_m, \omega_m, f_m)$
converges to $(X, \bfz, \omega, f)$ \emph{in the quasiconformal
  topology} if for some exhaustion~$K_m$ of~$X^s$, there is a
sequence of $L_m$-conformal maps $g_m\colon K_m \to X_m$ which exhibit
convergence in the quasiconformal topology on $\oldDehn$ and such that
$g_m^* \omega_m$ converges to~$\omega$ in the topology of weak locally
$L^2$ convergence.

We show below that these topologies agree.
\par
\begin{lm} \label{lm:converge_to_identity}
  Suppose $3g-3+n>0$. Let $(X_m, \bfz_m, f_m)$ be a sequence in
  $\oldDehn$ converging to $(X, \bfz, f)$ in the quasiconformal topology,
  and let $g_m\colon K_m\to X_m$ be a sequence of $L_m$-quasiconformal maps exhibiting this convergence, where
  \changed{$K_m$ is an exhaustion of~$X^s$}.  Then the maps $g_m$ (regarded as maps into the universal curve
  $\oldDehnfam$) converge uniformly on compact sets to the identity map on~$X$.
\end{lm}
\par
\begin{proof}
First, we claim that there is a subsequence which converges
uniformly on compact sets.    We show via the
Arzel\'a-Ascoli Theorem that a subsequence converges uniformly on
\changed{the thick part}
$K = \Thick[(X^s \setminus \bfz)][\epsilon_0]$, and convergence on compact sets
follows from the usual diagonal trick
%\changed{(recall from
 % Section~\ref{sec:augm-teichm-space} that $(X,\bfz)_\epsilon$
% denotes the $\epsilon$-thick part of $X\setminus \bfz$).}
We fix also $K' = \Thick[(X^s \setminus \bfz)][\epsilon_1]$
for some $\epsilon_1<\epsilon_0$
and assume $m$ is large enough that $g_m$ is defined on $K'$.
\par
Choose a Riemannian metric $\rho'$ on $\oldDehnfam^s$, the complement of
the nodes and marked
points in $\oldDehnfam$, whose restriction to the fibers is the
vertical hyperbolic metric~$\rho$.  The functions $g_m$ have a
uniform modulus of continuity on $K$ by \cite[\S~3.3.3]{lehto} and so
are uniformly equicontinuous.

To apply Arzel\'a-Ascoli, we just need that the $g_m$ are contained in a
compact subset of~$\oldDehnfam^s$. By Mumford's compactness criterion,
the $\epsilon$-thick part in the vertical hyperbolic metric
$\Thick[(\oldDehnfam^s)][\epsilon]$ is compact, so it suffices to show
that~$g_m$ maps~$K$ into $\Thick[(X_m^s \setminus \bfz_m)][\epsilon]$ for
some uniform~$\epsilon$. Again, since the $g_m$ have a uniform modulus
of continuity, there are small constants $\epsilon <\epsilon'$ so that if
$g_m(K)$ intersects the $\epsilon$-thin part, it must be contained in
the $\epsilon'$-thin part which is a union of annuli by the Margulis
lemma.  Since~$g_m$ is compatible with the markings, $g_m$ is $\pi_1$-injective
on~$K$, but $K$ is not an annulus as $X^s\setminus \bfz$
is finite type and hyperbolic, so $g_m(K)$ cannot be contained in the
$\epsilon'$-thin part.  It follows immediately that $g_m(K)$ is contained
in $\Thick[(X_{m}^s \setminus \bfz_{m})][\epsilon]$ for some uniform~$\epsilon$.

We now show convergence to the identity. The preceding argument in fact shows that every
subsequence of~$g_m$ has a further uniformly convergent
  subsequence.  As the~$g_m$ are $L_m$-quasiconformal, with $L_m\to 1$, any subsequential limit is a conformal automorphism of
  $X\setminus \bfz$.  As the~$g_m$ are compatible with the markings, this map is homotopic to
  the identity, so must in fact be the identity, since $X\setminus \bfz$ is finite type and hyperbolic.
  Thus any subsequence of $g_m$ has a further subsequence which converges to the identity map, and it
  follows that $g_m$ converges to the identity.
\end{proof}

\begin{prop}
  \label{prop:topologies_are_the_same}
  The vector bundle, conformal, and quasiconformal topologies on $\ooldDehn$ coincide.
\end{prop}
\par
\begin{proof}
  On the base surface $(\Sigma, \bfs)$, choose $g$ disjoint
  homologically independent ``$\alpha$-curves''
  $\alpha_1, \dots, \alpha_g$, such that each $\alpha_i$ is either
  part of the multicurve ~$\Lambda$ or disjoint from each curve
  in~$\Lambda$.
  % they are dual to a basis of relative forms.  Let $\alpha_i^*$ be
  % smooth, compactly supported $1$-forms on $\Sigma\setminus \bfs$
  % which are dual to the $\alpha_i$.  Then there are relative
  % one-forms
  Let $\eta_1, \dots, \eta_g$ be relative holomorphic $1$-forms on the
  universal curve~$\oldDehnfam$ over $\oldDehn$ such that in each
  fiber,
  \begin{equation*}
    \int_{\alpha_i} \eta_j \= \delta_{ij}.
  \end{equation*}
  Now suppose a sequence $(X_m, \bfz_m, \omega_m, f_m)$ converges to
  $(X, \bfz, \omega, f)$ in the quasiconformal topology, and let
  $g_m\colon K_m \subset X\to X_m$ be the sequence of quasiconformal
  maps exhibiting this convergence.  \changed{Let
    $\alpha_1^*, \ldots, \alpha_g^*$ be smooth and compactly supported
    $1$-forms on $X^s$ which are dual to the curves $f(\alpha_i)$.}
  We may write each $\omega_m$ and~$\omega$ as a linear combination of
  the $\eta_i$:
  \begin{equation}
    \label{eq:omega_in_basis}
    \omega_m \= \sum_{i=1}^g c_{mi}\, \eta_i|_{X_m} \quad\text{and}\quad
    \omega \= \sum_{i=1}^g\, c_i \eta_i|_{X}\,.
  \end{equation}
  Convergence in the vector bundle topology is then equivalent to convergence of each~$c_{mi}$ to
  $c_i$.  Since each $c_{mi}$ can be recovered as the integral of
  $g_m^*\omega_m \wedge \alpha_i^*$, this follows from the weak
  convergence of $g_m^*\omega_m$ to~$\omega$.
\par
  Now suppose the sequence converges in the vector bundle topology.  Writing
  the form~$\omega_m$ in
  the basis $\eta_i$ as in \eqref{eq:omega_in_basis}, this means that $c_{mi}$ converge to $c_i$ for each~$i$.
  Then $(X_m, \bfz_m, f_m)$ converge to $(X, \bfz, f)$ as marked pointed surfaces, and by
  Theorem~\ref{thm:topologiesarethesame} there is a sequence of maps $g_m\colon K_m \to X_m$,
  defined on an exhaustion of~$X$, which exhibit convergence in the conformal topology.  By
  Lemma~\ref{lm:converge_to_identity}, these $g_m$ converge uniformly on compact sets to the
  identity, and so do their derivatives.  It follows that $g_m^* \eta_i \to \eta_i|_X$ uniformly on
  compact sets, so $g_m^* \omega_m \to \omega$ as well, and so the
  sequence converges in the conformal topology.

  Since conformal convergence obviously implies quasiconformal
  convergence, it follows that the three topologies are equivalent.
\end{proof}

These notions of convergence will appear in several similar contexts.  We will often need convergence of
one-forms on part of~$X$ only.  A sequence of stable differentials $(X_m, \bfz_m, \omega_n)$ converges to $(X, \bfz, \omega)$
on an irreducible component $Y\subset X$ if there are conformal maps $g_m\colon K_m\to X_m$ so that
$g_m^* \omega_m $ converge to $ \omega$ uniformly on compact sets, where~$K_m$ is an exhaustion of~$Y$.  In
another direction, one may allow the $\omega_m$ to have poles of prescribed order at the marked
points.  These notions of convergence may be formalized similarly to the vector bundle topology
described above by twisting the relative cotangent bundle, giving a notion of convergence
equivalent to conformal convergence.

%%%%%%%%%%%%%%%%%%%%%
\subsection{Strengthening conformal convergence}
%%%%%%%%%%%%%%%%%%%%%

We have defined conformal convergence of one-forms as uniform convergence of the
pullbacks $g_m^*\omega_{m}$ to $\omega$ on compact sets.  A natural strengthening is to require
the pullbacks to be equal to~$\omega$.  This is not always possible: if
$\omega_m$ and~$\omega$ have different relative periods, then they cannot be
identified by any conformal map.  It turns out that relative periods
are the only obstruction.
\par
\begin{thm} \label{thm:conformal_convergence}
Let~$X$ be a closed Riemann
surface, containing open subsurfaces~$U$ and~$W$ with
$\overline{U}\subset W \subsetneq X$,
and let $Z, P \subset U$ be disjoint discrete sets. Suppose moreover that
the boundaries of~$U$ and~$W$ are smooth and that~$U$ is a deformation
retract of~$W$. Let $\nu_m$ and $\eta_m$
be two sequences of meromorphic one-forms on~$W$ converging uniformly
on compact sets to a single non-zero meromorphic form~$\omega$.
Suppose moreover that
\begin{enumerate}[(i)]
\item all of the forms $\nu_m$, $\eta_m$ and~$\omega$ have the same set of poles~$P$ and the same set of zeroes~$Z$, and moreover
\item the orders $\ord_z \nu_m $, $ \ord_z\eta_m $ and $ \ord_z\omega$ coincide
for every~$m$ and $z\in U$, and
\item for each~$m$, the classes $[\nu_m]$ and $[\eta_m]$ in $H^1(U\setminus P, Z; \CC)$ are equal.
\end{enumerate}
Then for~$m$ sufficiently large, there exists a conformal map $h_n\colon U\to W$ fixing
each point of $Z\cup P$, and such that $h_m^* (\nu_m) =  \eta_m$. Moreover one can
choose~$h_m$ to converge uniformly to the identity as $m\to\infty$.
\end{thm}
\par
The proof will follow from applying the Implicit Function Theorem to a suitable holomorphic map on
an open subset of $\mathcal{H}\times E$, where $\mathcal{H}$ is a space of holomorphic maps
$U\to X$, and $E$ is a Banach space parameterizing one-forms on $W$.  In the next Lemma, we give
$\mathcal{H}$ the structure of a Banach manifold modeled on a space of vector fields on~$U$.

Given an open set~$V$ in some Banach space, we denote by $\calO_V$ the Banach space of bounded
holomorphic functions on~$V$ equipped with the sup norm.  More generally, if $E$ is a normed vector
space, $\calO_V(E)$ will denote the Banach space of bounded holomorphic functions $V\to E$, equipped
with the sup norm.
We use the following notation for derivatives of maps between Banach spaces. We
denote by $D^n_iF_z$ the~$n$-th partial derivative with respect to the~$i$-th variable at $z$ and we
let $D^nF_z$ denote the derivative of~$F$ at $z$.  We use several times that standard results from
calculus and complex analysis hold in the context of holomorphic maps on Banach spaces.  See
\cite{Mujica_complex_analysis,Nachbin} for details.
\par
\begin{lm}
  Let $Y\subset \cx^3$ be a smooth analytic curve, $U\subset Y$ a relatively compact open set, and
  $S\subset U$ a finite subset.
  In the space $\mathcal{O}_U(\cx^3)^S$ of bounded holomorphic functions $g\colon U\to \cx^3$ which fix $S$
  pointwise, let $\mathcal{H}$ be the locus of those functions sending~$U$ into~$Y$, and let
  $\mathcal{B}_\epsilon$ be the $\epsilon$-ball centered at the identity map $\id$.  Then for some
  $\epsilon>0$ the intersection $\mathcal{H}_\epsilon = \mathcal{B}_\epsilon \cap \mathcal{H}$ has
  the structure of a Banach manifold isomorphic to an open ball in $\mathcal{V}(U)^S$, the space of
  bounded holomorphic tangent vector fields to $U$ which vanish at each point of $S$.
\end{lm}

\begin{proof}
By \cite[Corollary~1.5]{BFComplInt}, every analytic curve $Y$ in $\cx^3$ is
an ideal-theoretic complete intersection, meaning there are holomorphic
functions $F_1, F_2\colon \cx^3\to \cx$ so
that~$Y$ is defined by the equations $F_1 = F_2 = 0$ and the derivative
$DF \colon \cx^3\to \cx^2$ (where $F= (F_1, F_2)$) is surjective at each point of $Y$.
\par
Now $\mathcal{O}_U(\cx^3)^S\subset \mathcal{O}_U(\cx^3)$ is a finite-codimension
affine subspace, which may be identified with the Banach space $\mathcal{O}_U(\cx^3)_S$
of functions which vanish on~$S$.  We define $\Phi\colon \mathcal{O}_U(\cx^3)_S \to
\mathcal{O}_U(\cx^2)_S$ by $\Phi(g) = F\circ (\id + g) - \id$, a holomorphic map
with derivative $D\Phi_{\id}(g) = DF \cdot g$.  The space $\mathcal{H}$ is then the
fiber of~$\Phi$ over~$0$.  If we could show that $D\Phi_{\id}$ is a split surjection
by constructing a right-inverse to $DF$, it would then follow that
$\mathcal{H}_\epsilon$ is a Banach manifold  modeled on the kernel of $DF$,
which is clearly $\mathcal{V}(U)^S$, as claimed.
\par
The derivative $DF$ is explicitly the $2\times 3$ matrix whose $ij$th entry is the
entire function~$\frac{\partial F_i}{\partial z_j}$.  Let $M_i$ be the $3\times 2$
matrix obtained by replacing
the $i$th row of $DF^T$ by zeros, and let $\mu_i$ be the $i$th minor of $DF$, so that
  \begin{equation}
    \label{eq:right_inverse}
    DF\cdot M_i \= \mu_i I.
  \end{equation}
  Since $DF$ is surjective, the minors $\mu_i$ have no common zero on $Y$.  In other words the
  functions $F_1, F_2, \mu_1, \mu_2, \mu_3$ have no common zero in $\cx^3$.  Let $\mathfrak{a}$ be
  the ideal generated by these functions in the ring $\mathcal{O}_{\cx^3}$ of entire
  holomorphic functions.  By a version of Forster's analytic Nullstellensatz (see \cite{ABF}), the
  radical ideal $\sqrt{\mathfrak{a}}$ is dense in $\mathcal{O}_{\cx^3}$ (in the topology of normal
  convergence).  There are then entire functions $\alpha_{k}, \beta_{k}, h$ and
  an integer $n$ so that
  \begin{equation*}
    h^{n} \= \alpha_{1} \mu_1 + \alpha_{2}\mu_2 + \alpha_{3} \mu_3 + \beta_{1}F_1 +
    \beta_{2} F_2\,,
  \end{equation*}
  with $h$ nonzero on $U$.  Using \eqref{eq:right_inverse}, we then have that $h^{-n} \sum_k \alpha_{k} M_k$ is the desired
  right-inverse to $DF$ on $U$.
\end{proof}

\begin{rem}
When this Lemma is applied below, $Y$ is an algebraic curve.  In this case,
by the Ferrand-Szpiro Theorem $Y$ is a set-theoretic complete intersection
(see \cite{Szpiro}) which may not be an ideal-theoretic complete intersection.
So even when~$Y$ is algebraic, we are forced to use analytic equations defining~$Y$.
\end{rem}

\begin{proof}[Proof of Theorem~\ref{thm:conformal_convergence}]
  Choose $Q \in X \setminus W$ and fix an embedding of $Y = X \setminus Q$ in $\cx^3$ as an affine
  space curve.  By the previous Lemma, the space of holomorphic maps $U\to Y$ which fix the subset $S$ and are
  sufficiently close to the identity may be identified with a $\delta$-ball $\mathcal{V}(U)^S_\delta$.

  Let $\calO_W^0$ be the closed subspace of $\calO_W$ consisting of those $f$ such that $f\omega$
  has trivial periods in $W\setminus P$.  We can thus write $\nu_m = (1+f_m)\omega$ and
  $\omega_m =(1+f_m+g_m)\omega$ with $f_m \in \calO_W$ and $g_m \in \calO_W^0$ both converging to
  zero as $m \to \infty$.
  \par
  Let $\calB_\epsilon$ denote the $\epsilon$-ball in $\mathcal{V}(U)^S\times \calO_{W} \times \calO_{W}^0$
  centered at $(id, 0, 0)$. Consider the map $\Psi\colon \calB_\epsilon \to \calO_{U}^0$ defined by
 \begin{equation*}
   \Psi(\phi, f, g) \= \frac{\phi^*\big( (1+f+g)\omega \big) - (1+f)\omega}{\omega}\,.
\end{equation*}
Once we have shown that $\Psi$ is well-defined, holomorphic and
that the tangent map $D_1\Psi_{(id, 0, 0)}$ is a split surjection,
the Implicit Function Theorem allows to construct a family of
maps $\phi(f,g)$, parameterized by $f$ and $g$ in some $\epsilon$-ball
\changed{(and some neighborhood of zero of the kernel of $D_1\Psi_{(id, 0, 0)}$;
see the proof of Lemma~\ref{lm:supercool})},
so that $\phi(0,0)$ is the identity map and $\Psi(\phi(f,g), f, g) = 0$.
We can then set $h_m = \phi(f_m,g_m)$.
\par
To check that $\Psi$ is well-defined, that is, the image is contained in $\mathcal{O}_U^0$, note
first that the numerator and denominator have the same zeros and poles, since they are fixed by~$\phi$.  Moreover, the right hand is bounded on~$U$ for $\epsilon$ sufficiently small, as it extends
to a holomorphic function on a neighborhood of $\overline{U}$, so it does indeed belong to
$\calO_U^0$.  Note that for $\epsilon$ sufficiently small, the maps $\phi$ are sufficiently close to the
identity map~$U$ to~$W$, so that the pullback in the definition of $\Psi$ is defined.
\par
We claim that $\Psi$ is holomorphic. We write~$\calX$ for $X\times
\calB_\epsilon$, and similarly~$\calU$ and~$\calW$ for the trivial
families of subsets of~$X$ over $\mathcal{B}_\epsilon$. We have a universal holomorphic map
$\Phi\colon\calU\to\calX$ whose fiber over a point in
$\mathcal{V}(U)^S$ is the map which that point represents. Similarly,
there are universal bounded holomorphic functions $F, G\colon \calW\to
\CC$ associated to the factors~$\calO_W$ and $\calO_W^0$ of
$\calB_\epsilon$.  The form~$\omega$ can be regarded as a relative
one-form~$\Omega$ on~$\calW$.  The function
\begin{equation*}
  H \= \frac{\Phi^*\big((1+F+G)\Omega\big) - (1+F)\Omega}{\Omega}
\end{equation*}
is holomorphic on $\calU$ and uniformly bounded on $\bdry \calU$.  Here we use
the Cauchy Integral formula to bound the ``vertical'' derivatives of~$\Phi$
on $\bdry \calU$.  By Lemma~\ref{lem:banach_universal_property}, this induces
a holomorphic map into $\calO_U^0$ which is none other than $\Psi$, and
moreover $D\Psi$ can be computed with \eqref{eq:univ_derivative}.
\par
The derivative operator $D_1\Psi_{(id, 0, 0)}\colon \calV(U)^S
\to \calO_U^0$  is
  \begin{equation*}
    D_1\Psi_{(id, 0, 0)}(v) \= {\calL_v \omega} / {\omega}\,,
  \end{equation*}
where $\calL_v$ is the Lie derivative. We now show that this map is a split
surjection by constructing a right inverse~$\Upsilon$ to~$D_{1}\Psi$.
We define
\begin{equation*}
\Upsilon\colon \calO_U^0 \to \calV(U)^S\,,
\qquad \Upsilon(f) \=  \frac{1}{\omega} \int_{z_0} f\omega\,
\end{equation*}
and argue now that this is well-defined.  The integral is over any path
starting at $z_o$ which is either an arbitrary choice of basepoint  in~$Z$,
or an arbitrary basepoint if~$Z$ is empty.  The integral depends only on the
endpoints of the path, since $f\omega$ has trivial absolute periods, and
moreover since it has trivial relative periods, it vanishes at each point in $Z$
to order one larger than~$\omega$.  It follows that $\Upsilon(f)$ is
a holomorphic vector field on~$U$ which vanishes at $Z\cup P$.
\par
This defines an operator $V\colon \calO_U^0\to \calV(U)^S$ which
is evidently bounded. It is a left inverse to $D_{1}\Psi$ by Cartan's equality
$\calL_v\omega = d (\omega(v))$ (for closed~$\omega$). This completes the
verification of the hypothesis of the Implicit Function Theorem.
\end{proof}
\par
To complete the proof of Theorem~\ref{thm:conformal_convergence} it remains
to verify the following universal property.
\begin{lm}
  \label{lem:banach_universal_property}
  Let $E$ and $F$ be complex Banach spaces containing open sets~$U$ and $V$
  respectively, and let
  $f\colon U\times V\to \CC$ a bounded holomorphic function.  Then the map
$F \colon U \to \calO_V$ defined by $F(z)(w) = f(z,w)$ is a
holomorphic function with
\begin{equation}
  \label{eq:univ_derivative}
  DF_z (w)\= D_1f_{(z,w)}\,.
\end{equation}
\end{lm}
\par
\begin{proof}
Given $(z,w)\in U\times V$, suppose $B_R(z)$ is contained in~$U$.
By the Cauchy integral formula, we then have
  \begin{equation}
    \label{eq:cauchy_inequality}
    \|D_1^p f_{(z,w)}\| \,\leq\, p!\, \frac{M}{R^p}\,,
  \end{equation}
where $M$ is a uniform bound for $|f|$ on $U \times V$. Given $z\in U$,
let $D_z\colon E \to \calO_V$ be the bounded operator $D_z(h)(w) = D_1 f_{(z,w)}(h).$
We claim that $F$ is differentiable at $z$ with first derivative $D_z$.
Since $D_z$ is complex linear, it implies that $F$ is holomorphic.
This follows immediately from the bound
  \begin{align*}
    |F(z+h) - F(z) - D_z(h)| &\= \sup_{w\in V} |f(z+h, w) - f(z,w) - D_1f_{(z,w)}(h)|\\
    & \,\leq\, \frac{M}{(R - |h|)^2} |h|^2,
  \end{align*}
where the last inequality follows from the bound \eqref{eq:cauchy_inequality}
for the second derivative and Taylor's Theorem.
\end{proof}
\par

%%%%%%%%%%%%%%%%%%%%%%%%%%%%
\subsection{Compactness for meromorphic differentials}
\label{sec:compactdiff}
%%%%%%%%%%%%%%%%%%%%%%%%%%%

In this subsection, we study convergence for sequences of curves equipped with a
meromorphic differential, establishing a compactness result which we later use
in Section~\ref{sec:msdastopo} to obtain compactness of the moduli space of \msds.
For an alternative approach to these compactness results, see the appendix
to \cite{mcmullen_amenability}.
\par
Given a pointed stable curve $(X, \bfz)$ we denote the punctured surface $X^s\setminus\bfz$ by~$X'$,
which will always be equipped with its \Poincare hyperbolic metric $\rho$.
%\changed{We use $\Thick[X'][\epsilon]$ to denote the $\epsilon$-thick part of~$X'$ (with $\epsilon$ smaller than the Margulis constant).}
%Recall that $\Thick[X][\epsilon]$ denotes the $\epsilon$-thick part of~$X$ (with $\epsilon$ smaller than the Margulis constant).
\par
Consider a degenerating sequence of pointed meromorphic differentials $(X_m, \bfz_m, \omega_m)$ in
$\omodulin(\mu)$ such that the underlying pointed curves converge to some pointed stable curve $(X, \bfz)$.
It may happen that on some components of the thick part of $X_m'$ the flat metric~$|\omega_m|$ is much smaller
than on other components.  As a result the limit of~$\omega_m$ may be non-zero on some components of $X'_m$, and
vanish identically on others.  In order to get non-zero limits everywhere, we allow ourselves to rescale the differential on
different components at different rates.  These rescaling parameters arise from a notion of size for
the thick parts of the $X_m'$ which we now define.

Given a meromorphic differential $(X, \omega)\in \omodulin(\mu)$, for any $p\in X'$, let $|\omega|_p$ be
its norm at $p$ with respect to the hyperbolic metric.  If $Y$ is a component of the thick part~$\Thick[X'][\epsilon]$, we define the \emph{size} of $Y$ by
\begin{equation}\label{eq:size}
  \lambda(Y) \= \sup_{p\in Y} |\omega|_p\,.
\end{equation}
A similar notion of size is defined in \cite{RafiThTh}.

\begin{thm}
  \label{thm:compactness_for_differentials}
  Suppose $(X_m, \bfz_m, \omega_m)$ is a sequence of
  meromorphic differentials in $\omodulin(\mu)$ such that $(X_m, \bfz_m)$ converges to some
  $(X, \bfz)$ as a sequence of pointed stable curves.  Let $Y\subset X$ be a component
  and choose $\epsilon$ small enough so that \changed{the $\epsilon$-thick part} $\Thick[Y][\epsilon]$ is connected.  For large~$m$, choose
  $\Thick[(Y_m)][\epsilon]$ to be the sequence of components of $\Thick[(X_m)][\epsilon]$ such that
  $\Thick[(Y_m)][\epsilon]$ converges to $\Thick[Y][\epsilon]$.  Let $\lambda_m =\lambda(\Thick[(Y_m)][\epsilon])$.  Then we may pass to a  subsequence so that the sequence of
  rescaled differentials $\omega_m / \lambda_m$ has a non-zero limit on~$Y$.
\end{thm}

Note that if we only wanted a limiting differential defined on $\Thick[Y][\epsilon]$, since
$|\omega_m/\lambda_m|$ is bounded on $\Thick[(Y_m)][\epsilon]$, this would be a trivial
consequence of Montel's
Theorem.  To get convergence on all of $Y$, we establish \emph{a priori} bounds (depending only on
$\epsilon$ and $\mu$) for the size of any component of the $\epsilon$-thick part of $X'$,
in terms of the norm $|\omega|_p$ at any point of~$Y$.
\par
To this end, we introduce the \emph{\Poincare distortion function} of a pointed meromorphic
differential $(X, \bfz, \omega)$ as the function $\daleth\colon X'\to \reals$ defined by
\begin{equation*}
  \daleth(p) \= |\beta|_p \quad\text{where}\quad \beta \=d \log|\omega/\rho|\, .
\end{equation*}
This function measures how quickly the flat metric $|\omega|$ varies with respect
to the hyperbolic metric~$\rho$.  Note that $\daleth$ is independent of the scale of~$\omega$, so can be
regarded as a function on the punctured universal curve over~$\pomodulin(\mu)$.

\begin{lm}
  \label{lm:distortion_bound}
  There is a constant $C$ depending only on $\mu$ and $\epsilon$ so that
  for any $(X, \bfz, \omega)\in\omodulin(\mu)$, the distortion function $\daleth$
  is bounded by $C$ on the $\epsilon$-thick part of~$X'$.
\end{lm}

\begin{proof}
We wish to define a compactification of $\pomodulin(\mu)$ so that $\daleth$ extends
continuously to the universal curve over the compactification.  To this end, let
$\PP\obarmoduli[g,n]^{\rm{ninc}}(\mu)$ be the normalization of the Incidence Variety
Compactification, with the universal curve $\pi\colon \tilde{\calX} \to \PP
\obarmoduli[g,n]^{\rm{ninc}}(\mu)$. The universal curve is equipped with a family of
one-forms~$\omega$, defined up to scale. Its divisor consists of horizontal
components (whose $\pi$-image is $\PP\obarmoduli[g,n]^{\rm{ninc}}(\mu)$) along
the marked zeros and poles, and also some vertical components (whose $\pi$-image
is a boundary divisor of $\PP\obarmoduli[g,n]^{\rm{ninc}}(\mu)$).
\par
Suppose $D\subset \tilde{\calX}$ is an irreducible vertical component of the zero divisor
of~$\omega$ \changed{ and let $D'$ be the image of~$D$ in the base $\PP\obarmoduli[g,n]^{\rm{ninc}}(\mu)$}.
Since the base is normal, by Proposition~\ref{prop:normal-rescalable} below near any point
$p\in D'$ there is a regular function~$f$ defined near~$p$ so that $\omega/f$ is regular
on~$D$ near the fiber over~$p$, and moreover such an $f$ is unique up to multiplication
by the pullback of a regular function which does not vanish at~$p$. \changed{Let $\tilde{\calX}'$ be the punctured universal curve by removing the marked zeros and poles from $\tilde{\calX}$.} The family of one-forms
  \begin{equation*}
    \tilde{\beta} \= d \log|\omega/\rho f|
\end{equation*}
is then a continuous extension of~$\beta$ which is defined in a neighborhood of the fiber
over~$p$ in the punctured universal curve $\tilde{\calX}'$ and depends neither on the choice
of~$f$ nor on the scale of~$\omega$.  Here we are using the fact that the vertical hyperbolic metric
is $C^1$ on $\tilde{\calX}'$
by \cite{WolHyp}.  Since we may extend $\beta$ on a neighborhood of any such vertical zero
divisor, this gives a continuous extension of $\tilde{\beta}$ over all of $\tilde{\calX}'$.  The
function $\tilde{\daleth}(p) = |\tilde{\beta}|_p$ is then the desired continuous extension
of $\daleth$ to $\tilde{\calX}'$.  Since the $\epsilon$-thick part of $\tilde{\calX}$ is
compact, we see that $\daleth$ is bounded on the $\epsilon$-thick part.
\end{proof}
\par
\begin{cor}
  \label{cor:stupid_bound}
  There exists a constant $L$ that depends only on $\mu$ and~$\epsilon$, such that for any pointed meromorphic differential $(X, \bfz, \omega) \in \omodulin(\mu)$, for any points
  $p$ and $q$ in the $\epsilon$-thick part of~$X$, we have
  \begin{equation*}
    |\omega|_q \,\leq\, L |\omega|_q\,.
  \end{equation*}
\end{cor}

\begin{proof}
  By Lemma~\ref{lm:distortion_bound}, $\log|\omega/\rho|$ is $C$-Lipschitz on $\Thick[X][\epsilon]$ for a
  uniform constant $C$.  The diameter of $\Thick[X][\epsilon]$ is bounded by a uniform constant $E$, so we can
  take $L = C e^E$.
\end{proof}
\par
\begin{proof}[Proof of Theorem~\ref{thm:compactness_for_differentials}]
  Let $f_m\colon K_m \to X_m$ be conformal maps on an exhaustion $\lbrace K_m\rbrace$ of $Y'$ that exhibit the
  convergence of the $(X_m, \bfz_m)$.  By Corollary~\ref{cor:stupid_bound}, and convergence of the
  \Poincare metrics of $X_m'$ to that of $Y'$, the differentials $f_m^*(\omega_m/\lambda_m)$ are
  uniformly bounded on the $(1/k)$-thick part of $Y$ for every~$k$.  By Montel's Theorem, there
  is a subsequence which converges uniformly on $\Thick[Y][1/k]$.  The diagonal trick gives a sequence
  converging uniformly on compact subsets of~$Y$.
\end{proof}

%%%%%%%%%%%%%%%%%%%%%%%%%%%%%%%%%%%%
\section{Normal forms for differentials}
\label{sec:NF}
%%%%%%%%%%%%%%%%%%%%%%%%%%%%%%%%%%%%%

This section provides auxiliary statements for the normal forms for
differentials on Riemann surfaces, and for \changed{families of differentials whose underlying Riemann surfaces
 degenerate to a nodal Riemann surface}. There are two types of statements. The first
is for a fixed Riemann surface, in fact a disk or an annulus. If moreover
the differential is fixed, this goes back to Strebel. For a varying
differential we proved such a normal form statement
in~\cite[Section~4.2]{strata} and we give a slight generalization below.
The second type of normal form theorem is for differentials on a family of
surfaces whose topology changes. This statement has two subcases
corresponding to the local situation at
vertical nodes and horizontal nodes, respectively.
\par
We first recall Strebel's standard local coordinates for meromorphic
differentials in the complex plane. A meromorphic differential~$\omega$ defined
on a neighborhood of $0$ in~$\CC$ has two local conformal invariants,
its order of vanishing $k = \ord_0\omega$ and its residue $r = \Res_0\omega$.
Strebel constructed a normal form for~$\omega$, which depends
only on~$k$ and~$r$.
\par
\begin{thm}[{\bf Normal form on a disk}, \cite{Strebel}]
  \label{thm:standard_coordinates}
  Consider a meromorphic differential~$\omega$ on the $\delta$-disk $\Delta_\delta\subset\CC$ with $k=
  \ord_0\omega$ and $r = \Res_0\omega$.  Then for some $\epsilon>0$, there exists a \changed{biholomorphic coordinate change map}
  %conformal map
  $\phi\colon(\Delta_\epsilon, 0) \to (\Delta_\delta,0)$ such that
  \begin{equation}\label{eq:standard_coordinates}
    \phi^*\omega \=
    \begin{cases}
      z^k\, dz &\text{if $k\geq 0$,}\\
      r\frac{dz}{z} &\text{if $k = -1$,}\\
      \left(z^{k+1} +r\right)\frac{dz}{z} &\text{if $k < -1$.}
    \end{cases}
  \end{equation}
The germ of $\phi$ is unique up to multiplication by a $(k+1)$-st root of unity
when $k \geq 0$, and up to multiplication by a non-zero constant if $k=-1$.
For $k<-1$ the map~$\phi$ is uniquely determined
by~\eqref{eq:standard_coordinates} and the specification of the image
of some point~$p$ in $\Delta_\ve \setminus \{0\}$. Moreover, if $\phi$
satisfies~\eqref{eq:standard_coordinates} and $\phi(p) = q$, then
there exists a neighborhood~$U$ of~$q$ such that for every~$\wt{q} \in U$
there exists a map $\wt{\phi}$ satisfying~\eqref{eq:standard_coordinates}
and with $\wt{\phi}(p) = \wt{q}$.
\end{thm}
\par
This statement also holds for families of differentials~$\omega_\bft$ on
families of disks, as long as the order $\ord_0 \omega_\bft = k$ is the same
for all $\bft$. For families of differentials such that the order
$\ord_0 \omega_\bft$ is not constant, the situation is more complicated.
This has essentially been dealt with
in \cite[Section~4.2]{strata}, and we implement here two minor generalizations.
First, the differential is given a priori only over an annulus, and second,
the locus where the differential is assumed to have the normal form is an arbitrary
closed subvariety of some open ball $U \subset \CC^N$. Let $A_{\delta_1,\delta_2}
\coloneqq \{z:\delta_1 < |z| < \delta_2\} \subset\Delta_{\delta_2}$ be an annulus
\index[other]{$\bfe(z) \= \exp(2\pi \sqrt{-1} z)$}
and let $\zeta_j\coloneqq \bfe(j/(k+1))$ be a $(k+1)$-st root of unity (where we denote $\bfe(z) \= \exp(2\pi \sqrt{-1} z)$).
\par
\begin{thm}[{\bf Normal form of a deformation on an annulus}]
 \label{thm:deformed_standard_coordinates}
Let~$\omega_{\bft}$ be a holomorphic family of nowhere vanishing holomorphic
differentials on $ U\times A_{\delta_1,\delta_2}$ such that its restriction
over a closed complex subspace $Y \subset U$ is in normal
form~\eqref{eq:standard_coordinates}.
\par
\changed{Given a basepoint $p\in A_{\delta_1,\delta_2}$ and a holomorphic
map $\varsigma\colon U\to A_{\delta_1,\delta_2}$ such that $\varsigma(Y) =
\zeta_j p$,}  there exists a neighborhood $U_{\bfzero}\subset U$ of $Y$,
together with $\delta_1 < \ve_1 < \ve_2 <\delta_{2}$ and a holomorphic
map $\phi\colon U_{\bfzero} \times A_{\ve_1,\ve_2}\to A_{\delta_1,\delta_2}$
such that $\phi_{\bft}^* (\omega_\bft)$ has normal
form given by~\eqref{eq:standard_coordinates}, and such that $\phi|_{Y \times A_{\ve_1,\ve_2}}$
is the inclusion of annuli composed with multiplication by $\zeta_j$, and
such that $\phi_\bft(p) = \varsigma(\bft)$ for all $\bft\in U_{\bfzero}$.
\end{thm}
\par
\medskip
We now pass to families where the topology of the underlying Riemann
surfaces changes.  We establish below the existence of a normal form in a neighborhood
of a node, which will be used in the unplumbing construction of Proposition~\ref{prop:unplumbconst}.
Fix some arbitrary complex (base) space $B$, possibly singular and possibly non-reduced,
with a base point $p \in B$. Any family of Riemann surfaces  over~$B$ with at worst nodal
singularities can be  locally embedded in $\widetilde V=\widetilde V_\delta =
\Delta^2_\delta \times B$, for some radius~$\delta$, where the family is given by
$V(f,\delta)=\lbrace uv=f\rbrace$, where $f$ is a holomorphic function on~$B$ and where~$u$ and~$v$ are the two
coordinates on the disk (see \cite[Proposition~X.2.2.1]{acgh2}). For simplicity we sometimes write $V(f)$ or $V$ for $V(f, \delta)$ when there is no confusion.
\index[family]{b010@$uv=f$!Local equation of a nodal family}We denote  the ``upper'' component of the nodal fibers by $X^+ = \{f=0, v=0\}$,
and the ``lower" component by  $X^- = \{f=0, u=0\}$ respectively. The next statement gives a local normal form
for a family of differentials on~$V$ near the nodal locus $X^+\cap X^-$. 
\par
\begin{thm}[{\bf Normal form near vertical nodes}] \label{thm:NF}
Let~$\omega$ be a family of holomorphic differentials on~$V$,
not identically zero on every irreducible component of~$V$,
which does not vanish at a generic point of $X^{+}$ and \changed{restricted to each nodal fiber vanishes
to order exactly~$k = \kappa-1 \geq 0$ at the node in the nodal locus $X^+\cap X^-$}. Suppose that $f^\kappa$ is not identically zero and that there exists an adjusting function~$h$ on~$B$ such that
$\omega = h \eta$ for some family of
meromorphic differentials~$\eta$ on~$V$, which is holomorphic away from~$X^+$ and
nowhere zero.
\par
Then for some~$\ve >0$, after restricting~$B$
to a sufficiently small neighborhood of~$p$,  there exists an $r\in\calO_{B,p}$ divisible by $f^\kappa$, and a change of coordinates
$\phi\colon V(f,\ve) \to V(f,\delta)$, which lifts the identity map of $B$ to itself, such that
\begin{equation}\label{eq:ordkNF}
\phi^*\omega \= (u^\kappa +r) \frac{d u} u\,.
\end{equation}
Moreover, given a
section $\varsigma_0\colon B \to V$ and an (initial) section $\varsigma$ that both
map to $X^-$ along $f=0$ and with $\varsigma$ sufficiently close to $\varsigma_0$,
there exists a unique change of coordinates~$\phi$ as above
that further satisfies $\phi \circ \varsigma = \varsigma_0$.
\end{thm}
The notion of adjusting function will be formally defined and used later,
see Definition~\ref{def:rescalable}. We split the proof in several steps.
\par
\begin{lm}
 \label{lm:first_reduction}
Under the assumption of Theorem~\ref{thm:NF}, the following statements hold:
\begin{enumerate}[(i)]
\item There exists a holomorphic function~$g$ on $V$ such that we can write $\omega = u^\kappa g(u,v) \tfrac{du}u$. Moreover, $g$ can be taken with constant term $1$
after rescaling $u$ by a unit.
\item Up to multiplying $\eta$ by a unit, we can assume that $h=f^\kappa$.
\item We have $f^\kappa\ |\ r$.
\end{enumerate}
\end{lm}
\par
\begin{proof}
We will see that the second and third statements follow from the proof of the first one.
Using the defining equation of~$V$ and the fact that~$\omega$ is holomorphic, we
can expand~$\omega$ in series as
\begin{equation}\label{eq:defomega}
\omega \= \Bigl(\sum_{i \geq 0} c_i u^i + \sum_{i > 0} c_{-i}v^i \Bigr)
\frac{du}u
\end{equation}
for some local functions $c_i,c_{-i}$ on $B$. An arbitrary holomorphic function~$g$
on~$V$ can be uniquely written, possibly after shrinking the neighborhood to
guarantee convergence, as a series $g = \sum_{i \geq 0}a_i u^i + \sum_{i > 0} b_{-i}v^i$,
so that our goal is to write~$\omega$ as
\begin{equation}\label{eq:ugoal}
\omega \= u^\kappa g(u,v) \frac{du}u \= \Bigl(\sum_{i \geq \kappa} a_{i-\kappa}  u^i
+\sum_{0 \leq i < \kappa} b_{i-\kappa} f^{\kappa-i} u^i +  \sum_{i > 0} b_{-\kappa-i}f^\kappa v^i \Bigr) \frac{du}u\,.
\end{equation}
Since~$\eta$ is holomorphic outside the locus $v=0$, we can also expand it as
\begin{equation}
\eta \= \Bigl(\sum_{i \geq 0} e_i v^{-i} + \sum_{i > 0} e_{-i}v^i \Bigr) \frac{du}u\,.
\end{equation}
The hypothesis on the vanishing order of~$\omega$  implies that $c_i(p)=0$ for
$0 \leq i <\kappa$, but $c_\kappa(p) \neq 0$. We consider the equation $\omega =h\eta$
near $X^-$ and write $u^i = f^i v^{-i}$ in the defining power series~\eqref{eq:defomega}
of~$\omega$. Comparing the $v^{-\kappa}$ terms gives $c_\kappa f^\kappa = h e_\kappa$,
hence $h \mid f^\kappa$. On the other hand, the winding number argument as in the
proof of \cite[Theorem~1.3]{strata} implies that $e_\kappa(p) \neq 0$, so
that $f^\kappa \mid h$. (If~$B$ is topologically just a point, we can take any
lift of the family to a polydisk, run the argument there and the conclusion
persists after reduction.) Changing $\eta$ by a unit in~$\calO_{B,p}$, we can assume that $h=f^\kappa$ from now on, thus verifying (ii). Coefficient
comparison of the terms $v^i$ for $i>0$ in the equality $\omega =h\eta$ now implies that
$c_{-i} = f^\kappa e_{-i}$. It also implies that $f^i c_i = f^\kappa e_i$ for $i\geq 0 $.
Since~$f^i$ is a non-zero function on $B$ for those~$0 \leq i <\kappa$ by the non-vanishing
hypothesis of~$\omega$, this implies the remaining divisibility condition $f^{\kappa - i}
\mid c_i$ for $0 \leq i <\kappa$ needed for making~\eqref{eq:ugoal} equal
to~\eqref{eq:defomega}.
\par
The form of~$\omega$ we derived so far implies that the residue of~$\omega$ is equal
to $r=b_{-\kappa}f^\kappa$, which is in particular divisible by~$f^\kappa$, hence proving (iii).
\par
Finally we can multiply $u$ by a unit and $v$ by the inverse of the unit to make the constant term $a_0 = 1$ in $g$, thus completing the entire proof.
\end{proof}
\par
We write $r = r_0 f^\kappa$ from now on.
\par
\begin{lm}
We may assume that~$B$ is a polydisk.
\end{lm}
\par
\begin{proof}
Any (possibly reducible and non-reduced) analytic space can be embedded
locally into a polydisk. We thus replace~$f$ by any of its lifts to such a
polydisk. To put~$g$ into the form~\eqref{eq:ordkNF} we may assume that $B$ is
a polydisk with coordinates~$\bfb$ and zero is the base point. The coordinate
change~$\phi$ that puts the differential in the normal form over the polydisk
then restricts to a coordinate change over~$B$ with the desired properties.
\end{proof}
\par
\begin{proof}[Proof of Theorem~\ref{thm:NF}]
  We look for a solution of the form
  \begin{equation}
    \label{eq:def_phi}
    \phi_{(X,Y)}(u,v) \= (u e^{X(u) + Y(v)}, v e^{-X(u) - Y(v)})\,,
  \end{equation}
where $X(u) = \sum_{i=1}^\infty c_i(b) u^i$ is a holomorphic function of~$u$
and $b$ with no constant term, and similarly for $Y(v)$.  (In the sequel,
for a holomorphic function of $u,v$, and~$b$, the dependence on~$b$ will be
left implicit.)
\par
We first remark that the uniqueness of~$\phi$ follows from the observation
that any two holomorphic maps with the same pullback of a differential and
that agree at a marked point in the regular locus of the differentials agree
everywhere. This marked point is given by the section~$\varsigma$ over
$f \neq 0$.  Consequently, if $\phi_1$ and $\phi_2$ both satisfy the
hypothesis of the theorem, then $\phi_1 \circ \phi_2^{-1}$ is identity on the
locus in the family where $f \neq 0$, and hence $\phi_1 = \phi_2$ everywhere.
\par
By Lemma~\ref{lm:first_reduction}, we may write the relative form $\omega$ as
  \begin{equation*}
    \omega \= u^\kappa(1 + r_0 v^\kappa + g_0(u) + h_0(v)) \frac{du}{u}\,,
  \end{equation*}
where $g_0$ is a function of $u$ and $b$ with no constant term, and $h_0$
is a function of~$v$ and~$b$ with no constant term or $v^\kappa$-term.
\par
We first make a preliminary change of coordinate $\psi$ so that
$\psi^*\omega = \omega_0$, where
  \begin{equation*}
    \omega_0 \= u^\kappa(1 + r_0 v^\kappa + f g(u) + f h(v)) \frac{du}{u}\,.
  \end{equation*}
  This may be done by taking functions $\alpha(u) = u e^{A(u)}$ and $\beta(v) = v e^{-B(v)}$ such
  that (possibly after shrinking $\epsilon$) on $\Delta_\epsilon\times B$,
  \begin{eqnarray*}
    \alpha^* u^\kappa(1+ g(u))\frac{du}{u} & = & u^\kappa \frac{du}{u}\,, \quad\text{and}\\
    \beta^* v^{-\kappa}(1 + r_0 f^\kappa + h(v) )\frac{dv}{v} & = & v^{-\kappa}(1 + r_0 f^\kappa )\frac{dv}{v}\,,
  \end{eqnarray*}
using Strebel's normal form, Theorem~\ref{thm:standard_coordinates}.
Then it is straightforward to check that $\phi_{(X,Y)}(u,v) = (u e^{X(u) + Y(v)},
v e^{-X(u) - Y(v)})$ is of the desired form.
\par
  We now wish to find functions $X(u)$ and  $Y(v)$  so that
  $\phi_{(X,Y)}^* u^\kappa (1 + r_0 v^\kappa)\frac{du}{u} = \omega_0$, and
  $\phi_{(X,Y)}\circ \varsigma_0 = \varsigma$.  Explicitly this means that on $\Delta^2_\epsilon\times B$, the
  functions $X$ and $Y$ satisfy the equations,
  \begin{eqnarray*}
       (e^{\kappa(X + Y)} + r_0 v^\kappa)\left(1+ u  \frac{\partial X}{\partial u} - v
         \frac{\partial Y}{\partial v}\right) - (1 +v^\kappa + f g + f h) + (uv -f)W & = & 0\,, \\
     \tau_0 e^{-X(f/\tau_0) - Y(\tau_0)} - \tau & = & 0\,,
  \end{eqnarray*}
  where $W(u,v)$ is a holomorphic function on $\Delta^2_\epsilon$, and where the sections $\varsigma$ and $\varsigma_0$ are written as
  \begin{equation*}
\varsigma \= (f/\tau, \tau) \quad\text{and}\quad \varsigma_0
\= (f/\tau_0, \tau_0)
\end{equation*}
for some nowhere zero functions $\tau$ and $\tau_0$ on $B$.  Our approach
to solving these  equations will be by perturbing the trivial solution
$X=Y=0$ when $g=h=0$ and $f=0$ via the Implicit Function Theorem.  To do this,
we introduce an auxiliary complex parameter~$s$ and the
  rescaling maps $\rho_s(b) = sb$ on $B$ and $\tilde{\rho}_s(u,v,b) = (u,v,sb)$, so that we have
  the commutative diagram:
  \[\begin{tikzpicture}
\matrix (m) [matrix of math nodes, row sep=2.10em, column sep=3.5em,
text height=1.5ex, text depth=0.25ex]
{V(f) &V(f) \\
V(f\circ \rho_{s}) & V(f\circ \rho_{s})   \\};

\path[->,font=\scriptsize]
(m-2-1) edge [below] node {$\varphi_{(X,Y)}$} (m-2-2)
(m-1-1) edge[left] node {$\tilde\rho_{s}$} (m-2-1)
(m-1-1) edge [above] node {$\varphi_{(X,Y)}$} (m-1-2)
(m-1-2) edge [left] node {$\tilde\rho_{s}$} (m-2-2);
\end{tikzpicture}\]
Solving the original equations is then equivalent to solving on the polydisk
  $\Delta^2_\epsilon\times B$ the equations
  \begin{align*}
    \Phi_1(W,X,Y,\tau,s) &\= (e^{\kappa(X(u) + Y(v))} + (r_0\circ\rho_s)  v^\kappa)\left(1+ u
                           \frac{\partial X}{\partial u} - v
                           \frac{\partial Y}{\partial v}\right)
    \\
                         & \hfill\,-\, (1 +v^\kappa + (fg)\circ \tilde{\rho}_s +
                           (fh)\circ\tilde{\rho}_s) + (uv - f\circ \rho_s)W  \= 0\,,\\
    \Phi_2(W, X, Y, \tau, s) &\= \tau_0 e^{-X(f/\tau_0) - Y(\tau_0)} - \tau=0
  \end{align*}
  for any nonzero $s$. (Note that only the first equation has
  been rescaled.)
\par
  We fix some notation for the Banach spaces we need. Let $\banach{M}{m}$ denote the Banach space of
  holomorphic functions on $M$ whose first $m$ derivatives are uniformly bounded, equipped with the
  $C^m$-norm $\|F\|_m \,\coloneqq\, \sum_{j=0}^m \sup_{z\in M} |F^{(j)}(z)|.$ We let
  $U_B = \Delta_\epsilon\times B$, $V_B= \Delta_\epsilon\times B$, and $\widetilde{V}=
  \Delta^2_\epsilon\times B$ be polydisks with
  coordinates~$(u,b), (v,b)$, and~$(u,v,b)$ respectively. An upper index ${\rm nc}$ will refer to functions
  without constant term (in~$u$ resp.\ in~$v$) and an upper index ${\rm nr}$ (``no residue'') will refer
  to functions without $v^\kappa$-term.
\par
  In this notation we can view $\Phi=(\Phi_1,\Phi_2)$ as a map
  \begin{equation*}
    \Phi\colon \banach{\widetilde{V}}{0} \oplus \banach{U_B}{1}^{\rm nc} \oplus \banach{V_B}{1}^{\rm nc}
  \oplus  \banach{B}{0} \oplus \CC \to \banach{\widetilde{V}}{0}^{\rm
    nc, nr}\oplus \banach{B}{0}\,,
  \end{equation*}
  where the domain summand parameterize $W, X, Y, \tau$, and~$s$ respectively. In order to apply
  the Implicit Function Theorem, we need to show that
  \begin{equation*}
    D_1 \Phi\colon \banach{\widetilde{V}}{0} \oplus \banach{U_B}{1}^{\rm nc} \oplus \banach{V_B}{1}^{\rm nc}\to
    \banach{\widetilde{V}}{0}^{\rm nc,nr} \oplus \banach{B}{0}
  \end{equation*}
  is an isomorphism.  Here $D_1\Phi$ refers to the derivative at $(0,0,0, \tau_0, 0)$ with
  respect to $W, X$, and $Y$.
  This derivative is given explicitly by
  \begin{align} \label{eq:DPhi}
    D_1 \Phi_1(W, X, Y) &= W \cdot uv + \left(\kappa X + u(1 + r_0(0)v^\kappa) \frac{\partial X}{\partial u}\right)+ \left(\kappa Y - v(1+ r_0(0) v^\kappa)\frac{\partial Y}{\partial v}\right),\\
    D_1\Phi_2(W,X,Y) &= -\tau_0 X(f/\tau_0) - \tau_0 Y(\tau_0)\,.  \nonumber
  \end{align}
\par
  We will show that $D_1\Phi$ is an isomorphism by constructing an explicit inverse,
  \begin{equation*}
    S \colon \banach{U_B}{0}^{\rm nc} \oplus \banach{V_B}{0}^{\rm nc,nr} \oplus \banach{\widetilde{V}}{0} \oplus \banach{B}{0} \to\banach{\widetilde{V}}{0} \oplus \banach{U_B}{1}^{\rm nc} \oplus \banach{V_B}{1}^{\rm nc}\,,
  \end{equation*}
  identifying $\banach{\widetilde{V}}{0}^{\rm nc,nr}$ with $\banach{U_B}{0}^{\rm nc} \oplus
  \banach{V_B}{0}^{\rm nc,nr} \oplus \banach{\widetilde{V}}{0}$ by decomposing any holomorphic function
  in $\banach{\widetilde{V}}{0}^{\rm nc, nr}$ uniquely as $\aleph(u) + \beth(v) + \daleth(u,v) uv$.
\par
  We define bounded operators
  $S_X\colon \banach{U_B}{0}^{\rm nc}\to \banach{U_B}{1}^{\rm nc}$ and
$S_Y\colon \banach{V_B}{0}^{\rm nc,nr}\to \banach{V_B}{1}^{\rm nc}$ to be the solutions
to the differential equations
  \begin{align}
    \label{eq:diffX}
    \kappa X + u \frac{\partial X}{\partial u} &\= \aleph\,,\\
    \kappa Y - v( 1+ r_0(0) v^\kappa)  \frac{\partial Y}{\partial v} &\= \beth \label{eq:diffY}\,,
  \end{align}
  obtained from the $X$- and $Y$-components of \eqref{eq:DPhi} by deleting terms containing $uv$.
  Solving these equations explicitly using the method of integrating factors (see \cite{euler1732}) yields
  \begin{align*}
    S_X(\aleph) &\= \frac{1}{u^\kappa}\int u^{\kappa-1} \aleph \, du\,,\\
    S_Y(\beth) &\= \frac{-v^\kappa}{1 + r_0(0) v^\kappa} \int \frac{\beth}{v^{\kappa+1}} \, dv\,,
  \end{align*}
  where each antiderivative is chosen to have no constant term.  The second antiderivative exists
  because $\beth$ was assumed to have no $v^\kappa$ term.  The differential operator,
  \begin{equation*}
    T(X) =  \kappa X + u (1 + r_0(0) v^\kappa)\frac{\partial X}{\partial u}\,,
  \end{equation*}
  which is the $X$-component of \eqref{eq:DPhi}, then satisfies
  \begin{equation*}
    TS_X(\aleph) \= r_0(0)uv^k \frac{\partial S_X(\aleph)}{\partial u}  = E(\aleph)\,,
  \end{equation*}
  where
  \begin{equation*}
    E(\aleph) \= -\kappa r_0(0) v^{\kappa} S_X(\aleph)\,,
  \end{equation*}
  which is divisible by $uv$.
  Finally, we define $S$ by
  \begin{equation*}
    S(\aleph, \beth, \daleth, \tau) = \left( \daleth - \frac{1}{uv}E(\aleph), \ S_X(\aleph),\ S_Y(\beth) + C(\aleph, \beth, \tau)\mu(v)\right)\,,
  \end{equation*}
  where
  \begin{equation*}
    \mu(v) \= \frac{v^\kappa}{1 + r_0(0) v^\kappa}
  \end{equation*}
  is the kernel of the left-hand side of \eqref{eq:diffY}, and
  \begin{equation}
    \label{eq:defC}
    C(\aleph, \beth, \tau) \= - \frac{\tau+\tau_0 S_X(\aleph)(f/\tau_0) + \tau_0 S_Y(\beth)(\tau_0)}{\tau_0\mu(\tau_0)}
  \end{equation}
  is chosen so that $D_1 \Phi_2 \circ S(\aleph, \beth,  \daleth, \tau) = \tau$.  Note that since
  $\tau_0(0)\neq 0$, we may
  assume that the denominator $\tau_0\mu(\tau_0)$ of \eqref{eq:defC} is nonzero by possibly shrinking $B$.
\par
We then know that $D_1\Phi$ is surjective, since it has a right inverse.  Injectivity
of $D_1\Phi$ is easily checked, using that the solutions to~\eqref{eq:diffX}
and~\eqref{eq:diffY} are
  unique up to the kernel of \eqref{eq:diffY}, which is of the form $C\mu(v)$, and once $X$ and $Y$
  are fixed, there is a unique function $C$ such that $D_1\Phi_2 = 0$.
\par
  We can now apply the Implicit Function Theorem in a neighborhood of~$(s, \tau) = (0, \tau_0)$ to
  obtain functions $X_{s \tau},Y_{s \tau},W_{s \tau}$ with
  $\Phi(X_{s\tau},Y_{s \tau},W_{s\tau},\tau,s) = 0$.  Since $\phi_0$ is the identity, thus
    mapping $V(f,\ve)$ into $V(f,\delta)$, this inclusion  still holds for $(s,\tau)$ sufficiently
  small.  Consequently the map $\phi$ we constructed maps into $V(f,\ve)$ as required.
\end{proof}
\par
\begin{rem}
The change of coordinates~$\phi$ may also be represented as an explicit formal power
series via the following Ansatz, as a function of the form
\begin{equation}\label{eq:phiAnsatz}
  \phi(u,v) \= (u(1+Z)e^{X(u)+Y(v)},v(1+Z)^{-1}e^{-X(u)-Y(v)})\,,
\end{equation}
where $X(u)$ and $Y(v)$ are holomorphic functions as before, with expressions
$X(u) = \sum_{i>0} c_i u^i$ and $Y(v) = \sum_{i>0} d_i v^i$ respectively.
Here the $c_i$, $d_i$, and~$Z$ are holomorphic functions on $B$.left-hand side
Equation~\eqref{eq:ordkNF} is then equivalent to
\begin{equation}
(u^{\kappa}(1+Z)^\kappa e^{\kappa(X(u) + Y(v))} + r) (1+ uX'(u) - vY'(v)) \= u^{\kappa}g \,.
\end{equation}
\par
A formal solution of this differential equation can be constructed recursively.
We begin with solving the equation mod~$f$. The $v^i$-terms and the $u^j$-terms
for $j \leq \kappa$ are zero mod~$f$ on both sides. The $u^{\kappa}$-term implies
$Z=0 \mod f$. The $u^{\kappa+j}$-term involves a linear equation for $c_j \mod f$
with leading coefficient~$\kappa+j$ for $j > 0$. Next we solve mod~$f^2$, where
the $u^{\kappa-1}$-term gives a linear equation for $b_1 \mod f$. The coefficient
$Z \mod f^2$ is linearly determined by the $u^{\kappa}$-term mod $f^2$ and
the $u^{\kappa+j}$-term mod~$f^2$ linearly determine $c_j \mod f$. In the third round,
considering terms mod~$f^3$, we start with the $u^{\kappa-2}$-term,
which determines $b_2 \mod f$, then consider the $u^{\kappa-1}$-term to determine
$b_1 \mod f^2$. The $u^{\kappa}$-term and higher terms to compute $Z \mod f^3$ and
then the $c_j \mod f^3$. This clearly determines an algorithm, starting at
the $u^{\kappa-n}$-term at the step~``$\mathrm{mod}\, f^n$'',
where the consideration of a term~$u^{-i}$ should be read as the $v^{i}$-term.
The $u^0$-term determines there residue, but imposes no condition on $b_\kappa$
(since it appears with coefficient~$\kappa-\kappa$).
Making an arbitrary choice for that coefficient, the algorithm can be continued as
indicated. This choice can be used to adjust the section~$\varsigma$.
\par
\end{rem}

\par
The corresponding statement for horizontal nodes is a direct
adaptation of \cite[Lemma~7.4]{BHM}. In fact, the proof given there
uses no geometry of the base, and the convergence of
the given formal solution follows from straightforward estimates.
\par
\begin{prop}[{\bf Normal form near horizontal nodes}] \label{thm:NFhoriz}
Let~$\omega$ be a family of holomorphic differentials on~$V$,
whose restriction to the components $X^+$ and $X^-$
of the central fiber both have a simple pole at the nodal locus $X^+\cap X^-$.
\par
Then for some~$\ve >0$ there exists, after restricting~$B$
to a sufficiently small neighborhood of~$p$, a change of coordinates
$\phi\colon V(f,\ve) \to V(f,\delta)$ such that it is the identity on~$B$ and such that
\begin{equation}\label{eq:ord1NF}
\phi^*\omega \= r \frac{d u} u\,
\end{equation}
\changed{where $r$ is the residue of $\omega$ (divided by $2\pi i$) restricted to the horizontal node.} 
Moreover, given a
section $\varsigma_0\colon B \to V$ and an (initial) section $\varsigma$
that both map to $X^-$ along $f=0$ and with $\varsigma$ sufficiently
close to $\varsigma_0$, there is a unique change of coordinates~$\phi$ as above
that further satisfies $\phi \circ \varsigma = \varsigma_0$.
\end{prop}

%%%%%%%%%%%%%%%%%%%%%%%%%%%%%%%%%%%%
\section{Prong-matched differentials}
\label{sec:BSProng}
%%%%%%%%%%%%%%%%%%%%%%%%%%%%%%%%%%%%%

In this section we construct the \Teichmuller space $\ptwT$ of \ptwds
as a topological space. Subsequently the augmented \Teichmuller
space of flat surfaces will be constructed as a union of quotients of such
spaces $\ptwT$. Along the way, we introduce the key notions of degenerations
of multicurves, \prmas and weldings as well as several
auxiliary \Teichmuller spaces.
\par
To avoid overloading this section, we define in this section the points in
the moduli spaces by specifying the objects they represent. All these objects
have a natural notion of deformation that endows those spaces with a topology
that we address in Section~\ref{sec:AugTeich}, as well as modular
interpretations that we address in Section~\ref{sec:famnew}.

%%%%%%%%%%%%%%%%%%%%%%%%%
\subsection{Ordered and enhanced multicurves and their degenerations.}
\label{sec:order}
%%%%%%%%%%%%%%%%%%%%%%%%%

We continue to fix an $n$-pointed topological surface
$(\Sigma, \bfs)$, as in Section~\ref{sec:classaugteich}.  To every
multicurve $\Lambda \subset \Sigma\setminus\bfs$, we can associate the
dual graph $\Gamma(\Lambda)$ whose vertices correspond to connected
components of $\Sigma\setminus\Lambda$, whose edges correspond to
curves in~$\Lambda$, and whose half-edges correspond to the marked
points.  \index[graph]{b080@$\Lambda$!Multicurve in $\Sigma$} In the
setting of multicurves, we will generally imitate the standard
notation for level graphs from Section~\ref{sec:graphetc}.  We call
$\oL = (\Lambda,\ell)$ an \emph{ordered multicurve} and specify the
ordering relation between the components of $\Sigma\setminus\Lambda$
by~$\cleq$. The notions \emph{horizontal} and \emph{vertical} are
defined similarly. A multicurve is \emph{purely vertical}
(resp.~\emph{purely horizontal}) if all of its curves are vertical
(resp.~horizontal) edges of $\Gamma(\Lambda)$.
\par
An {\em enhanced multicurve}~$\eLp$ is a multicurve~$\eL$ such that
the associated graph~$\Gamma(\eL)$ has been provided with the extra
structure of an enhanced level graph.
\index[graph]{b090@$\eLp$, $\eL$! Enhanced multicurve}
In order to keep the notation simple, we will mostly denote an  enhanced
multicurve simply by~$\eL$.
\index[graph]{b100@$\eGp(\eLp)$, $\eL$! Enhanced graph associated to the enhanced multicurve $\eL$}
Moreover, by an abuse of notation, the enhanced level graph $\eGp(\eL)$
associated to~$\eL$ will be denoted by~$\eGp$, and often simply by~$\Gamma$.
\par
We adapt many notions for graphs to the context of multicurves. We denote by~$L^\bullet(\overline\Lambda)$
the set of all levels of the level graph associated to the multicurve, and call this set normalized
if $L^\bullet(\oL)=\lbrace 0,\dots,-N\rbrace$. We denote by $L(\oL)=L^\bullet(\oL)\setminus\lbrace 0\rbrace$
the set of all levels except the top one. We denote $\gamma_e$ the curve of~$\Lambda$ corresponding to an
edge~$e$ of $\Gamma(\Lambda)$, and for $i\in L^\bullet(\oL)$ call the union of the connected
components of $\Sigma\setminus\eL$ at level~$i$ the \emph{level~$i$
subsurface} $\Sigma_{(i)}\subset\Sigma$. Denote $\Sigma_v\subset \Sigma$ the subsurface corresponding to the vertex~$v$.
We write  $\Sigma_v^c$ and $\Sigma_{(i)}^c$ for the corresponding compact surfaces
where the boundary curves have been collapsed to points.
\par
\begin{df}
  \label{def:degeneration}
Suppose $(\eL_1, \ell_1)$ and $(\eL_2, \ell_2)$ are ordered multicurves
on a fixed pointed topological surface.  We say that $(\eL_1, \ell_1)$ is a
\emph{degeneration} of $(\eL_2, \ell_2)$ (or $\eL_2$ is an
\emph{undegeneration of $\eL_1$}), and we denote \changed{this binary relation} by $\degen\colon (\eL_2, \ell_2) \rightsquigarrow (\eL_1, \ell_1)$,
\index[graph]{b110@$\degen\colon \eL_2 \rightsquigarrow \eL_1$! Degeneration of the ordered multicurve $\eL_2$}
if the following conditions hold:
  \begin{itemize}
  \item  As a set of isotopy classes of curves $\eL_2 \subset \eL_1$.  Let then $\delta \colon \Gamma(\eL_1)\to \Gamma(\eL_2)$ be the simplicial homomorphism induced by the inclusion $\Sigma\setminus\eL_1\hookrightarrow \Sigma\setminus\eL_2$.
More concretely, the map $\delta$ is defined by collapsing every edge of $\Gamma(\eL_1)$ corresponding to a curve in $\eL_1 \setminus \eL_2$.
  \item The map $\delta$ is compatible with the orders $\ell_j$  in
    the sense that if  $v_1\cleq v_2$ then  $\delta(v_1) \cleq
    \delta(v_2)$.  It follows that  if $v_1 \asymp    v_2$ then
    $\delta(v_1) \asymp \delta (v_2)$, so $\delta$ induces a
    surjective, order non-decreasing map, still denoted by $\delta$,
    \index[graph]{b120@$\delta\colon \uN
      \twoheadrightarrow \uM$! Map defining a vertical undegeneration}
    on the (normalized) sets of levels $\delta\colon L^\bullet(\eL_1) \twoheadrightarrow
    L^\bullet(\eL_2)$.
    \item The map $\delta$ respects the labeling of the
      half-edges.
  \end{itemize}
The notion of degeneration of ordered multicurves extends to a notion
of \emph{degeneration of enhanced multicurves} by requiring that moreover
the map~$\degen$ preserves the weights~$\kappa_e$ of the edges $e$ that are not contracted.
\end{df}
As these constraints are all phrased only in terms of the dual graphs
of these multicurves, there is an analogous notion of a degeneration
of enhanced level graphs.
\par
We alert the reader that there are non-trivial degenerations that increase
the number of levels without changing the underlying multicurve, see Figure~\ref{fig:degen}.
  \begin{figure}[htb]
   \centering
\begin{tikzpicture}[scale=1]
\coordinate (a1) at (4,0);\fill (a1) circle (2pt);\node [] at (6.5,0) {$\ell_{1}=0$};
\coordinate (a2) at (3,-1);\fill (a2) circle (2pt);\node [] at (6.5,-1) {$\ell_{1}=-1$};
\coordinate (a3) at (4,-2);\fill (a3) circle (2pt);\node [] at (6.5,-2) {$\ell_{1}=-2$};
\coordinate (a4) at (5,-3);\fill (a4) circle (2pt);\node [] at (6.5,-3) {$\ell_{1}=-3$};
\draw (a1) -- (a2);
\draw (a1) -- (a3);
\draw (a1) -- (a4);

\draw[->,decorate,decoration={snake,amplitude=.4mm,segment length=2mm,post length=1mm}] (1,-.5) -- (2.8,-.5) coordinate[pos=.5] (b);
\node[above] at (b) {$\degen$};

\coordinate (a1) at (0,0);\fill (a1) circle (2pt);\node [] at (-2.5,0) {$\ell_{2}=0$};
\coordinate (a2) at (-1,-1);\fill (a2) circle (2pt);\node [] at (-2.5,-1) {$\ell_{2}=-1$};
\coordinate (a3) at (0,-1);\fill (a3) circle (2pt);
\coordinate (a4) at (1,-1);\fill (a4) circle (2pt);
\draw (a1) -- (a2);
\draw (a1) -- (a3);
\draw (a1) -- (a4);
\end{tikzpicture}
\caption{A degeneration that does not change the underlying multicurve.}
\label{fig:degen}
\end{figure}
\par
There are two kinds of undegenerations of $\eL_1$. First, for any subset $
D^\hor \subseteq \eL_1^\hor$ of the set of horizontal curves we can define a
 \index[graph]{b130@$D^\hor \subseteq \eL_1^\hor$! Subset of horizontal curves inducing a horizontal undegeneration}
{\em horizontal undegeneration} of $\Lambda_1$ by taking
$\eL_2\coloneq \eL_1 \setminus D^\hor$ and \changed{defining $\delta$ to contract all horizontal edges in $D^{\hor}$.}
%$\delta = {\rm id}$. 
Geometrically this
undegeneration smoothes out the horizontal nodes corresponding to $D^\hor$.
Second, suppose that $\eL_1$ has~$N+1$ levels. Then
any surjective, order non-decreasing map $\delta\colon \uN
\twoheadrightarrow \uM$ defines a {\em vertical undegeneration}
$\eL_2 \rightsquigarrow \eL_1$ of~$\Lambda_1$ as follows. We denote by \changed{$\Sigma_{(i)},\Sigma_{(j)}$ the level subsurfaces for $\eL_1$, and} let $\eL_2 \subseteq \eL_1$ be the multicurve obtained by deleting all curves
that lie in the boundaries of $\Sigma_{(i)}$ and $\Sigma_{(j)}$ for $i\ne j$ such that
$\delta(i) = \delta(j)$.
The level structure on $\eL_2$ is obtained by collapsing to a point
every edge joining levels~$i$ and~$j$ such that $\delta(i)=\delta(j)$.
Note that every ordered multicurve that is an undegeneration of~$\eL_1$ is obtained
uniquely as the composition of a vertical undegeneration and a horizontal
undegeneration.
Consequently, we refer to an undegeneration by the
 \index[graph]{b140@$(\delta,D^\hor)$, $\delta$! Undegeneration of an enhanced multicurve}
symbol $(\delta,D^\hor)$ or simply by~$\delta$.
\par
There is another way to encode vertical degenerations.  Consider \changed{$J = \{ j_{-1}, \dots, j_{-M}\}$ such that $0>j_{-1} > \dots > j_{-M} \geq -N$.}
% a decreasing sequence $J = \{ 0>j_{-1} > \dots > j_{-M} \geq -N\}$.  
We define
$j_0=0$ and $j_{-M-1} = -N-1$ (though they are not part of $J$).  The subset $J$ induces a
map $\delta_{J}\colon\uN\to\uM$ which maps integers (i.e.\ levels) in each
interval $(j_{k-1}, j_{k}]$ to $k$.  We denote the associated degeneration by
$\degen_{J}\colon\eL_{J}\rightsquigarrow \eL$. The two-level degenerations given
by $J=\left\{i\right\}$, and denoted by $\degen_i$, will be particularly
\index[graph]{b150@$\degen_{J}$, $\delta_{J}$! (Un)degenerations associated
  with the subset $J$}
useful (see Section~\ref{sec:levrot}). In the example of Figure~\ref{fig:degen}, we have $J=\lbrace -1 \rbrace$. The level $(j_{-1}, j_{0}]=(-1,0]$ is mapped to $0$ and the levels $(j_{-2}, j_{-1}]=(-4,-1]$ are mapped to $-1$.

%%%%%%%%%%%%%%%%%%%%%%%%%%%%%%%%%%%%%
\subsection{The \Teichmuller space of \twds.} \label{sec:TeichTwds}
%%%%%%%%%%%%%%%%%%%%%%%%%%%%%%%%%%%%%

For a reference surface $(\Sigma,\bfs)$ let $\Oteich$ be the
\Teichmuller space of $(\Sigma,\bfs)$-marked flat surfaces of type~$\mu$
and let  $\proj \Oteich = \Oteich/ \CC^*$ be its projectivization.
\index[teich]{d010@$\Oteich$!\Teichmuller space of marked flat surfaces of type~$\mu$}
We define the subsets $P_\bfs$ and $  Z_\bfs $ of $\bfs$ to be the marked points such that
their images under $f$ in~$X$ are respectively poles and zeros of~$\omega$.
\index[surf]{b040@$P_\bfs, Z_\bfs$!Subset of $\bfs$ mapped respectively to the poles and zeros of~$\omega$}
The complex structure on $\Oteich$ is induced by the
\emph{global period map}
\begin{equation*}
  {\rm Per}\colon \Oteich \to H^1(\Sigma \setminus P_\bfs, Z_\bfs; \CC )\,,
\end{equation*}
which is locally biholomorphic (see e.g\ \cite{Veech}, \cite{HubbardMasur},
\cite{kdiff}).
\par
The mapping class group~$\Mod[(\Sigma,\bfs)]$ of $(\Sigma,\bfs)$ acts properly
discontinuously on $\teich$ and on the twisted Hodge bundle over it, preserving the
submanifold~$\Oteich$.  The spaces $\Oteich$ are highly disconnected, and we do not address here the
question of classifying their connected components.  Moreover, we do not claim that $\proj \Oteich$ is
simply connected.
\par
We next define similarly strata of flat surfaces over the boundary
components of the augmented \Teichmuller space. We start with an
auxiliary object that will play no major role further on.
The upper index ``no'' indicates that no GRC and no matching residue condition
at the horizontal nodes is imposed here. This is mainly
introduced to contrast with the space defined later, where the residue conditions {\em are}
required. Moreover, recall that we denote an enhanced multicurve $\eLp$ simply by $\eL$.
\par
\begin{df}
The \emph{\Teichmuller space $\OBnoGRC$
of flat surfaces of type $(\mu,\eL)$} is the space of tuples
$(X,f, \bfz, \eta)$ where $(X,f,\bfz)$ is a marked (in the sense of
Definition~\ref{def:marking}) pointed stable
curve with enhanced pinched multicurve~$\eL$ and where
$\eta = \{\eta_v\}_{v \in V(\eL)}$ is a collection
of not identically zero meromorphic one-forms of type~$\mu$ that have
order $\pm \kappa_e - 1$ at~$e^+$ and $e^-$, respectively, for any edge $e \in \Gamma(\eL)$.
\index[teich]{d020@$\OBnoGRC$!\Teichmuller space of flat surfaces of type~$(\mu,\eL)$ without GRC}
\end{df}
\par
To construct $\OBnoGRC$ as an analytic space, \changed{we take the product of the twisted Hodge bundles over the \Teichmuller spaces for the components of $\Sigma \setminus \eL$ (with additional labeled marked points that arise from half-edges in $\Lambda$) and take a finite quotient to get rid of the symmetries of the graph $\Gamma (\Lambda)$ (eg two parallel edges with the same enhancement).}
%we take a finite unramified cover of the product of the twisted Hodge bundles over the \Teichmuller spaces for the components of $\Sigma \setminus \eL$ that encodes the identification of the marked points that are paired to form nodes. 
Then the subset defined by the vanishing conditions of~$\eta$ along~$\bfz$
and pole orders at the nodes is the space~$\OBnoGRC$.
\par
The group $(\CC^*)^{V(\eL)}$
acts on $\OBnoGRC$ with quotient $\BPteich$, since the one-forms
$\eta_v$ are uniquely determined up to scale by the required
vanishing conditions encoded in an enhanced multicurve.
\par
\begin{df}
\label{df:teich}
The \emph{\Teichmuller space  $\OBteich$
of twisted differentials of type $(\mu,\eL)$} is the subset of $\OBnoGRC$
consisting of $(X,f, \bfz, \eta)$ where~$\eta$ is a twisted
differential compatible with $\Gamma(\eL)$.
\index[teich]{d030@$\OBteich$!\Teichmuller space of twisted differentials of type $(\mu, \Lambda)$}
\end{df}
Said differently, $\OBteich$ is the subset of $\OBnoGRC$ cut out by the condition of matching
residues at the horizontal nodes and the global residue condition.
There is an action of $(\CC^*)^{L(\eL)}$ on $\OBteich$ preserving
the fibers of the map to $\BPteich$, but the full group  $(\CC^*)^{V(\eL)}$
no longer acts on $\OBteich$ because it does not necessarily preserve the matching  residues or the GRC.
\par
\medskip
We recall that as a consequence of Proposition~\ref{prop:topologies_are_the_same},
two natural topologies on $\OBteich$ agree. The first topology is the one used above
to define the complex structure, as a subset of a finite cover of the product
of the twisted
Hodge bundles over a product of \Teichmuller spaces. The second topology is the
product of the conformal topologies on the components of $X \setminus f(\eL)$.
By definition, this topology is the same as the conformal topology on $\OBteich$,
where a sequence $(X_n, f_n, \bfz_n, \eta_n)$ of marked pointed
\twds converges to $(X, f, \bfz, \eta)$ if for some exhaustion~$K_n$ of~$X$,
there is a sequence of conformal maps $g_n\colon K_n \to X_n$ such that
$f_n \simeq g_n \circ f$ and $g_n^* \eta_n$ converges to~$\eta$ uniformly
on compact sets.

%%%%%%%%%%%%%%%%%%%%%%%%%%
\subsection{Welded surfaces.}
\label{sec:weld}
%%%%%%%%%%%%%%%%%%%%%%%%%%

\Teichmuller markings of nodal surfaces are by definition insensitive
to the precomposition by Dehn twists around the vanishing cycles. Here
we introduce the concept of a welded surface to define a refined concept
of markings.
\par
Let~$X$ be a stable nodal curve with dual graph~$\Gamma$,
and let $\pi\colon X^\ast\to X$ be the normalization.
Given a node $q$ of~$X$, with preimage $\pi^{-1}(q) = \{x, y\}$, a \emph{welding of~$X$ at
  $q$} is a \changed{$\mathbb{C}$-antilinear} isomorphism $\sigma_q\colon T_x X^\ast\to T_y X^\ast$, modulo scaling by
positive real numbers. We alternatively think of the welding as an orientation-reversing metric
isomorphism $\sigma_q \colon S_x X^* \to S_y X^*$, where $S_pX^* = (T_pX^*\setminus\{0\})/\RR_{>0}$
denotes the real tangent circle to~$X^*$ at~$p$.  As we will explain below, this viewpoint is natural from the perspective
of real oriented blowups, which will be discussed in full generality in Section~\ref{sec:rob}.  The ordering of the fiber over $q$ is not part of the
structure, and we consider $\sigma_q^{-1}\colon T_y X^\ast\to T_x X^\ast$ to be the same welding as
$\sigma_q$.  The space of all weldings of a given node $q$ is a circle~$S^1$.
\par
A welding can otherwise be described in terms of a real blowup of~$X$ that we
now recall, see e.g.~\cite[Section~X.9 and~XV.8]{acgh2} and Section~\ref{sec:rob}.  Given the unit disk $\Delta\subset \CC$,
the real oriented  blowup $p\colon \Bl_0 \Delta \to \Delta$ is the locus
\begin{equation*}
  \Bl_0 \Delta = \{ (z, \tau) \in \Delta\times S^1 : z = |z| \tau\}\,,
\end{equation*}
with the projection $p$ given by $p(z, \tau) = z$.  It is a real manifold with
a single boundary circle $\{0\} \times S^1$.  The projection~$p$ collapses the
boundary circle to the origin and is otherwise a diffeomorphism.
\par
More generally, if~$X$ is a Riemann surface and $D\subset X$ is a finite set of points,
performing the above construction at each point $q\in D$ yields
the \emph{real oriented blowup} $p\colon \Bl_DX \to X$, which is a real manifold
such that its boundary maps to $D$, and consists of a circle over each point $q\in D$. Then $p$
restricts to a diffeomorphism $\operatorname{int}(\Bl_D X) \to X\setminus D$, and
for each $q\in D$ the boundary circle $\bdry_q \Bl_D X = p^{-1}(q)$ is naturally
identified with the real tangent circle $S_q X = (T_q X \setminus \{0\})
/ \reals_{>0}$ of~$X$ at $q$.  The conformal structure of~$X$ gives $\bdry_q \Bl_D X$
the structure of a metric circle of arc length~$2\pi$.
\par
Given a subset $D\subset \N$ of the set~$\N$ of nodes of~$X$, the \emph{real oriented blowup}
$p\colon \Bl_{D} X\to X$ is
the real oriented blowup of the partial normalization $X^*$ of~$X$ at $D$, at the set
of preimages of~$D$ on this partial normalization.  In other words, for each
node $q\in D$ the fiber $p^{-1}(q)$ is a pair of metric circles
$S_{q^+} \cup S_{q^-} \subset \bdry\Bl_{D} X$.  In these terms, a welding
of~$X$ at $D$ is a choice for each node $q\in D$ of an orientation-reversing isometry
$\sigma_q \colon S_{q^+} \to S_{q^-}$.
\par
\medskip
A {\em global welding~$\bfsigma$ of~$X$} is a choice of a welding at
each node of~$X$. If the dual graph is endowed with a level
structure~$\overline\Gamma$, then a {\em vertical welding~$\bfsigma$
  of~$X$} is a choice of a welding at each vertical node (horizontal
nodes will never be welded in this paper).
\par
Given a vertical welding $\bfsigma$ of~$X$, we define the
\emph{associated welded surface}~$\overline{X}_\bfsigma$ to be the surface
obtained by gluing the boundary components of $\Bl_{\Nver} X$ via~$\bfsigma$.
The associated welded surface has the following extra structures:
\begin{enumerate}
\item a multicurve $\eL^{\ver}$ on $\overline{X}_\bfsigma$, containing for each node
$q\in \Nver$ the simple closed curve that is the image
of $S_{q^+}\sim S_{q^-}$, called the (multicurve of) {\em seams} of $\overline{X}_\bfsigma$;
\item a conformal structure on $\overline{X}_\bfsigma \setminus \left( \eL^{\ver} \cup \Nhor \right)$; and
\item a metric on each component of $\eL^{\ver}$, of arc length $2\pi$.
\end{enumerate}
By a slight abuse of terminology, we call $\eL = \eL^\ver \cup \Nhor$ the
\emph{pinched multicurve} of $\overline{X}_\bfsigma$. Note that the
surface~$\overline{X}_\bfsigma$ can have horizontal nodes, and is smooth
elsewhere.
\par
These notions obviously extend locally to
\emph{equisingular families} $(\pi\colon \calX \to B, \bfz)$ of stable curves,
also called \emph{families of constant topological type}. These are
families where all the nodes are \emph{persistent}, i.e.\ for each
node~$q$ in each fiber of~$\pi$ there is a section of~$\pi$ passing though~$q$
and mapping to the nodal locus of~$\calX$. We briefly digress on these notions,
aiming for the definition of the topology in Section~\ref{sec:ptwds} and the
comparison in Proposition~\ref{prop:comptopo}.
We will return to these notions in detail in Section~\ref{sec:rob}.
\par
For an equisingular family, a \emph{family of weldings} $\bfsigma$ over an
open set~$U \subset B$ is a continuous choice of weldings for each fiber
over~$U$.
Here we use the fact that~$\pi$ is locally trivial in the $C^\infty$-category,
to compare the tangent spaces $T_q$ in nearby fibers. Equivalently,
we can perform the real oriented blowup in families over~$U$ (see e.g.\
\cite[Section~XV.9]{acgh2} and Section~\ref{sec:rob}), and then a family of
weldings is a continuous section of the $S^1$-bundle at each vertical node.
For each family of weldings $\bfsigma$ the {\em family of welded
surfaces~$\overline{\calX}_\bfsigma$} is obtained by identifying the
family of real oriented blowups of~$\calX$ along the identifications
provided by~$\bfsigma$. A \emph{marked family of welded surfaces}
is defined by requiring that the fiberwise markings vary continuously.
\par

%%%%%%%%%%%%%%%%%%%%%%%%%%%%%%%%%%%%
\subsection{Prongs and \prmas.}
\label{sec:prmatch}
%%%%%%%%%%%%%%%%%%%%%%%%%%%%%%%%%%%%

Any point $p$ of a meromorphic differential $(X, \omega)$ which is not a simple pole has a
set of horizontal directions which we call the \emph{prongs} of $(X, \omega)$ at $p$.  Intuitively speaking, the
prongs at $p$ are the directions in the unit circle $S_p X= (T_p X\setminus \{0\}) / \reals_{>0}$
which are
tangent to horizontal geodesics limiting to $p$ under the flat structure induced by~$\omega$.  In fact, the prongs can be naturally defined as
vectors rather than just directions. \changed{We denote by~$\zeta_n$ a primitive $n$'th root of unity.}
\begin{df} 
Suppose the meromorphic differential~$\omega$ on~$X$ has order $k\neq -1$ at some point~$p$ \changed{and~$\phi$ is a local coordinate centered at $p$ in which the differential has the normal form given in Theorem~\ref{thm:standard_coordinates}. A \emph{complex prong} $v\in T_pX$ of~$\omega$ at $p$ is one of the $2|k+1|$ vectors $\phi_*(\zeta_{2|k+1|)}^j\pderiv{z})$ for $j=1,\ldots ,2|k+1|$. We call a prong \emph{outgoing} if $j$ is even, and otherwise call it \emph{incoming}.} The $2|k+1|$ vectors in $S_pX$ obtained by projectivizing the complex prongs are the \emph{real prongs} of~$\omega$ at $p$.
\end{df}
\par
\changed{When $p$ is a pole of order $k<-1$ with a non-zero residue, while there are infinitely many choices of a local coordinate in which the differential has normal form, there are still only $2|k+1|$ prongs, as the prongs only depend on the first derivative of~$\phi$ at $p$}.
\par
Since complex and real prongs are in natural bijection, we will simply refer to them as prongs when we do not need to make the distinction.
\par
We denote the set of incoming prongs at $z$ by $\pin$
and the set of outgoing prongs by~$\pout$.  Each has cardinality $\kappa_{z}=|1+k|=|1 + \ord_z \omega|$.
Each set of prongs is equipped with the counterclockwise cyclic ordering when embedded in the complex
plane with coordinate~$z$.
\par
Now suppose $q$ is a vertical node of a \twd $(X, \eta)$.
The matching orders condition~(1) of a \twd  equivalently says that the
zero at~$q^+$ and the pole at~$q^-$ have the same number of prongs (equal to $\kappa_q$).
\par
\begin{df}
  \label{def:pm1}
A \emph{local \prma} of $(X, \eta)$
at $q$ is a cyclic-order-reversing bijection $\sigma_q\colon \pin[q^-][\eta] \to \pout[q^+][\eta]$.
\par
A \emph{(global) \prma $\bfsigma$}
\index[twist]{b008@$\bfsigma$!Global \prma for~$X$}
for a \twd is a choice of a local \prma $\sigma_q$ at each vertical node~$q$ of~$X$.
\end{df}
\par
Note that \prmas at horizontal nodes are  not defined.  The following
equivalent definition of a local \prma will be useful for studying families
in Section~\ref{sec:famnew}. \changed{Let~$X^*$ denote the normalization of~$X$.}  
\begin{df} \label{def:pm2}
A \emph{local \prma} of $(X, \eta)$ at a node $q$ is an element~$\sigma_q $ of
$T^*_{q^+}X^*\otimes T^*_{q^-}X^*$ such that the equality $\sigma_{q}(v_+ \otimes v_-)^{\kappa_q} = 1$ holds for any pair $(v_+, v_-)$ of an outgoing and an incoming prong.
\end{df}
To see the equivalence of these definitions, note that any such $\sigma_{q}$ corresponds to an order-preserving
bijection $\pin[q^-][\eta] \to \pout[q^+][\eta]$ by assigning to an incoming prong $v_-$ the unique
outgoing prong~$v_+$ such that $\sigma_{q}(v_-\otimes v_+) = 1$.
\par
\changed{For example, if $\eta$ has normal form $u^\kappa \tfrac{du}{u}$ (for $\kappa\geq0$) in a local coordinate $u$ centered at
  $q^+$, and normal form $-t^\kappa (v^{-\kappa}+r)\tfrac{dv}{v}$ in a
  local coordinate $v$ centered at $q^-$, then $t^{-1} \, du\otimes dv$ is a local \prma. This will be the standard example in the plumbing construction in Section~\ref{sec:defPlmap}.}
\par
A \prma $\sigma_q$ at the node~$q$ determines a welding of~$X$ at~$q$
by identifying a prong $v \in S_{q^-} X$ with
the prong $\sigma_q(v) \in S_{q^+}X$, and extending this to an
orientation-reversing isometry of these tangent circles. We denote by $\ol{X}_{\bfsigma}$ or simply
by~$\ol{X}$ the associated welded surface constructed using the welding defined by the
\prma~$\bfsigma$.
\index[surf]{b065@$\overline{X}_\bfsigma$!Welded surface associated to the \prma $\bfsigma$}

\begin{df} \label{def:ptwd}
A \emph{\ptwd of type $\mu$ compatible with $\Gamma$}, or just
  {\em \ptwd} for short, is the datum
$(X, \bfz, \eta, \bfsigma)$ consisting of a \twd
$(X, \bfz, \eta)$ of type $\mu$, compatible with $\Gamma$, and a global \prma~$\bfsigma$.
\end{df}
\par
An isomorphism between \ptwds is an isomorphism of stable curves which
identifies the forms on each component and is additionally required to
commute with all of the local \prmas.
\par
\begin{df} \label{df:nearsm}
Given two \ptwds $X_1$ and $X_2$ with associated welded surfaces
$\overline{X}_1$ and $\overline{X}_2$, an \emph{almost-diffeomorphism}
$f\colon \overline{X}_1\to \overline{X}_2$ is a continuous map which satisfies:
  \begin{enumerate}
  \item The preimage of each horizontal node is either a horizontal
node or a simple closed curve disjoint from the nodes of $\overline{X}_1$,
and the restriction of $f$ to each component of
$\overline{X}_1\setminus f^{-1}(N_{X_2}^h)$
is a diffeomorphism onto a component of $\overline{X_2} \setminus N_{X_2}^h$.
  \item The map $\Gamma(X_1) \to \Gamma(X_2)$ induces a degeneration
    of enhanced level graphs.
  \end{enumerate}
If $f$ contracts no simple closed curves, we call it a \emph{diffeomorphism}.
\par
\changed{An \emph{almost-homeomorphism} $f\colon \overline{X}_1\to \overline{X}_2$
is defined the same way, replacing ``diffeomorphism'' with ``homeomorphism.''}
\end{df}
%\changed{We will occasionally need the notion of an
%  \emph{almost-homeomorphism}, which we define as in the previous
%  definition, replacing ''diffeomorphism'' with ``homeomorphism.''}
\par 
\changed{For a family of twisted differentials  $(\pi\colon \calX \to B,
  \bfcalZ, \eta)$ on an equisingular family of curves} 
  we define \emph{a family of \prmas} to be a family of global
weldings  that is a \prma in each fiber of~$\pi$.
\par
The \emph{prong rotation group} associated with an enhanced level graph $\eG$
is the finite group
\begin{equation}\label{eq:Prot}
  \Prot \= \prod_{e \in \eL^\ver} \ZZ / \kappa_e \ZZ\,.
\end{equation}
\index[twist]{d005@$\Prot$!Prong rotation group}The number of \prmas for a
given twisted differential is then equal to $|\Prot|$. Moreover, for any given
twisted differential $(X,\eta)$, the prong rotation group acts on the set
of \prmas as follows. An element $(j_e)_{e\in \eL^\ver} \in P_\Gamma$ acts by
composing the local \prma at the node $q = q_e$ with the bijection
$\pin[q^-][\eta]\to \pin[q^-][\eta]$ defined by turning counterclockwise~$j_e$
times. Here, and for other similar notions depending on graphs, we also write
$\Prot[\eL]$ or $P_\Gamma$ as shorthand for $\Prot[\eG(\eL)]$.
\par

%%%%%%%%%%%%%%%%%%%%%%
\subsection{The \Teichmuller space of \ptwds.} \label{sec:ptwds}
%%%%%%%%%%%%%%%%%%%%%

We now define the notion of a marking of a \ptwd and construct the
\Teichmuller space $\ptwT$ of marked \ptwds of type $(\mu,\eL)$
as a complex manifold.  The notion of a marking is modeled on the
definition of a marked stable curve, except the target of the marking
is the associated welded surface.

\begin{df}
  \label{df:marking}
  A \emph{marking} of a \ptwd $(X, \bfz, \eta, \bfsigma)$ is a
  continuous map
  $f\colon (\Sigma, \bfs) \to (\overline{X}_\bfsigma, \bfz)$ which
  satisfies:
  \begin{enumerate}
  \item The preimage of every horizontal node is a simple
    closed curve on $\Sigma$.
  \item If we denote by $\Lambda^h\subset\Sigma$ the
    ``horizontal'' multicurve consisting of the preimage of the
    set of horizontal nodes $N^h_X$ of $X$, then the restriction of $f$ to
    $\Sigma\setminus \Lambda^h$ is an orientation-preserving
    diffeomorphism $\Sigma\setminus \Lambda^h \to \overline{X}_\bfsigma
    \setminus \N^h$.
  \item The map $f$ preserves the marked points, that is,
    $f\circ \bfs = \bfz$.
  \end{enumerate}
  We say that a \emph{marked \ptwd} $(X, f,\bfz, \eta, \bfsigma)$ is
  \emph{of type $\Lambda$} if~$\Lambda$ is the enhanced multicurve
  obtained by pulling back the seams of $\overline{X}_\bfsigma$.
\par
  Two marked \ptwds are equivalent if there is an
  isomorphism of \ptwds that identifies the marking of their
  associated welded surfaces up
  to isotopy rel~$\bfs$.
\end{df}
\changed{In particular two marked \ptwds that differ by the action of
a Dehn twist around a horizontal curve are equivalent.}  
\par
\begin{df}
The \emph{\Teichmuller space $\ptwT$ of marked \ptwds of type $(\mu,\eL)$}
is the set of isomorphism classes of marked \ptwds of type $\mu$ with
marking of type~$\Lambda$.
\index[teich]{d040@$\ptwT$!\Teichmuller space  of \ptwds}
\end{df}
\par
Given any contractible open set $U\subset\OBteich$ together with a
\prma~$\bfsigma$ and a marking $f\colon\Sigma\to\overline{X}_\bfsigma$
for some basepoint $(X, \eta)\in U$, we may uniquely extend~$\bfsigma$
to a continuous family of prong-matchings over $U$.  The corresponding
family of welded surfaces is then topologically trivial over $U$, so
the marking $f$ may be extended uniquely (up to isotopy) to a
continuous family of markings over the base $U$.  This defines a lift
$U\to \ptwT$.  We give $\ptwT$ the structure of a complex manifold
such that these lifts are holomorphic local homeomorphisms.  The
forgetful map $\ptwT\to\OBteich$ is then a holomorphic covering map. 
\changed{Note that in the domain space of the map the vertical nodes are welded, and the marking records Dehn
twists around the nodes, while in the target space such a Dehn twist
is isotopic to the identity map. Therefore, this map has infinite degree (unless $\Lambda$ has only horizontal edges, which are not welded).} 
\par
Once we define families of marked \ptwds, in Proposition~\ref{prop:univPTWT}
we will observe that $\ptwT$ is the fine moduli space for families
of marked \ptwds, i.e.~that every family of marked \ptwds can be obtained as
a pullback from the family $\calX\to\ptwT$ of \ptwds over $\ptwT$.

%%%%%%%%%%%%%%%%%%%%%%
\subsection{Turning numbers on \ptwds} \label{sec:turning}
%%%%%%%%%%%%%%%%%%%%%
Recall that for any flat surface~$(X,\bfz,\omega)$ the \emph{Gauss map}
is defined on the tangent circle bundle of the punctured surface $S(X\setminus\bfz) = (T(X\setminus\bfz)\setminus \lbrace 0\rbrace)/ \reals_{>0}$ by 
\begin{equation*}
  G(v) \= \frac{\omega(v)}{|\omega(v)|}.
\end{equation*}
We define the \emph{turning number} $\tau(\gamma)$ of any smooth immersed
arc $\gamma:[a,b]\to X\setminus\bfz$ as $\tau(\gamma) = g(b) - g(a)$,
where $g\colon[a, b]
\to \reals$ is a continuous lift of $G\circ \tilde{\gamma}$, that is $e^{2 \pi i g}
= G\circ \tilde{\gamma}$, and where $\tilde{\gamma}$ is the natural lift
of~$\gamma$ to~$S(X\setminus\bfz)$.
\par
\changed{Turning numbers are invariant under regular isotopies (meaning
isotopies through immersed curves) preserving the endpoints
of~$\gamma$ as well as the tangent vectors at these endpoints, by
the continuity of the Gauss map.}
\par
\changed{Let now $(X, \bfz,\omega, \bfsigma)$ be a \ptwd and choose any
extension~$\bfsigma'$ of the vertical welding~$\bfsigma$ to a global welding.
We call the image of an immersed arc on $\overline{X}_{\bfsigma'}\setminus\bfz$ under the
projection to $\overline{X}_\bfsigma\setminus\bfz$ an immersed arc on the welded punctured surface
and define regular isotopies also by reference to~$\overline{X}_{\bfsigma'}\setminus\bfz$.}
\par
\begin{lm} \changed{
The Gauss map extends continuously to the tangent circle bundle
of the welded punctured surface~$\overline{X}_{\bfsigma'}\setminus\bfz$.}
\par
\changed{Consequently, turning numbers of immersed arcs on welded punctured
surfaces~$\overline{X}_\bfsigma\setminus\bfz$ are invariant under regular isotopies.}
\end{lm}
\par
\begin{proof}
\changed{The problem is local near the nodes. We deal with the case of a vertical node~$q$
and leave the straightforward horizontal case to the reader. We choose coordinates~$u$
on a neighborhood~$\Delta_u$ of $q^+$ and~$v$ on a neighborhood~$\Delta_v$ of~$q^-$.}
\changed{Consider the (real oriented) blowup
$$B \= \{(u, U) \in \Delta_u \times S^1\colon u = |u| \cdot U \}\,.$$
We parametrize~$B$ using polar coordinates $(r, \theta)$ where $u = r e^{i\theta}$
and $U = e^{i \theta}$ for $r \geq 0$ and $\theta \in [0, 2\pi)$.
The tangent bundle with the basis $\partial / \partial r$ and $\partial / \partial \theta$ extends to all of~$B$, including at the boundary $r=0$. The differential form on $\Delta_u$ is
$\omega = u^\kappa du/u = r^\kappa e^{i \kappa \theta} (dr/r + i d\theta)$, and
the Gauss map at the point $(r,\theta)\in B$ is given by 
$$G\big(a (\partial / \partial r)  + b (\partial / \partial \theta)\big) = \frac{e^{i\kappa \theta} (a+i br)}{\sqrt{a^2 + b^2 r^2}}.$$ }
\par
\changed{Next, we carry out the same calculation in coordinate $v$ on $\Delta_v$, using polar
coordinates $v = r_v e^{i\theta_v}$. We obtain similarly 
\begin{eqnarray*}
G\big(a (\partial / \partial r_v)  + b (\partial / \partial \theta_v)\big)
& = & - \frac{ (r_v^{-\kappa} e^{-i\kappa \theta_v} + s) (a+ibr_v) }
{|r_v^{-\kappa} e^{-i\kappa \theta_v} + s| \cdot \sqrt{a^2 + b^2 r_v^2}} \\
& = & - \frac{ ( e^{-i\kappa \theta_v} + s r_v^{\kappa} ) (a+ibr_v) }
{|e^{-i\kappa \theta_v} + sr_v^{\kappa} | \cdot \sqrt{a^2 + b^2 r_v^2}}
 \end{eqnarray*}
where $s$ denotes the residue of $\bfomega$ at the point $q^-$, i.e. at $v=0\in \Delta_v$.}
\par
\changed{To weld the two half-planes to a plane we have to identify $r_v
\mapsto -r$ and
$\theta_v \mapsto \theta$, hence $\partial / \partial r_v \mapsto - \partial /
\partial r$ and $\partial / \partial\theta_v \mapsto -\partial / \partial\theta$.
Since $\kappa>0$ the $r_v^\kappa$ eliminates the residue term at $r_v=0$ and the
two definitions of the Gauss map glue continuously at $r=0$.}
\end{proof}
\par
\changed{In what follows, we always consider Gauss maps and turning numbers on the punctured surfaces only for arcs not passing through any points in~$\bfz$. To keep notation simpler, we will suppress `punctured' and $\setminus \bfz$ from now on.}
\par
\begin{df}\label{df:tnp}
  Consider a \ptwd $X$, a sequence of \ptwds $\{X_m\}_{m\in \NN}$, and a sequence
  of almost-diffeomorphisms $h_m\colon X_m \to X$. We say that the
  sequence~$\{h_m\}$ is {\em asymptotically \tnp}
  if for any immersed arc $\gamma$ on $\overline{X}_\sigma$, we have
  $\tau(h^{-1}_m(\gamma)) \to \tau(\gamma)$ as $m\to\infty$.
\end{df}
\par
\begin{prop} \label{prop:jacuzzi} Consider a twisted differential $X$
  with a prong-matching $\bfsigma$ and a sequence of prong-matchings
  $\bfsigma_n$.  Let $h_n$ be an asymptotically \tnp sequence of diffeomorphisms of \ptwds
  $h_n\colon\overline{X}_{\bfsigma_n} \to \overline{X}_\bfsigma$.
  Suppose that the diffeomorphisms $h_n^{-1}$ converge $C^1$-uniformly on compact sets of
  $X^s \setminus \bfz$ to the identity map.  Then for $n$
  sufficiently large, $\bfsigma = \bfsigma_n$, and moreover
  $h_n$ is isotopic to the identity map on
  $\overline{X}_\bfsigma$.
\end{prop}
\par
\begin{proof}
Fix a compact subsurface $K$ that is a deformation retract of
  $X^s \setminus \bfz$, and let $\gamma$ be a curve joining two
  boundary components and crossing a single seam corresponding to a
  node $p$.  For $n$ large, $h_n^{-1}$ is $C^1$-close to the identity
  on $K$, so the endpoints of~$\gamma$ and $h_n^{-1}(\gamma)$ are close as
  are their corresponding tangent vectors.  Since the turning
  number of a curve is determined (mod $\zed$) by its endpoints and
  tangent vectors, we have
  $\tau(h_n^{-1}(\gamma))-\tau(\gamma) \sim \theta_n \pmod \zed$,
  where the prong-matchings at $p$ are related by
  $e^{2\pi i \theta_n}\sigma = \sigma_n$.  Since the $h_n$ are
  asymptotically \tnp, we have $\theta_n\to 0$,
  so $\sigma_n = \sigma$ eventually, as the set of prong-matchings is a fixed
  discrete set.

  Now take $n$ large enough so that $\bfsigma=\bfsigma_n$ and, moreover,
  $h_n^{-1}$ moves each point $q$ of~$K$ by a distance smaller than the
  injectivity radius of $X^s \setminus \bfz$ at $q$.  We may then on~$K$
  take the nearest-point isotopy from $h_m^{-1}$ to the identity, and
  extend it to an isotopy on $\overline{X}_\bfsigma$ from $h_m^{-1}$ to a
  map $k$ that is the identity on $K$.  Each seam is contained
  in an annular component $A$ of $\overline{X}_\bfsigma\setminus
  K$, and to show that $k$ is isotopic to the identity on $A$
  (rel $\bdry A$) it suffices to show that $k(\gamma)$ is isotopic
  to $\gamma$ (rel $\bdry A$), with $\gamma$ as before joining the two
  boundary components.  But the isotopy class of $\gamma$ (through
  immersed curves with the endpoints and vectors fixed) is determined
  by the turning number of~$\gamma$, since the core curve of $A$ has
  nonzero turning number.  As $\tau(h_n^{-1}(\gamma)) -
  \tau(\gamma)\to 0$, and this difference of turning numbers is
  integral, the difference is eventually zero, so these curves are isotopic.
\end{proof}
\par
Using these notions, we can now give an alternative definition for the
topology on $\ptwT$, closer to what we will use later for the augmented
\Teichmuller\ space of flat surfaces.
\par
\begin{df} \label{df:coftop} We say that a sequence
$X_m = (X_m, \bfz_m, \eta_m, \cleq, \bfsigma_m, f_m)$
of marked \ptwds in   $\ptwT$  converges in the \emph{conformal topology}
to $X = (X, \bfz, \eta, \cleq, \bfsigma, f)$ if and only if for any sufficiently
large~$m$  there exists a diffeomorphism $g_m \colon \overline{X}_m\to\overline{X}$
and a sequence of positive numbers $\epsilon_m$ converging
to~$0$, such that the following conditions hold:
\begin{enumerate}[(i)]
  \item The function  $g_m$ is compatible with the markings in the sense
that $f$ is isotopic to $g_m \circ f_m$ rel marked points.
  \item The function $g_m^{-1}$ is conformal on the $\epsilon_m$-thick
part  $\Thick[(X, \bfz)][\epsilon_m]$.
\item The differentials $(g_m)_*\eta_m$ converge to~$\eta$ uniformly
on compact sets of the $\epsilon_m$-thick part of~$X$.
\item The functions $g_m$ are asymptotically \tnp. \qedhere
\end{enumerate}
\end{df}
\par
Here (and in the sequel) we use the pushforward notation
$g_* = (g^{-1})^*$ for the action on differentials.
\par
\begin{rem} \label{rem:coftop}
In order to verify that the $g_m$ are asymptotically \tnp,
it suffices to choose a collection of arcs that contains, for every
seam, an arc that crosses only this seam and no others, and does so
exactly once. Indeed, if turning numbers converge for these arcs, then
together with (ii) this forces the convergence of the turning
numbers of all other arcs as well.
\end{rem}
\par
\begin{prop} \label{prop:comptopo}
The conformal topology on $\ptwT$ and the topology defined in
Section~\ref{sec:ptwds} as the covering space of $\OBteich$ agree.
\end{prop}
\par
\begin{proof}
  Consider a sequence $X_m\to X$ as in the above definition.  Let
$U\subset\OBteich$ be a contractible neighborhood of the point corresponding
to~$X$, and let
  $\calX\to U$ be the restriction of the universal curve over $\OBteich$
  to $U$.  The prong-matching $\bfsigma$ of $X$ extends uniquely to a
  continuous family of prong-matchings over $U$.  Choose a smooth (in
  the sense of Definition~\ref{df:nearsm}) trivialization $h\colon
  X\times U \to \calX$ of this family whose restriction to the fiber
  over $X$ is the identity.  This defines a lift
  $U\to\ptwT$, and our goal is to show that $X_m$ is eventually in
  $U$.

  Restricting the marking maps to the complement of the seams defines
  a projection $\pi\colon\ptwT\to\OBteich$.  Items (i)--(iii) above, together
  with Proposition~\ref{prop:topologies_are_the_same} ensure that
  this projection is continuous, so that $\pi(X_m)\to \pi(X)$ in
  $\OBteich$.  The trivialization~$h$ then defines smooth maps
  $\overline{h}_m\colon \overline{X}_\bfsigma \to
  \overline{(X_m)}_{\bfsigma_m}$.  The maps $\overline{h}_m^{-1} \circ g_m$
  then satisfy the hypotheses of Proposition~\ref{prop:jacuzzi}, which
  ensures $g_m$ eventually lies in the lift of $U$.
\par
  Conversely, consider a sequence $X_m$ such that $\pi(X_m)\to \pi(X)$ in
$\OBteich$, with the points corresponding to~$X_m$ eventually in the lift
of~$U$ constructed
  above.  The desired maps~$g_m$ can then be constructed by modifying (as in
  the proof of Theorem~\ref{thm:topologiesarethesame})
  the maps~$h_m$ constructed above so that they are conformal on an exhaustion.
\end{proof}
\par
Using a continuous trivialization is obviously impossible
for a degenerating family of stable curves with varying topological types. For this reason,
the conformal topology on the augmented \Teichmuller space of flat surfaces defined
in Section~\ref{sec:AugTeich} will be more involved.

%%%%%%%%%%%%%%%%%%%%%%%%
\subsection{\changed{Turning numbers of Jordan arcs}}
\label{sec:turning_Jordan}
%%%%%%%%%%%%%%%%%%%%%%%%

At times we will want to relax the requirement from
Definition~\ref{df:coftop} that the maps~$g_m$ are smooth. To then
make sense of item~(iv), we should have a notion of turning number for
arcs which are not necessarily smooth.  While this seems impossible
for arbitrary continuous arcs, one can define a notion of turning
number for Jordan (that is, embedded and continuous) arcs.
\par
Consider Jordan arcs  $\gamma, \delta \colon[a,b]\to X$ in a translation
surface $(X, \omega)$ which are differentiable at their endpoints.
We will say that a Jordan arc~$\delta$ is \emph{regular}, if it is smooth
and immersed. We will say that two arcs~$\gamma,\delta$ have \emph{common
end-vectors}, if they have the same endpoints, are differentiable at those
endpoints, and the tangent vectors at those endpoints agree.  
\par
We define the turning number $\tau(\gamma)$ of any Jordan arc to be the turning number~$\tau(\delta)$
of a regular Jordan arc~$\delta$ that is isotopic to $\gamma$, via an isotopy through
Jordan arcs (as always: on $X\setminus \bfz$) with common end-vectors. The following proposition shows that
this is well-defined.
%
%
%Let $\delta\colon [a,b]\to X$ be a regular Jordan arc which has common
%end-vectors with $\gamma$ and which is isotopic to $\gamma$ through
%Jordan arcs with common end-vectors.  We define the
%turning number $\tau(\gamma)$ to be the turning number of $\delta$.
%The following Proposition shows that this is well-defined.
\par
\begin{prop}
  \label{prop:turning_Jordan}
  Let $\delta_1, \delta_2\colon[a,b]\to X$ be regular Jordan arcs in a
  translation surface $(X, \omega)$ which are isotopic through (not
  necessarily regular) Jordan arcs with common end-vectors.  Then
  $\delta_1$ and $\delta_2$ have the same turning number.
\end{prop}

\begin{proof}
  Let $\overline{X}\to X$ be the real blowup of $X$ at the endpoints
  of the $\delta_i$.  Since the~$\delta_i$ have common end-vectors,
  they lift to Jordan arcs $\tilde{\delta}_i\colon[a,b]\to\overline{X}$ joining
  the same points on the boundary of~$\overline{X}$. The isotopy lifts to a continuous
  isotopy between $\tilde\delta_1$ and $\tilde\delta_2$, relative to the endpoints, which may
  be approximated by a smooth homotopy.  Smoothly homotopic arcs are
  smoothly isotopic, again relative to the endpoints (see
  \cite{primer}, Proposition~1.10 and Section~1.2.7).  Since such an isotopy
  leaves turning numbers constant, the desired conclusion follows.
\end{proof}
% Since the $\delta_i$ are isotopic, we may choose two lifts
  % $\tilde{\delta}_i\colon[a,b]\to \widetilde{X}$ to the universal
  % cover $\widetilde{X}$ with the same endpoints.  By the Homotopy
  % Lifting Theorem, the isotopy lifts to $\widetilde{X}$. 

  % Choose a regular Jordan arc $\alpha$ which has common end-vectors
  % with $\tilde{\delta}_1$ and $\tilde{\delta}_2$ and which is disjoint
  % from both, except at the endpoints.  We claim that $\alpha$ is
  % isotopic to both $\tilde{\delta}_1$ and $\tilde{\delta}_2$ through
  % regular Jordan arcs with common end-vectors.  Since such an isotopy
  % leaves turning numbers constant, the desired conclusion follows.

  % To construct the desired isotopy, note that
  % $\alpha\cup\tilde{\delta}_i$ bounds a disk $D$.  By Carath\'eodory's
  % Theorm, the Riemann mapping $f\colon\Delta\to D$ extends to a homeomorphism
  % between the closures, and we may normalize $f$ so that it sends the
  % upper half-circle to $\alpha$ and the lower half-circle to
  % $\tilde{\delta}_i$.  Composing an isotopy between the two
  % half-circles with $f$ gives the desired isotopy.
  % \Matt{The idea of the last paragraph comes from here:
  %   \url{https://mathoverflow.net/questions/57766/why-are-there-no-wild-arcs-in-the-plane/57770\#57770}.
  % Maybe we should credit Thurston.}
%\end{proof}

%%%%%%%%%%%%%%%%%%%%%%%%%%%
\section{Twist groups and level rotation tori} \label{sec:twrot}
%%%%%%%%%%%%%%%%%%%%%%%%%%%

The goal of this section is to define the twist group~$\Tw$ and the level rotation
torus~$T_\Lambda$ associated with an enhanced multicurve~$\eL$. The twist group is
generated by appropriate combinations of Dehn twists, such that the quotient of some
augmented \Teichmuller space by the twist group is the flat geometric counterpart of
the classical Dehn space introduced in Section~\ref{sec:dehn-space-deligne}.
This augmented \Teichmuller space of flat surfaces, to be defined in
Section~\ref{sec:AugTeich}, requires a \lw projectivization of the
space of prong-matched differentials, and we define here the appropriate
actions of multiplicative groups, the level rotation tori,  for this
projectivization. We will provide various viewpoints on the  level rotation torus
that will be used in the definition of families of model differentials
and \msds in the later sections.

%%%%%%%%%%%%%%%%%%%%%%%%%%%
\subsection{The action of $\CC^{L^\bullet(\eL)}$ on the space of
prong-matched differentials}
\label{sec:ConPMD}
%%%%%%%%%%%%%%%%%%%%%%%%%%%

Recall from Section~\ref{sec:TeichTwds} that $(\CC^*)^{L^\bullet(\eL)}$ acts
on~$\OBteich$ by simultaneously scaling forms at the same level and preserving
the fibers of the projection to $\BPteich$.  However, the group
$(\CC^*)^{L^\bullet(\eL)}$ does not act naturally on $\ptwT$, since a loop around the
origin in~$\CC^*$ in general returns to the same differential with a
different \prma and a different marking. To get a continuous action on $\ptwT$,
we have to pass to the universal cover $\CC^{L^\bullet(\eL)}$ of $(\CC^*)^{L^\bullet(\eL)}$,
which acts continuously on $\ptwT$ by \emph{level rotations}, as we now describe.
\begin{enumerate}
\item On the level of forms, the tuple $\bfd =  (d_i)_{i \in L^\bullet(\eL)}\in \CC^{L^\bullet(\eL)}$ acts through the
quotient $(\CC^*)^{L^\bullet(\eL)}$ by multiplying the form at level~$i$\index[twist]{b010@$\bfd = (d_i)_{i \in L^\bullet(\eL)}$!Tuple in $\CC^{L(\eL)}$ acting on prong-matched differentials}
\index[twist]{b020@$\cdot$!Action of $\CC^{L(\eL)}$ on \ptwds} by~$\bfe(d_i)$
(recall that we denote $\bfe(z)=\exp(2\pi \sqrt{-1}  z)$).
\item On a \prma~$\sigma$ we act by
shifting the angles by  the real parts of the~$d_i$, i.e.\
for a twisted differential $(X,\eta)$, a \prma~$\sigma$,
 and for $\bfd = (d_i)_{i \in L^\bullet(\eL)}$ we define
\begin{equation}\label{eq:ConPMD}
\bfd \cdot (X, \eta, \sigma) \= (X,
\{\bfe(d_i)\eta|_{X_{(i)}}\}_{i\in L^\bullet(\eL)}, \{\bfd \cdot \sigma\})\,,
\end{equation}
where for each vertical node~$q$ we let
\begin{equation}\label{eq:ConPMD2}
\bfd \cdot \sigma_{q}\colon
\pin[q^-][\eta] \to \pout[q^+][\eta]
\end{equation}
be the map $\sigma_q$ precomposed and post-composed with rotations by the
angle $-2\pi \Re(d_{\lqbot}/ \kappa_q)$ and $2\pi \Re(d_{\lqtop}/ \kappa_q)$,
so that $\bfd\cdot\sigma_q$ remains to be a \prma.

Alternatively, following Definition~\ref{def:pm2}, $\sigma_q$ can be regarded as an element of
$T^*_{q^+}X\otimes T^*_{q^-}X$ in which case $\bfd\cdot\sigma_q =
\bfe((d_{\lqtop} - d_{\lqbot})/ \kappa_q) \sigma_q$.
\item  On the marking $f$, the element $\bfd\in \CC^{L^\bullet(\eL)}$ acts by composition with
the map 
  \begin{equation}\label{eq:defFD}
    F_\bfd \colon \ol{X}_\sigma \to \ol{X}_{\bfd \cdot \sigma}
\end{equation}
which is the identity map outside a union of annular neighborhoods~$A_q$ of
the corresponding seams and a fractional Dehn twist in their interior.
\end{enumerate}
\par
\changed{Here a \emph{fractional Dehn twist}\index[twist]{b040@$F_\bfd$!Fractional
Dehn twist} is a diffeomorphism, which is isotopic to
the map $S^1 \times [0,1] \to S^1 \times [0,1]$ given by $(t,h) \mapsto (t + 2\pi qh,h)$
for some $q \in \RR$ in a neighborhood of the weldings. The case $q=1$ gives a usual
full Dehn twist.}

%%%%%%%%%%%%%%%%%%%%%%%%%%%
\subsection{Twist groups.} \label{sec:levrot}
%%%%%%%%%%%%%%%%%%%%%%%%%%%

The restriction of the action of~$\CC^{L^\bullet(\eL)}$ on $\ptwT$ to the
subgroup $\ZZ^{L^\bullet(\eL)}\subset\CC^{L^\bullet(\eL)}$ acts by modifying the \prmas and markings, while
preserving the underlying differentials. We call the group $\ZZ^{L^\bullet(\eL)}$  the {\em level rotation group}.
\index[twist]{d008@$\ZZ^{L^\bullet(\eL)}$!Level rotation group}
Considering only the action on prongs defines a homomorphism from the level
rotation group $\ZZ^{L^\bullet(\eL)}$ to the prong rotation group $P_{\eL}$ defined
in~\eqref{eq:Prot}:
\begin{equation}\label{eq:definephi}
\phi_\eL^\bullet\colon \ZZ^{L^\bullet(\eL)} \to P_\eL\,, \qquad \bfn \mapsto
\left(n_{\ltop} - n_{\lbot} \mod \kappa_e \right)_{e\in \eL^\ver}\,.
\end{equation}
This map allows us to introduce an important equivalence \index[twist]{d009@$\phi_\eL^\bullet$!Map from level rotation group to prong rotation group} relation.
\par
\begin{df}\label{def:eqPM}
Two prong-matchings are called \emph{equivalent} if there exists an element of the level rotation group that transforms one into the other.
\end{df}
\par
The homomorphism $\phi_\eL^\bullet$ fits into the following commutative diagram of group homomorphisms
\index[twist]{d010@$\vTw[\eL]$!Vertical twist group}
\[\begin{tikzpicture}
\matrix (m) [matrix of math nodes, row sep=2.10em, column sep=3.5em,
text height=1.5ex, text depth=0.25ex]
{\vTw[\eL]  & \ker(\phi_\eL^\bullet) & \ZZ^{L^\bullet(\eL)} & \\
\Mod[(\Sigma, \bfs)] & \ker(\overline{\psi})  & \ZZ^{\eL^\ver}  & P_\eL  \\};
\path[(->,right hook-latex,font=\scriptsize]
(m-2-2) edge [below] node {$\psi$} (m-2-1)
(m-2-2) edge (m-2-3)
(m-1-1) edge (m-2-1)
(m-1-2) edge  (m-1-3);
\path[->,font=\scriptsize]
%(m-1-2) edge    (m-1-1)
(m-1-2) edge [above] node {$\tau_\eL^\bullet$} (m-1-1)
(m-1-3) edge [above] node {$\phi_\eL^\bullet$} (m-2-4)
(m-1-3) edge [left] node {$\wt{\phi}_\eL^\bullet$} (m-2-3)
(m-2-3) edge [below] node {$\overline{\psi}$} (m-2-4)
(m-1-2) edge [left] node {$\wt{\phi}_\eL^\bullet$} (m-2-2);
\end{tikzpicture}\]
that we now describe.
The group $\ZZ^{\eL^\ver}$ acts on the space $\ptwT$ via {\em edge rotations}
by the fractional Dehn twists, i.e. the tuple $ (n_e)_{e \in \eL^\ver}$
twists the \prma of the edge~$e$ by $\kappa_e \{n_e/\kappa_e \}$ (the remainder
of~$n_e$ mod $\kappa_e$)
and precomposes the marking by $\lfloor n_e/\kappa_e \rfloor$ left Dehn twists
around the curve corresponding to~$e$. Taking the quotient by the subgroup
of full Dehn twists at such an edge~$e$ gives a map $\ZZ\to\ZZ/\kappa_e\ZZ$,
and doing this for all vertical edges induces a map
$\overline{\psi}\colon \ZZ^{\eL^\ver} \to P_\eL$ onto the prong rotation group.
The kernel $\ker(\overline{\psi})$ is thus generated by (full) Dehn twists
around $\eL^{\ver}$ and is thus a subgroup of the mapping class group.
We denote by $\psi\colon\ker(\overline{\psi}) \hookrightarrow \Mod[(\Sigma, \bfs)]$
this inclusion.

There is a natural homomorphism $\widetilde{\phi}_\eL^\bullet\colon \ZZ^{L^\bullet(\eL)}
\to \ZZ^{\eL^\ver}$ defined by
\begin{equation}\label{eq:definetildephi}
  \widetilde{\phi}_\eL^\bullet(\bfn)
\= \left(n_{\ltop} - n_{\lbot}\right)_{e\in \eL^\ver}\,.
\end{equation}
The composition of  $\widetilde{\phi}_\eL^\bullet$  followed by $\overline{\psi}$
recovers the homomorphism $\phi_\eL^\bullet\colon \ZZ^{L^\bullet(\eL)}\to \Prot[\eL]$
defined in \eqref{eq:definephi}. The kernel
$\ker(\phi_\eL^\bullet)$ is in other words the subgroup of $\ZZ^{L^\bullet(\eL)}$
whose action on $\ptwT$ fixes the underlying \ptwds, only changing the markings.
This defines a homomorphism $\tau_\eL^\bullet = \psi \circ  \widetilde{\phi}_\eL^\bullet
\colon \ker(\phi_\eL^\bullet)\to \Mod[(\Sigma, \bfs)]$ sending $\bfn$ to the
product of Dehn twists,
\begin{equation}\tau_\eL^\bullet(\bfn) = \prod_{e\in \eL^\ver} \tw_{\gamma_e}^{m_e},
\quad\text{with $m_e$ defined by}\quad \widetilde{\phi}_\eL^\bullet(\bfn) = (m_e
  \kappa_e)_{e\in \eL^\ver}\,,
\end{equation}
where $\gamma_e$ is the seam corresponding to $e$, and $\tw_{\gamma_e}$ is the Dehn twist around it.
The image of $\tau_\eL^\bullet$ is called the \emph{vertical $\eL$-twist
group} $\vTw[\eL]\subset \Mod[(\Sigma, \bfs)]$. Tracking the above definitions, we conclude the following.
\par
\begin{prop} \label{prop:rank_of_twist_group} The vertical $\eL$-twist group $\vTw[\eL]$ is a free
  abelian group of rank~$N$. Moreover, $\ker(\tau_\eL^\bullet)\subset\ker(\phi_\eL^\bullet)$
  is isomorphic to $\ZZ$, generated by $\boldsymbol{1}=(1,\dots,1)$, and $\ker(\phi_\eL^\bullet) =  \vTw\oplus\ker(\tau_\eL^\bullet) \isom \vTw\oplus\ZZ$.
\end{prop}
\par
\begin{proof}
Since $P_\eL$  is a torsion group, the rank of $\ker(\phi_\eL^\bullet)$ is equal to the
rank of the level rotation group~$\ZZ^{L^\bullet(\eL)}$ which is~$N+1$.  A tuple $\bfn$ lies in $\ker(\tau_\eL^\bullet)$
if and only if $n_{\ltop} = n_{\lbot}$ for every
vertical edge $e$.  Since the dual graph is connected, $\bfn$ is a multiple of
$\boldsymbol{1}$, so $\boldsymbol{1}$ generates $\ker(\tau_\eL^\bullet)$. The vector
$\boldsymbol{1}$ is primitive in the level rotation group, so it is also primitive in
$\ker(\phi_\eL^\bullet)$.  Hence there is a splitting of the short exact sequence
\begin{equation*}
    0\to \ZZ \to \ker(\phi_\eL^\bullet) \to \vTw  \to 0. \qedhere
  \end{equation*}
\end{proof}
\par
We define the \emph{horizontal $\eL$-twist group} to be the subgroup
$\hTw[\eL]\subset\Mod[(\Sigma, \bfs)]$ generated by Dehn twists around the horizontal
curves $\eL^h$.
\index[twist]{d020@$\hTw[\eL]$!Horizontal twist group}
We then define the \emph{$\eL$-twist group} to
be the direct sum
\index[twist]{d030@$\Tw$!Twist group}
\begin{equation*}
  \Tw \= \vTw \oplus \hTw\,.
\end{equation*}
\par
\medskip
Let $\left(\bff_i\right)_{i=0,\dots, -N}$ be the \emph{standard basis}
of~$\CC^{L^\bullet(\eL)}\cong \CC^{N+1}$, where $\bff_i = (0^{-i}, 1, 0^{N+i})$.
In order to describe the above groups in simpler terms, we will also use the
\emph{lower-triangular basis} $\left(\bfb_i\right)_{i=0,\dots, -N}$ defined as
$$\bfb_i = \sum_{k = -N}^{i} \bff_{k} = (0^{-i}, 1^{N+i+1})\,.$$
Then for $v_i \in \CC$, the element
$v_i \bfb_i$ acts by simultaneously multiplying the forms on all
levels $j\le i$ by $ \bfe(v_i)$. In particular, $v_0\bfb_0$ simultaneously
scales the form on every irreducible component of~$X$ by $\bfe(v_0)$.
\par
Recall from Section~\ref{sec:order} that for every level $i\in L(\eL)$ there
is a two-level undegeneration ${\rm dg}_i\colon\eL_i\rightsquigarrow\eL$ that contracts the
(vertical) edges of~$\Gamma(\Lambda)$ strictly above level~$i$ and the edges below or at level~$i$.
We denote by $\sTwi = ({\rm dg}_i)_* (\vTw[\eL_i])\subset \vTw[\eL]$
the corresponding subgroup of the vertical $\eL$-twist group.  Note that
$\vTw[\eL_i]$ is the cyclic group generated by the element $(0, a_i)$ with
$a_i = \lcm_e \kappa_e$ over all edges $e$ connecting the graph $\Gamma_{>i}$
to $\Gamma_{\le i}$. 
Moreover, $({\rm dg}_i)_{*}(0, a_i) = (0^{-i}, a_i^{N+i+1})
\in \ZZ^{L^\bullet(\eL)}$. \changed{Here we implicitly identify $\vTw[\eL]$ with $\ker(\phi_\eL^\bullet)\subset\ZZ^{L^\bullet(\eL)}$ modulo the diagonal subgroup $\ZZ$ as in Proposition~\ref{prop:rank_of_twist_group}.} It follows that
\index[twist]{d040@$\sTwi$!Simple vertical twist group of level~$i$}
$\sTwi$ is the cyclic subgroup of $\vTw[\eL]$ generated by
\index[twist]{b050@$a_i$, $m_{e,i}$! Defined by $a_i\= \lcm_e \kappa_e$ and $m_{e,i} \= a_i/\kappa_e$}
\begin{equation}
  \label{eq:aidef}
\tau_\eL^\bullet(a_i \bfb_i) \= \prod_e \tw_{\gamma_e}^{m_{e,i}},
\quad\text{with}\quad a_i\= \lcm_e \kappa_e  \quad\text{and}\quad
m_{e,i} \= a_i/\kappa_e\,,
\end{equation}
where the product and the $\lcm$ are taken for the set of vertical edges $e$
connecting $\Gamma_{>i}$ to $\Gamma_{\le i}$.   The collection of $\sTwi$ for
all $i\in L(\eL)$ generates a subgroup of the twist group, which
\index[twist]{d050@$\svTw$!Simple vertical twist group}
we call the \emph{simple vertical $\eL$-twist group $\svTw$}.
\par
\begin{lm}  \label{lm:defsvTw}
The simple vertical $\eL$-twist group is a finite index
subgroup of the vertical twist group that can be written as the direct sum
\begin{equation} \label{eq:defsvTw}
  \svTw \= \bigoplus_{i\in L(\eL)} \sTwi \,\subset\, \vTw[\eL]\,.
\end{equation}
\end{lm}
\par
\begin{proof}
The vectors $\bfb_i$ are linearly independent in the level rotation
group~$\ZZ^{L^\bullet(\Lambda)}$, and moreover, $\bfb_0 = (1, \dots, 1)$ generates
$\ker(\tau_\Lambda^\bullet)$.  Combining with Proposition~\ref{prop:rank_of_twist_group},
it follows immediately that the $N$ simple twists $\tau_\eL^\bullet(a_i \bfb_i)$ for
$i<0$ generate a rank $N$ subgroup of the vertical twist group $\vTw\isom\ZZ^N$.
\end{proof}
\par
We denote by $K_\Lambda$ the finite quotient $K_\Lambda = \vTw [\eL] / \svTw$.
\index[twist]{d055@$K_\Lambda $! Finite group defined by $\vTw/\svTw$}
In general, the inclusion in~\eqref{eq:defsvTw} is strict; see
Example~\ref{ex:trianglegraph}. We will see that this phenomenon is
responsible for the quotient singularities of our moduli space at
the boundary; see also Section~\ref{sec:modeldomain}. Finally, we denote by
$\sTw = \svTw \oplus \hTw$ the
\index[twist]{d060@ $\sTw$!Simple twist group}
\emph{simple $\eL$-twist group}.
\par
The $\CC^{L^\bullet(\eL)}$-action restricts to the $\CC^{L(\eL)}$-action, which
acts by scaling all but the top level.  All of the above objects have
analogues using this restricted action, which we denote by dropping the
superscript~$\bullet$.
For example, the restriction of $\phi_\eL^\bullet$ to $\ZZ^{L(\eL)}$ is denoted by~$\phi_\eL$.
\par
The homomorphisms $\phi_\eL^\bullet$ and $\phi_\eL$ have the same image in the prong rotation
group~$\Prot[\eL]$.  Similarly $\tau_\eL^\bullet$ and $\tau_\eL$ have the same image
$\vTw$ in $\Mod[(\Sigma, \bfs)]$.  Intuitively, the actions both of  $\CC^{L(\eL)}$ and
$\CC^{L^\bullet(\eL)}$ yield the same subgroups of the prong rotation group and of the mapping class group,
because the top-level factor~$\CC$ of $\CC^{L^\bullet(\eL)}$ acts (in terms of the lower-triangular basis) by simultaneously
scaling the differentials at all levels by the same factor, which has no effect on the markings or \prmas.
\par
\medskip
We remark that there is another characterization of the twist group as
a subgroup of the full twist group $\oldtwist$. Recall that the full twist group
was defined in Section~\ref{sec:dehn-space-deligne} as the group generated
by Dehn twists around all curves of $\eL$, and is isomorphic
to~$\ZZ^{E(\eL)}$. To prove the following proposition one checks that the quotient
of $\oldtwist$ by $\Tw$ and by the group $\rottwist$ defined implicitly in the
proposition are torsion free, and that $\Tw$ and $\rottwist$ have the same rank.
\par
\begin{prop}\label{prop:TwRotInv}
Let $(X,\eta,\sigma,f) \in \ptwT$. The twist group $\Tw$ is the
subgroup of $\oldtwist$ that fixes the turning number of every immersed
arc in~$\ol{X}_\sigma$ that starts and ends at the same level.
\end{prop}

%%%%%%%%%%%%%%%%%%%%%%%%%%%%%%%%%
\subsection{Level rotation tori} \label{sec:levrottori}
%%%%%%%%%%%%%%%%%%%%%%%%%%%%%%%%%

We define the \emph{level rotation torus} $\Tprong[\eL]$ to be the
quotient \changed{
$\Tprong[\eL] = \CC^{L(\eL)} / \vTw[\eL] $.} 
%\cong \CC^{L(\eL)} /\ker(\phi_\eL)$. 
Similarly, the \emph{simple level rotation torus} is
the quotient $\Tprong[\eL]^s = \CC^{L(\eL)} / \svTw[\eL]$.
\index[twist]{d070@$\Tprong[\eL]^s$, $\Tprong[\eL]$!(Simple) level
  rotation torus} They will play a prominent role in defining families
of \msds. The level rotation tori depend obviously only on the
enhanced level graph $\Gamma(\Lambda)$ rather than on the multicurves
and we will thus write~$\Tprong$ and $\Tprong[\eL]$ interchangeably.
\par
The following is an alternative characterization of the level rotation torus.
Similarly to the twist groups, the ambient $\CC^{L(\Lambda)}$ and also
$(\CC^*)^{L(\Lambda)}$ can be parameterized using the standard and the
triangular basis.
\par
\begin{prop} \label{prop:LRTchar}
The level rotation torus $\Tprong[\eL]$ is the identity component of the subgroup of
\begin{equation}\label{eq:torusproduct}
(\CC^*)^{L(\eL)} \times (\CC^*)^{E(\eL)}
\= \changed{\left( (r_i)_{i\in L(\eL)}, (\rho_e)_{e\in E(\eL)}\right)}
%\left( (r_i, \rho_e)\right)_{i\in L(\eL),e\in E(\eL)}
\end{equation}
cut out by the set of equations
\begin{equation}\label{eq:rrho}
  r_{\lbot} \dots  r_{\ltop-1} \= \rho_e^{\kappa_e}
\end{equation}
for all edges $e$, where the $r_i$ are the coordinates in the triangular
basis.
\par
There is an identification $T^s_\eL \cong (\CC^*)^{\changed{L(\eL)}}$
such that the quotient map $T^s_\eL \to T_\eL$ is given in coordinates by
\begin{equation}\label{eq:TsT}
(q_i) \,\mapsto\,
(r_i, \rho_e) \= \left(q_i^{a_i}, \prod_{i=\lbot}^{\ltop-1} q_i^{a_i/\kappa_e} \right)
\end{equation}
with the numbers $a_i$ defined in~\eqref{eq:aidef}.
\end{prop}
\par
\begin{proof}
Consider first the projection of the subgroup of $(\CC^*)^{L(\eL)} \times
(\CC^*)^{E(\eL)}$ cut out by Equations~\eqref{eq:rrho} (not just its identity
component) onto the $(\CC^*)^{L(\eL)}$ factor. Since each~$\rho_e$ is determined
by the~$r_i$'s up to roots of unity, this projection is an unramified (possibly
disconnected) cover with fiber equal to the prong rotation group $P_\Gamma =\prod_e
\ZZ/\kappa_e\ZZ$.
\par
Next we determine the connected component of the identity within this subgroup.
The fundamental group of $(\CC^*)^{L(\eL)}$ is equal to $\ZZ^{L(\eL)}$, and an
element $\bfn\in \ZZ^{L(\eL)}$ acts by multiplying each coordinate~$\rho_e$
by $\bfe((n_{\ltop}-n_{\lbot})/\kappa_e)$. Recalling Equation~\eqref{eq:definephi}
that defines $\phi_\eL$, we see that $\ker (\phi_\eL)$ is precisely the set of
elements $\bfn\in\ZZ^{L(\Lambda)}$ that act by trivial monodromy. Thus the connected
component of the identity is an unramified cover of $(\CC^*)^{L(\Lambda)}$ with deck
transformation group being the image of monodromy, i.e.~$\ZZ^{L(\eL)}
/\ker(\phi_\eL)$. As by definition the level rotation torus $\Tprong[\eL]$ is a
Galois cover of $(\CC^*)^{L(\eL)}$ with the same Galois group, it is equal to the
connected component of the identity. This shows the first statement of the claim,
from which~\eqref{eq:TsT} follows, since we exhibit a map of tori of the same
dimension and the right-hand side satisfies~\eqref{eq:rrho}.
\end{proof}
\par
These constructions can also be regarded as covariant functors on the
category of ordered enhanced multicurves on $(\Sigma, \bfs)$.
More precisely, a degeneration of enhanced multicurves $\degen\colon\eL_1
\rightsquigarrow \eL_2$ induces a monomorphism \changed{$\widehat{\degen_*}
\colon \CC^{L(\eL_1)} \to \CC^{L(\eL_2)}$}. Using Proposition~\ref{prop:rank_of_twist_group}
to think of twist groups as kernels of the map to the prong rotation group, up to a
$\ZZ$-summand, we obtain a monomorphism $\degen_*\colon \Tw[\eL_1]\inj \Tw[\eL_2]$.
\par
\begin{lm} \label{lm:monoLRT}
A degeneration of enhanced multicurves $\degen\colon\eL_1 \rightsquigarrow \eL_2$
induces an injective homomorphism $\degen_* \colon T_{\eL_1}\to T_{\eL_2}$.
In the coordinates~\eqref{eq:torusproduct} the image is cut out by equations $\rho_e=1$
for every edge~$e$ of~$\eL_2$ that is contracted in~$\eL_1$, and
respectively $r_i=1$ for every level~$i\in L(\eL_2)$ such that the images of~$i$
and $i+1$ are the same in~$L(\eL_1)$. \changed{Moreover, if $\degen'\colon\eL_1
\rightsquigarrow \eL_2$ and $\degen''\colon\eL_2
\rightsquigarrow \eL_3$ are two degenerations, and if $\degen\colon\eL_1
\rightsquigarrow \eL_3$ denotes their composition, then $\degen_{*} = \degen_{*}'' \circ \degen_{*}'$. }
\end{lm}
\par
\begin{proof} The description of the image is obvious. For injectivity
we have to show that an element in $\Tw[\eL_2]$  in the image of~$\widehat{\degen_*}$
already belongs to $\Tw[\eL_1]$. This is obvious from the description
of the twist group in Proposition~\ref{prop:rank_of_twist_group}. \changed{These descriptions also directly imply the last part of the claim.}
\end{proof}
\par
We will also  need the rank~$N+1$ \emph{extended level rotation torus}
$\Textd[\eL] = \CC^{L^\bullet(\eL)} / (\vTw[\eL] \oplus \ZZ)$, as
well as its simple variant $\Tsextd[\eL] = \CC^{L^\bullet(\eL)} / (\svTw[\eL] \oplus \ZZ)$.
\par
\medskip
The level rotation torus $\Tprong[\eL]$ acts on a \ptwd $(X,\eta,\sigma)$,
where~$\sigma$ is a \prma, via
\index[twist]{b030@$\ast$! Action of $T_{\eL}$ on \ptwds}
 \begin{equation}\label{eq:actionLRT}
(r_i, \rho_e)\,\ast\, \bigl(X,(\eta_{(i)}),\,  (\sigma_e)\bigr)
\= \biggl(X, \bigl(r_i\dots r_{-1}\eta_{(i)}\bigr),\,
(\rho_e \ast \sigma_e)\biggr)
\end{equation}
where $\rho_e \ast \sigma_e$ is the \prma
$\pin[q^-][\eta] \to \pout[q^+][\eta]$ at the node~$q$
corresponding to~$e$
given by $\sigma_e$ post-composed with the rotation by $\arg(\rho_{e})$.
Note that this is the exponential version of the action described
in item~(2) of Section~\ref{sec:ConPMD}. If~\changed{$\ol{X}_{\bfsigma}$} is moreover marked
by~$f$ we define $(\rho_e) \ast f$ to be the marking of
$(r_i, \rho_e)\,\ast\, \bigl(X, (\eta_{(i)}),\,  (\sigma_e) \bigr)$
obtained by post-composing~$f$ with a fractional Dehn twist of
angle $\arg(\rho_e)$ on each vertical edge~$e$.
This marking is well-defined up to an element in $\Tw$ only.
\par
Analogously, the simple level rotation torus acts on the set of \ptwds.
We can also assume that these differentials are marked by $f$, defined modulo
the action of the simple twist group. Using the map $\Tsimp \to\Tprong[\eL]$ given
in Proposition~\ref{prop:LRTchar}, the action $\ast$ defined in
Equation~\eqref{eq:actionLRT} is given in the triangular basis  by
\begin{equation}\label{eq:Tsimp_action}
\bft \ast (X, (\eta_{(i)}), (\sigma_{e}), f)
\= (X, ( t_i^{a_i} \dots \,t_{-1}^{a_{-1}} \eta_{(i)}), (f_{e}\ast \sigma_{e}),
(f_{e})\ast f)\,,
\end{equation}
where \changed{$\bft = (t_i)_{i\in L(\Lambda)}$ and} $f_{e}=\prod_{i=\lbot}^{\ltop-1} t_i^{a_i/\kappa_e}$ with
the integers $a_i$ defined in~\eqref{eq:aidef}. For later use, we recast
Proposition~\ref{prop:LRTchar} in terms of this action.
\par
\begin{cor} \label{cor:compLRT}
Equivalence classes of \ptwds up to the action~\eqref{eq:actionLRT} of the
level rotation torus are in bijection with connected components of the
subgroup of $(\CC^*)^{L(\eL)} \times (\CC^*)^{E(\eL)}$ cut out by
Equation~\eqref{eq:rrho}.
\end{cor}
\changed{
\begin{proof}
Recall that equivalence classes of \ptwds are the equivalence classes of prongs under the action of the level-wise twists, i.e. the orbits of $P_{\eL}$ under the action of $\ZZ^{L^\bullet(\eL)}$ given by the homomorphism $\phi_\eL^\bullet$ defined in \eqref{eq:definephi}. Hence there are in bijection with $\ZZ^{L^\bullet(\eL)}/\ker(\phi_\eL^\bullet)$. As discussed in the proof of Proposition \ref{prop:LRTchar}, $\ZZ^{L^\bullet(\eL)}$ acts transitively on the set of connected components of the subgroup given by Equation~\eqref{eq:rrho}, with the connected component of the identity mapped to itself precisely by $\ker(\phi_\eL^\bullet)$. Thus the set of connected components is also in bijection with $\ZZ^{L^\bullet(\eL)}/\ker(\phi_L^\bullet)$.
\end{proof}}

%%%%%%%%%%%%%%%%%%%%%%%%%%%%%%%%%
\subsection{The covering viewpoint} \label{sec:coverGH}
%%%%%%%%%%%%%%%%%%%%%%%%%%%%%%%%%

So far we have analyzed the group $\CC^{L(\Lambda)}$ acting on $\ptwT$
and defined twist groups as finite index subgroups of $\ZZ^{L(\Lambda)}$.
In what follows we will use compactifications of quotients of $\ptwT$ by
twist groups. We can alternatively construct them as finite covers
of $\OBteich$, as we explain now.
\par
The triangular basis provides an identification of $\Tsimp[\eL]$ with
$(\CC^*)^{L(\Lambda)}$ and we denote by $\Tsimp[\Lambda,i]$ the~$i$-th
factor of this torus. Recall the direct sum expression of $\svTw$
in~\eqref{eq:defsvTw}. We define the {\em \lw ramification groups}
to be $H_i = \ZZ_i/\svTw[\Lambda, i]$, where $\ZZ_i$ is the~$i$-th factor
of $\ZZ^{L(\Lambda)} \subset \CC^{L(\Lambda)}$.
By definition, we have the cardinality $|H_i| = a_i$ (defined in~\eqref{eq:aidef}) and
the identification
$H :=\oplus_{i \in L(\Lambda)} H_i = {\rm Ker}(\Tsimp  \to (\CC^*)^{L(\Lambda)})$.
On the other hand, we may define the {\em (full) ramification group}
associated with an enhanced level graph~$\Lambda$ to be
$G := {\rm Ker}(\Tprong[\eL] \to (\CC^*)^{L(\Lambda)})$.
\index[twist]{d080@$H$, $G$! Ramifications groups}
By definition we have an exact sequence of finite abelian groups
\begin{equation}\label{eq:exseqcover}
0 \to K_{\eL} = \vTw/\svTw \to H \to G \to 0 \, .
\end{equation}
Note that the map $H_i \to G$ is injective for every $i\in L(\eL)$, since
an element in~$H_i$ and its image in~$G$ act by the same fractional Dehn
twists the seams. The situation is summarized by the following diagram:
\bes
\begin{tikzpicture}[scale=1.2,decoration={
    markings,
    mark=at position 1 with {\arrow[]{>}}}]
%Gauche
\coordinate   (x0) at (-1,0); \node [] at  (x0) {$\ptwT/\sTw$};
\coordinate (x1) at (1,-1); \node [] at (x1) {$\ptwT/\Tw$};
\coordinate (x2) at (-1,-2) ; \node [] at (x2) {$\OBteich$};

 \draw[postaction=decorate](x0)+(.5,-.2) -- node[above right]{$/K_{\eL}$}  (1,-.75) ;
 \draw[postaction=decorate] (x0)+(0,-.2) --node[left]{$/H$}  (-1,-1.75) ;
\draw[postaction=decorate] (x1)+(0,-.2) -- node[below]{$/G$} (-.8,-1.75) ;
\end{tikzpicture}
\ees
Of course, all the maps in the diagram are unramified covers, but
they will become ramified with local ramification groups~$H_i$ at the
appropriate boundary divisors, once we consider the compactifications.
\par
\begin{exa} ({\em The twist group quotient $K_\Lambda$ can be non-trivial.})
\label{ex:trianglegraph}
Consider the enhanced level graphs $\eG_{1}$ and $\eG_{2}$ of our running example
in Section~\ref{sec:runningex}. In both cases, the \lw undegenerations
$\degen_{-1}$ and $\degen_{-2}$ as introduced in Section~\ref{sec:order} are both
equal to the graph with two vertices connected by two edges, labeled by~$1$ and~$3$
respectively.
 \par
First consider the enhanced level graph $\Gamma_1$ on the left of
Figure~\ref{cap:running}. Then the map
$$ \phi_{\Lambda_{1}}^{\bullet}\colon \ZZ^{L^\bullet(\Lambda)} \cong \ZZ^3 \to P_{\Gamma_1} \cong\ZZ/3\ZZ \times \ZZ/3\ZZ \times \ZZ/\ZZ $$
is given by
$$(n_0, n_{-1}, n_{-2}) \mapsto (n_0 - n_{-1}, n_{-1} - n_{-2}, n_0 - n_{-2}) $$
in the standard basis. Then $\ker(\phi_{\Lambda_{1}}^{\bullet})$  is the subgroup of $\ZZ^{3}$ consisting  of elements of
the form $(m, m+3k_1, m + 3k_2)$ for $m, k_1, k_2\in \ZZ$, and hence $\vTw = \Tw$
 is generated by the vectors $(0, 3, 0)$ and $(0, 0, 3)$.
The simple $\Lambda_{1}$-twist group $\svTw[\eL_{1}]$ is the direct sum
$\svTw[\Lambda_{1}, -1] \oplus \svTw[\Lambda_{1}, -2]$
where $\svTw[\Lambda_{1}, i] = ({\rm dg}_i)_{*}(\vTw[\Lambda_{1,i}])$. In this
case $\svTw[\Lambda_{1}, -1]$ is generated by $ (0, 3, 3)$ and $\svTw[\Lambda_{1}, -2]$
is generated by $(0, 0, 3)$, hence $\svTw[\eL_{1}]$ coincides with $\vTw[\eL_{1}]$.
The level rotation torus $T_{\Lambda_{1}}$ is an unramified cover of  $(\CC^{*})^2$
of degree~$9$ with Galois group equal to the prong rotation group  $P_{\Gamma_1}$.
The action of the level rotation group by $\phi_{\Lambda_{1}}^{\bullet}$ has a unique orbit, hence the nine prong-matchings are all equivalent.
\par
Next consider the enhanced level graph $\Gamma_2$ on the right of Figure~\ref{cap:running}. Then we have
$$ \phi_{\Lambda_{2}}^{\bullet}\colon \ZZ^{L^\bullet(\Lambda_{2})} \cong \ZZ^3 \to P_{\Gamma_2} \cong\ZZ/\ZZ \times \ZZ/\ZZ \times \ZZ/3\ZZ $$
given by
$$(n_0, n_{-1}, n_{-2}) \mapsto (n_0 - n_{-1}, n_{-1} - n_{-2}, n_0 - n_{-2}) $$
in the standard basis. Then $\ker(\phi_{\Lambda_{2}}^{\bullet})$ is the subgroup of $\ZZ^{3}$ consisting of elements of the form $(m, m+k_1, m + 3k_2)$ for $m, k_1, k_2\in \ZZ$, and hence $\vTw[\eL_{2}] = \Tw[\eL_{2}]$  is generated by the vectors  $(0, 1, 0)$ and $(0, 0, 3)$. The simple
$\Lambda_{2}$-twist group $\svTw[\eL_{2}]$ is the direct sum $\svTw[\Lambda_{2}, -1] \oplus \svTw[\Lambda_{2}, -2]$
where $\svTw[\Lambda_{2}, -1]$ is generated by
$(0, 3, 3)$ and $\svTw[\Lambda_{2}, -2]$ is generated by $(0, 0, 3)$, hence $\svTw[\eL_{2}]$ is a subgroup of index $3$ in $\vTw[\eL_{2}]$.
The level rotation torus~$T_{\Lambda_{2}}$ is an unramified cover of $(\CC^{*})^3$ of degree~$3$ with Galois group equal to the prong rotation group $P_{\Gamma_2}$. The action of the level rotation group has a unique orbit, hence the three prong-matchings are all equivalent.
\par
In both cases the local ramification groups are $H_i \cong \ZZ/3\ZZ$ for $i=-1,-2$,
and the group~$G$ coincides with the prong rotation group in each case. However, in
the first case $H \cong G$, while in the second case $H \twoheadrightarrow G$ has
kernel $K_{\eL_{2}}=\ZZ/3$. In particular, in the second case the quotient map of a
smooth space by $K_{\eL_{2}}$  will produce quotient singularities in our
compactification, which will be illustrated in Example~\ref{ex:quotientsing}.
\end{exa}
\smallskip
\par
In order to conclude this section we give a cautionary example.
\par
\begin{exa}
({\em The number of non-equivalent prong-matchings may decrease under degenerations.})
\label{ex:pmdecrease}
 We consider the degeneration of enhanced level graphs as shown in Figure~\ref{fig:pmdecrease}. The first graph is a two-level graph with two edges $e_{1}$ and $e_{2}$  between two vertices. Moreover we set $\kappa_{1}=\kappa_{2}=2$. We degenerate this graph to a three-level triangle with three edges $e_{1}$, $e_{2}$ and $e_{3}$ labeled as in Figure~\ref{fig:pmdecrease}. The edges are labeled by $\kappa_{1}=\kappa_{2}=2$ and $\kappa_{3}=1$.
 \begin{figure}[htb]
   \centering
\begin{tikzpicture}[scale=1]
\coordinate (a1) at (4,0);\fill (a1) circle (2pt);
\coordinate (a2) at (4.4,-1);\fill (a2) circle (2pt);
\coordinate (a3) at (4,-2);\fill (a3) circle (2pt);
\draw (a1)  .. controls ++(-30:.5) and ++(90:.2) .. (a2) coordinate[pos=.5] (b2);
\draw (a1)  .. controls ++(-150:.5) and ++(150:.5) ..  (a3) coordinate[pos=.5] (b1);
\draw (a2)  .. controls ++(-90:.2) and ++(30:.5) .. (a3) coordinate[pos=.5] (b3);
\node[left] at (b1) {$e_{1}$};
\node[right] at (b2) {$e_{2}$};
\node[right] at (b3) {$e_{3}$};

\draw[->,decorate,decoration={snake,amplitude=.4mm,segment length=2mm,post length=1mm}] (1.2,-.5) -- (2.8,-.5) coordinate[pos=.5] (b);
\node[above] at (b) {$\degen_{-1}$};

\coordinate (a1) at (0,0);\fill (a1) circle (2pt);
\coordinate (a2) at (0,-1);\fill (a2) circle (2pt);
\draw (a1) .. controls ++(-120:.3) and ++(120:.3) ..  (a2) coordinate[pos=.5] (b1);
\draw (a1) .. controls ++(-60:.3) and ++(60:.3) .. (a2) coordinate[pos=.5] (b2);
\node[left] at (b1) {$e_{1}$};
\node[right] at (b2) {$e_{2}$};
\end{tikzpicture}
\caption{A degeneration that decreases the number of prong-matchings.}
\label{fig:pmdecrease}
\end{figure}
\par
Clearly the action of the level rotation group of Equation~\eqref{eq:definephi} has two orbits in the first case. On the other hand one can check that in the second case it has only one orbit. Hence this degeneration decreases the number of non-equivalent prong-matchings from two to one.
\end{exa}

%%%%%%%%%%%%%%%%%%%%%%%%%%%%%%%%%
\subsection{The level rotation torus closure} \label{sec:LRTC}
%%%%%%%%%%%%%%%%%%%%%%%%%%%%%%%%%

The partial closures of tori we define here will give local models of
the toroidal part of our compactification. Recall that
the level rotation torus $T_\eL$ is by Proposition~\ref{prop:LRTchar}
naturally embedded in~$(\CC^*)^{L(\eL)} \times (\CC^*)^{E(\eL)}$, where it is
the connected component of the identity of the torus cut out
by Equation~\eqref{eq:rrho}. We define the {\em level rotation
torus closure $\ol{T}_{\eL}$} to be the closure of this identity component
in $\CC^{L(\eL)} \times \CC^{E(\eL)}$.
\par
On the other hand, the simple level rotation torus $T^{s}_\eL$ is naturally
identified with $(\CC^*)^{L(\eL)}$, with closure $\ol{T}_{\eL}^{s} =  \CC^{L(\eL)}$.
The group~$K_\eL = \Tw/\Tw^s$ introduced in the previous section acts
on $T^{s}_\eL = \CC^{L(\Lambda)}/\sTw$. Since each element in~$K_\eL$ acts diagonally
by a tuple of roots of unity, this action extends to an action of $K_\eL$ on
$\ol{T}_{\eL}^{s}$. The quotient will be the local model for the toroidal part
of the compactification that we will construct. Our goal here is to relate
this viewpoint with the closure of the level rotation torus.
\par
\begin{prop} \label{prop:charLRTC}
The projection map $p \colon T^{s}_\eL \to T_\eL$ given by Equation~\eqref{eq:TsT}
extends to $\ol{T}_{\eL}^{s}$ and descends to an isomorphism $\ol{p}\colon
\ol{T}_{\eL}^{s}/K_\eL \to \ol{T}_{\eL}^{{n}}$ to the normalization of the
level rotation torus closure.
\end{prop}
\par
\begin{proof}
The map~$p$ extends to a map $p_2 \colon \ol{T}_{\eL}^{s} \to \ol{T}_{\eL}$
since it is given explicitly by monomials, in the coordinates used for
taking the closure. Since $K_\eL$ acts on~$\ol{T}_{\eL}^{s}$ and since~$p$
is the quotient by~$K_\eL$, the map $p_2$ factors through the quotient
to give $p_3\colon \ol{T}_{\eL}^{s}/K_\eL \to \ol{T}_{\eL}$. Since a quotient
of $\CC^{L(\eL)}$ by a finite group is normal, the map $p_3$ factors as
the normalization map precomposed with a map~$\ol{p}$. \changed{The map~$p_2$
is finite, since its composition with the projection onto the $(r_i)_{i
\in L(\Lambda)}$-components is the finite map $(q_i) \mapsto (q_i^{a_i})$.
Consequently,} the map $\ol{p}$ is finite and \changed{moreover it is}
birational since on the open set
$T^{s}_\eL/ K_\eL = T_{\eL}$. Since the target is normal, it follows that the map
$\ol{p}$ is an isomorphism (see \cite[Lemma~28.52.8]{stacks-project}).
\end{proof}
\par
\begin{exa}
({\em In general, $\ol{T}_{\eL}$ is not normal.})
\label{ex:nonnormal}
This can be seen from the example of a graph~$\Gamma$ with two vertices on two
levels, connected by two edges~$e_1$ and~$e_2$ with~$\kappa_1 = 2$ and $\kappa_2 = 3$.
Then $\ol{T}_{\eL}$ has a cusp, locally modeled on $\CC[f_1,f_2,s]/\langle f_1^2 -s,
f_2^3 -s \rangle$. Its normalization is $\CC[t]$, with the normalization map given
by $f_1=t^3$ and $f_2=t^2$. This change of coordinates also describes~$\ol{p}$ here,
since $\Tw = \Tw^s$ in this example.
\end{exa}

%%%%%%%%%%%%%%%%%%%%%%%%%%%%%%%%%%%%
\section{Augmented \Teichmuller space of marked \msds}
\label{sec:AugTeich}
%%%%%%%%%%%%%%%%%%%%%%%%%%%%%%%%%%%%%

In this section, we formally introduce the notion of a \msd as well
as markings.  We introduce the augmented \Teichmuller space,
parameterizing equivalence classes of marked \msds, define its
topology, and establish basic topological properties. The main results
are the Hausdorff property of quotients of augmented \Teichmuller space
by any subgroup of the mapping class group in Theorem~\ref{thm:PaugHausdorff}
and the compactness of the moduli space in Theorem~\ref{thm:MSDcompact}.

%%%%%%%%%%%%%%%%%%%%%%%
\subsection{\Msds and markings} \label{sec:singleMSD}
%%%%%%%%%%%%%%%%%%%%%%%%

The notation we developed so far enables us to make the
Definition~\ref{def:one_msd} of multi-scale differentials precise. The
notion of general families of multi-scale differentials is not needed for
the construction of augmented \Teichmuller space and Dehn space, and will thus
only be given in Section~\ref{sec:famnew}.
\par
\begin{df}
  \label{df:msd}
  Given an enhanced level graph $\Gamma$,
  a \emph{\msd  of type $(\mu, \Gamma)$} consists of the following data:
  \begin{enumerate}
  \item A pointed stable curve $(X, \bfz)$,
  \item an identification of the dual graph $\Gamma_X$ with $\Gamma$,
    \item a collection of differentials
      $\bfomega = (\omega_{(i)})_{i\in L^\bullet(\Gamma)}$ that
      give $(X, \bfz)$ the structure of a twisted differential of type
      $(\mu, \Gamma)$ compatible with $\Gamma$, and
    \item a global \prma $\bfsigma = (\sigma_e)_{e\in  E(\Gamma)}$,
    \end{enumerate}
where $(\bfomega,\bfsigma)$ is defined to be equivalent to $(\bfomega',\bfsigma')$
if and only if there exists an element of the level rotation torus $T_\Gamma$ that
sends $(\bfomega,\bfsigma)$ to $(\bfomega',\bfsigma')$ under the
action~\eqref{eq:actionLRT}.
\end{df}
\par
Given an enhanced multicurve $\Lambda\subset \Sigma \setminus \bfs$, a {\em
marked \msd of type $(\mu,\Lambda)$} is the \changed{equivalence class of $(\bfomega, \bfsigma, f)$ (up to the action
of $\CC^{L(\Lambda)}$ as in Section~\ref{sec:ConPMD})}, 
where $f$ is a marking (in the sense of Definition~\ref{df:marking}) of
the \ptwd defined by $(\bfomega, \bfsigma)$. 

Note that, since the vertical twist group
$\vTw\subset \CC^{L(\Lambda)}$ fixes $(\bfomega,\bfsigma)$, a marked
\msd may alternatively be defined as a $\Tprong$-equivalence class of
the data $(\bfomega, \bfsigma, [f])$, where $[f]$ is the equivalence
class of markings up to the action of $\Tw$.  More generally, we will
call this a $G$-marked \msd if the marking is taken up to the action of
a subgroup $G$ of the mapping class group.
\par
Projectivized \msds and their marked analogues are defined similarly,
replacing $\Tprong[\Gamma]$ with $\Textd[\Gamma]$.

%%%%%%%%%%%%%%%%%%%%%%%%%%%%%%%%%
\subsection{Augmented \Teichmuller space of \msds as a set.}
\label{sec:AugSet}
%%%%%%%%%%%%%%%%%%%%%%%%%%%%%%%%%

As a first step towards defining the augmented \Teichmuller space of \msds of type $\mu$, we introduce it as a set.
\begin{df}
\index[teich]{d060@$\Oaugteich$!Augmented \Teichmuller space of
marked flat surfaces of type $\mu$} The \emph{(projectivized) augmented
\Teichmuller space of \msds of type $\mu$} is the set of equivalence classes of
(projectivized) marked \msds of type $\mu$, which we denote by $\Oaugteich$ or
  $\Paugteich$ respectively.
\end{df}
\par
Given an enhanced multicurve $\Lambda\subset \Sigma \setminus \bfs$, we
define the \emph{(projectivized) $\Lambda$-boundary stratum}
\index[teich]{d065@$\msT[\Lambda]$! $\Lambda$-boundary stratum}
to be the set $\msT[\Lambda] \subset \Oaugteich$  or
$\pmsT[\Lambda]\subset \Paugteich$ of (projectivized) marked \msds of
type $(\mu, \Lambda)$. Note that the special case
$\msT[\emptyset] = \ptwT[\emptyset] = \Oteich$ is the \Teichmuller space
of marked flat surfaces of type $\mu$.
\par
\begin{lm} \label{le:OmBsmooth}
\changed{The boundary strata $\msT[\Lambda] = \ptwT[\Lambda] / \CC^{L(\Lambda)}$
and their projectivizations $\pmsT[\Lambda] = \ptwT / \CC^{L^\bullet(\Lambda)}$
are smooth complex manifolds.  Moreover, the groups $\CC^{L(\Lambda)}$ and
$\CC^{L^\bullet(\Lambda)}$ act freely on $\ptwT[\Lambda]$.}
\end{lm}
\par
\begin{proof}
\changed{Recall from Section~\ref{sec:ptwds} that there is an unramified
covering map $\msT[\Lambda] \to \OBteich$ where the fibers correspond to
simultaneous Dehn twists around the vertical nodes. The group
$(\CC^*)^{L(\Lambda)} = \CC^{L(\Lambda)} / \ZZ^{L(\Lambda)}$ acts on $\OBteich$
by projectivizing each (but the top) level of the level-wise product
of Teichmüller spaces of certain meromorphic differentials constrained by
residue conditions. Each of these quotients is smooth. The group
$\ZZ^{L(\Lambda)}$ acts by permuting the sheets of the covering. This implies
the first claim and the argument in the projectivized case is the same.
The second claim is an obvious consequence of the level graph being
connected, and of there being at least one vertex at each level.}
\end{proof}
\par
The augmented \Teichmuller spaces are the disjoint union of these strata
over the set of enhanced multi\-curves~$\Lambda\subset \Sigma \setminus \bfs$:
\be \label{eq:defbdstrata}
  \Oaugteich \= \coprod_{\eL} \msT[\eL]\quad \text{ and }\quad
  \Paugteich \= \coprod_{\eL} \pmsT[\eL]\,.
\ee
\par
The mapping class group $\Mod$ acts on $\Oaugteich$ and~$\Paugteich$
by pre-compo\-sition of the marking, that is, given $g\in \Mod$, an
equivalence class of
markings $f\colon \Sigma \to \overline{X}_\bfsigma$ is replaced with
$[f\circ g^{-1}]$
\par
\begin{prop}
  \label{prop:stabilizers}
  The subgroup of $\Mod$ fixing the boundary stratum $\msT[\Lambda]$
  pointwise is exactly the twist group $\Tw[\Lambda]$.  Moreover, if
  $\Lambda'$ is a degeneration of~$\Lambda$, then the twist group
  $\Tw[\Lambda]$ fixes the boundary stratum $\msT[\Lambda']$
  pointwise.  Both statements hold for the projectivizations as well.
\end{prop}
\par
\begin{proof}
  The first statement follows directly from the definition of the
  equivalence relation on markings.  For the second statement, if
  $\Lambda'$ is a degeneration of~$\Lambda$, then
  $\Tw[\Lambda]\subset \Tw[\Lambda']$. Hence~$\Tw[\Lambda]$ fixes
  $\msT[\Lambda']$ pointwise.
\end{proof}

%%%%%%%%%%%%%%%%%%%%%%%%
\subsection{Augmented \Teichmuller space  as a topological space.}
\label{sec:topspace}
%%%%%%%%%%%%%%%%%%%%%%%%

We now give both augmented \Teichmuller spaces $\Oaugteich$ and $\Paugteich$
a topology.  We give a sequential definition of the topology  first (see e.g.\
\cite[Section~I.8.9]{borelji} for a precise discussion of defining a topology in
this way). In the proofs below we also give a definition by specifying a basis
of the topology. Unless stated otherwise,
$\bfomega$ in the tuple $(X, \bfz, \bfomega, \cleq, \bfsigma, f)$ refers to a
chosen representative of the equivalence class.  We also
write $\overline{X}_{\bfsigma_n}$ as shorthand for $(\overline{X}_n)_{\bfsigma_n}$.
\par
\begin{df}  \label{def:augmented_topology_def}
A sequence  $(X_n, \bfz_n, \bfomega_n, \cleq_n, \bfsigma_n, f_n)
\in \pmsT[\Lambda_n]$ converges to a point
$(X, \bfz, \bfomega, \cleq, \bfsigma, f)\in\pmsT[\Lambda]\subset\Paugteich$,
if there exist representatives (that we denote with the same symbols)
in $\ptwT[\Lambda_n]$ and $\ptwT$, a sequence of positive numbers~$\epsilon_n$ converging to~$0$,
and a sequence of vectors
$\bfd_n = \{d_{n,i}\}_{i \in L^\bullet(\Lambda)} \in \CC^{L^\bullet(\Lambda)}$ such that the
following conditions hold, where we denote $c_{n,i} = \bfe(d_{n,i})$:
\begin{enumerate}[(1)]
\item For sufficiently large~$n$ there is an undegeneration of
  enhanced multicurves $(\delta_n,D_n^\hor)$ with
  $\delta_n\colon L^\bullet(\Lambda) \to L^\bullet(\Lambda_n)$ (see Definition~\ref{def:degeneration}).
\item  For sufficiently large~$n$ there exists an \diffeogen $g_n
\colon \overline{X}_{\bfd_n \cdot\bfsigma_n}\to \overline{X}_{ \bfsigma}$   that
is compatible with the markings (in the sense that $g_n \circ (\bfd_n
\cdot f_n)$ is isotopic to~$f$ rel marked points) and such that $g_n^{-1}$
is conformal on the $\epsilon_n$-thick part $\Thick[(X, \bfz)][\epsilon_n]$.
\item The restriction of $c_{n,i} (g_n)_*(\bfomega_n)$ to the $\epsilon_n$-thick part of
the level~$i$ subsurface of $(X, \bfz)$ converges
uniformly on compact sets to $\omega_{(i)}$.
\item For any $i,j \in L^\bullet(\Lambda)$ with $i>j$, and any subsequence
along which $\delta_n(i)=\delta_n(j)$, we have
    \begin{equation*}
      \lim_{n\to\infty} \frac{|c_{n,i}|}{|c_{n,j}|} = 0\,.
    \end{equation*}
  \item The \diffeogens $g_n$ are asymptotically \tnp.
\end{enumerate}
\par
For convergence in $\Oaugteich$, we require moreover that $c_{n,0} = 1$
for the rescaling constant corresponding to the top level
of~$\Lambda$.  
\end{df}
\par
Note that the notion of convergence does not depend on the choice of
representative of $X$ in $\ptwT$
since if $X' = \bfd' \cdot X$ is another representative, then using $\bfd' + \bfd$
certifies convergence to~$X'$. Note moreover, that in item~(5) we could as well require
the difference of turning numbers to tend to $0$ for a fixed collection of arcs
dual to the collection of seams, since for any fixed arc $\gamma$
disjoint from the seams, $\tau(g_n^{-1}(\gamma)) \to \tau(\gamma)$, as
$g_n$ converges uniformly $C^1$ to the identity on $\gamma$.
\par
This topology can be given equivalently by a basis of open sets which
we now describe.  For a given marked \msd $X$, let $V_\epsilon(X)$ be
the set of marked \msds $(X_0, \bfz_0, \bfomega_0, \cleq_0, \bfsigma_0,
f_0)$ such that there exists $\bfd = \{d_i\}_{i \in L^\bullet(\Lambda)} \in\CC^{L^\bullet(\Lambda)}$ and
\begin{enumerate}[(i)]
\item a degeneration
  $(\Lambda_0,\cleq_0)\rightsquigarrow (\Lambda,\cleq)$
  given by a map~$\delta\colon L(\Lambda) \to L(\Lambda_0)$ and a subset of~$\Nhor$,
\item an \diffeogen $g\colon \overline{X}_{\bfd
    \cdot\bfsigma_0}\to\overline{X}_{ \bfsigma}$ with $g^{-1}$ conformal on the $\epsilon$-thick
  part of~$X$ and compatible with the markings,
\item letting $c_i = \bfe(d_i)$,
  the bound $||c_i g_*\omega_{0, (i)} - \omega_{(i)}||_{\infty} < \epsilon$ holds
  on the $\epsilon$-thick part of~$X$,
\item  for any $i,j \in L(\Lambda)$ with $i>j$ and $\delta(i)=\delta(j)$,
  we have $|c_i/c_j| < \ve$,
\item for each immersed arc $\gamma$ on $\overline{X}_\bfsigma$,
  $$|\tau(g^{-1}  \circ \gamma)-\tau( \gamma)|< \ve.$$
\end{enumerate}

\begin{prop}
  \label{prop:same_topology}
  The sets $V_\epsilon(X)$ are a basis for a topology on $\Oaugteich$, and
  the convergent sequences in this topology agree with the notion of
  convergence of Definition~\ref{def:augmented_topology_def}.
\end{prop}

Note that the topology defined by this basis is apparently stronger
than that defined by Definition~\ref{def:augmented_topology_def}
because (v) requires uniform control on turning numbers of all immersed
arcs.  In the next lemma we show that one can  uniformly control
the change in turning numbers of all immersed arcs by restricting to a
finite collection of arcs dual to the seams.  For this, fix a
collection $\{\gamma_1, \ldots, \gamma_k\}$ of immersed arcs on a \msd~$X$
which are dual to the seams of $X$, meaning that to each seam
corresponds exactly one $\gamma_i$ which crosses that seam once and is
disjoint from the others.  We define $V'_{\epsilon}(X)$ to be the set
of marked \msds which satisfy conditions (i)-(iv) above as well as
\begin{enumerate}
\item[(vi)] for each  arc $\gamma_i,$ we have
  $|\tau(g^{-1}  \circ \gamma_i)-\tau( \gamma_i)|< \ve.$
\end{enumerate}
\par
\begin{lm}
  \label{lm:dual_to_seams}
Consider $X\in \Oaugteich$ with a collection  of arcs dual to its
vertical seams as above. Then for any $\epsilon>0$, there is a $\delta>0$
such that $V'_\delta(X) \subset V_\epsilon(X)$.
\end{lm}

\begin{proof}
  We first claim that there is a $\delta_1>0$ such that for any
  $g\in V_{\delta_1}(X)$ for closed immersed loop $\gamma$ in
  $\Thick[(X, \bfz)][\epsilon]$ (where $\epsilon$ is small enough that
  $\Thick[(X, \bfz)][\epsilon]$ is the complement of neighborhoods of
  the punctures and nodes), $g$ preserves the turning number of $\gamma$.  To see
  this, choose a finite set of immersed curves $\{\alpha_i\}$ which
  form a basis of $H_1(T^1 \Thick[(X, \bfz)][\epsilon]; \zed)$.
  Choose $\delta_1$ small enough that $g$ is $C^1$-close-enough to the
  identity that it changes the turning number of each $\alpha_i$ by at
  most $1/2$.  Since closed curves have integral turning numbers,
  they must be fixed.  Since the $\alpha_i$ are a basis of homology,
  it follows that all closed curves must be fixed.
\par
  Now let $\gamma$ be any immersed arc which crosses exactly one seam, and
  let $\gamma_i$ be the chosen arc which crosses the same seam.  We
  may take arcs $\beta_1, \beta_2$ in $\Thick[(X, \bfz)][\epsilon]$
  which have bounded length (in terms of the genus of $X$) and join
  the endpoints of $\gamma$ to those of~$\gamma_i$, so that
  $\gamma'=\beta_1+\gamma+\beta_2$ is an immersed arc which has the same
  endpoints as $\gamma_i$.  They then differ by closed curves in the
  thick part, so have the same turning number.  Take $\delta_2$ small
  enough that $g$ changes the turning number of the $\beta_i$ by at
  most $\epsilon/2$.  Taking $\delta < \max(\epsilon/2, \delta_1,
  \delta_2)$ then ensures that~$g$ changes the turning number of $\gamma$
  by at most~$\epsilon$.
\end{proof}
\par
\begin{proof}[Proof of Proposition~\ref{prop:same_topology}]
  Suppose that $X_0 \in V_\ve(X)$. To check that the sets defined
  above are a basis of topology we want to find~$\rho$ such that
  $V_\rho(X_0) \subset V_\ve(X)$. Suppose that $X_1 \in
  V_\rho(X_0)$. Let
  $g_0\colon \overline{X}_{0,\bfd_0 \cdot\bfsigma_{0}} \to
  \overline{X}_{\bfsigma}$ and
  $g_1\colon
  \overline{X}_{1,\bfd_1\cdot\bfsigma_{1}}\to\overline{X}_{0,\bfsigma_{0}}$
  be the \diffeogens given by the definitions of~$V_\ve(X)$ and
  $V_\rho(X_0)$. We define $\bfd=\bfd_1 + \bfd_{0}$, denote the
  rescaling constant by $\bfc_{0} =\bfe(\bfd_{0})$
  and~$\bfc_1 =\bfe(\bfd_1)$, and set $\bfc = \bfe(\bfd_0 + \bfd_1)$.
\par
  First we remark that item (i) is
  automatically satisfied in $V_\ve(X)$. For item (ii), choose~$\rho$
  small enough such that the $\rho$-thick part of~$X_0$ contains the
  $g_0$-image of the $\ve$-thick part of~$X$, and let
  $g =  g_0 \circ F_{\bfd_0}\circ g_1\circ F_{\bfd_0}^{-1}\colon
  \overline{X}_{1,\bfd_1\cdot\bfsigma_{1}}\to\overline{X}_{\bfsigma}$. Then $g$ clearly
  satisfies (ii).
\par
  For (iii), to simplify notation and illustrate the main idea, we
  treat the case that~$X_1$ is smooth and that~$X$ and~$X_0$ have the
  same level graph with two levels.  Moreover, we assume that all
  rescaling constants on the top level are equal to one and we denote the
  rescaling constants on lower level by $c_0,c_{1}$ and
  $c=c_0c_{1}$. The general case follows by the same idea.  Under
  these assumptions, let $\omega^-$ and $\omega^-_0$ be differentials
  on the lower level of~$X$ and $X_0$ respectively.  By assumption,
  the norm
  $\epsilon'=|| c_{0} (g_{0})_*\omega^-_{0}-\omega^-||_{\infty}$
  satisfies $\epsilon'<\epsilon$. We estimate the sup-norms on the
  lower level subsurface of the $\ve$-thick part of~$X$ as follows:
  \bas ||c g_{\ast}\omega_{1} -\omega^-||_{\infty} &\,\leq\,
  ||c (g_{0})_{\ast}(g_1)_{\ast}\omega_{1} -
  c_{0}(g_{0})_{\ast}\omega_{0}^- ||_{\infty}
  \+ ||c_{0}(g_{0})_{\ast}\omega_{0}^- -\omega^-||_{\infty},\\
&\,<\, c_{0} \, || (g_{0})_{\ast}(c_1 (g_1)_{\ast}\omega_{1} -\omega_{0}^-)
||_{\infty} \+ \epsilon'\\
  &\, <\, c_{0} \,C_{g_{0}} \rho +\epsilon'\,.  \eas
Here $C_{g_0}$ is the supremum of the norm of the derivative $Dg_0$, with respect
to the hyperbolic metrics, on the $\ve$-thick part.  We now take $\rho$
  small enough so that $c_{0} \,C_{g_{0}} \rho +\epsilon'<\ve$.  This
  shows that item~(iii) holds for $X_1 \in V_\ve(X)$.
\par
  Item (iv) follows if we moreover choose $\rho$ such that
  $\rho < \ve c_0$.  Finally, item (v) follows from the triangle
  inequality.  Consequently, the $V_\ve(X)$ are indeed a basis of a
  topology.

  It is obvious that a sequence which converges in the topology
  defined by this basis also converges in the sense of
  Definition~\ref{def:augmented_topology_def}.  Conversely, given any
  $V_\epsilon(X)$ choose curves~$\{\gamma_i\}$ and $\delta$ such that
  $V_\delta'(X) \subset V_\epsilon(X)$.  Any sequence converging to
  $X$ in the sense of Definition~\ref{def:augmented_topology_def} is
  then eventually in $V_\delta'(X)\subset V_\epsilon(X)$.
\end{proof}
\par
\begin{thm} \label{thm:PaugHausdorff}
For any subgroup $G < \Mod$, the quotient $\Oaugteich/G$ and its
projectivized version $\Paugteich/G$ are Hausdorff topological spaces.
\end{thm}
\par
\begin{proof}
Suppose that $(X_n, \bfz_n, \bfomega_n, \cleq_n, \bfsigma_n, f_n)$ is a
sequence of marked \msds which converges to $(X, \bfz, \bfomega, \cleq, \bfsigma, f)$,
and suppose moreover we have a sequence $\{\gamma_n\}$ in $G$ so that
  $(X_n, \bfz_n, \bfomega_n, \cleq_n, \bfsigma_n, f_n\gamma_n^{-1})$
converges to $(X', \bfz', \bfomega', \cleq', \bfsigma',  f')$.
Let~$g_n$ and $g_n'$ be the respective sequences of maps
  exhibiting this convergence, and let $h_n = g_n'\circ g_n^{-1}\colon
  \overline{X}_\bfsigma \to \overline{X}'_{\bfsigma'}$ (strictly
  speaking, $h_n$ may be only defined on an exhaustion of the
  complement of the horizontal nodes).
\par
  Forgetting all but the underlying pointed stable curves, our
  topology gives the conformal topology on the Deligne-Mumford
  compactification, which is a Hausdorff space, so the maps
  $h_n$ must converge uniformly on compact sets to
  an isomorphism $h\colon (X, \bfz) \to (X',\bfz')$ of pointed stable
  curves.  We next show that~$\cleq$ and $\cleq'$ are the same (weak)
  full order. Suppose, for contradiction, that there exist irreducible
  components~$X_u$ and~$X_v$ of~$X$ such that $X_u \succ X_v$ but
  $X_u \cleq' X_v$. Since $\cleq_n$, for~$n$ sufficiently large, is an
  undegeneration of both~$\cleq$ and $\cleq'$, this is possible only
  if $X_u \asymp_n X_v$. We denote by~$\ell$ and~$\ell'$ some level
  functions inducing the full orders~$\cleq$ and~$\cleq'$,
  respectively. The specific choices of these level functions are not
  important, as we will only use them to match notation. Then
  condition~(3) of convergence of sequences implies that
  $||c_{n,\ell(u)} (g_n)_*\bfomega_n - \omega_u||_\infty < \ve_n$ and
  $||c'_{n,\ell'(u)} (g_n')_*\bfomega_n - \omega'_u||_\infty < \ve_n$,
  where $\omega_{u}$ is the restriction of~$\bfomega$ to the
  $\epsilon_n$-thick part of~$X_{u}$.  Pulling back the second
  inequality by~$h$ and choosing $\ve_n$ small enough, these
  conditions imply that the ratios $c_{n,\ell(u)}/c'_{n,\ell'(u)}$ are
  bounded away from zero and infinity.  Similarly, the same holds for
  $c_{n,\ell(v)}/c'_{n,\ell'(v)}$. However, condition~(4) of
  convergence implies that $|c_{n,\ell(u)}/c_{n,\ell(v)}| \to 0$,
  while on the other hand the hypothesis $X_u \cleq' X_v$ implies that
  (after possibly passing to a subsequence)
  $c'_{n,\ell'(u)}/c'_{n,\ell'(v)}$ is bounded away from
  zero. Combining these inequalities yields a contradiction.

  To verify that the form~$\bfomega$ is equal to $h^*\bfomega'$, we use that
  for every level~$i$ both inequalities
  $||c_{n,i} (g_n)_*(\bfomega_n)-\omega_{(i)}||_\infty <\ve_n$ and
  $||c'_{n,i} (g_n)_*(\bfomega_n)- h^*\omega'_{(i)}||_{\infty}< C\ve_n$
  hold for some constant~$C$ that depends on the map~$h$ but not
  on~$n$.  We multiply the second inequality by $c_{n,i}/c'_{n,i}$,
  use that this quantity is bounded away from zero and infinity, and
  thus deduce that
  $||c_{n,i}/c_{n,i}' \cdot h^*\omega_{(i)}'- \omega_{(i)}||_{\infty}$
  tends to zero on the $\ve_n$-thick part of~$X_{(i)}$. This implies
  the convergence of the sequence $c_{n,i}/c_{n,i}'$ for each~$i$, and
  also the equivalence as projectivized differentials, as desired.
\par
Finally, the maps~$h_n$  are asymptotically \tnp, so by Proposition~\ref{prop:jacuzzi}, $h^*\bfsigma' =
  \bfsigma$, and moreover the induced map $\overline{h}\colon
  \overline{X}_\bfsigma \to \overline{X}'_{\bfsigma'}$ is eventually isotopic to
  $h_n$.  We have $h_nf\simeq f'\gamma_n$ for each $n$, so eventually
  $hf\simeq f'\gamma_n$, so $h$ exhibits a $G$-equivalence of $X$ and
  $X'$, as desired.
\end{proof}
\par
Our next goal is to show that $\Oaugteich$ is a second countable
topological space with countable basis $\{V_\epsilon(X_n)\}$, where
$\epsilon\in \ratls$ and $\{X_n\}$ is a dense sequence.

\begin{lm}
  \label{lm:eventally_contains}
  Let $X_n\to X$ be a convergent sequence in a single stratum of
  $\Oaugteich$.  For any $\epsilon>0$, the neighborhoods
  $V_\epsilon(X_n)$ eventually contain $X$, and
  $V_\epsilon(X_n) \subset V_{4\epsilon}(X)$.
\end{lm}

\begin{proof}
  Let $g_n\colon \overline{X}_n\to \overline{X}_\bfsigma$ be a
  sequence of maps which exhibit convergence of this sequence as in
  Definition~\ref{def:augmented_topology_def}.  We wish to show that
  eventually $g_n^{-1}$ exhibits $X\in V_\epsilon(X_n)$.  Since $X$ and
  $X_n$ lie in the same stratum, it suffices to check items (ii),
  (iii), and (v) in the definition of $V_\epsilon(X)$.
\par
Let $N$ be large enough so that $\epsilon_N < \epsilon$.  By
Lemma~\ref{lm:converge_to_identity}, the $g_n^{-1}$ converge uniformly to
the identity as maps to the universal curve.  Since the vertical hyperbolic
metric is continuous, the image of $\Thick[(X, \bfz)][\epsilon_N]$
under $g_n^{-1}$ of eventually contains $\Thick[(X_n, \bfz_n)][\epsilon]$,
and moreover
  \begin{equation*}
    \frac{1}{C_n} \,\leq\, \| Dg_n^{-1} \|_{\epsilon_N, \infty} \,\leq\, C_n
  \end{equation*}
(meaning the sup-norm on the $\epsilon_N$-thick part) for $C_n \searrow 1$,
where the norm is defined via the \Poincare metrics.
\par
As a consequence the map~$g_n$ is eventually defined on
$\Thick[(X_n, \bfz_n)][\epsilon]$, and moreover on $\Thick[(X_n, \bfz_n)][\epsilon]$
  \begin{equation*}
    \|\bfomega_n - g_n^*\bfomega\|_{\epsilon,\infty} \leq C_n\| (g_n)_* \omega_n
    - \omega\|_{\epsilon_N, \infty}\to 0,
  \end{equation*}
  so $\|\bfomega_n - g_n^*\bfomega\|_\infty < \epsilon$ eventually.
\par
  Finally, $g_n$ changes turning numbers of immersed arcs as much as
  $g_n^{-1}$ does, so eventually~$g_n$ changes turning numbers of immersed
  arcs by at most $\epsilon$, so $X\in V_\epsilon(X_n)$.
\par
  Now, suppose $X' \in V_{\epsilon}(X_n)$ is exhibited by
$g_n'\colon \overline{X}'_{\bfsigma'}\to \overline{X}_n$.  We wish
to show that $g_n\circ g_n'$ eventually exhibits $X'\in V_{4\epsilon}(X)$.
The composition $(g_n\circ g_n')^{-1}$ is
eventually conformal on $\Thick[(X, \bfz)][\epsilon_N]$, and moreover
\begin{align*}
\|(g_n\circ g_n')_* \bfomega'- \bfomega\|_{\epsilon_N, \infty}
&\leq C_n \|(g_n')_*\bfomega ' - g_n^* \bfomega\|_{\epsilon,    \infty}\\
&\leq C_n \| (g_n')_*\bfomega' - \bfomega_n\|_{\epsilon,  \infty}
+ C_n \|\bfomega_n -g_n^*\bfomega\|_{\epsilon, \infty}\\
&\leq 2 C_n \epsilon < 4\epsilon. \qedhere
  \end{align*}
\end{proof}
\par
\begin{prop}
  \label{prop:second_countable}
  The augmented \Teichmuller space $\Oaugteich$ is second countable.
\end{prop}
\par
\begin{proof}
  Each stratum $\msT[\Lambda]$ is a complex manifold and thus separable,
  so $\Oaugteich$ is separable as well as it is a countable union of
  these strata.  Let $\{X_n\}$ be a sequence whose intersection with
  each stratum is dense.  We claim that the family
  $\mathcal{F} = \{V_\epsilon(X_n) : \epsilon\in \ratls\}$, is a basis
  of the topology on $\Oaugteich$.

  Consider any $X\in \msT[\Lambda]$ and $\epsilon>0$.  Take any
  subsequence $X_{n_k}\to X$ within $\msT[\Lambda]$ and rational
  $\epsilon'<\epsilon/4$.  By the previous Lemma, eventually
  $X \in V_{\epsilon'}(X_{n_k})\subset V_\epsilon(X)$, so
  $\mathcal{F}$ is a countable basis, as desired.
\end{proof}
\par
We finish this section with the following Proposition which allows us
to relax some of the conditions of
Definition~\ref{def:augmented_topology_def}.
\par
\begin{prop}
  \label{prop:weak_convergence}
  For convergence in Definition~\ref{def:augmented_topology_def},
  it suffices that the maps $g_n$ satisfy all of the conditions listed, with
  the following modifications:
  \begin{enumerate}
  \item[(2')] The $g_n^{-1}$ are \changed{almost-homeomorphisms which are}
    conformal on an open set $U$ which intersects each irreducible
    component of $X$ \changed{and which are} $K_n$-quasiconformal on the rest
    of the $\epsilon_n$-thick part, where $K_n\to 1$.
  \item[(3')] These forms converge in the weak locally $L^2$ topology.
  \item[(5')] Convergence of turning numbers is only required for
    \changed{regular Jordan} arcs whose endpoints lie in $U$.
  \end{enumerate}
\end{prop}

\changed{
  In item (5'), since $g_n$ is not smooth, the images of regular Jordan
  arcs under $g_n$ may not be smooth.  We use the discussion of
  Section~\ref{sec:turning_Jordan} to define the turning numbers of
  these arcs.
}

\begin{proof}
  Apply Lemma~\ref{lm:supercool} to produce a sequence $k_n\colon X\to
  X$ of $K_n$-quasiconformal maps, converging uniformly to the
  identity, such that \changed{$\bar{g}_n= k_n^{-1} \circ g_n$} is
  conformal \changed{on the $\epsilon_n$-thick part}, and the
  restriction of $k_n$ to $U$ is holomorphic.  \changed{Let $\tilde{g}_n$ be an
  almost-diffeomorphism isotopic to $\bar{g}_n$ that agrees with
  $\bar{g}_n$ on the $2\epsilon_n$-thick part.  This can be constructed by
  first approximating $\bar{g}_n$ by a piecewise-linear map (see
  \cite[Theorem~6.4]{Moise}) and then approximating the piecewise-linear map with a
  smooth one, which is always possible in dimension two (see \cite[Theorem~5.2]{Munkres}).}

  We claim that these maps $\tilde{g}_n$ satisfy all of the
  requirements of Definition~\ref{def:augmented_topology_def}.
  The uniform convergence of forms required by~(3) follows from
  Proposition~\ref{prop:topologies_are_the_same}.  For~(5), note that
  Lemma~\ref{lm:dual_to_seams} allows us to consider only turning
  numbers of Jordan arcs whose endpoints are contained in $U$.  Since
  the maps~$k_n$ converge uniformly, their derivatives converge as
  well on $U$.  \changed{By Proposition~\ref{prop:turning_Jordan}, the
    turning number of a Jordan arc can be computed using only its
    isotopy class (relative to endpoints and tangent vectors at
    endpoints).}  It follows that (5') implies that the $\tilde{g}_n$
    are asymptotically \tnp.
\end{proof}

%%%%%%%%%%%%%%%%%%%%%%%%%%%%%%%
\subsection{The moduli space of \msds as a topological space.}
\label{sec:msdastopo}
%%%%%%%%%%%%%%%%%%%%%%%%%%%%%%%

We are now in a position to define our central moduli space as a
topological space and establish its main topological properties.
\par
\begin{df}
\label{df:msds}
  The moduli space of \msds is the quotient space
  $\LMS = \Oaugteich/\Mod$, and its projectivization is the quotient
  $\PP\LMS = \Paugteich/\Mod = \LMS / \cx^*$.
\end{df}
\par
It follows immediately from Theorem~\ref{thm:PaugHausdorff} and
Proposition~\ref{prop:second_countable} that these spaces are Hausdorff and
second countable.
\par
\begin{thm} \label{thm:MSDcompact}
  The moduli space $\PP\LMS$ of
  projectivized \msds of type~$\mu$ is compact.
\end{thm}
\par
\begin{proof}
  In a second countable space, compactness is equivalent to sequential
  compactness (see \cite[Proposition~1.6.23]{beattie_butzmann}), so it
  suffices to establish sequential compactness.

  Let $\left\{(X_n,\bfz_n,\bfomega_n, \cleq_n, \sigma_n, f_n)\right\}$
  be a sequence in~$\Paugteich$.  We wish to exhibit a convergent
  subsequence, after pre-composing the markings $f_n$ by a sequence in $\Mod$.

  Since $\barmoduli[g,n]$ is compact, after pre-composing the markings,
  we can pass to a subsequence so that
  $\left\{(X_n,\bfz_n, f_n)\right\}$ converges in $\oldaugteich$ to a
  marked surface $(X,\bfz, f)$.  Since there are finitely many
  enhanced multicurves up to the action of $\Mod$, we may pass to a further
  subsequence so that these surfaces lie in a single stratum $\msT$.

  By definition of convergence in $\oldaugteich$, there exists an
  exhaustion $K_n$ of $X^s \setminus \bfz$, and conformal maps
  $h_n\colon K_n\to X_n$ compatible with the markings.  Transporting the orders on $X_n$ to $X$ by
  $h_n$ induces a full order on~$X$ that we denote by $\cleq_0$.
  \par
  Now we use the sizes of the forms $\bfomega_n$ to refine the order
  $\cleq_0$.  Choose $\epsilon$ small enough so that for each irreducible
  component $Y$ of $X$, the $\epsilon$-thick part of $Y$ minus the
  marked points is connected.  Let $\lambda_n(Y)$ be the size of the
  corresponding component of $X_n$, in the sense of
  Equation~\eqref{eq:size}.  Passing to a subsequence, we may assume
  that for any pair of components $Y$ and $Y'$ the ratios $\lambda_n(Y)/ \lambda_n(Y')$
  converge, to a number in $\RR\sqcup\lbrace\infty\rbrace$.
  We then define an order $\cleq$ on the set of
  components of $X$ so that $Y\cleq Y'$ when $Y \cleq_0 Y'$, and
  moreover if $Y\asymp_0 Y'$, we define $Y \cleq Y'$ if $\lambda_n(Y)/\lambda_n(Y')\not\to\infty$.
  Theorem~\ref{thm:compactness_for_differentials} then implies that we
may pass to a subsequence so that on each component~$Y$ of~$X$
the rescaled forms  $\left\{h_n^{\ast} \bfomega_n/\lambda_n(Y)\right\}$ converge
to some form $\bfomega$.
\par
  For each level~$i$ of $\cleq$ we pick a component $Y_i$ at that
  level and define $c_{n,i} = \lambda(Y_i)^{-1}$.  Then the sequence of differentials
  $\left\{c_{n,i} h_n^{\ast}  \bfomega_n\right\}$ converges to some
  differential~$\omega_{(i)}$ on the $i$-th level $X_{(i)}$ of $X$
  with respect to $\cleq$.  We define~$\bfomega$ on~$X$ to be
  the collection of those differentials~$\omega_{(i)}$. We have to prove
  that~$\bfomega$ is a twisted differential compatible with the
  order~$\cleq$. The crucial conditions (matching orders, matching
  residues and GRC) can be verified by~$\bfomega$-path integrals or
  turning numbers (compare Section~4 in \cite{strata}). Hence these
  conditions carry over from the corresponding integrals on the
  sequence of surfaces $X_n$, using the convergence of one-forms and
  using (for the GRC) the fact that the rescaling constants~$c_{n,j}$
  depend on the levels only.
\par
  For each vertical node~$q$ choose a preliminary \prma
  $\widetilde{\sigma_q}$, forming together a global
  \prma~$\widetilde{\sigma}$, and choose a preliminary
  almost-diffeomorphism
  $\tilde{g_n}\colon \overline{X}_{\bfd \cdot \widetilde{\sigma}} \to
  \overline{X}_{\sigma_n}$ which is isotopic to $h_n^{-1}$ on the
  complement of the seams.  Conditions (1)--(4) of
  Definition~\ref{def:augmented_topology_def} are clearly satisfied
  for these $g_n$, taking $d_{n,i} = \frac{1}{2\pi i} c_{n,i}$, and it
  remains to show that the prong-matchings and~$g$ may be
  modified so that the $g_n$ are asymptotically \tnp.
\par
  Now choose a collection of immersed arcs $\gamma_i$, dual to the seams
  of $\overline{X}_{\widetilde{\sigma}}$ as in Remark~\ref{rem:coftop}.
  By convergence of the forms, $\tau(g_n^{-1}(\gamma_i))\to
  \tau(\gamma_i) \pmod \zed$.  We first modify each prong-matching
so that these turning numbers converge mod $\kappa_i$, and then modify~$g$
by an appropriate twist around each seam so that the turning numbers converge.
By  Remark~\ref{rem:coftop} convergence of the turning numbers of these
arcs ensure the~$g_n$ are asymptotically \tnp.
\end{proof}

%%%%%%%%%%%%%%%%%%%%%%%%%%%%%%%%%%%%
\section{The model domain}
\label{sec:modeldomain}
%%%%%%%%%%%%%%%%%%%%%%%%%%%%%%%%%%%%%

In this section, we construct the \emph{model domain} $\barMD$, an
orbifold which will serve as a local model for the boundary of the
moduli space $\LMS$.  The model domain is constructed as the finite
quotient orbifold of the \emph{simple model domain} $\barMDs$, a
complex manifold which is in turn constructed as a bundle over a
product of \Teichmuller spaces.  We moreover construct a family of
differentials over the model domain, which we will call the {\em universal
family of model differentials}.  These objects will be used in
Section~\ref{sec:Dehn}, where we provide $\LMS$ with an atlas of
``plumbing maps'', which are defined on open subsets of the
model domain by a plumbing construction on the universal family.
The families of model differentials over arbitrary bases will be defined
by combining the definition of families of \msds in Section~\ref{sec:famnew}
and the families of markings in Section~\ref{sec:rob}. In Section~\ref{sec:uniMD} we will finally establish
that the family over the model domain is indeed universal.

%%%%%%%%%%%%%%%%%%%%%%%%
\subsection{Construction of the model domain}
\label{sec:constructionMD}
%%%%%%%%%%%%%%%%%%%%%%%%

As a first step in the construction, we define
\begin{equation}
\MDs  \= \PP\ptwT / \svTw \quad \text{and}  \quad \MD \= \PP\ptwT / \vTw\,,
\end{equation}
whose points represent prong-matched twisted differentials defined up
to rescaling all components simultaneously, together with a marking
defined up to the action of $\svTw$ or $\vTw$, respectively.  Recall from
Lemma~\ref{lm:defsvTw} that $\svTw$ is a finite index subgroup of
$\vTw$. The finite quotient group $K_\Lambda = \vTw/\svTw$ defined in
Equation~\eqref{eq:exseqcover} thus acts on~$\MDs$, with quotient~$\MD$.
\par
The model domains
$\barMD$ and $\barMDs$ are constructed informally by adding a boundary
consisting of differentials which are allowed to be identically zero
on some set of levels below the top level.
\par
More precisely, the simple level rotation torus $\Tsimp = \CC^{L(\Lambda)}/
\svTw\isom (\CC^*)^{L(\Lambda)}$
acts freely on~$\MDs$ via the action defined in~\eqref{eq:Tsimp_action}.
Recall from Section~\ref{sec:AugSet} that the projectivized~$\Lambda$-boundary
stratum~$\pmsT$ is the quotient $\MDs / \Tsimp$, so $\MDs$ is a
principal $(\CC^*)^{L(\Lambda)}$-bundle over~$\pmsT$ \changed{by the second
statement in Lemma~\ref{le:OmBsmooth}.} We define $\barMDs$
as the associated $\CC^{L(\Lambda)}$-bundle over~$\pmsT$, where
$(\CC^*)^{L(\Lambda)}$ acts on $\CC^{L(\Lambda)}$ by coordinate-wise
multiplication in the usual way.
\par
As the $K_\Lambda$-action on $\MDs$ commutes with the $\Tsimp$-action,
$K_\Lambda$ acts on $\barMDs$, and we define (as a complex orbifold)
\begin{equation*}
  \barMD \= \barMDs / K_\Lambda\,.
\end{equation*}
\par
We now provide notation to describe
the boundary $\partial\MDs = \barMDs\setminus \MDs$ of the simple model domain.
The boundary $\partial\MDs$  is a normal crossing divisor given by
$D = \cup_{i \in L(\Lambda)} D_i$ in $\barMDs$, where $D_i$ is fiber-wise defined by $\left\{t_i=0\right\}\subset\CC^{L(\Lambda)}$.
There is a stratification
\index[teich]{f020@$\barMDs$,  $\barMD$!(Smooth) model domain}
\begin{equation}\label{eq:stratMDs}
  \barMDs \=
  \coprod_{J\subset L(\eL)} \MD^{s,\Lambda_{J}}
\end{equation}
indexed by the vertical undegenerations of the multicurve~$\Lambda$ or, equivalently, by
the subsets $J=\left\{j_{1},\dots,j_{m}\right\}$ of $L(\Lambda)$
(see Section~\ref{sec:order} for the correspondence), where we define
$D_J = D_{j_{1}}\cap \dots \cap D_{j_{m}}$, and $\MD^{s,\Lambda_{J}}$ is defined by
\[\MD^{s,\Lambda_{J}} \= D_J  \setminus \bigcup_{J'\varsupsetneq J}D_{J'}\,.\]
In these terms, the space $\MDs$ corresponds to the subset~$J=\emptyset$,
or equivalently to the degeneration $\degen\colon \bullet \rightsquigarrow\Lambda$
from the trivial graph to~$\Lambda$. Moreover, $D_i$ corresponds to the subset
$J=\left\{i\right\}$, or equivalently to the two-level (un)degeneration
$\degen_i \colon \Lambda_i\rightsquigarrow\Lambda$ of~$\Lambda$.
\par
The model domain has an obvious non-projectivized variant.  The
quotient $\Omega\MDs \= \ptwT / \svTw$ is a
$(\CC^*)^{L(\Lambda)}$-bundle over~$\msT$.  We define $\Omega\barMDs$
as the associated $\CC^{L(\Lambda)}$-bundle,  and let
$\Omega\barMD = \Omega\barMDs / K_\Lambda$.
\par
The smoothness of $\pmsT$ \changed{(Lemma~\ref{le:OmBsmooth})}
and this description imply immediately the following result.
\par
\begin{prop} \label{prop:MDsmooth}
The simple model domain~$\barMDs$ is smooth, while $\barMD$
 has only finite quotient singularities.
\end{prop}
\par
\begin{exa}\label{ex:quotientsing}({\em A model domain with finite quotient
singularities})
To see that finite quotient singularities can actually occur in this way, we
analyze the second case of our running example in Section~\ref{sec:runningex}.
There, $\barMDs[\eL_{2}]$ is locally the product
of $\pmsT[\eL_{2}]$ with~$\CC^2$, and the two boundary divisors are the
coordinate axes. The generators of the \lw ramification groups $H_1$ and $H_2$
(see Section~\ref{sec:coverGH} and Example~\ref{ex:trianglegraph}) act on
the~$\CC^2$ factor by $(z_1,z_2) \mapsto (\zeta_3 z_1, z_2)$ and $(z_1,z_2)
\mapsto (z_1, \zeta_3 z_2)$ respectively, where~$\zeta_3$ is a third root of unity.
Consequently, the generator
of $K_{\Lambda_{2}} = {\rm Ker}(H \to G)$ acts by $(z_1,z_2) \mapsto (\zeta_3 z_1,
\zeta_3^{-1} z_2)$. The ring of invariant polynomials under this action is generated
by $u = z_1^3$, $v= z_2^3$ and $z= z_1z_2$, hence
the quotient has a singularity locally given by the equation $uv - z^3 = 0$.
\end{exa}
\par

%%%%%%%%%%%%%%%%%%%%%%%%%%%
\subsection{The universal family}
\label{sec:universalMD}
%%%%%%%%%%%%%%%%%%%%%%%%%%%

We now construct the universal family of model differentials over the
model domain. Once we formally define the notion of families of model
differentials over arbitrary bases,  we will see  in Proposition~\ref{prop:sMDuniv}
that this family is in fact the universal family of model differentials.
\par
Over the open part~$\Omega\MDs$, this universal family of model differentials
is simply given by the $\svTw$-quotient of the equisingular
family $(\pi\colon\calX \to \ptwT, \changed{\bfomega}, \bfz, \bfsigma, f )$ over
$\ptwT$, where $\bfomega$ is a universal relative one-form, $\bfsigma$
is a family of prong-matchings, $\bfz$ are sections marking the zeros
and poles, and $f$ is a family of markings (to be defined precisely in
Section~\ref{sec:rob}) up to the group $\svTw$.
\par
As the action of $\Tsimp$ on $\ptwT$ is trivial on the level of
underlying curves, the universal curve over $\Omega\MDs$ is the pullback of
the universal curve over the quotient $\pmsT$.  It follows that the universal
curve over $\Omega\MDs$ extends to a universal curve $\calX\to
\Omega\barMDs$ as the pullback of this bundle. \changed{By definition of
an associated $\CC^{L(\Lambda)}$-bundle the form $\bfomega$ extends
to a relative one-form on~$\calX$, also denoted by~$\bfomega$ and referred
to as the universal relative one-form. It vanishes on the component on~$i$-th
level precisely over the divisor~$D_i = \{t_i = 0\}$.}
\par
Consider an open set $\calV\subset \msT$ together with a section
$ \mathscr{S}\colon \calV \to \Omega\MDs$ of the $\Tsimp$-bundle.  Let
$\calW\subset \Omega\barMDs$ be the preimage of $\calV$ and
$\calX_\calW\to \calW$ the universal curve over it.  Informally, a point in
$\calV$ represents a $\Tsimp$-orbit of forms and compatible
prong-matchings, and $\mathscr{S}$ represents a holomorphic choice of
representative ``rescaled forms''~$\bfeta$ and compatible
prong-matchings $\bfsigma$.  The section $\mathscr{S}$ determines a
trivialization $\calW\to \calV \times \cx^{L(\Lambda)}\times \cx^*$,
and composing with the projection to $\cx^{L(\Lambda)}\times \cx^*$
determines a tuple of holomorphic functions
$\bft\colon \calW\to \cx^{L(\Lambda)}\times \cx^*$ which we call
\emph{simple rescaling parameters}.
\index[family]{d003@$\bft = (t_i)_{i \in L}$! Simple rescaling parameters}
\par
The rescaled differentials $\bfeta$ can be regarded as a tuple of
relative one-forms on~$\calX_\calW$ which do not vanish on any vertical
component of the universal curve, \changed{which are independent of
  $\bft$,} and which satisfy $\bft\ast \bfeta =
\bfomega$, where $\bfomega$ is the universal relative one-form over
$\Omega\barMDs$. Here we recall from~\eqref{eq:Tsimp_action} the
definition of the action
\begin{align} \label{eq:omTeta}
\index[twist]{b060@$\prodt$!Product of the $t_j^{a_j}$ for $j\geq i$}
\bft\ast \bfeta \= \left(\prodt\cdot \eta_{(i)} \right)_{i\in L(\Lambda)}
&\= \left(t_{-1}^{a_{-1}}\cdot t_{-2}^{a_{-2}}\dots t_i^{a_i} \cdot
              \eta_{(i)} \right)_{i\in L(\Lambda)}\\
&\= \left(s_{-1}\cdot\ldots\cdot s_{i} \cdot\cdot \eta_{(i)}
\right)_{i\in L(\Lambda)}\,,\nonumber
\end{align}
where we define the \emph{rescaling parameters} $\bfs$ to be the
powers $(s_i) = (t_i^{a_i})$.
\index[family]{d007@$\bfs = (s_i)_{i \in L}$! Rescaling parameters}
The prong-matchings $\bfsigma$ can be regarded as a continuously
varying family of prong-matchings for~$\bfeta$.

\changed{Together $\bfeta$ and $\bfsigma$ give each fiber $X$ of
  $\calX_\calW$ the structure of a \msd.  We may also assign this
  fiber a marking (as in the definition of a marked \msd following
  Definition~\ref{df:msd}) as the $(\cx^*)^{L(\Lambda)}$-orbit of the
  data $(\bfeta, \bfsigma, [f])$, where $[f]$ is a marking
  $(\Sigma, \bfs) \to (\overline{X}_\bfsigma, \bfz)$ up to
  precomposition by $\svTw$.}
\par
To summarize, we have defined locally, depending on a choice of
section $\mathscr{S}\colon \calV \to\Omega\MDs$, which we will in the
sequel call a \emph{local trivialization} of $\Omega\MDs$ over
$\calV$, a collection of rescaled differentials $\bfeta$, compatible
prong-matchings $\bfsigma$,  holomorphic functions~$\bft$, and
rescaling parameters $\boldsymbol{s}$ such that
the product $\bft\ast \bfeta$ agrees with the universal one-form
$\bfomega$ on the universal curve over $\Omega\barMDs$.
% \changed{To
% provide boundary points with a marking~$f$ we take the limit of
% the markings given over $\Omega\MDs$ along a ray in~$\bft$ of constant
% arguments. (Again this is well-defined only up to the group $\svTw$,
% consistent with the definition in Section~\ref{sec:rob}, and the use
% of rays here motivates the level-wise real blow-up in
% Section~\ref{sec:rob}.)}
\par
The product
of $\bft$ with the projection to $\msT$ determines an isomorphism
$\calW \to \calV \times\cx^{L(\Lambda)}\times \cx^*$.  In working with
model differentials, we will often implicitly assume we have chosen
such a local trivialization.  In the sequel, the functions $\bft$
together with local coordinates on $\calV$ will give a convenient
system of local coordinates on \changed{$\Omega\barMD$.}  In
Section~\ref{sec:famnew}, the equivalence class of
$(\bfeta, \bfsigma, \bft)$ will be part of the data describing a
general family of model differentials.

%%%%%%%%%%%%%%%%%%%%%%%%%%
\subsection{Topology of $\Omega\MD$}
%%%%%%%%%%%%%%%%%%%%%%%%%%

The topology on the model domain may also be expressed in the language
of conformal maps. The following proposition follows immediately from
the definition of the conformal topology in
Section~\ref{sec:turning} and the topology on the
$\CC^{L(\Lambda)}$-bundle associated with a
$(\CC^*)^{L(\Lambda)}$-bundle. Given $\bft\in\CC^{L(\Lambda)}$, we
define $J(\bft) \subseteq L(\Lambda)$ to be the subset of indices~$i$
such that $t_i =0$.
\begin{prop} \label{prop:MDsconv}
A sequence $(X_m, \bfz_m, \bfeta_m, \bft_m, \cleq, \sigma_m, f_m)$
of (simple) \auxds in $\Omega\barMDs$ converges
to $(X, \bfz, \bfeta, \bft, \cleq, \sigma, f)$
if and only if, taking representatives with $t_{m,i}, t_i \in \{0,1\}$
for $X_m$ and $X$, there exist a sequence of positive numbers $\epsilon_m$
converging to~$0$ and a sequence of vectors $\bfd_m = \{d_{m,j}\}_{j \in J(\bft)}
\in \CC^{J(\bft)}$ such that the following conditions hold for sufficiently
large~$m$, where we let $d_{m,0}=0$ and denote $c_{m,j} = \bfe(d_{m,j})$:
\par
\begin{enumerate}[(i)]
\item There is an inclusion
  $\iota_m\colon J(\bft_m) \hookrightarrow J(\bft)$.
\item  For sufficiently large~$m$ there exists an \diffeogen $g_m
\colon \overline{X}_{\bfd_m \cdot\bfsigma_m}\to \overline{X}_{ \bfsigma}$   that
is compatible with the markings (in the sense that $g_m \circ\bfd_m \cdot f_m$
is isotopic to $f$ rel marked points) and such that $g_m^{-1}$ is conformal
on the $\epsilon_m$-thick part $\Thick[(X, \bfz)][\epsilon_m]$.
\item The restriction of $(g_m)_*(c_{m,i}\bfomega_m)$ to the $\epsilon_m$-thick part of
the level~$i$ subsurface of $(X, \bfz)$ converges
uniformly to $\omega_{(i)}$.
\item For any $i,j \in J(\bft)$ with $i>j$ and any subsequence
along which $[j,i] \cap {\rm im}(\iota_{m}) = \emptyset$, we have
    \begin{equation*}
      \lim_{m\to\infty} \frac{|c_{m,i}|}{|c_{m,j}|} \= 0 \,.
    \end{equation*}
  \item The \diffeogens $g_m$ are asymptotically \tnp.
\end{enumerate}
\end{prop}

%%%%%%%%%%%%%%%%%%%%%%%%%%%%%%%%%%%%
\section{Modifying differentials and perturbed period coordinates }
\label{sec:modif}
%%%%%%%%%%%%%%%%%%%%%%%%%%%%%%%%%%%%%

The first goal of this section is to define {\em modifying differentials}~$\bfxi$ as
a preparation for the plumbing construction in Section~\ref{sec:Dehn}, which will enable us to give complex charts on~$\LMS$. The
second goal is to define local coordinates, which we will call {\em perturbed
period coordinates}, on the simple model domain. Once we define the plumbing construction and define families of
\msds, the universal property of the family of \auxds over the model domain will allow us to prove the universal property of $\LMS$ in
Section~\ref{sec:UnivDehn}.
\par
For defining perturbed period coordinates, in this section we restrict to the case with only vertical nodes. In Section~\ref{sec:extendperturbed} we will define extended perturbed period coordinates, to also account for the periods through horizontal nodes. This extension will require the plumbing setup introduced in Section~\ref{sec:Dehn}.
The perturbed period coordinates are similar to the usual period
coordinates, with the following modifications that will allow us to
transition from stable curves with many nodes to curves with fewer nodes.
\par
First, they are coordinates for the universal differential~$\bfeta$, but perturbed by the modifying
differentials~$\bfxi$, and rescaled by~$\bft$ as defined in \eqref{eq:omTeta}. The
reason for this is that the perturbed differential lives on the universal family
over the Dehn space~$\ODehn$, which will be defined after plumbing. Consequently, once
the plumbing construction is completed, perturbed period coordinates will turn out to be
coordinates on~$\ODehn$.
\par
Second, the plumbing construction cuts out the zero that used to be at the top
end of any vertical node. Thus to keep track of the relative period
corresponding to such a zero, we compute a period not to this zero, but to a
suitably chosen nearby point. The choice of this nearby point will be made
in such a way that under degeneration to the boundary of~$\Omega\barMDs$ the
difference between the perturbed period and the original period tends to zero.
\par
Third, the perturbed period coordinate system contains for each level one entry which
measures the scale of degeneration. This is not actually a period,
but rather an $a_i$-th root of a period of~$\bfeta$.
\par
In the whole of this section we work in the preimage
$\calW \subset \Omega\barMDs$ of an open set $\calV \subset \msT$
where we have chosen a local trivialization as in Section~\ref{sec:universalMD}.

%%%%%%%%%%%%%%%%%%%
\subsection{Modifying differentials and the global residue condition revisited.}
\label{sec:modifGRC}
%%%%%%%%%%%%%%%%%%%

In order to construct the plumbing map, we need modifying
differentials as in \cite{strata}, but now defined on the universal
family over an open subset of the model domain.  In this section, we
prove the existence of such families of modifying differentials, for
families that may have both horizontal and vertical nodes.
\changed{ We let $\tilde{\calX}\to\calX$ denote the partial normalization
at the vertical nodes, and let $\tilde\pi\colon\tilde\calX\to \calW$
denote its composition with $\pi$.}
\par
\begin{df}\label{df:modif}
A \emph{family of modifying differentials} over $\calW\subset\Omega\barMDs$ is a
family of meromorphic differentials $\bfxi$ on
$\tilde\pi\colon\tilde{\calX}\to\calW$,  such that:
\begin{enumerate}[(i)]
\item $\bfxi$ is holomorphic, except for possible simple poles along
  both horizontal and vertical nodal sections as well as marked poles;
\item \changed{$\xi_{(-N)}$ vanishes identically, and $\xi_{(i)}$ is divisible
  by $t_{i-1}^{a_{i-1}}$ for each
  $i\in L^\bullet(\Lambda) \setminus \{-N\}$, where $\xi_{(i)}$
  denotes the restriction of $\bfxi$ to the $i$'th level component of $\tilde{\calX}$;}
\item $\bft\ast(\bfeta  + \bfxi)$ has opposite residues at the two preimages of every
 node.
\end{enumerate}
\index[plumb]{b030@$\bfxi$!Family of modifying differentials}
\vskip-6mm
\end{df}
\changed{We emphasize that $\bfxi$ does not satisfy the
opposite-residue condition at the vertical nodes, hence it should
not be regarded as a family of differentials on $\calX\to\calW$
(although it could alternatively be defined as a meromorphic section of the
relative dualizing sheaf).  Also note that, while $\eta_{(i)}$ is independent
of the rescaling parameters~$\bft$, the modifying differential $\xi_{(i)}$ as
constructed below is only independent of the~$t_j$ for $j\geq i$.
Condition (iii) forces $\xi_{(i)}$ to depend on~$t_{i-1}$, and in general,
as it will be constructed recursively level by level, it may depend on
any~$t_j$ for $j<i$.}
\par
Recall that $q_e^\pm\colon \calW\to\tilde\calX$ denote the sections
corresponding to the top and bottom preimages of the vertical node~$e$, with
images~$Q_e^\pm\subset\tilde\calX$. Moreover, let $\calP$ be the reduced
divisor associated to~$\calZ^\infty$. Then~$\bft\ast\bfxi$ is a holomorphic
section of
$$\tilde\pi_* \omega_{\tilde{\calX} / \calW}
\left(\sum_{e\in\vertedge}(Q_e^+ + Q_e^-) + \calP  \right)\,,$$
which is divisible by $\bft^\bfa_{\lceil i-1\rceil}$ at level $i$, and chosen so that
as functions on~$\calW$
\begin{equation}
    \label{eq:opp_res}
    \Res_{q_e^+}\bft\ast (\bfeta  + \bfxi) + \Res_{q_e^-} \bft\ast (\bfeta  + \bfxi) \= 0
\end{equation}
for every vertical node $e\in\vertedge$.
\par
We start the construction of families of modifying differentials by recalling from \cite{kdiff}  a topological
restatement of the global residue condition. Consider the subspace
$V \subseteq H_1(\Sigma \setminus P_\bfs; \QQ)$ spanned by the
vertical curves~$\Lambda^v$, where $P_\bfs$ is the set of marked poles.
The order on $\Lambda$ determines a filtration
\begin{equation}\label{eq:filtration}
0\= V_{-N-1} \,\subseteq \, V_{-N} \,\subseteq\, \dots \,\subseteq\, V_{-1} \= V\,,
\end{equation}
where $V_i$ is generated by the image in~$V$ of all those
vertical curves in $\Lambda$ such that $\lbot\le i$.
Note that this convention differs slightly from the one of \cite{kdiff}: we
allow horizontal nodes, our~$V_i$
corresponds to~$V_{-i}$ there, and our~$N$ corresponds to~$N-1$ there.
\par
Suppose we are given a marked differential $(X, \eta)$ on a pointed stable curve
that satisfies the axioms (0)-(3) of a twisted differential. Fixing an orientation
of the individual curves of~$\Lambda^{v} $, the differential~$\eta$ defines
a \emph{residue assignment} $\rho\colon \Lambda^{v} \to \CC$. With the help of
these maps we give an alternative statement of the global residue condition.
\par
\begin{prop}[{\cite[Proposition~6.3]{kdiff}}] \label{prop:kdiffGRC}
A residue assignment $\rho\colon \Lambda^{v} \to \CC$
satisfies the global residue condition if and only if there exist
\emph{period homomorphisms}
\index[plumb]{b050@$\rho_i$!Period homomorphism}
$$ \rho_i\colon V_i/V_{i-1} \to \CC \quad \text{for any $i \in L(\Lambda)$}\,,$$
such that $\rho_i(\lambda) = \rho(\lambda)$ for all simple closed curves
$\lambda$ in $\Lambda^v$, where~$i = \ell(\lambda^-)$.
\end{prop}
\par
In what follows it will be convenient for us to lift the period homomorphisms to linear maps $\rho_i\colon V_i\to\CC$ such that $\rho_i(V_{i-1})=0$ for all $i\in L(\Lambda)$.
We are now ready to construct the family of modifying differentials, and we will then demonstrate the constructions in the proof by an example.
\par
\begin{prop}   \label{prop:constrModif}
The family $\pi\colon \OMDsfam   \to \calW$ equipped with the universal
differential $\bft \ast \bfeta$ has a family of modifying
differentials $\bfxi$.
\end{prop}
\par
\begin{proof}
Choose a maximal multicurve $\Lambda_{\rm \max} \supseteq\Lambda$
decomposing~$\Sigma\setminus P_\bfs$ into pants. Let $V'\subset
H_1(\Sigma \setminus P_\bfs; \QQ)$ be the subspace of homology generated by
the classes of all curves in $\Lambda_{\rm \max}$. Note that~$V'$ contains~$V$,
and projects to a Lagrangian subspace of $H_1(\Sigma; \QQ)$.  The restriction of
$\bft \ast \bfeta$ to levels $i$ or below determines a holomorphic
period map (extending $\rho_i$ above to families)
  \begin{equation*}
    \rho_i \colon \calW \to \Hom_\QQ(V_i, \CC)\,,
  \end{equation*}
such that $\rho_i$ restricts to zero on $V_{i-1}$. In period coordinates, $\rho_i$ is
simply a linear projection.  By \eqref{eq:Tsimp_action}, the map $\rho_i$ is
$\Tsimp$-equivariant, i.e.\
  \begin{equation}
    \label{eq:Tsimp_equivariant}
\rho_i(\bfq \ast (\calX, \bft \ast \bfeta)) \= \prod_{j \geq i} q_j^{a_j}
\,\cdot \,\rho_i(\calX, \bft \ast \bfeta)
\quad \text{for any} \quad \bfq \in \Tsimp\,.
\end{equation}
For each~$i\in L(\Lambda)$ we choose a sub-multicurve $B_i \subset \Lambda_{\rm \max}$
whose image in~$V'$ is a basis of $V'/V_i$, such that for any $i\in L(\Lambda)$
the inclusion $B_i\subset B_{i-1}$ holds. We then define the extension
$\widetilde{\rho_i} \colon \calW \to \Hom_\QQ(V', \CC)$ of $\rho_i$ by the
requirement $\widetilde{\rho_i}(b_i) = 0$ for all $b_i \in B_i$.
\par
Since $\Lambda_{\rm \max}$ is a maximal multicurve on $\Sigma\setminus P_\bfs$,
a meromorphic form on $\calX$, holomorphic except for at worst simple poles at
the nodes and at the marked poles $P_{\bfs}$, is specified uniquely by its periods
on~$V'$. We define $\bft\ast\bfxi$ \changed{on~$\tilde{\calX}$} so that its $V'$-periods, for
$\gamma\in \Lambda_{\rm \max}$ \changed{lying on $\tilde{\calX}_{(i)}$}, are
\begin{equation}
\label{eq:xi_def}
    \int_\gamma\changed{(\bft \ast  \bfxi)_{(i)}}(u) \= \sum_{j < i} \widetilde{\rho}_j(u)(\gamma)
   \quad \text{for all} \quad u \in \calW \,.
\end{equation}
By the equivariance \eqref{eq:Tsimp_equivariant}, we see that $\widetilde{\rho}_j$
is divisible by \changed{$t^{a_{i-1}}_{i-1}$}, and hence $\bft\ast\bfxi$ is also
divisible by $\bft^\bfa_{\lceil i-1\rceil}$.
\par
Given a curve $\gamma\in \Lambda^{v}$ joining level $i$ to level $j<i$,
we verify the opposite residue condition (iii) for
a modifying differential, which is given by~\eqref{eq:opp_res}, and states that
  \begin{equation*}
    \int_\gamma((\bft \ast (\bfeta+ \bfxi))_{(j)}  \=  \rho_{j}(\gamma)+ \sum_{k < j} \widetilde{\rho}_k(\gamma)
\= \sum_{k < i} \widetilde{\rho}_k(\gamma) \=\int_\gamma((\bft \ast (\bfeta+ \bfxi))_{(i)}\,.
  \end{equation*}
In the above the first equality follows from the fact that $\rho_j$ is the period map determined by $(\bft*\bfeta)_{(j)}$, and from the definition of $\xi_{(j)}$. The second equality follows from the global residue condition of~$\bft \ast \bfeta$ as restated in Proposition~\ref{prop:kdiffGRC}, which implies that $\rho_k(\gamma) = 0$ for all $j< k \leq i$. The last equality again follows from the definition of $\xi_{(i)}$ and the fact that $\rho_i(\gamma) = 0$.
\end{proof}
\par
This proof shows in particular the following.
\par
\begin{cor}\label{cor:modif}
The modifying differential~$\bfxi$ is uniquely determined by the choice of the
subspace $V'\subset  H_1(\Sigma \setminus P_\bfs; \QQ)$ and the multicurves~$B_i$.
Its \lw components $\xi_{(i)}$ depend only on~$t_j$ and~$\eta_{(j)}$ for $j<i$.
\end{cor}
\par
\begin{exa}
We illustrate the objects introduced in the proof of Proposition~\ref{prop:constrModif} in the context of a slight simplification of our running example, as pictured in Figure~\ref{cap:modifdiff}, with one pole denoted by~$p$ (so the level graph is still a triangle, but the irreducible components are simpler).  The family of modifying differentials~$\bfxi$ depends on the parameters $\bft=(t_{0},t_{-1},t_{-2})$.
 \par
 % pic_plumbing_dot_tag_3_colored.tex

\begin{figure}[htb]
	\centering
\begin{tikzpicture} [scale=.54]
\tikzset{->-/.style={decoration=
		{markings, mark= at position .8 with {\arrow{#1}},},postaction={decorate},},
	lw/.style={line width=.2}, % dünnere Strichstaerke
	color1/.style={color=red}, %in colored versoin: red circle
	color2/.style={color=violet}, %in colored versoin: violet arrows
	color3/.style={color=green}, %in colored versoin: green circle	
    color4/.style={color=blue} %in colored versoin: blue circle				
}
% % % % % % % % % % % % % % % % % % % % % % % % % % % % % % % % % % % % % % % % % % % % % % % %
%Rechtes Figur
% % % % % % % % % % % % % % % % % % % % % % % % % % % % % % % % % % % % % % % % % % % % % % % %
% aeussere Figur
\draw plot [smooth cycle,tension=.6] coordinates {(0,0) (-.7,-2.8) (-.3,-4.1) (1.5,-4.8) (3.5,-4.5) (4.2,-3.5) (4,-2.2) (5.4,-1.2) (5.7,-.2) (5.4,.6) (4,1.5) (4.8,3.8) (4.2,5.8) (1.7,6.7) (-.8,5.3)};

% innere Figur
\draw plot [smooth cycle,tension=.6] coordinates {(0,0) (.3,-2) (1.7,-2.8) (4,-2.2) (2.7,0) (4,1.6)  (2.9,2.8) (1.8,2.7) (.7,1.8)};

% Punkte
\coordinate (P1) at (0.25,1.34);
\coordinate (P2) at (4,2.35);
\coordinate (P3) at (4.1,-1.57);
\coordinate (z1) at (1.67,5);
\coordinate (z2) at (3.9,.35);
\coordinate (z3) at (1.8,-3.5);

% Sterne und Kreuz
\begin{scope}[lw] % Strichstärke

% Kreuz
\begin{scope}[xshift=4.65cm,yshift=-5.25cm]
\draw[rotate around={135:(.8,4.3)}] (.8,4.2) -- (.8,4.4) node[right] {$p$};
\draw[rotate around={225:(.8,4.3)}] (.8,4.2) -- (.8,4.4);
\end{scope}

% unterer Stern z_3
\begin{scope}
\draw (1.8,-3.4) -- (1.8,-3.6)  node[below] {$z$};
\draw[rotate around={60:(z3)}] (1.8,-3.4) -- (1.8,-3.6);
\draw[rotate around={120:(z3)}] (1.8,-3.4) -- (1.8,-3.6);
\end{scope}
\end{scope}

% Tori

\begin{scope}[xshift=1.1cm, yshift=4cm, rotate=5]
\draw (0,0) arc (180:360:.6cm and .3cm);
\draw (.1,-.15) arc (180:0:.5cm and .2cm);
\end{scope}

\begin{scope} [xshift=4.1cm, yshift=-.25cm, rotate=0]
\draw (0,0) arc (180:360:.6cm and .3cm);
\draw (.1,-.15) arc (180:0:.5cm and .2cm);
\end{scope}

\draw (2,-5.9) node[] {$(X,\bfeta)$};
\node at (9,-3.9) {$(X_{(-2)},\eta_{(-2)})$};
\node at (9,-.5) {$(X_{(-1)},\eta_{(-1)})$};
\node at (9,4) {$(X_{(0)},\eta_{(0)})$};

%%%%%%
%%Knote
\node at (.4,-.4) {$q_{1}$};
\node at (4.45,1.45) {$q_{2}$};
\node at (4.4,-2.5) {$q_{3}$};

% % % % % % % % % % % % % % % % % % % % % % % % % % % % % % % % % % % % % % % % % % % % % % % % % % % %
%linkes Figur
% % % % % % % % % % % % % % % % % % % % % % % % % % % % % % % % % % % % % % % % % % % % % % % % % % % %
\begin{scope}[xshift=-11cm]
% aeussere Figur
\draw plot [smooth cycle,tension=.6] coordinates {(-.25,0) (-.7,-2.8) (-.3,-4.1) (1.5,-4.8) (3.5,-4.5) (4.2,-3.5) (4.2,-2.2) (5.4,-1.2) (5.7,-.2) (5.4,.6) (4.2,1.5) (4.8,3.8) (4.2,5.8) (1.7,6.7) (-.8,5.3)};

% innere Figur
\draw plot [smooth cycle,tension=.6] coordinates {(.25,0) (.3,-2) (1.7,-2.8) (3.6,-2.2) (2.7,0) (3.7,1.5)  (2.9,2.8) (1.8,2.7) (.7,1.8)};

% Punkte
\coordinate (P1) at (0.25,1.34);
\coordinate (P2) at (4,2.35);
\coordinate (P3) at (4.1,-1.57);
\coordinate (z1) at (1.67,5);
\coordinate (z2) at (3.9,.35);
\coordinate (z3) at (1.8,-3.5);

%\fill[] (4,1.65) circle (2pt);
%\fill[] (P3) circle (2pt);
%fill[] (4.01,-2.15) circle (2pt);
%\fill[] (0,0) circle (2pt);
%\fill (P1) circle (2pt);

%\draw[->-={latex}, color=red] (z2) .. controls (3.9,.3) and (2.9,-.6) .. (P3);

% Sterne und Kreuz
\begin{scope}[lw] % Strichstärke

% Kreuz
\begin{scope}[xshift=4.65cm,yshift=-5.25cm]
\draw[rotate around={135:(.8,4.3)}] (.8,4.2) -- (.8,4.4) node[right] {$s_{2}$};
\draw[rotate around={225:(.8,4.3)}] (.8,4.2) -- (.8,4.4);
\end{scope}
% unterer Stern z_3
\begin{scope}
\draw (1.8,-3.4) -- (1.8,-3.6)  node[below] {$s_{1}$};
\draw[rotate around={60:(z3)}] (1.8,-3.4) -- (1.8,-3.6);
\draw[rotate around={120:(z3)}] (1.8,-3.4) -- (1.8,-3.6);
\end{scope}
\end{scope}

% Tori

\begin{scope}[xshift=1.1cm, yshift=4cm, rotate=5]
\draw[color1] (.42,-.92) circle [x radius=.13cm, y radius=.67cm,  rotate=-23];
\draw[color1, -latex, rotate=-23] ($(.75,-.8) + (10:.13cm and .67cm)$(P) arc (10:15:.13cm and .67cm);
\draw (0,0) arc (180:360:.6cm and .3cm);
\draw (.1,-.15) arc (180:0:.5cm and .2cm);
\node[color=black] at (3.7,-6.2) {$\lambda_{6}$};
\end{scope}

\begin{scope} [xshift=4.1cm, yshift=-.25cm, rotate=0]
\draw[color1] (.6,-.9) circle [x radius=.13cm, y radius=.61cm, rotate=-5];
\draw[color1, -latex, rotate=-5] ($(.67,-.9) + (0:.13cm and .61cm)$(P) arc (0:5:.13cm and .61cm);
\draw (0,0) arc (180:360:.6cm and .3cm);
\draw (.1,-.15) arc (180:0:.5cm and .2cm);
\node[color=black] at (-2.3,2.3) {$\lambda_{8}$}; 
\end{scope}

% einzelne blaue Ringe
\begin{scope}[color4, xshift=.8cm, yshift=-.5cm, rotate=-90]
\draw (.6,-.9) circle [x radius=.13cm, y radius=.25cm];
\draw[-latex] ($(.6,-.9) + (0:.13cm and .61cm)$(P) arc (0:5:.13cm and .61cm);
\node[color=black] at (.2,-1.85) {$\lambda_{1}$};
\end{scope}

\begin{scope}[color4, xshift=4.91cm, yshift=1.6cm, rotate=-90]
\draw (.6,-.9) circle [x radius=.13cm, y radius=.61cm];
\draw[-latex] ($(.6,-.9) + (0:.13cm and .61cm)$(P) arc (0:5:.13cm and .61cm);
\node[color=black] at (.45,-1.95) {$\lambda_{2}$};
\end{scope}

\begin{scope}[color4, xshift=4.74cm, yshift=-1.75cm, rotate=-90]
\draw (.6,-.9) circle [x radius=.13cm, y radius=.31cm];
\draw[-latex] ($(.6,-.9) + (0:.13cm and .61cm)$(P) arc (0:5:.13cm and .61cm);
\node[color=black] at (.8,-.2) {$\lambda_{3}$};
\end{scope}

% einzelne raute Ringe
\begin{scope}[color1, xshift=4.45cm, yshift=.15cm, rotate=-90]
\draw (.6,-.9) circle [x radius=.13cm, y radius=.70cm];
\draw[-latex] ($(.6,-.9) + (0:.13cm and .61cm)$(P) arc (0:5:.13cm and .61cm);
\node[color=black] at (.6,-1.95) {$\lambda_{4}$};
\end{scope}

\begin{scope}[color1, xshift=5.46cm, yshift=1.12cm, rotate=-40]
\draw (.6,-.9) circle [x radius=.13cm, y radius=.40cm];
\draw[-latex] ($(.6,-.9) + (0:.13cm and .61cm)$(P) arc (0:5:.13cm and .61cm);
\node[color=black] at (.6,.05) {$\lambda_{5}$};
\end{scope}

\begin{scope}[color1, xshift=1.05cm, yshift=6.31cm, rotate=0]
\draw (.6,-.9) circle [x radius=.13cm, y radius=1.29cm];
\draw[-latex] ($(.6,-.9) + (0:.13cm and .61cm)$(P) arc (0:5:.13cm and .61cm);
\node[color=black] at (.7,1.0) {$\lambda_{7}$};
\end{scope}
\draw (2,-5.9) node[] {$\Sigma$};

\draw[->] (6.5,0) -- node[above] {$f$}  (9.5,0) ; 
\end{scope}

\end{tikzpicture}
	        \caption{The marked surface together with the multicurves $\Lambda$ and $\Lambda_{\rm \max} $.}
	\label{cap:modifdiff}
\end{figure}
%\end{document}
\par
The vertical multicurve $\Lambda^{v}$ (in \textcolor{blue}{blue} in
Figure~\ref{cap:modifdiff}) consists of curves $\lambda_{1}$, $\lambda_{2}$
and~$\lambda_{3}$, which are all homologous to each other. The filtration of
the $V_{i}$ induced by the multicurve $\Lambda^{v}$ is then given by
$$0 \= V_{-3}\subset V_{-2} \= \langle\lambda_{1}\rangle \= V_{-1} \= V\,,$$
where $\langle\cdot\rangle$ denotes the linear span. Hence the maps $\rho_{i}$ are
given by $\rho_{-2}(\bft)(\lambda_{1}) = t_{-2}^{a_{-2}} \cdot a \in \CC$ and $\rho_{-1}(\bft)(\lambda_{1})=0$, where $a$ is $2\pi\sqrt{-1}$ times the residue of $\eta$ at the corresponding node.
We choose the maximal multicurve $\Lambda_{\rm max}\supset \Lambda$ by adding the
curves $\lbrace \lambda_4,\dots,\lambda_8\rbrace$, shown in \textcolor{red}{red}.
Then we have in homology the equalities
$$\lambda_1 \= \lambda_2 \= \lambda_4+\lambda_5 \= \lambda_3 \=
\lambda_4+\lambda_6 \= \lambda_7 + \lambda_8\,,$$
and thus $V'=\langle\lambda_{1},\lambda_{4},\lambda_{7}\rangle$. We choose the sets $B_{i}$ to be  $B_{-2}=\lbrace \lambda_{4},\lambda_{7}\rbrace=B_{-1}$. Then the extension $\tilde\rho_{-2}$ of the map $\rho_{-2}$ defined on $V_{-2}=\langle\lambda_1\rangle$ is given by requiring $\tilde\rho_{-2}(\bft)(B_{-2})=0$. That is,
\[\tilde\rho_{-2}(\bft)(\lambda)=\left\{
  \begin{array}{rl}
    t_{-2}^{a_{-2}}\cdot  a\,,  & \text{ if }  \lambda \= \lambda_{1}\,, \\
    0\,,  & \text{ if } \lambda \= \lambda_{4},\lambda_{7} \,.\\
  \end{array}
\right.\]
Similarly, the extension $\tilde\rho_{-1}$ of $\rho_{-1}$ from
$V_{-1}=\langle\lambda_1\rangle$ to $V'$ is defined by requiring $\tilde\rho_{-1}(\bft)(B_{-1})=0$,
hence $\tilde\rho_{-1}$ is simply identically zero.
\par
We can now define the modifying differentials $\xi_{(i)}$. Following the construction
in the proof, we see that the differential $\xi_{(-2)} = 0$ identically. The
differential $\xi_{(-1)}$ is supported on the component on the right, which is at
level~$-1$. It has simple poles at~$q_2$ and~$q_3$ with residues $\pm t_{-2}^{a_{-2}}\cdot a/2\pi \sqrt{-1}$,
is holomorphic at $p$, has period zero over $\lambda_4$, and has periods~$t_{-2}^{a_{-2}}\cdot a$
over~$\lambda_{5}$ and~$\lambda_{6}$.  Finally, the differential $\xi_{(0)}$ has
simple poles at~$q_1$ and~$q_2$, with residues $\pm t_{-1}^{a_{-1}}t_{-2}^{a_{-2}}\cdot a/2\pi \sqrt{-1}$, has period
zero over~$\lambda_7$, and period~$t_{-1}^{a_{-1}}t_{-2}^{a_{-2}}\cdot a$ over $\lambda_{8}$.  To see these,
consider for example the period of $\xi_{(0)}$ at $\lambda_{8}$. By definition it is given by
$$\int_{\lambda_8} \xi_{(0)} \= \tilde\rho_{-2}(\lambda_{8}) + \tilde\rho_{-1}(\lambda_{8})
\=\rho_{-2}(\lambda_{1}) + \rho_{-1}(\lambda_{1}) \= t_{-2}^{a_{-2}}\cdot a + 0 \= t_{-2}^{a_{-2}}\cdot a\,,$$
since $\lambda_1$ is homologous to $\lambda_7 + \lambda_8$, and since $\tilde\rho_{-2}
(\lambda_7) = \tilde\rho_{-1} (\lambda_7) = 0$ for $\lambda_7 \in B_{-2}, B_{-1}$.
The other cases can be computed similarly.
\end{exa}

%%%%%%%%%%%%%%%%%%%%%%%%%%%%%%%%
\subsection{Perturbed period coordinates}
\label{sec:perturbed}
%%%%%%%%%%%%%%%%%%%%%%%%%%%%%%%%
We will now perturb the usual notion of period coordinates, to avoid using marked
points and zeros that are at the nodes, and instead choosing different \changed{points near the nodes to be endpoints of period integrals}.
We first introduce the preparatory material in full generality, and then define
the perturbed period coordinates under the simplifying assumption that there are
no horizontal nodes. We extend these coordinates to the case with horizontal nodes
in Section~\ref{sec:extendperturbed}.
\par
To define the perturbed period map we need to specify additional
marked points near the vertical \changed{nodes} of~$\eta$ and we need to
recall various spaces defined by residue conditions, together with the
dimension estimates from~\cite{kdiff}.
\par
Recall that $\Sigma_{(i)}^{c}$ was defined in the paragraph preceding  Definition~\ref{def:degeneration} as the subsurface of $\Sigma$ at level $i$ where the boundary curves have been collapsed to points.
The \Teichmuller\ markings up to twist group of the welded surfaces
in the model domain induce markings $f_i\colon  \Sigma_{(i)}^{c} \to \calX_{(i)}$ of the
families of connected components of the subsurfaces at level~$i$.
Denote by $P_{\bfs,i}$ and $Z_{\bfs,i}$ \changed{the} 
marked poles and zeros that lie on the compact level~$i$ subsurface~$\Sigma_{(i)}^c$.
We denote by $Q_{E,i}^{\pm}$ \changed{the set of 
%those zero and pole 
sections} on~$\Sigma_{(i)}^c$ mapping to the preimages of the nodes~$E$. We define the sets of points
\begin{equation}
\label{eq:PiZi}
P_i = P_{\bfs, i} \cup Q_{E,i}^{-} \quad \text{and} \quad Z_i = Z_{\bfs,i} \cup Q_{E,i}^{+}
\end{equation}
for each level~$i$.
\par
\medskip
The perturbed period coordinates are roughly the product of the coordinates~$t_i$
and coordinates of the projectivization of certain subspaces~$\calR_i^{\rm grc}$
of $H^1(\Sigma^c_{(i)}\! \setminus\! P_i, Z_i; \CC)$.
Coordinates on the latter are as usual given by all but one of the periods.
\par
To define~$\calR_i^{\rm grc}$  we start with the map
$\changed{H_1}(\Sigma^c_{(i)}\! \setminus\! P_i, Z_i; \CC) \to \CC^{|P_i|}$ given by taking the
integrals over small loops around the points~$P_i$. Note that the image of this map
is contained in the subspace cut out by the residue theorems on each component.
Let $R_i^{\rm grc} \subset \CC^{|P_i|}$ be the subspace cut out further by the matching
residue condition at the horizontal nodes, and the global residue condition, as stated in
Section~\ref{sec:deftwd}. The {\em GRC space} $\calR_i^{\rm grc} \subseteq
\changed{H_1}(\Sigma^c_{(i)}\! \setminus\! P_i, Z_i; \CC) $
is then defined as the preimage of $R_i^{\rm grc}$. If we denote by~$H$ the number
of horizontal nodes of~$\Lambda$, then \cite[Theorem~6.1]{kdiff} can
be restated as follows.
\par
\begin{prop} \label{prop:kdiffDim}
The (open) simple model domain $\Omega\MDs$ is locally modeled on
the sum of the GRC spaces $\oplus_i  \calR_i^{\rm grc}$. This space has dimension
\bes \sum_{i \in L^\bullet(\Lambda)} \dim(\calR_i^{\rm grc}) \= \dim \omodulin(\mu) - H\,.
\ees
\end{prop}
\par
For each half-edge~$h$ of~$\Gamma(\Lambda)$ with non-negative $m_{h}$, i.e.\ for
each non-polar marked point in the smooth part of~$\calX$, we denote by
$\bfz(h)$ the corresponding section  of~$\calX \to \calW$. We choose nearby
sections $\sigma_e^+\colon \calW \to \calX$ and  $\sigma_h\colon \calW
\to \calX$ so that
\begin{equation}\label{eq:constdistance}
\int_{\changed{q^+_e(w)}}^{\sigma_e^+(w)} \eta_{(i)} \= {\rm const}
\quad \text{and} \quad
\int_{\changed{\bfz(h)(w)}}^{\sigma_h(w)} \eta_{(j)} \= {\rm const}\,,
\end{equation}
where $i = \ell(e^+)$ and $j = \ell (h)$ are the corresponding levels that contain
the (short fiber-wise) integration paths respectively.
\par
As the final preparation step, note that the form $\bft \ast (\bfeta + \bfxi)$
on $\calX$ may no longer have a zero of the prescribed order at $\bfz(h)$ because of
the modifying differential~$\bfxi$. In the process of plumbing in
Section~\ref{sec:plumbing}, we will describe a local surgery
of~$\calX$ in a neighborhood of the sections $\bfz(h)$ corresponding to the
half-edges~$h$, such that the images of the sections~$\sigma_h$
are untouched by the surgery, and the extension of $\bft \ast (\bfeta + \bfxi)$
to the resulting family again has a
zero of order $\ord_{\bfz(h)} \bfeta$ along a section that we still denote by
$\bfz(h)\colon \calW \to \calX$.
\smallskip
\par
Finally, we can now define the {\em perturbed period map at level $i$}
under the hypothesis that there are no horizontal nodes. We fix
homology classes $\gamma_1,\dots, \gamma_{n(i)}$  such that their periods
$\int_{\gamma_j} \eta_{(i)}$ form a basis of $\calR_i^{\rm grc}$.
Stability of the curve~$\calX$ implies that for each~$i$ at least one
of the periods $\int_{\gamma_j} \eta_{(i)}$ is non-zero, say for
$j=n(i)$. \changed{Since~$\xi_{(i)}$ is small, in fact divisible by a power
of~$t_{i-1}$, we may and will assume in the sequel that
the rescaled forms~$\bfeta$ (and along with them the rescaling parameters~$\bft$)
have been chosen in Section~\ref{sec:universalMD}  so that the last period at each is normalized, by the condition $\int_{\gamma_{n(i)}}
(\eta_{(i)} + \xi_{(i)}) = 1$.}
We thus denote by $\calR_i'\subset\calR_i^{\rm grc}$ the
codimension one subspace  generated by the periods of
$\gamma_1,\dots,\gamma_{n(i)-1}$ for all levels $i<0$, and we let
$\calR_0'=\calR_0^{\rm grc}$, as the differential on the top level is
not considered up to scale. \changed{We let $n'(i) = \dim \calR_i'$, i.e,
$n'(i) = n(i)-1$ for $i\neq 0$ and $n'(0)= n(0)$.}
\par
The perturbed period map is then built with the help of
\begin{equation}\label{eq:defPPer}
\PPer_i\colon \begin{cases} \begin{array}{ccl} \calW &  \to &% \quad
\calR_i'\,,  \\
[(X,\bfeta,\bft)] &\mapsto &
\left(\int_{\gamma_j}  (\eta_{(i)} +\xi_{(i)})\right)_{j=1}^{n'(i)}\,.
     \end{array}
\end{cases}
\end{equation}
  Here the integrals are over the $f_i$-images
of the cycles, but \changed{\emph{we integrate from the points
$\sigma^+_e(w)$} for cycles starting or ending at a point in $Q_{E,i}^{+}$},
where $w = [(X,\bfeta,\bft)]$.%, rather than from the nearby zeros of~$\eta_{(i)}$.
\par
\changed{In what follows we denote by $\calW_\Lambda = \cap_{i \in L(\Lambda)} D_i$ the most degenerate stratum in $\calW$, where 
$D_i$ is fiber-wise defined by $\{ t_i = 0\}\subset \mathbb{C}^{L(\Lambda)}$ as introduced right above~\eqref{eq:stratMDs}.}
\par
\begin{prop} \label{prop:pertper}
The perturbed period map
\begin{equation}\label{eq:defPPnohor}
\PPer\colon \calW \to \CC^{L^\bullet(\Lambda)} \times \bigoplus_{i\in L^\bullet(\Lambda)}
\calR_i'\,, \qquad
[(X,\bfeta,\bft)] \mapsto \Big(\bft\,;\oplus_{i \in L^\bullet(\Lambda)}\PPer_i\Big)
\end{equation}
is open and locally injective on a neighborhood of $\calW_\Lambda$ inside of~$\calW$. 
%the most degenerate stratum $\calW_\Lambda = \cap_{i \in L(\Lambda)} D_i$ inside of~$\calW$.
\end{prop}
\par
\begin{proof}
We need to show that the derivative of $\PPer$ is surjective
along the boundary stratum $\calW_\Lambda$, since surjectivity is an open condition,
and since this surjectivity implies openness. At a point of $\calW_\Lambda$
the $i$-th summand $\PPer_i$ consists of the usual period coordinates for $\eta_{(i)}$,
shifted by a constant since we integrate from a nearby point
(using the absence of horizontal nodes by our assumption). Here along the most degenerate
stratum $\calW_\Lambda$, the integral
of $\xi_{(i)}$ is identically zero, because by definition $\xi_{(i)}$ is
divisible by~$t_{i-1}^{a_{i-1}}$, and $\calW_\Lambda$ is defined by the equations $t_j = 0$ for
all $j\in L(\Lambda)$. In the complementary directions, surjectivity is
obvious since the $t_i$ are coordinates on the domain and are included in the target
of $\PPer$.
\par
Since $\calW$ is smooth and of the same dimension as $\oplus_{i\in L^\bullet(\Lambda)}
\calR_i'$ by Proposition~\ref{prop:kdiffDim} (recall that we are still working under the assumption of no horizontal
nodes), surjectivity of the derivative
of $\PPer$ implies injectivity of the derivative map at any boundary point
in $\calW_\Lambda$, and hence local injectivity in some neighborhood.
\end{proof}
\par
\begin{exa}
 We give the description of the perturbed period coordinates in the setting of our
 running example of Section~\ref{sec:runningex}.  Hence the differentials that we consider are in the closure of the meromorphic stratum  $\Omega\calM_{5,4}(4,4,2,-2)$. More precisely, we consider the enhanced dual graph $\eG_{2}$ on the right of Figure~\ref{cap:running}. In this case the differential $\eta_{(-2)}$ is in $\omoduli[0,3](4,-2,-4)$, $\eta_{(-1)}$ is in  $\omoduli[1,4](0,4,-2,-2)$ and $\eta_{(0)}$ is in $\omoduli[3,3](2,2,0)$. Since the global residue condition imposes precisely the condition that the residue of~$\eta_{(-1)}$ at $q_{e_{1}}^{-}$ is zero, the GRC space is the product of the top and bottom~$H^{1}$ with the hyperplane of the middle $H^{1}$ given by this residue condition.
 \par
 \input{pic_plumbing_dot_tag_3_colored}
 \par
We will consider the deformations over a disc $\Delta^{2}=\Delta_{t_{-1}}\times
\Delta_{t_{-2}}$ which parameterizes the smoothing of the levels of $(X,\eta)$. Note that
the residues at the poles of the differential $\eta_{(-2)}$ are non-zero
(see \cite[Lemma~3.6]{strata}). The family of modifying differentials~$\bfxi$ consists
of $\xi_{(0)}$ on $X_{(0)}$ and $\xi_{(-1)}$ on $X_{(-1)}$, where $\xi_{(0)}$ is divisible
by $t_{-1}^{3}$, and $\xi_{(-1)}$ is divisible by $t_{-2}^{3}$. Moreover, $\bfxi$ vanishes
identically on~$X_{(-2)}$. In Figure~\ref{cap:changebasis} we show a basis of the cycles
of integration before and after the plumbing construction described later in
Section~\ref{sec:Dehn}. In this basis, the map $\PPer_{0}$ is given by the map which
associates the integrals of $\eta_{(0)}+ \xi_{(0)}$ along the cycles belonging to~$X_{(0)}$. The maps $\PPer_{-1}$
and $\PPer_{-2}$ are defined analogously.
 \par
 We now describe how the perturbed period coordinates behave over the base~$\Delta^{2}$.
Note that since our construction is local, we can identify the circles $\alpha_j$
and $\beta_j$ for $j=1,\dots,4$ with the circles $\alpha_j^{0}$ and $\beta_j^{0}$. On the
subsurface~$X_{(0)}$ the restriction of the differential $\eta_{(0)}+ \xi_{(0)}$ on $\alpha_j^{0}$  and  of the plumbed differential to~$\alpha_j$
clearly coincide (where all the~$t_i$ are non-zero) under this identification. The
case of the subsurface~$X_{(-1)}$ is similar. Note that if the modifying differentials
vanish, then the period of each cycle on~$X_{(0)}$ would be a constant and the period
of each cycle on~$X_{(-1)}$ would be of a constant times~$t_{-1}^{3}$.
\par
We now consider the relative cycles $\gamma_{k}$ which degenerate to the relative cycles~$\gamma_{k}^{0}$. The period for $\gamma_{1}$ is equal to the period for $\gamma_{1}^{0}$ plus a function of~$t_{-1}$ and $t_{-2}$ which is zero on $\{ t_{-1}t_{-2}=0 \}$. This function depends on the choice of the points near $z_{1}$, near the node, and the way that we glue the plumbing fixture in the nodal differential. The case of the cycles $\gamma_{k}$ for $k=2, 3$ is similar.
\par
Finally, note that the period of $\bft \ast \bfeta$ at the homotopic cycles
$\delta_1$ and $\delta’_1$ is a function $2\pi  i r$ that is divisible by
$(t_{-1} t_{-2})^3$, where~$r$ is the residue at the corresponding node.  This
is consistent with the GRC.
\end{exa}

%%%%%%%%%%%%%%%%%%%%%%%%%%%%%%%%%%%%
\section{The Dehn space and the complex structure }
\label{sec:Dehn}
%%%%%%%%%%%%%%%%%%%%%%%%%%%%%%%%%%%%%

In order to understand the structure of~$\LMS$ at the boundary,
we introduce an auxiliary space~$\ODehns$, the \emph{simple Dehn
space}, which is a direct analogue of the classical Dehn space.
The goal of this section is to give, for each~$\Lambda$, the topological
space~$\ODehns$ the structure of a complex manifold, which will then be used
to give~$\LMS$ its complex structure.  This complex structure is induced
by \emph{plumbing maps}  $\OPl\colon  U\to \ODehns$, defined by a local
plumbing construction on the universal family of \auxds, which we will
show give an atlas of complex coordinate charts on the simple Dehn space.
\par
There is a natural open forgetful map $\ODehns\to\LMS$, and~$\LMS$ is
covered by the images of these maps as~$\Lambda$ ranges over all
enhanced multicurves. The conclusion of this section is summarized in
Theorem~\ref{thm:LMScomplex}, where we use the complex structures on
the $\ODehns$ to give $\LMS$ the structure of a smooth complex
orbifold.  \changed{Roughly speaking, this means that $\LMS$ has an
  atlas of charts locally modeled on quotients of domains in $\cx^N$
  by finite groups of biholomorphic automorphisms.  We refer the
  reader to the appendix to \cite{chenruan} or \cite[\S3.4]{acgh2} for
  careful introductions to orbifolds.}
\par
Throughout this section, we fix an enhanced multicurve~$\Lambda$ with
dual graph~$\Gamma$ having $N+1 = |L^\bullet(\Lambda)|$ levels and $H$ horizontal
nodes.

%%%%%%%%%%%%%%%%%%%%%%%%%%%%%%%%5
\subsection{The Dehn space}
%%%%%%%%%%%%%%%%%%%%%%%%%%%%%%%%5

Informally, $\ODehn$ is the moduli space of $\Tw$-marked \msds of type
$(\mu,\Lambda')$, where $\Lambda'$ is any undegeneration of $\Lambda$.
The simple Dehn space~$\ODehns$ is the analogous space of $\sTw$-marked
differentials.
\changed{Our definition of the Dehn space in this section
parallels the definition of the classical Dehn space in Section~\ref{sec:dehn-space-deligne}.
See also the construction in \cite[Section 7]{HubKoch}.}
\par
More formally, the \emph{Dehn space associated with~$\Lambda$}   \index[teich]{f040@$\ODehn $!Dehn space associated with~$\Lambda$}
is the topological space
\begin{equation}\label{eq:defDehn1}
\ODehn \=
\Bigl(\coprod_{\Lambda' \rightsquigarrow \Lambda} \msT[\Lambda']\Bigr)/\vTw\,,
\end{equation}
where we endow this disjoint union with the subspace topology induced from the
topology of the augmented \Teichmuller\ space of flat surfaces $\Oaugteich$, and
recall from~\eqref{eq:defbdstrata} that $\msT[\Lambda']$ are boundary strata
in $\Oaugteich$). We can write the space equivalently as
\begin{equation}\label{eq:defDehn2}
\ODehn \= \coprod_{\Lambda' \rightsquigarrow\Lambda} \ODstratum
\quad \text{where} \quad \ODstratum \= \msT[\Lambda'] / \vTw[\Lambda] \,.
\end{equation}
The \emph{simple Dehn space} is defined by
\begin{equation}\label{eq:defDehns2}
\ODehns \=
\Bigl(\coprod_{\Lambda' \rightsquigarrow \Lambda} \msT[\Lambda']\Bigr)/\svTw
\= \coprod_{\Lambda' \rightsquigarrow \Lambda} \ODstratums\,,
\end{equation}
where $\ODstratums \= \msT[\Lambda'] / \svTw[\Lambda]$.
\index[teich]{f050@$\ODehns $!Simple Dehn space associated
  with~$\Lambda$}
The \emph{simple vertical Dehn space} $\ODehnsv \subset \ODehns$ is
\index[teich]{f055@$\ODehnsv  $!simple vertical Dehn space}
the locus consisting of surfaces where every horizontal edge of
$\Lambda$ corresponds to a horizontal node.  In other words,
\begin{equation*}
  \ODehnsv  = \coprod_{\Lambda' \rightsquigarrow \Lambda} \ODstratums\,,
\end{equation*}
where the union is over all \emph{vertical} undegenerations $\Lambda'
\rightsquigarrow \Lambda$.
\par
The finite group $K_\Lambda = \vTw [\eL] / \svTw$ acts on $\ODehns$
with topological quotient $\ODehn$.  We will see that $\ODehns$ is in
fact a smooth manifold with quotient orbifold $\ODehn$.
\par
We write similarly $\Dehn$ and $\Dehns$ for the corresponding spaces where the top
level is projectivized, that is, $\Dehn$ is the quotient of $\ODehn$ under
the $\CC^*$-action.
\par
We refer to a point in the Dehn space (resp.~simple Dehn space) as (the moduli
point of the equivalence class of) a marked \msd $(Y,\bfz,\bfomega,\bfsigma,f)$
where the marking is up to the action of~$\Tw$ (resp.~$\sTw$). Pointwise this is
justified by definition.
\par
\medskip
We now outline the \emph{plumbing construction}.  We divide the
construction into two steps, \emph{vertical plumbing} and
\emph{horizontal plumbing}.
The vertical plumbing construction starts with the universal curve $\OMDsfam \to
\barOMDs$, which we restrict to a neighborhood $\calW_\epsilon$ of a
point $P$ in the deepest boundary stratum.  By cutting out neighborhoods of the
vertical nodes and gluing in standard plumbing fixtures, we construct a new
family of curves $\calY^v \to \calW_\epsilon$ whose generic fiber has only
horizontal nodes.
\par
The horizontal  plumbing construction requires an extra complex parameter for
each horizontal node, parameterizing the modulus and a twist parameter for the
annulus that is glued in.  We consider \changed{$\calY^v$} as a family over the product
$\calW_\epsilon\times \Delta^H$ (which does not depend on the second factor).
By cutting out a neighborhood of each horizontal node and gluing in a standard
plumbing fixture, with parameter given by the second factor, we then construct a
new generically smooth family of curves $\calY\to\calW_\epsilon\times \Delta^H$.
\par
We equip our standard plumbing fixtures with families of one-forms and
choose our gluing maps to identify these forms with those on the target,
so that the family $\calY$ comes with a degenerating family of one-forms $\omega$.
In fact, with more care we give the fibers of $\calY$ the structure of
$\sTw$-marked \msds.
\par
This horizontal plumbing construction is a version of the standard
plumbing construction used to construct coordinates near the boundary
of $\barmoduli[g,n]$, dating back at least to Bers \cite{bersplumb}.
We emphasize that the vertical construction differs from the usual plumbing
in that it does not require extra parameters to describe opening up the
nodes---this is rather prescribed by the relative size of the differentials,
so that they would glue on the plumbed surface.  These plumbing constructions
are not canonical and depend on choices made at several points in the construction.
\par
If the universal property for $\ODehns$ were available, the plumbed
family $\calY \to \calW_\epsilon\times \Delta^H$ would give a holomorphic map
$\calW_\epsilon\times \Delta^H \to \ODehns$.  Unfortunately,
the universal property is not available yet, as we wish to give
$\ODehns$ its complex structure, and then use it in establishing the
universal property.  Instead, we define a \emph{plumbing map}
$\OPl\colon \calW_\epsilon\times \Delta^H \to \ODehns$ stratum-by-stratum, using the
universal property for the boundary strata $\msT[\Lambda']$ parameterizing
equisingular loci in the augmented Teichm\"uller space.
Similarly, the family $\calY^v \to \calW_\epsilon$ will give rise to a
\emph{vertical plumbing map} $\OPlv\colon\calW_\epsilon\to \ODehnsv$.  As the
plumbing constructions are not canonical, neither are these plumbing maps, as
they depend on several choices.
\par
In Section~\ref{sec:PLUMBstandard_coordinates} below, we use the
normal forms from Section~\ref{sec:NF} to construct the gluing
maps used in the vertical plumbing construction.  In
Section~\ref{sec:plumbing} we define a vertical plumbing
construction and a vertical plumbing map.  In
Section~\ref{sec:plhomeo} we show that this
map is a local homeomorphism.  In Section~\ref{sec:plhoriz} we
introduce the horizontal plumbing construction and the full plumbing
map, and show that it is a local homeomorphism, which yields the following
main results of this section.
\par
\begin{thm}\label{thm:plumbing}
For any point~$P$ in the deepest stratum \changed{$\OMDstratum[\Lambda][s,\Lambda]$}
of the model domain  $\barOMDs$, there exists a neighborhood~$\calW_\epsilon\times
\Delta^H$ of $P \times {\bf 0} \in \barOMDs \times \Delta^H$,
and a plumbing map
$$ \OPl\colon \calW_\epsilon\times \Delta^H\to \ODehns\,, $$
which is a local homeomorphism. This map preserves the
stratifications~\eqref{eq:stratMDs} and~\eqref{eq:defDehns2}, is holomorphic
on each stratum,  and is $K_\Lambda$-equi\-variant.
Moreover, the plumbing map $\OPl$ can be chosen to be $\CC^*$-equivariant, and
thus to descend to a plumbing map
$$ \OPl\colon (\calW_\epsilon\times \Delta^H) / \CC^* \to \Dehns $$
that is also \changed{a local homeomorphism, stratum-preserving, holomorphic on each stratum, 
and $K_\Lambda$-equi\-variant.} 
\end{thm}
\par
Note that the deepest stratum of the model domain $\barOMDs[\Lambda]$
can be canonically identified with the deepest stratum of the
corresponding Dehn space (and we implicitly do so throughout this
section).
\par
\begin{thm} \label{thm:DehnIsOrbi}
The collection of all plumbing maps gives an atlas of charts which makes $\ODehns$
and the projectivized version $\Dehns$ a complex manifold. Moreover, for each point
of these spaces the plumbing construction provides a corresponding \msd.
\par
The spaces $\ODehn$ and $\Dehn$ have the structure of complex analytic spaces
(or orbifolds) with at worst abelian quotient singularities.
\end{thm}
\par
\changed{The key point of the proof of this theorem is to show that
  the transition functions of these plumbing maps are biholomorphic.
  This is done in Lemma~\ref{lem:transitions_are_holomorphic} using
  the Riemann Extension Theorem and the fact that plumbing maps are
  local homeomorphisms and holomorphic on generic strata.  (This was
  Bers' approach to define the complex structure of the
  classical Dehn space; see \cite{bers81}.) This shows
  in particular that the complex structure defined by the plumbing
  maps does not depend on any of the many choices made in the plumbing
  construction.   
}

We will discuss the universal
properties of the Dehn spaces in Section~\ref{sec:UnivDehn}.
\par
We now collect some of the properties already proved, which proves the first of
our main results, Theorem~\ref{intro:main} except for item~(4). Recall
that a divisor in an orbifold is said to be normal crossing if it
is the image of a normal crossing divisor in orbifold charts.
\par
\begin{thm} \label{thm:LMScomplex}
The moduli space of \msds is a  complex orbifold $\LMS$ containing
$\omoduli$ as an open dense suborbifold with complement a normal crossing
boundary divisor. The quotient  $\PP\LMS = \LMS/\CC^*$ is compact.
\par
The connected components of  $\LMS$ are in bijection with the connected
components of $\omoduli(\mu)$.
\end{thm}
\par
\begin{proof} The statement combines Theorem~\ref{thm:MSDcompact}
and Theorem~\ref{thm:DehnIsOrbi}. The orbifold structure and the
normal crossing boundary carry over from the model domain, as defined
in Proposition~\ref{prop:MDsmooth} and along with~\eqref{eq:stratMDs}.
The statement about components follows from the smoothness of the
orbifold chart.
\end{proof}

%%%%%%%%%%%%%%%%%%%%%%%%%%%%%
\subsection{The setup and notation for vertical plumbing.} \label{sec:the_setup}
%%%%%%%%%%%%%%%%%%%%%%%%%

We now set up the notation for the neighborhoods in which the vertical
plumbing construction is performed and for the plumbing
fixtures we need. We fix for the remainder of the plumbing
construction \changed{a base point in $\msT$ which represents the
  equivalence class of some surface $(X_0, \bfeta_0)$} and a local
chart $\Psi\colon \Delta_\epsilon^M \to \msT$, where $M = \dim \msT$,
parameterizing a neighborhood $\calV$ of
$\Psi(\bfzero) = (X_0, \bfeta_0)$.%
\index[plumb]{b010@$(X_0, \bfeta_0)$!Base surface in
  the~$\Lambda$-boundary stratum} \index[plumb]{e020@$\Psi$!Local
  chart center at the base surface}
As in Section~\ref{sec:universalMD}, we let $\calW\subset\barOMDs$ be
the neighborhood of $(X_0, \bfeta_0)$ that is the preimage of $\calV$.  We fix
for the rest of this section a local trivialization of the model
domain over $\calW$, which determines holomorphic functions
$\bft\colon \calW\to \cx^N\times \cx^*$, rescaled forms~$\bfeta$
(\changed{with the normalization of a period over $\gamma_{n(i)}$ at each
level $i\in L(\eG)$ as in the paragraph before~\eqref{eq:defPPer}}), and
prong-matchings $\bfsigma$, so that the tautological one-form
is~$\bft\ast \bfeta$.  
\par
The product map $\Psi\times\bft$ identifies $\calW$ with
$\Delta_\epsilon^{M}\times\cx^N\times\cx^*$.  We will subsequently work on
\index[plumb]{e010@$\calW_\epsilon $!Base of the plumbing construction}
\begin{equation*}
\calW_\epsilon \= \Delta_\epsilon^{M} \times \Delta_\epsilon^{N}
\times \CC^*\subset\calW\,,
\end{equation*}
for $\ve = \ve(X_0,\bfeta_0)$ sufficiently small and to be determined
(first by Theorem~\ref{thm:gluing_maps}, and then to be reduced a finite number of
times in the course of the construction).
\par
In the remainder of the section, we will make the above identification implicit and
simply write $\OMDsfam \to \calW_\epsilon$ for the restriction of the
universal curve to the domain of the chart. We will denote points in $\calW_\ve$
as $(\bfw,\bft)$ with $\bfw \in \Delta_\epsilon^M$, or by \changed{$(X,\bfeta,\bft)$.}
The boundary stratification of the model domain induces a
stratification of  $\calW_\epsilon$.  Given a subset $J\subset
L(\Lambda)$, we define $\calW_\epsilon^J =
\calW_\epsilon\cap\MD^{s,\Lambda_{J}}$.  In other words,
$\calW_\epsilon^J$ is the locus where $t_i=0$ if and only if $i\in J$.
\par
\medskip
We now introduce the notation for our standard annuli and plumbing fixtures,
and families of such. We define the standard round annulus
\index[plumb]{d@$ A_{\delta_1, \delta_2}$!Annulus of inner radius $\delta_{1}$ and outer
radius $\delta_{2}$}\bes
  A_{\delta_1, \delta_2} \= \{z\in \CC : \delta_1 < |z| < \delta_2\}
\ees
and use the base point $p = \sqrt{\delta_1\delta_2}\in A_{\delta_1, \delta_2}$ unless
specified differently. For $\delta = \delta(X_0,\bfeta_0)$ to be determined below,
and $s\in\CC$, we define the standard \emph{plumbing fixture}
\index[plumb]{d@$V({s}) $!Standard plumbing fixture}
\begin{equation}\label{eq:PlFix}
V({s}) \= \{(u,v) \in \Delta_\delta^2: uv = s\}
\end{equation}
together with the \emph{top plumbing annulus} and \emph{bottom plumbing annulus}
  \index[plumb]{d@$A^{\pm} $!Top and bottom plumbing annuli}
  \index[plumb]{d@$R,\delta $!Defining constants of the plumbing annuli}
  \begin{equation}
    \label{eq:PlAnnuli}
A^+ \= \{ \delta/R < |u| < \delta\} \quad\text{and}\quad
A^- \= \{ \delta/R < |v| < \delta\}
  \end{equation}
for some $R$ still to be specified. Unless specified otherwise, we will
use the basepoints 
  \begin{equation}
    \label{eq:pdef}
    p^\pm \= \delta/\sqrt{R} \in A^\pm
  \end{equation}
respectively. For $s=0$ the plumbing fixture is simply
$$ V(0) \= \Delta_{\delta}^+ \cup \Delta_{\delta}^-,$$
i.e.,~two disks joined at a node, with $u$ being the coordinate on~$\Delta_\delta^+$
and~$v$ on~$\Delta_\delta^-$.  \index[plumb]{d@$p^\pm$!Top and bottom marked points}
\par
For each vertical edge $e$ of $\Gamma = \Gamma(\Lambda)$, we define
the \emph{plumbing fixture} $\VV_e\to \calW_\epsilon$ to be the standard
model family of nodal curves over~$\calW_\ve$
\begin{equation}
  \label{eq:plumbing_fixture}
  \VV_e \= \left\{(\bfw, \bft, u, v) \in \calW_{\epsilon} \times \Delta_{\delta}^2
  \,\colon\,  uv = \bft_e \right\}, \quad\text{where}\quad \bft_e = \prod_{i
    = \lbot}^{\ltop-1} t_i^{m_{e,i}} ,
\end{equation}
\index[twist]{d@$\bft_e$! Product over the $t_i^{m_{e,i}}$ over the level
crossed by~$e$}
and where the integers $m_{e,i}$ are defined in \eqref{eq:aidef}. Note that the
fiber of $\VV_e \to \calW_\epsilon$ is an annulus if each $t_i$ in the product
in~\eqref{eq:plumbing_fixture} is non-zero, and a pair of disks meeting at
a node otherwise.
\par
We  denote by $r_e, r_e' \colon \calW_\epsilon \to \CC$ the residue functions
\begin{equation*}
  \changed{r_e \=  \Res_{q_e^-}  \bfeta \quad\text{and}\quad
  r'_e \=  \Res_{q_e^-}  (\bfeta+\bfxi)\,,}
\end{equation*}
%\begin{equation*}
%r_e \=  \Res_{q_e^-} \bft\ast \bfeta \quad\text{and}\quad
%r'_e \=  \Res_{q_e^-} \bft\ast (\bfeta+\bfxi)\,,
%\end{equation*}
where $\bfxi$ denotes \changed{a choice of a} modifying differential, satisfying conditions of Definition~\ref{df:modif}, constructed in Proposition~\ref{prop:constrModif}.
\par
We equip $\VV_e$ with the relative one-form~$\Omega_e$, given in
coordinates by
\begin{equation}\label{eq:reloneform}
  \changed{\Omega_e \= \prodttop\left( u^{\kappa_e} - r_e'\right) \frac{du}{u}
  \= \prodtbot\left(
    -v^{-\kappa_e} + r_e'\right)\frac{dv}{v}\,,}
\end{equation}
%\begin{equation}\label{eq:reloneform}
%  \Omega_e \= \left(\prodttop\cdot u^{\kappa_e} - r_e'\right) \frac{du}{u}
% % \quad \text{and} \quad \Omega_e 
%  \= -\left(
%    \prodtbot\cdot v^{-\kappa_e} - r_e'\right)\frac{dv}{v}\,,
%\end{equation}
where the notation ${\prodttop}$ was introduced in~\eqref{eq:omTeta}. The
two expressions agree if $uv \neq 0$.
\par
In what follows we will carefully choose the sizes of $\delta$ and $\epsilon$ for
the plumbing fixtures in~\eqref{eq:plumbing_fixture}, so that the moduli of the
annuli are sufficiently large, as required by some later parts of the plumbing
construction. We start by fixing a constant
$R>1$, and denote $\delta =\delta(R)$ and $\epsilon = \epsilon(R)$ the corresponding
constants that will be provided by Theorem~\ref{thm:gluing_maps} below.
\par
We define families of disjoint annuli $\calA_e^+, \calA_e^- \subset
\VV_e$ by
\begin{align*}
\calA_e^+ &\= \{(\bfw, \bft, u,v) : |w_i|,|t_j| < \epsilon \text{ for all $i,j$,
and } \delta/R < |u|
 < \delta\} \quad\text{and}\\
\calA_e^- &\= \{(\bfw, \bft, u,v) : |w_i|,|t_j| < \epsilon \text{ for all $i,j$,
and } \delta/R < |v|
  < \delta\}\,.
\end{align*}
We will refer to $\calA_e^+$ and $\calA_e^-$ as the {\em top and bottom
plumbing annuli} corresponding to the vertical edge~$e$. \changed{We call
$\VV_e \setminus \{ \calA_e^+ \cup \calA_e^-\}$ the \emph{trimmed plumbing fixture.}}
\changed{At several points in the sequel we will implicitly identify $\calA_e^+$
and $\calA_e^-$ with $A^+\times \calW_{\epsilon}$ and $A^-\times
\calW_{\epsilon}$ via the coordinates~$u$ and~$v$.}

%%%%%%%%%%%%%%%%%%%%%
\subsection{Standard coordinates.} \label{sec:PLUMBstandard_coordinates}
%%%%%%%%%%%%%%%%%%%%%%%

We now apply the normal form theorems of Section~\ref{sec:NF} to the
family $\calX\to \calW_\epsilon$.  Several of these normal forms are
not unique, in which case we simply make an arbitrary choice.
\par
By an application of Strebel's original result
(Theorem~\ref{thm:standard_coordinates}) in families, we know that
for some $\delta_1>0$ and for each node \changed{$e$} there exist \changed{the following maps that are biholomorphic onto their images}
%local coordinates
\begin{equation}
  \label{eq:standard_phi}
  \phi_e^+ \colon \calW_\epsilon \times \Delta_{\delta_1}\to \calX_{\ltop}
  \quad\text{and}\quad
  \phi_e^- \colon \calW_\epsilon \times \Delta_{\delta_1}\to \calX_{\lbot}
\end{equation}
(to keep the notation manageable, we write simply $\calX_{\ltop}$
instead of $\calX_{(\ltop)}$) whose restrictions to $\calW_\epsilon\times\{0\}$ correspond  to the loci $Q_e^+$ and $Q_e^-$ respectively, and
which put the form $\bft\ast \bfeta$  in the normal form.
\index[plumb]{b015@$\phi_e^\pm$!Local coordinates at a node $e$}
For a vertical node~$q_{e}$, this normal form is
\index[plumb]{b020@$r_{e}(\bft)$! Residue of $\bfomega$ at the node
  $e$}
\begin{align*}
  (\phi_e^+)^*( \bft\ast\bfeta) &\=  \prodttop u^{\kappa_e} \frac{du}{u}
                                  \quad\text{and}\quad\\
  \changed{(\phi_e^-)^*( \bft\ast\bfeta )} &\changed{\=  \prodtbot \left( -v^{-\kappa_e} +
                                   r_e\right) \frac{dv}{v}\,.}
\end{align*}
\par
As these standard coordinates are not unique, we use the prong-matching $\sigma_e$
to restrict their choice as follows.  Given a choice of $\phi_e^\pm$, in these
coordinates the prong-matching must be of the form $\sigma_e = \zeta du \otimes dv$,
where $\zeta$ is some $\kappa_e$'th root of unity.  We require that our choice of
$\phi_e^\pm$ makes this root equal to 1, so that
\begin{equation}
  \label{eq:sigmae}
  \sigma_e \= du \otimes dv.
\end{equation}
\par
In general, the modified differential $\bft\ast (\bfeta + \bfxi)$ does not admit such
a simple normal form in a neighborhood of a vertical node.  Consider a vertical node
with top section $q_e^{+}\colon \calW_\epsilon \to \calX_{\ltop}$, which is a zero of order
$\kappa_e-1$ of $\bft\ast \bfeta$. Then this zero breaks up into a simple
pole and $\kappa_e$ nearby zeros of the differential $\bft\ast (\bfeta+\bfxi)$.
These extraneous nearby zeros should not belong to our plumbed family, so we will
construct a family of disks~$\calE_e^+$ containing these nearby zeros, which
we will then cut out of~$\calX$.  These disks will be bounded by a family of annuli
$\calB_e^+$, and come with a family of \emph{gluing maps}
$\Upsilon_e^+\colon\calA_e^+\to\calB_e^+$ in which the differential~$\bft\ast (\bfeta+\bfxi)$ has normal form on a family of annuli over~$\calW_\epsilon$.  These objects are constructed in
Theorem~\ref{thm:gluing_maps} below.  This is the basic analytic ingredient in our plumbing
construction.  In Section~\ref{sec:plumbing}, we will use these gluing maps to glue
in the standard plumbing fixture $\VV_e$ defined above.
\par
Adding the modifying differential $\bfxi$ creates a similar problem at the
zero sections~$\bfz(h)$ of~$\calX$.  When the modifying differential
is added, a zero of order~$m_{h}$  breaks into~$m_h$ nearby zeros, but we wish to
construct a family where the order of the zero remains constant.  The solution
is similar, that is, we construct below a family of disks $\calE_h$ around $\bfz(h)$,
and gluing maps that put $\bft\ast (\bfeta + \bfxi)$ into the normal form on a
family of annuli surrounding $\calE_h$. In Section~\ref{sec:plumbing} we
will then cut out these disks and glue in a standard family of disks $\DD_h$.
\par
\begin{thm}
  \label{thm:gluing_maps}
  There exists a constant $\delta>0$ such that for any $R>0$ there
  exists a constant $\epsilon>0$  such that for each vertical
edge $e$ and for each half-edge $h$ of~$\Gamma$ there are families of conformal maps
of annuli
\index[plumb]{e060@$ \upsilon_e^+ $, $\upsilon_e^-$, $\upsilon_h $! Conformal maps
on annuli putting $\bft\ast (\bfeta + \bfxi)$ in
standard form}
\begin{align*}
    \upsilon_e^+ &\colon \calW_\epsilon \times A_{\delta/R, \delta} \to
    \calX_{\ltop}\,,\\
    \upsilon_e^- &\colon \calW_\epsilon \times A_{\delta/R, \delta} \to
    \calX_{\lbot}\,,\quad\text{and}\\
    \upsilon_h &\colon \calW_\epsilon \times A_{\delta/R, \delta} \to \calX_{\ell(h)}\,.
\end{align*}
These maps commute with the projections of the source and target to $\calW_e$,
and have the following properties:
\begin{enumerate}[(i)]
  \item The images of $\upsilon_e^+$, $\upsilon_e^-$, $\upsilon_h$ are families of annuli
    $\calB_e^+, \calB_e^-$, $\calB_h$ that do not contain any zeros or
poles of $(\calX, \bft\ast \bfeta)$.  The families of  annuli $\calB_e^+,\calB_e^-$,
and $\calB_h$ bound families of disks $\calE_e^+, \calE_e^-,$ and $\calE_h$, respectively, where
    \begin{equation*}
      Q_e^+ \subset \calE_e^+ \subset  \calX_{\ltop}\,, \quad
      Q_e^- \subset \calE_e^- \subset  \calX_{\lbot}\,,\quad\text{and} \quad
      \bfz(h) \subset \calE_h \subset  \calX_{\ell(h)}\,.
    \end{equation*}
\item The pullback of $\bft\ast (\bfeta + \bfxi)$ under each of the maps
$\upsilon_e^+$, $\upsilon_e^-$, and $\upsilon_h$, has the normal form on the annulus,
that is
    \begin{align*}
      (\upsilon_e^+)^*(\bft\ast (\bfeta + \bfxi)) &\= \changed{\prodttop\left( z^{\kappa_e} - r_e'\right)
      \frac{dz}{z}}\,, \\
      (\upsilon_e^-)^*(\bft\ast (\bfeta + \bfxi)) &\= \changed{\prodtbot\left( -z^{-\kappa_e} +
        r_e')\right) \frac{dz}{z}}\,, \quad\text{and} \\
      \upsilon_h^*(\bft\ast (\bfeta + \bfxi)) &\= \prodthor\cdot z^{m_{h}}  dz \,.
    \end{align*}
\item The maps $\upsilon_e^+$, $\upsilon_e^-$, and $\upsilon_h$ agree with the corresponding
maps $\phi_e^+$, $\phi_e^-$, and $\phi$ of Theorem~\ref{thm:standard_coordinates}
on the subset of  $\calW_\epsilon\times  A_{\delta/R, \delta}$
where $t_{L-1}= \dots = t_{-N} =0$ with $L=\ltopbot$ or $L=\ell(h)$
respectively.
 \end{enumerate}
Moreover, we may take $\delta$ sufficiently small that the maps
$\upsilon_e^\pm$ and $\upsilon_h$ are injective and have mutually disjoint images.
\end{thm}
We need to allow the constant $R$ to be arbitrarily large to
facilitate the proof that $\OPlv$ is locally injective.  See
Lemma~\ref{lm:acrossthin}, where the choice of $R$ is made.
\par
The location of these annuli is illustrated in the left part of
Figure~\ref{cap:plumbing}. The images of the marked points $p_e^\pm$ and $p_h$
in~$\calX$ are denoted by $b_e^\pm\in\calB^\pm$ and $b_h\in\calB_h$, respectively, for
each vertical edge or half-edge.
\index[plumb]{e070@$b_e^\pm$, $b_h$! Image of $p^{\pm}$ in $\calB^{\pm}$}
\changed{We  use the points $b_e^+$ as the nearby points used
for defining perturbed period coordinates, which were denoted by~$\sigma_e^+$
in~\eqref{eq:constdistance}.}
\par
\begin{proof}
  In the $\phi_e^+$-coordinates, the modifying differential $\bft\ast\bfxi$ becomes
  \begin{equation*}
    (\phi_e^+)^* (\bft\ast\bfxi) = \prodttop\cdot\alpha_e \frac{du}{u}\,,
  \end{equation*}
where $\alpha_e$ is a holomorphic function on the product $\calW_\epsilon\times
\Delta_\epsilon^N \times \Delta_{\delta_1}$  satisfying
\begin{equation*}
  \changed{\prodttop\cdot\alpha_e(\bfw, \bft, 0)= -r_e'(\bfw, \bft)}\quad
  \text{and}\quad \alpha_e(\bfw, {\bf 0}, z) \equiv 0.
\end{equation*}
%(Using Corollary~\ref{cor:modif}, we see that in fact $\bfxi$ depends only on the $t_i$ with $i<L$).
\par
By
Theorem~\ref{thm:deformed_standard_coordinates}, after possibly decreasing
$\epsilon$, there is a family of conformal maps
  \begin{equation*}
\psi_e \colon \calW_\epsilon \times A_{\delta/R, \delta} \to\calW_\epsilon
\times \Delta_{\delta_1}\,,
  \end{equation*}
which cover the identity map of $\calW_\epsilon$, fix the section
  $\calW_\epsilon\times \{p\}$ (where $p=\delta/\sqrt{R}$), and put $(\phi_e^+)^*(\bft\ast (\bfeta + \bfxi))$
  in the desired normal form as follows:
  \begin{eqnarray*}
(\phi_e^+\circ \psi_e)^*(\bft\ast (\bfeta + \bfxi))
      &=& \changed{\prodttop\psi_e^*\left(  z^{\kappa_e} + \alpha_e(\bfw, \bft, z)\right)\frac{dz}{z}} \,,\\
      &=& \changed{\prodttop \left(     z^{\kappa_e} - r_e'\right) \frac{dz}{z}}\, .
  \end{eqnarray*}
Since $\alpha_e$ is holomorphic,   we may in particular choose $\epsilon$ small
enough so that the zeros of the rightmost form belong to the disk of
radius $\delta/R$. Over the locus where $t_{L-1}=\dots= t_{-N}=0$, the modifying
differential~$\bfxi$ vanishes on level~$L$ (see Corollary~\ref{cor:modif}), which means $\psi_e$ preserves
the form $z^{\kappa_e}\frac{dz}{z}$ and fixes the point $b_e$, and
  thus~$\psi_e$ is the identity over this locus.   We then define
  $\upsilon_e^+ = \phi_e^+\circ \psi_e$. The desired family of disks
  $\calE_e^+$ is then $\phi_e^+(\calV_e)$, where $\calV_e$ is the
  bounded component of the complement of the family of annuli
  $\psi_e(\calW_\epsilon \times A_{\delta/R, \delta})$.
\par
The construction of $\upsilon_e^-$ is much simpler at a pole, as then we need only
to apply Theorem~\ref{thm:standard_coordinates} to construct a map
$\upsilon_e^-$ putting $\bft\ast (\bfeta  + \bfxi)$ in the normal form.
This works in a neighborhood of the node, and we may of course restrict to a family
of annuli.
\par
In the case of a half-edge, the construction of the map $\upsilon_h$ follows
from the same technique.  In this case, the modifying differential $\bfxi$ is
holomorphic along the zero section $\bfz_{h}$, so the resulting
normal form of $\upsilon_h^*(\bft\ast (\bfeta+\bfxi))$ has no residue.
\end{proof}

%%%%%%%%%%%%%%
\subsection{The vertical plumbing construction}
\label{sec:plumbing}
%%%%%%%%%%%%%%

We now present the basic plumbing construction. The plumbing starts from a
family $\OMDsfam \to \calW_\epsilon$ equipped with the family of differentials
$\bft\ast \bfeta$ (as defined in Section~\ref{sec:the_setup}) together with a \changed{chosen}
modifying differential~$\bfxi$, and builds a family of meromorphic stable differentials
$(\calY^v \to \calW_\epsilon,\omega,\bfz)$ that vanishes identically on the lower
level components over the boundary divisor and is elsewhere holomorphic and
nonzero, except for the prescribed zeros and poles $\bfz(h)$.
\par
We define conformal isomorphisms $\Upsilon_e^\pm \colon \calA_e^\pm \to
\calB_e^\pm\subset\calX$ by
\begin{equation*}
\Upsilon_e^+(\bfw, \bft, u, v) \= \upsilon_e^+(\bfw, \bft, u)
\quad\text{and}\quad
\Upsilon_e^-(\bfw, \bft, u, v) \= \upsilon_e^-(\bfw, \bft,v)\,,
\end{equation*}
where $\calB_e^\pm$ and $\upsilon_e^\pm$ are defined in Theorem~\ref{thm:gluing_maps}.
These maps identify each standard form $\Omega_e$ on the annuli with $\bft\ast (\bfeta + \bfxi)$, as desired.
By abuse of notation, we will refer to both $\calA^+_e$ and its image $\calB^+_e$ as
the top plumbing annuli, and to both $\calA^-_e$ and $\calB^-_e$
as the bottom plumbing annuli corresponding to the edge~$e$.
\index[plumb]{e040@$\calB^-$, $\calB^+$!Bottom and top plumbing annuli in~$\calX$}
\par
For each half-edge $h$, we denote the family of conformal isomorphisms provided in
Theorem~\ref{thm:gluing_maps} by
\begin{equation*}
  \Upsilon_h \= \upsilon_h\colon \calA_h \to \calB_h\subset \calX\,.
\end{equation*}
We let $\calY^v\to \calW_\epsilon$ be the family of curves obtained by removing
from~$\calX$ the families of disks $\calE_e^\pm$ and $\calE_h$ and attaching
each family $\VV_e$ and~$\DD_h$ by identifying the $\calA$-annuli to the
$\calB$-annuli via the $\Upsilon$-gluing maps.  As the gluing maps respect the one-forms,
the family $\calY^v$ inherits a relative one-form $\bfomega$.
\par
We denote the plumbing annuli as subsurfaces of $\calY^v$ by
$\calC_e^\pm$ and $\calC_h$, denote the middle annuli bounded by the
$\calC_e^\pm$ as $\calF_e$, and denote $c_e^\pm$ the images of the
points~$p^{\pm}$ in $\calC_e^\pm$.
\index[plumb]{e050@$\calC^-,\calC^{+}$!Bottom and top plumbing annuli
  in $\calY^v$} \index[plumb]{e080@$c_e^\pm$! Image of the points
  $p^{\pm}$ in $\calC_e^\pm$} These points are defined near the
corresponding vertical and horizontal nodes, but for the latter the
sign is an arbitrary auxiliary choice. The final result of plumbing is
illustrated on the right in Figure~\ref{cap:plumbing}.
\input{pic_plumbing_dot_tag}
\par
We will later use the following consequence of the construction and the fact
that the modifying differential~$\bfxi$ on level~$i$ depends only on the
levels below~$i$ and the topological data, see Corollary~\ref{cor:modif}.
\par
\begin{prop} \label{prop:annuluslocation}
\changed{Given the modifying differential~$\bfxi$,} for each edge~$e$ the location of the family
of annuli $\calB_{e}^+ \subset \calX_{\ltop}$ depends only on the
subsurfaces $(\calX_{(i)}, \eta_{(i)})$ and on the values of~$t_i$ for $i \leq \ltop$.
\end{prop}

%%%%%%%%%%%%%%%%%%%%%%%%%%%%%%%%5
\subsection{ \changed{The vertical plumbing map}} \label{sec:defPlmap}
%%%%%%%%%%%%%%%%%%%%%%%%%%%%%%%% 
We now construct the vertical plumbing map $\OPlv\colon \calW_\epsilon \to
\ODehnsv$.  \index[plumb]{a020@$\OPlv$! Vertical plumbing map}
Over the generic stratum $\calW_\epsilon^\emptyset$, as~$\omega$ does not vanish
identically on any fiber, this map arises from the standard universal
property of the stratum of marked differentials $\Oteich$.
\par
Now consider a boundary stratum $\calW_\epsilon^J$, with
$J\subset L(\Lambda)$ a nonempty set of levels, and let
$\calX^{v, J} \to \calW_\epsilon^J$ and
$\calY^{v, J} \to \calW_\epsilon^J$ be the restrictions of $\calX^v$
and $\calY^v$.  The family $\calY^v$ is then equisingular with dual
graph the enhanced multicurve $\Lambda_J\rightsquigarrow\Lambda$.  To
define the plumbing map on this stratum, we restrict appropriate
rescalings $\bfomega^J$ of $\omega$ to the irreducible components of
$\calY^{v, J}$, construct prong-matchings $\sigma^J_e$ along the
vertical nodes of $\calY^{v, J}$, and define a marking of
the welding of $\calY^{v, J}$ along these prong-matchings.
\par
We first equip each level of $\calX^J$ with the rescaled forms
$\eta_{(i)}^J$, defined by the equations, 
\begin{equation}  \label{eq:therescaledform}
\eta_{(i)}^J = \bft^\bfa_{\lceil i \rceil \setminus J}  \eta_{(i)} 
\end{equation}
where $\bft^\bfa_I = \prod_{i \in I} t_i^{a_i}$ and 
$\lceil i \rceil = \{j: j \geq i\}$, extending the definition
in~\eqref{eq:omTeta}.  We assign to each edge $e$ of $\Lambda$ the
prong-matching 
\begin{equation}
  \label{eq:sigmaeJ}
\sigma_e^J \= (\bft_{e}^J)^{-1} \, du\otimes dv, \quad\text{where}\quad \bft_{e}^J
\= \prod_{\substack{\lbot \leq k < \ltop \\ k \not\in J}} t_k^{m_{e, k}},
\end{equation}
which are easily checked to be prong-matchings for $\bfeta^J$.  These
prong-matchings define for each fiber $X_{(\bfw,\bft)}$ of
$\calX^{v, J}$ a welding $\overline{X}_{(\bfw, \bft)}$ which together 
with a marking
$f\colon(\Sigma, \bfs) \to (\overline{X}_{(\bfw, \bft)}, \bfz)$
is defined up to precomposition by $\svTw$ (see  Section~\ref{sec:universalMD}).
For each vertical edge $e$ that  does not persist in $\Lambda_J$,
define on $\VV_e$ the form
\begin{equation*}
\Omega_e^J \= \bft^\bfa_{\lceil \ltop \rceil \setminus J}\left(u^{\kappa_e} - r_e'\right)
\frac{du}{u} \=  \bft^\bfa_{\lceil \lbot \rceil \setminus J}
\left( -v^{-\kappa_e} + r_e'\right) \frac{dv}{v}\,. 
\end{equation*}
\par
The gluing maps for such
edges identify these rescaled forms with $\bfeta^J$, which defines on
$\calY^{v, J}$ a collection of rescaled forms $\bfomega^J$ that do not
vanish identically on any fiber.  For the edges $e$ that do persist
in $\Lambda_J$, the prong-matchings $\sigma_e^J$ from
\eqref{eq:sigmaeJ} may also be regarded as prong-matchings for
$\bfomega^J$, which defines for each fiber $Y_{(\bfw, \bft)}$ of
$\calY^{v, J}$ a welding $\overline{Y}_{(\bfw, \bft)}$.
\par
We  define markings for the fibers of $\calY^{v, J}$ by
transporting the markings of the fibers of $\calX$ by \diffeogens
\be \label{eq:fJ}
f^J_{(\bfw, \bft)} \colon \overline{Y}_{(\bfw, \bft)}\to
\overline{X}_{(\bfw, \bft)}
\ee
between the corresponding welded
surfaces. We define these maps in Section~\ref{sec:mpa} below.
\par
These forms, prong-matchings, and markings give each fiber of
$\calY^{v, J}$ the structure of a marked multi-scale differential.
Applying the standard universal property of the boundary stratum
$\msT[\Lambda_J]$ to the family $(\calY^{v, J}\to \calW_\epsilon^J,
\bfomega^J, \bfsigma^J)$ defines a plumbing map $\OPlv$ on $\calW_\epsilon^J$.

%%%%%%%%%%%%%%%%%%%%%%%%
\subsection{\changed{Mapping plumbing fixtures}}
\label{sec:mpa}
%%%%%%%%%%%%%%%%%%%%%%%%

In this section, we define continuous maps of families $f^J\colon \calY^{v, J}
\to \overline{\calX}^{v,J}$ whose restrictions to the fiber over~$(\bfw, \bft)$
give the desired maps in~\eqref{eq:fJ}.  These maps complete the
construction of the vertical plumbing map $\OPlv$ and have further properties
which will be needed below in proving that $\OPlv$ is a local homeomorphism.  
\par
We will define these maps $f^J$ by defining maps $f_e^J\colon \VV_e^J\to
\overline{\calX}^{v,J}$ on each plumbing fixture, which agree with the gluing
maps $\Upsilon_e^\pm$ on its boundaries.  On subannuli $\AsharpJ_{e}$ and $\AflatJ_{e}$
(see~\eqref{eq:fiveannuli} below) 
of modulus roughly a quarter of the modulus of~$\VV_e$, we define~$f^J$ to
be the standard coordinates~$\phi^\pm$ for~$\bfeta$, defined
in~\eqref{eq:standard_phi}.  When the modulus of~$\VV_e$ is large, these
maps~$\Upsilon_e^\pm$ and~$\phi_e^\pm$ will nearly agree, and we interpolate
between them by nearly conformal maps on annuli $\Adoublesharp_{e}$
and~$\Adoubleflat_{e}$, using the straight line extension, defined below.
The remainder of $\VV_e$ is an annulus~$\Anatural_{e}$, which we map
to~$\overline{\calX}^{v,J}$ by a smooth map. We now define these maps formally.
\par
Consider the round annulus $A = \{r_1 < |z| < r_2\}\subset\cx^*$ for
$0<r_1<r_2<\infty$, which we equip with the flat metric $|du|/u$,
together with a given $C^1$ map $\alpha\colon \bdry A \to \cx^*$.  For~$\alpha$
sufficiently close to the identity embedding we define
the \emph{straight line extension} of $\alpha$ to be the map
$F\colon A \to \cx^*$ sending the radial geodesic from $r_1 e^{i \theta}$
to $r_2 e^{i\theta}$ to the geodesic joining $\alpha^-(r_1 e^{i \theta})$
to $\alpha^+(r_2 e^{i \theta})$, taking the parameterization proportional
to arclength, where~$\alpha^-$ and~$\alpha^+$ denote the restrictions
of~$\alpha$ to the inner and outer boundaries of~$A$.
\par
\begin{lm} \label{lm:SLE}
Suppose $\alpha_n\colon \bdry A \to \cx^*$ is a sequence of $C^1$
maps converging $C^1$-uniformly to the identity embedding, and let
$F_n\colon A \to \cx^*$ be their straight line extensions.  Then
the $F_n$ also converge $C^1$-uniformly to the identity, and are
eventually $C^1$ diffeomorphisms onto their image. In particular, $F_n$
is $K_n$-quasiconformal with~$K_n \to 1$. 
\end{lm}
\par
\begin{proof}
We identify~$A$ with the annulus $B = \{0 < \operatorname{Im} z < h\}
\subset \cx / \zed$ by using logarithmic coordinates.
In these coordinates, the straight line extension is simply
\begin{equation*}
  F_n(x+i y) \= \alpha_n^-(x)(1-y/h) + \alpha_n^+(x+ i h) y/h,
\end{equation*}
and we compute
\begin{equation*}
  F_n(x+ iy) - (x+iy) \= (\alpha_n^-(x) - x)(1 - y/h) + (\alpha_n^+(x +ih)
  - (x+ ih))y/h,
\end{equation*}
which easily implies the second claim.  It follows that if $\alpha_n$
is sufficiently close to the identity, the Jacobian determinant of
$F_n$ is nowhere vanishing, so $F_n$ is a $C^1$ diffeomorphism onto
its image.
\end{proof}
\par
We now construct the maps $f^J\colon \calY^{v, J}\to \overline{\calX}^{v,J}$,
following the idea outlined above. 
%It suffices to construct for each edge $e$ which does not persist in
%$\Lambda_J$ a continuous map $f_e^J\colon \VV_e^J\to
%\overline{\calX}^{v,J}$ which agrees with the gluing maps $\Upsilon_e^\pm$
%on the plumbing annuli $\calA_e^\pm$.  
To construct~$f_e^J$, first subdivide the trimmed plumbing fixture
$\VV_e^J\setminus \bigcup_e\calA_e^\pm$ into five annuli:
\ba \label{eq:fiveannuli}
  \AdoublesharpJ_{e} &\= \{\delta/R^2 < |u| < \delta/R\}\\
  \AsharpJ_{e} &\= \{|\bft_e^J|^{1/4}\delta^{1/2}/R < |u| <\delta/R^2\}\\
  \AnaturalJ_{e}&\= \{ |\bft_e^J|^{3/4}R/\delta^{1/2} < |u| < |\bft_e^J|^{1/4}\delta^{1/2}/R \}\\
      &\= \{ |\bft_e^J|^{3/4}R/\delta^{1/2} < |v| < |\bft_e^J|^{1/4}\delta^{1/2}/R \}\\
  \AflatJ_{e} &\= \{|\bft_e^J|^{1/4}\delta^{1/2}/R < |v| <\delta/R^2\}\\
  \AdoubleflatJ_{e} &\= \{\delta/R^2 < |v| < \delta/R\}
\ea
On the upper plumbing annulus $\mathcal{A}_e^+=\{\delta/R< |u|  <\delta\}$, we define
$f_e^J= \Upsilon_e^+$, and similarly on the lower annulus
$\mathcal{A}_e^-=\{\delta < |v| < \delta/R\}$, we define $f_e^J= \Upsilon_e^-$.
On $\AsharpJ_{e}$ and $\AflatJ_{e}$, we define
$f_e^J(u,v)=\phi_e^+(u)$ and $f_e^J(u,v)=\phi_e^-(v)$, respectively.
This defines $f_e^J$ on both boundaries of each of $\AdoublesharpJ_{e}$ and
$\AdoubleflatJ_{e}$, and we extend it across these annuli by the
straight line extension.  Finally, we define $f_e^J$ on $\AnaturalJ$
to be a smooth map sending radial curves (with $\arg u$ and $\arg v$
constant) to radial curves in the standard coordinates near the node
$q_e$ of 
$\overline{\calX}^{v,J}$, and sending the circle $\{|u|=|v|=
|\bft_e^J|^{1/2}\}$ to the seam.  The prong-matchings of $\calX^{v,J}$
have been chosen so that this is continuous along the seam.
\par
We write $f_{e,(\bfw, \bft)}^J$ for the restriction to the fiber
$X_{(\bfw, \bft)}$ of $\calW_\epsilon^J$.  We summarize below the
relevant properties of these maps.
\begin{prop}
\label{prop:properties_of_extension}
Let $(\bfw_n, \bft_n)$ be a sequence in $\calW_\epsilon^{J'}$ converging to
$(\bfw, \bft) \in \calW_\epsilon^{J}$ for some set of levels $J' \subset J$. Denote
by $g_n\colon X_{(\bfw_n, \bft_n)} \to
X_{(\bfw, \bft)}$ the maps on fibers exhibiting this convergence, as in Proposition~\ref{prop:MDsconv} , and let
$h_n= g_n\circ f_{e,(\bfw_n, \bft_n)}.$  Let $e\in E(\Lambda)$ be an edge joining levels $i>j$ that does
not persist in $\Lambda_J$. Then: 
\begin{itemize}
\item The map $h_n^{-1}$ is $K_n$-quasiconformal, with $K_n\to 1$,
  on an exhaustion of a neighborhood of the node $q_e$.
\item The rescaled differentials $\bft^{-\bfa}_{n,\lceil \ltop \rceil \setminus J}(h_n)_*
  \Omega_e^J $ and $\bft^{-\bfa}_{n,\lceil \lbot \rceil \setminus J}(h_n)_* \Omega_e^J$
converge in the weak locally $L^2$ topology to $\eta_{(i)}$ and $\eta_{(j)}$.
\item For any radial curve $\gamma$ on $\overline{X}_{(\bfw, \bft)}$ whose endpoints
belong to $\calA_e^\pm$, the turning numbers $\tau(h_n^{-1}(\gamma))$ converge to $0$.
\end{itemize}
\end{prop}
\par
\begin{proof}
By Theorem~\ref{thm:gluing_maps}, the gluing maps $\Upsilon_e^\pm$
converge to the standard coordinates $\phi_e^\pm$ as $\bft_e^{J'}\to 0$.
Working in the $\phi_e^\pm$-coordinates, Lemma~\ref{lm:SLE} then
implies that the map $f_{e,(\bfw_n, \bft_n)}^{J'}$ is~$C^1$ on
$\AdoublesharpJ_{e}$ and $\AdoubleflatJ_{e}$, and converges in $C^1$ to
the identity on these annuli, as $n\to\infty$.  Since $f_{e,(\bfw_n, \bft_n)}^J$ is conformal on
$\AsharpJp_{e}$ and $\AflatJp_{e}$, it is $K_n$-quasiconformal on 
$\VV_e^J\setminus \AnaturalJp_{e}$ with $K_n\to 1$.  The annuli
$h_n(\AsharpJp_{e})$ and $h_n(\AflatJp_{e})$ then have moduli tending
to infinity, so their images form an exhaustion of a neighborhood of
the node, on which $h_n^{-1}$ is $K_n$-quasiconformal.
\par
Convergence of the forms follows from the $C^1$ convergence of the
maps on the interior of each of the five annuli defined in \eqref{eq:fiveannuli}.
\par
Recall from Section~\ref{sec:turning_Jordan} that the turning number of a Jordan arc is
  determined by the value of the Gauss map on its end-vectors together
  with its homotopy class.  Convergence of the Gauss maps on the
  end-vectors of $h_n^{-1}(\gamma)$ follows from the above convergence
  of forms, which is uniform on the plumbing annuli $\calA_e^\pm$.
  The curves $h_n^{-1}(\gamma)$ are nearly radial, so the turning
  numbers converge.
\end{proof}

%%%%%%%%%%%%%%%%%%%%%%%%%%%%%%%%
\subsection{Plumbing is a local homeomorphism}
\label{sec:plhomeo}
%%%%%%%%%%%%%%%%%%%%%%%%%%%%%%%%

Our next goal is to establish that the  vertical plumbing
map is a local homeomorphism, and we start by proving its continuity.

\begin{prop} \label{prop:Plcont}
The vertical plumbing map $\OPlv\colon \calW_\epsilon \to \ODehnsv$ constructed
above is continuous.
\end{prop}
\par
\begin{proof}
We fix a sequence $(X_n, \bfeta_n, \bft_n)$ in $\Omega\barMDs$ that
converges to $(X,\bfeta,\bft)\in \calW_\epsilon^J$ and let
$(Y_n, \bfomega_n)$ and $(Y, \bfomega)$ denote the corresponding
plumbed surfaces.  Passing to a subsequence, we may assume this
sequence belongs to a fixed stratum of $\Omega\barMDs$.  We define
rescaled differentials $\bfeta^J$ and $\bfeta_n^J$ as in
\eqref{eq:therescaledform} as well as prong-matchings $\bfsigma^J$
and $\bfsigma_n^J$ as in \eqref{eq:sigmaeJ}.  Let $\overline{X}$ and
$\overline{X}_n$ be the weldings defined by these prong-matchings,
and denote by $g_n\colon \overline{X}_n \to \overline{X}$ the maps
exhibiting the convergence as in Proposition~\ref{prop:MDsconv}.
This means that there are subsurfaces
$X^\circ_n\subset X_n \setminus \bfz_n$ such that~$g_n$ is conformal
on~$X^\circ_n$, the images $g_n(X^\circ_n)$ exhaust $X\setminus \bfz$, and
$(g_n)_*\bfeta_n^J\to \bfeta^J$ uniformly on compact sets.
Moreover, the $X^\circ_n$ eventually contain the plumbing annuli
$\calB_e^\pm$ because these annuli vary continuously in the
universal curve, and $(g_n)^{-1}$ converges uniformly on compact sets to
the identity map on~$X$ by Lemma~\ref{lm:converge_to_identity}.
\par
As defined in~\eqref{eq:fJ}, we have \diffeogens
$f_n \colon \overline{Y}_n \to \overline{X}_n$ and
$f \colon \overline{Y} \to \overline{X}$ which are conformal on the
complements~$Y_n'$ and~$Y'$ of the trimmed plumbing fixtures. We define
\changed{$\widetilde{h}_n=f^{-1} \circ g_n\circ f_n\colon \ol{Y}_n \to \ol{Y}$.}
% $h_n=f\circ g_n\circ f_n^{-1}\colon Y_n \to Y$.
These maps satisfy all of the requirements to exhibit convergence
from Definition~\ref{def:augmented_topology_def}, except that $\widetilde{h}_n^{-1}$ is
conformal only on~$Y'$.  \changed{We address this by modifying~$\widetilde{h}_n$ on the
trimmed plumbing fixtures to obtain homotopic maps ${h}_n$ such that the~$h_n^{-1}$ are
$K_n$-quasiconformal with $K_n\to 1$ on an exhaustion of~$\ol{Y}$. This suffices thanks to
the modifications allowed by Proposition~\ref{prop:weak_convergence}.}
\par
\changed{On the trimmed plumbing fixtures of $\ol{Y}_n$ corresponding to nodes~of~$X$
where~$Y$ has been obtained by plumbing, we define ${h}_n$ as the straight line
extension of the restriction of~$\widetilde{h}_n$ to the boundary of~$Y_n'$ (which is
also the boundary of the trimmed plumbing fixture).  Since
$(\widetilde{h}_n)_*\bfomega_n\to \bfomega$ on $Y'$, Lemma~\ref{lm:SLE} implies
that on the whole trimmed plumbing fixture $\tilde{h}_n$ is $K_n$-quasiconformal with
$K_n\to 1$, and $(\tilde{h}_n)_*\bfomega_n\to \bfomega$.}
\par
\changed{On the trimmed plumbing fixtures of~$\ol{Y}_n$ corresponding to nodes of~$X$
where~$Y$ has \emph{not} been plumbed, we define~${h}_n$ to be the map~$h_e^J$
defined in Proposition~\ref{prop:properties_of_extension} (composed with the
identification of~$X$ with~$Y$ in a neighborhood of this node). This proposition
then gives the necessary properties for convergence.}
\end{proof}
\par
\begin{prop}\label{prop:plumbopen}
%The vertical plumbing map $\OPlv\colon\changed{\calW_\ve}\to\ODehnsv$
%is open at any point in the deepest stratum $\OMDstratum[\Lambda][\Lambda,s]$.
%\par
%  \Matt{I rewrote the statement of
%    this Proposition because it seems to claim more than we actually
%    prove.  A map $f$ is open at $x$ if every neighborhood of $x$ is
%    sent to a neighborhood of $f(x)$.  See
%    \url{https://en.wikipedia.org/wiki/Almost_open_map}.  Saying $f$
%    is open on a neighborhood of $x$ is a stronger statement.  Delete
%    the old statement below if this is ok.}
%
The vertical plumbing map $\OPlv\colon\changed{\calW_\ve}\to\ODehnsv$ is
open in a neighborhood of any point in the deepest stratum
$\OMDstratum[\Lambda][\Lambda,s]$.
\end{prop}
\par
For a surface~$X$ that corresponds to a point in the
subset~$\calW_\ve$ of the model domain we denote by $X_{(\leq i)}^-$
the subsurface consisting of the levels $\leq i$, including the
plumbing annuli~$\calB_e^-$ for all $e$ with $\lbot \leq i$
and~$\calB_h$ for $h$ with $\ell(h) \leq i$, but excluding
the discs $\calE_e^-$ for all $e$ with $\ltop >i$ and~$\calE_h$ for
$h$ with $\ell(h)=i$.  We
let $X_{(\leq i)}^+$ be the subsurface consisting of the levels
$\leq i$, including all the plumbing fixtures connecting to higher
levels all the way up to the top plumbing annuli $\calB_e^+$ for
$\ltop >i$.
\par
In the proof below, we will work with sequences of model differentials
$X_n\to X$ and adopt the convention that for any object defined in the
plumbing construction for $X$, the analogous object for $X_n$ includes
a subscript $n$, for example the basepoints $b_e^+\in X$ defined above
are denoted by $b_{n,e}^+=\phi_{n,e}(\delta/\sqrt{R})\in B_{n,e}^+ \subset X_n$.
\par 
\begin{proof}
\changed{To show~$\OPlv$ is open, it suffices to construct, for any convergent
sequence $(Y_n, \bfy_n, \bfomega_n)\to (Y,\bfy, \bfomega)$ in~ $\ODehnsv$, a sequence
$(X_n, \bfz_n,\bfeta_n, \bft_n)\to (X, \bfz, \bfeta, \bft)$ in~$\barOMDs$ such that
$\OPlv(X_n, \bfz_n, \bfeta_n, \bft_n) = (Y_n, \bfy_n,\bfomega_n)$
and $\OPlv(X, \bfz, \bfeta, \bft)=(Y, \bfy, \bfomega)$. We deal only with the case
that~$Y_n$ is smooth and $(X, \bfz, \bfeta, \bft)$ in the deepest stratum,
equivalently, $\bft=\bfzero$ . The general case is similar but notationally more
involved, since one needs to carry around notation for the edges that persist in~$Y_n$
and those that do not degenerate in~$X$.}
\par
We can choose representatives of $(X, \bfeta)$ in $\ptwT$ and
$(Y_n, \bfomega_n)$ in $\Oaugteich$ so that convergence still holds.
By the definition of convergence in $\Oaugteich$ as given in
Section~\ref{sec:AugTeich}, there is a sequence
$\bfd_n=\{d_{n, i}\} \in \CC^{L^\bullet(\Lambda)}$ and a sequence of
\diffeogens $g_n\colon Y_n \to \overline{X}$, 
\changed{which respect the marked points,} which are compatible with the markings,
asymptotically \tnp, \changed{and
whose inverses are conformal on an exhaustion~$K_{n, (i)}$ of each
level~$X_{(i)}$ minus the nodes.}  
\par
We start by choosing the sequence of coordinates~$\bft_n$ defined in
terms of these $d_{n, i}$ by
\begin{equation}\label{eq:deftfromd}
  t_{n,i} \= \bfe\left(\frac{d_{n,i+1} - d_{n,i}}{a_i}\right)\,.
\end{equation}
Similarly to $\prodt$ defined in Equation~\eqref{eq:omTeta}, we
denote  $\bft^\bfa_{n, \lceil i \rceil}= \prod_{j\geq i} t_{n,j}^{a_j}$.
Since we are not rescaling the top level, we have $d_{n,0}=0$ and it follows
that $\bft^\bfa_{n, \lceil i \rceil}= \bfe(-d_{n, i})$. By definition of convergence,
\begin{equation}\label{eq:1}
\bfe( d_{n, i}) (g_{n, (i)})_* \omega_n  \xrightarrow[n\to \infty]{}
\eta_{(i)}, \quad \text{that is,} \quad \frac{1}{\bft^\bfa_{n, \lceil i \rceil}}
(g_{n, (i)})_* \omega_n  \xrightarrow[n\to \infty]{} \eta_{(i)}.
\end{equation}
\changed{Recall from Section~\ref{sec:perturbed} that we normalize
the forms $\bfeta$ and parameters $\bft$ so that on each level,
$\int_{\gamma_{n(i)}}(\eta_{(i)} + \xi_{(i)})=1$.  We may assume
that the $d_{n,i}$ are chosen so that we have the corresponding
normalization, namely
\begin{equation}
  \label{eq:normalization}
  \int_{\gamma_{n(i)}}\frac{1}{\bft^\bfa_{n, \lceil i \rceil}}
(g_{n, (i)})_* \omega_n \=1\, .
\end{equation}
} 
\par
We will now construct \changed{by induction on the set of levels,
  starting from the bottom level,} the $(X_n, \bfeta_n)$ such that
$\OPlv (X_n, \bfeta_n,\bft_n) = (Y_n, \bfomega_n)$.  Recall that by
Proposition~\ref{prop:annuluslocation} it makes sense to consider
the effect of plumbing only on the bottom part of a surface and to
write $\OPlv(X_{n, (\leq i)}^\pm)$, suppressing the dependence on
$(\bfeta_n,\bft_n)$ for notational convenience. Note also that the
set of connected components of the $\ve_n$-thick part of~$Y_n$ is
eventually the disjoint union of sets of level~$i$
components~$Y_{n,(i)}$, where~$Y_{n,(i)}$ are those components
that~$g_n$ maps to $X_{n,(i)}$.  \changed{Alongside choosing the
surfaces, we inductively construct conformal maps
$h_{n, (i)}^+\colon\OPlv(X_{n, (\leq i)}^+)\to Y_n$ which identify
$\bfomega_n$ with the forms on the plumbed surface.  At the end of
the induction, we obtain the sequence
$(X_n, \bfeta_n)\to (X, \bfeta)$ together with a sequence of
isomorphisms $h_n \colon \OPlv(X_{n})\to Y_n$.}
\par
The base case of induction is to pick appropriately the surfaces with the
correct bottom level piece $(X_{n,(-N)}^-,\eta_{n,(-N)})\to(X_{(-N)}, \eta_{(-N)})$
among all surfaces parameterized by~$\calW_\ve$ and to construct a conformal
map on the bottom level
\bes h_{n, (-N)}\colon\OPlv(X_{n, (-N)}^-) \to Y_{n,(-N)}
\ees
which identifies the two differentials.   The conformal map $g_{n, (-N)}^{-1}$ is
eventually defined on a fixed subsurface $K_{(-N)}$ containing $X_{(-N)}^-$, but
this map only approximately identifies the rescaled differentials, as is indicated
in~\eqref{eq:1}.  We choose a sequence of surfaces
$(X_{n,(-N)}, \eta_{n, (-N)})$ of the same topological type as
$(X_{(-N)}, \eta_{(-N)})$ such that
\begin{equation}
  \label{eq:period}
  \Per(X_{n, (-N)}, \eta_{n, (-N)})
  \= \frac{1}{\prodt[n, \changed{\lceil-N\rceil}]} \Per_{(-N)}(Y_n,
  \bfomega_n)
\end{equation}
and such that $(X_{n, (-N)}, \eta_{n, (-N)}) $ converges to $ (X_{(-N)}, \eta_{(-N)})$.
In this equation $\Per_{(i)}\subset H^1(\Sigma_i \setminus P_i, Z_i; \CC)$ denotes
the relative periods in the level~$i$ subsurface.  We may choose
such a sequence because the period map is open.  By convergence in
the \changed{strong} conformal topology, there exist maps
$(g_{n, (-N)}^X) \colon X_{n,(-N)} \to X_{(-N)}$, \changed{respecting the
marked zeros and poles,}  whose inverses are conformal on the same subsurface.
We apply Theorem~\ref{thm:conformal_convergence} to the sequences
$(g_{n, (-N)}^X)_* (\bfeta_n)$ and $(\prodt[n, \lceil-N\rceil])^{-1}(g_{n, (-N)})_*
\bfomega_n$ to produce a conformal map~$k_n$ defined on~$K_{(-N)}$ which
identifies these forms.  The composition
\begin{equation*}
  h_{n,(-N)} \= g_{n, (-N)}^{-1}\circ k_n \circ g_{n, (-N)}^X
\end{equation*}
provides eventually a conformal map
\begin{equation*}
  h_{n,(-N)}\colon X_{n,
    (-N)}^- \to Y_n \quad \text{such that} \quad h_{n,(-N)}^*
  \bfomega_n \= \prodt[n, \changed{\lceil-N\rceil}] \eta_{n, (-N)}.
\end{equation*}
\par
The second step of the base case is to extend this map by analytic
continuation across the plumbing fixtures to obtain a conformal map
$h_{n, (-N)}^+\colon \OPlv(X_{n, (-N)}^+) \to Y_{n,(-N)}$.
For this analytic continuation, recall that the plumbed surface is
obtained by gluing into $X_n$ for all vertical nodes~$e$ the
plumbing fixtures $\VV_{n,e} \cong V(\bft_{n,e})$ equipped with the
standard form $\Omega_{n,e}$, where $\bft_{n,e}$ and $\Omega_{n,e}$
are as in~\eqref{eq:plumbing_fixture} and \eqref{eq:reloneform}.
At this point, the gluing map on the bottom plumbing annulus is
known, as we have chosen the lower level surface, and the gluing on
the top annulus will be known when we have chosen the upper level
surface $X_{n, (\ltop)}$.  By construction, the composition
\begin{equation}
  \label{eq:defnu}
  \nu_{n,e} \= h_{n, (-N)} \circ \upsilon_{n,e}^- \colon
  A_{\delta/R, R} \to Y_n
\end{equation}
(where $\upsilon_{n,e}^-$ was defined in
Theorem~\ref{thm:gluing_maps}) identifies the form $\bfomega_n$ on
$Y_n$ with the standard form $\Omega_{n,e}$ on the bottom plumbing
annulus. We show in Lemma~\ref{lm:acrossthin} below that for some
sufficiently large $R>0$ the maps $\nu_{n,e}$ can eventually be
analytically continued to a conformal map
$\nu_{n,e}\colon V_{n,e}' \to Y_n$, where $V'_{n,e}\subset V_{n,e}$
is a round subannulus containing the basepoint $p_{n,e}^+\in V_{n,e}$
whose $u$-coordinate is
$\delta/R^{1/2}$.   \emph{A fortiori} the analytic
continuation also identifies $\bfomega_n$ with $\Omega_{n,e}$.
We define basepoints $c_{n,e}^+ = \nu_{n,e}(p_{n,e}^+)$ in $Y_n$, for
those edges where $\nu_{n,e}$ has been defined.  
\par
%  For  the corresponding marked  of $Y_n$, we define $c_{n,h} =
%  g_n^{-1}(b_{n})$.  We will refer to these basepoints $b_{n,e}$ and
%  $b_{n,h}$ of $X_n$ collectively as the basepoints of $X_n$, and
%  similarly for the basepoints of $Y_n$.
For the inductive step, suppose that we have defined conformal maps
\bes h_{n, (i-1)}^+\colon\OPlv(X_{n, (\leq i-1)}^+) \to Y_{n,(\leq i-1)}\ees
which identify the forms on either side.  We also assume that we have
defined base points~$c_{n,e}^+$ near each of the lower boundary components
of $Y_{n, (i)}$ as in the second step of the base case.
\par
\changed{We will also need analogous basepoints $c_{n,h}$ near each marked
zero~$y_{n,h}$ of $Y_{n}$.  We define $c_{n,h}=\phi_{n,h}^Y(\delta/R^{1/2})$,
where~$\phi_{n,h}^Y$ is a choice (amounting to a choice of a prong at~$y_{n,h}$)
of standard coordinate sending~$0$ to the zero~$y_h$.  We choose~$\phi_{n,h}^Y$ by
requiring that $g_n\circ\phi_{n,h}^Y(c_{n,h})\to b_h$.}
\par
We wish to choose appropriately $(X_{n, (i)}, \eta_{n, (i)})$ and construct
conformal maps
\begin{equation*}
   h_{n, (i)}\colon \OPlv\left(X_{n, (\leq
    i)}^-, \eta_{n,(\leq i)}, \bft_n\right) \to Y_{n,(\leq i)}
\end{equation*}
which identify the forms and agree with $h_{n, (i-1)}^+$ on the
top plumbing annuli. We choose the
sequence~$X_{n, (i)}$ so that the perturbed period coordinates
satisfy 
\begin{equation}\label{eq:Pperiod}
  \PPer(X_{n, (i)}, \eta_{n, (i)} \changed{+ \xi_{n, (i)}}) \= \frac{1}{\prodt[n, \changed{\lceil i\rceil}]}
  \PPer(Y_{n,(i)}, \omega_n) \,
\end{equation}
using the basepoints $b_{n,e}^+$ and $b_{n,h}$ as the ``nearby
sections'' $\sigma_e^+$ in the definition of the perturbed period
coordinates.  The $\PPer$ on the right-hand side is a shorthand to
express that we compute on $Y_n$ periods in the same way as in the
definition of $\PPer$, i.e.\ we consider the surface cut open at the
lower ends and use integration based at the basepoints $c_{n,e}^+$
determined by the induction hypothesis instead of the $b_{n,e}^+$,
and similarly for the $c_{n,h}$.  Note that, while the perturbed
period coordinates do not specify the period over $\gamma_{n(i)}$, the
normalization \eqref{eq:normalization} ensures these periods agree
as well.  The choice of $(X_{n, (i)}, \eta_{n, (i)})$ with the
required perturbed periods is eventually possible, since the
perturbed period map is open by Proposition~\ref{prop:pertper}.
\par
\changed{ We would like to ensure that the conformal maps~$g_{n,(i)}$
and~$g_{n,(i)}^X$ respect the basepoints, that is, $g_{n,(i)}(c_{n,e}^+)
= g_{n,(i)}^X(b_{n,e}^+)=b_e^+$.  We first show that $g_{n,(i)}(c_{n,e}^+)$
and $g_{n,(i)}^X(b_{n,e}^+)$ both converge to~$b_{e}^+$.  Convergence of
$g_{n,(i)}^X(b_{n,e}^+)$ follows from continuity of the coordinates~$\phi_{n,e}^+$
and the convergence of the maps~$g_{n,(i)}^X$ to the identity map on~$X$.  Let
$f_{n,e}\colon V_{n,e}\to \overline{X}_n$ be the maps constructed in
Section~\ref{sec:mpa}.  The maps $\rho_{n,e} \coloneqq g_{n, (j)}\circ \nu_{n,e}
\circ f_{n,e}^{-1}\circ (g_{n, (i)}^X)^{-1}$ converge, in a neighborhood of the
node~$q_e$ of~$X$, to an automorphism of the one-form $\eta_{(i)}$, which must
be a rotation~$\rho_e$ of order~$\kappa_e$.  Note that
$\rho_{n,e}\circ g_{n, (i)}^X(b_{n,e}^+)= g_{n,(j)}(c_{n,e}^+)$, so to
show that $g_{n,(j)}(c_{n,e}^+)\to b_e^+$, it suffices to show that~$\rho_e$
is the identity.}
\par
\changed{On the lower level, we have similarly that $g_{n, (j)}\circ \nu_{n,e}
\circ f_{n,e}^{-1}\circ (g_{n,(j)}^X)^{-1}$ converges to an automorphism
of~$\eta_{(j)}$ which must in fact be the identity because by its construction,
it extends to an automorphism of~$X_{(j)}$ respecting its marking.  The compositions
$g_{n}\circ \nu_{n,e} \circ f_{n,e}^{-1}\circ (g_{n}^X)^{-1}$ are
moreover asymptotically \tnp, since the~$g_n$ and~$f_{n}$ are by
Proposition~\ref{prop:properties_of_extension}, and~$\nu_{n,e}$
are \tnp, since they respect the one-forms.  It
follows that $\rho_e=1$, so $g_{n,(j)}(c_{n,e}^+)\to b_e^+$, as desired.}
\par
\changed{We now modify the conformal maps $g_{n,(i)}$, using the strong
conformal topology from Section~\ref{sec:dehn-space-deligne}, to
obtain nearby conformal maps $\tilde{g}_{n,(i)}$, such that
$\tilde{g}_{n,(i)}(c_{n,e}^+)=b_e^+$ for each~$e$ and similarly
$\tilde{g}_{n,(i)}(c_{n,h})=b_h$ for each~$h$ corresponding to a
marked zero. To this end, let $D_e\subset X_{(i)}$ be small
embedded disks (with respect to the flat metric) centered at $b_e^+$.
Since the $g_{n,(i)}(c_{n,e}^+)$ converge to $b_e^+$, there is a
sequence $j_n\colon X_{(i)}\to X_{(i)}$ of $K_n$-quasiconformal maps
equal to the identity outside the disks
$D_e$, with $K_n\to 1$, and with
$j_n\circ g_{n,(i)}(c_{n,e}^+)=b_e^+$ for each $e$, and similarly for
the $c_{n,h}$.  For example, $j_n$
could be defined by mapping each geodesic segment from $c_{n,e}^+$ to
$p\in \bdry D_e$ to the radial segment from $b_e^+$ to $p$.  Regarding
the $c_{n,e}^+$ and $b_e^+$ as additional marked points, the maps
$j_n\circ g_{n,(i)}$ exhibit convergence of the $Y_n$ to $X$ in the
quasiconformal topology, and
Theorem~\ref{thm:topologiesarethesame} produces the desired conformal
maps $\tilde{g}_{n,(i)}$.}
\par
\changed{As in the base case, we now apply Theorem~\ref{thm:conformal_convergence}
to the sequences $(g_{n, (i)}^X)_* (\eta_{n, (i)}+ \xi_{n, (i)})$ and
$(\prodt[n, \lceil i\rceil])^{-1}(g_{n, (i)})_* \bfomega_n$ to produce a
conformal map~$k_n$ defined on~$K_{(i)}$ which identifies these forms and
fixes the basepoints~$b_{e}$ and~$b_{h}$.  The composition
  \begin{equation*}
    h_{n,(i)} \= g_{n, (i)}^{-1}\circ k_n \circ g_{n, (i)}^X
  \end{equation*}
  provides eventually a conformal map
  \begin{equation*}
    h_{n,(i)}\colon X_{n,(i)}^- \to Y_n \quad \text{such that} \quad h_{n,(i)}^*
    \bfomega_n \= \prodt[n, \lceil i\rceil] (\eta_{n, (i)}+\xi_{n, (i)}).
  \end{equation*}}
\par
\changed{We next show that the $h_{n,(i)}$ 
agree with the already constructed maps $h^+_{n, (j)}$ for any $j< i$
on any plumbing annulus $\mathcal{B}_e^+$ which joins them.  Since
these maps pull back $\eta_n$ to the same form, it suffices to show
that they agree at the basepoints $b_{n,e}^+$.  We check,
\ba    \label{eq:hrespectsbasepoints}
    h_{n, (i)}(b_{n, e}^+) &\=  g_{n, (i)}^{-1}\circ k_n \circ g_{n,
(i)}^X(b_{n,e}^+) \= g_{n, (i)}^{-1}\circ k_n(b_e^+) \\
& \= g_{n,(i)}^{-1}(b_e^+)\= c_{n,e}^+ \= h_{n, (j)}^+(b_{n,e}^+). 
\ea
The same argument shows that $h_{n, (i)}(b_{n,h}) = c_{n,h}$.}
\par
\changed{We finish the construction of $h_{n,(i)}$ by extending it over the
zeros which have been plumbed while constructing $\OPlv(X_{n, (i)})$.  For this,
we note that $h_{n,(i)} \circ \phi_h(\delta/R^{1/2}) = c_{n,h}=
\psi_n(\delta/R^{1/2})$, so we must have $h_{n,(i)} \circ \phi_h =
\psi_n$, since they agree at a point and identify the forms on their
respective targets with the standard form $z^m\, dz$ on
$A_{\delta/R, \delta}$.  It follows that we may extend $h_{n, (i)}$
over the plumbed disk by defining $h_{n, (i)} = \psi_n$ in the
standard coordinates on this disk.}
\par
\changed{If $i=0$, this finishes the induction.  Otherwise, for the second
part of the inductive step, we define as in \eqref{eq:defnu}
\begin{equation}
  \label{eq:defnu2}
  \nu_{n,e} \= h_{n, (i)} \circ \upsilon_{n,e}^- \colon A_{\delta/R, R} \to Y_n
\end{equation}
for each edge $e$ joining level $i$ to a higher level, and we define
the basepoints $c_{n,e}^+= \nu_{n,e}(p_{n,e}^+)$.}
% To specify~$\overline{X}_n$ as a marked surface, we define a marking
% $\tilde{f}_n$ of $\overline{X}_n$ by composing the marking of $Y_n$
% with the appropriate map $f_n\colon Y_n\to\overline{X}_n$ from
% Proposition~\ref{prop:ft}, which is defined when $Y_n$ is sufficiently
% close to $\overline{X}$ as an unmarked \msd, \changed{where the convergence of $Y_n$ to $\overline{X}$ 
% is measured in the context of Definition~\ref{df:msds}.}
%
% To complete the proof that $\overline{X}_n\to \overline{X}$ as model
% differentials, we need to check that the sequence of markings
% $\tilde{f}_n$ is asymptotically turning number preserving.  Since the
% sequence of markings of the $Y_n$ is asymptotically turning number preserving by
% definition, and the sequence $f_n$ is as well by Proposition~\ref{prop:ft},
% so is their composition~$\tilde{f}_n$.
\end{proof}
\par
In order to conclude the proof of the openness of the plumbing map,
it remains to justify the extension of the conformal map  across the thin part.
\par
Given $(Y_n, \omega_n)\to(X, \bfeta)$ as in the above proof, we
consider a fixed level~$i$ and continue to let $Y_{(i)}$ denote the
union of the components at this level of the $\varepsilon$-thick part
of $Y$.  For each edge $e$ with $\ltop=i$, we have a sequence of conformal maps
$\upsilon_{n,e}^-\colon A_{n,e}^-\to Y_{n,(<i)}$ of the bottom plumbing
annuli of the plumbing fixtures $\VV(\bft_{n,e})$ such that
$(\upsilon_{n,e}^{-})^{\ast}(\omega/ \prodt[n, \lceil i \rceil])
=\Omega_{e}$.
\par
\begin{lm} \label{lm:acrossthin} In the above situation, if $R$ is
  sufficiently large (depending only on the geometry of $(X, \eta)$),
  then the maps $\upsilon_{n,e}^-$ eventually extend to conformal maps
  $\upsilon_{n,e}$ whose domains contain the annulus
  $V_{n,e}^\circ = \{\delta/\changed{|\bft_{n,e}|} < |u| < \delta/\sqrt{R} \}$, such
  that $\upsilon_{n,e}^{\ast}(\omega/ \prodt)=\Omega_{e}$, and moreover,
  the images of $V^\circ_{n,e}$ and $V^\circ_{n,e'}$ are disjoint for any
  $e\neq e'$.
\end{lm}
\par
\begin{proof} [Proof of Lemma~\ref{lm:acrossthin}]
If $R$ is sufficiently large, the subsurface below the outer boundary
$\gamma_n = \gamma_{n,e}$ of $A_{n,e}^-$, given by $|v| = \delta/R$, is
convex (\changed{with respect to the flat metric}) for $R$ sufficiently small,
since $r_e'(\bft_n)$ tends to~$0$.  We use the orthogonal projection to~$\gamma_n$
to extend~$\upsilon_{n,e}^{-}$. That is, we map the equidistant curve of
\changed{flat} distance~$\ell$ to~$\gamma_n$ to the equidistant curve of
distance~$\ell$ to $\upsilon_{n,e}^{-}(\gamma_n)$, by mapping geodesic
rays orthogonal to $\gamma_n$ to geodesic rays orthogonal
to~$\upsilon_{_n,e}^{-}(\gamma_n)$.
\par
This gives a well-defined conformal map~$\upsilon_{n,e}$ of $V^\circ_{n,e}$ onto
its image as long as the $\Omega_e$-distance between the boundary components
of~$V^\circ_{n,e}$ is smaller than the $\omega_n$-distance from any
zero of~$Y_{n, (i)}$ to~$\upsilon_{_n,e}^{-}(\gamma_n)$.
The distance between these boundary components tends to
$\delta^\kappa / \kappa R^{\kappa/2}$.  Since
$(Y_n, \omega_n)\to(X, \bfeta)$, it suffices to take $R$ large
enough that $\delta^\kappa / \kappa R^{\kappa/2}$ is smaller than
the distance between any marked zero of $\eta_{(i)}$ and any zero at a node.
\par
Similarly, to ensure the $V^\circ_{n,e}$ are disjoint, it suffices
to take $R$ large enough that $\delta^\kappa / \kappa R^{\kappa/2}$
is smaller than half of the distance between any two nodal zeros of
$\eta_{(i)}$.
\end{proof}
\par
We now deal with the local injectivity of the plumbing map.

\begin{prop}
  \label{prop:plumbinj}
  The map $\OPlv$ is injective in a neighborhood of any point
  $(X,\bfeta)$ in the deepest stratum
  $\Omega\MDstratums[\Lambda][\Lambda]$.
\end{prop}
\par
\begin{proof}[\changed{Proof.}]
Consider a sequence $Y_n$ in $\ODehnsv$ converging to $X$ in the
deepest stratum, where $Y_n$ is obtained by plumbing two
sequences of model differentials $X_n^1$ and $X_n^2$, both converging
to $X$.  The proof of Proposition~\ref{prop:plumbopen} constructs
two sequences of isomorphisms $h_n^i\colon \OPlv(X_n^i, \bfeta_n^i, \bft_n)\to Y_n$, and
composing them gives isomorphisms
\begin{equation*}
  \phi_n \=(h_n^2)^{-1}h_n^1\colon\OPlv(X_n^1,\bfeta_n^1,\bft_n)\to\OPlv(X_n^2,\bfeta_n^2,\bft_n).
\end{equation*}
We will show by induction on the level $i$ that
$(X_{n,(i)}^1,\eta_{n,(i)}^1,\bft_n)=(X_{n,(i)}^2,\eta_{n.(i)}^2,\bft_n)$, starting from the
bottom level. For the base case $i=-N$, note that $\phi_n$ defines isomorphisms
$(X_{n,(-N)}^1,\eta_{n, (-N)}^1)\to (X_{n, (-N)}^2,\eta_{n, (-N)}^2)$ on the complement of the
top plumbing annuli, respecting the forms.  Extending these
isomorphisms to the entire surface, we see that $(X_{n,
  (-N)}^1,\eta_{n, (-N)}^1)= (X_{n, (-N)}^2,\eta_{n, (-N)}^2)$.
\par
For the inductive step, suppose $(X_{n,(-j)}^1,\eta_{n, (-j)}^1)=
(X_{n, (-j)}^2,\eta_{n, (-j)}^2)$ for all $j<i$.  We will show that
$\phi_n$ respects the basepoints in all bottom plumbing annuli of the
$i$'th level as well as the basepoints near the marked zeros.  From
this, it follows that $(X_{n,(-i)}^1,\eta_{n, (-i)}^1, \bft_n)$ and
$(X_{n,(-i)}^2,\eta_{n, (-i)}^2, \bft_n)$ have the same perturbed
periods, and thus they must be isomorphic by the local injectivity of perturbed
period coordinates.
\par
  It follows from the induction hypothesis that for each edge $e$
  connecting level $i$ to a lower level, the gluing maps $\nu_{n,e}$
  (defined in \eqref{eq:defnu2}) agree, so the two sequences define
  the same basepoints $c_{n,e}^+=\nu_{n,e}(p_{n,e}^+)$ in $Y_{n,(i)}$.
  By \eqref{eq:hrespectsbasepoints}, we have
  $h_{n,e}^i(b_{n,e}^{i,+})= c_{n,e}^+$, so $\phi_n(b_{n,e}^{1,+})=b_{n,e}^{2,+}$.
  Since the choice of basepoints $c_{n,h}$ near the marked zeros does
  not involve the $X_n$, we have $\phi_n(b_{n,h}^1)=b_{n,h}^2$ by the same argument.
\par
Alternatively, injectivity can be established by analyzing the proof
of Proposition~\ref{prop:plumbopen} to show that at every stage of
the construction of the sequence $X_n$ and the isomorphism $X_n\to Y_n$,
there is eventually a unique choice.  This uses in particular that the
perturbed period coordinates are injective, which was established in
Proposition~\ref{prop:pertper}.
\end{proof}

\par
\changed{
Combining Proposition~\ref{prop:Plcont}, Proposition~\ref{prop:plumbopen},
and Proposition~\ref{prop:plumbinj} we get:
}
\par
\begin{cor}
  \label{cor:plumblocdiff}
  The vertical plumbing map is a local homeomorphism on $\calW_\epsilon$.
\end{cor}
\par

%%%%%%%%%%%%%%%%%%%%%%%%%%%%%%%%
\subsection{The horizontal plumbing construction}
\label{sec:plhoriz}
%%%%%%%%%%%%%%%%%%%%%%%%%%%%%%%%
In this section, we complete the plumbing construction by  plumbing
the horizontal nodes of the family $(\calY^v, \omega) \to \calW_\epsilon$
constructed by the vertical plumbing map~$\OPlv$ to
obtain a generically smooth family $\calY\to \calW_\epsilon \times
\Delta^H$.  \changed{This version of the plumbing construction,
  defining coordinates by exponentiating relative periods of differentials with
  simple poles at the nodes, appeared in \cite{bainbridge07}.}
\par
We enumerate the horizontal edges as $e_1, \ldots, e_H$ and label the
branches through the corresponding nodes $q_i$ by an arbitrary choice of
sign.  Each $q_i$ is adjacent to two half-infinite cylinders
which each contain at least one zero of $\omega$.  We denote by~$z_i^+$
and~$z_i^-$ an arbitrary choice of a zero in the boundary of
each cylinder. We then apply Theorem~\ref{thm:standard_coordinates} to
choose standard coordinates
\begin{equation*}
  \upsilon_i^+ \colon \calW_\epsilon \times \Delta_{1}\to \calY^v
  \quad\text{and}\quad
  \upsilon_i^- \colon \calW_\epsilon \times \Delta_{1}\to \calY^v
\end{equation*}
covering a neighborhood of $q_i$ in the corresponding branch and such
that
\begin{equation*} (\upsilon_i^+)^*(\omega) \=
r_i'\frac{du}{u} \quad\text{and}\quad
(\upsilon_i^-)^*(\omega) \= -r_i'\frac{dv}{v}.
\end{equation*}
As these coordinates are unique up to multiplication by an
arbitrary constant, we normalize them  by requiring that the zeros
$z_i^+$ and $z_i^-$ correspond to $u=1$ and $v=1$ respectively.
\par
We define for each $j$ the (horizontal) plumbing fixture
\begin{equation}\label{eq:horplumbfix}
  \WW_j^\delta \= \left\{ (\bfw, \bft, \bfx, u, v) \in \calW_\epsilon\times
    \Delta_\epsilon^H\times \Delta_\delta^2 : uv=x_j\right\}\,,
\end{equation}
equipped with the standard relative holomorphic one-form
\begin{equation}\label{eq:horplumbform}
  \Omega_j \= -r_{e_j}'(\bft)\frac{du}{u} \= r_{e_j}'(\bft)\frac{dv}{v}\,,
\end{equation}
and define two families of annuli $\calA_j^\pm \subset \WW_j^1$ by
removing the upper or lower branches of the singular fibers:
\begin{equation*}
\calA_j^+ \= \WW_j^1 \setminus \{v=0\} \quad\text{and}\quad
\calA_j^- \= \WW_j^1 \setminus \{u=0\}\,.
\end{equation*}
We define two families of conformal maps
$\Upsilon_j^\pm\colon\WW_j^1\to \calX'$ by
\begin{equation}\label{eq:horizplumb}
\Upsilon_j^+(\bfw, \bft, \bfx, u, v) \= \upsilon^+_{j}(\bfw, \bft, u)
\quad\text{and}\quad
  \Upsilon_j^-(\bfw, \bft, \bfx, u, v) \= \upsilon^-_{j}(\bfw, \bft, v)\,,
\end{equation}
which identify $\calA_j^\pm$ with families of annuli
$\calB_j^\pm\subset \calY^v \times \Delta^H$.  Note that, in contrast
to the vertical plumbing construction, these families of annuli have
moduli tending to $\infty$ over the singular fibers.
\par
We define the plumbed family $\calY \to \calW_\epsilon \times
\Delta^H$ by removing from $\calY^v$ the disks bounded by the annuli
$\calB_j^\pm$ and gluing in the plumbing fixtures $\WW_j^1$ by the maps
$\Upsilon_j^\pm$.  As the gluing maps preserve the one-forms, $\calY$
is equipped with a relative one-form which we continue to denote by
$\omega$.
\par
Alternatively, in terms of flat geometry, the family $\calY$ can be
obtained by cutting each half-cylinder bounding $q_i$ along a closed
geodesic, and gluing the corresponding boundary components by an
isometry to obtain a finite annulus $C_j$.  The heights of the core
geodesics and the gluing maps are determined by requiring that
\begin{equation*}
  \int_{\gamma_j} \omega \= r_j' \log x_j,
\end{equation*}
where the integral is along a curve $\gamma_j$ in $C_j$ from the
chosen zero $z_j^-$ to $z_j^+$.
\par
As for the vertical plumbing construction, standard universal
properties applied to the restriction of $\calY$ to the strata
$\calW_\epsilon^J$ define the plumbing map $\OPl \colon \calW_\epsilon
\times \Delta^H \to \ODehns$. \index[plumb]{a030@$\OPl$! Plumbing map}
\par
\begin{proof}[\changed{Proof of Theorem~\ref{thm:plumbing}}]
We claim that $\OPl$ is a homeomorphism onto its image.
\par
To see this, consider a sequence $(X_n, \bft_n,\bfx_n)$ and point $(X, \bft, \bfx)$
in $\calW_\epsilon \times \Delta^H$ (which we are implicitly identifying
with $\ODehnsv\times \Delta^H$ via the vertical plumbing map), and
let $Y_n$ and $Y$ in
$\ODehns$ be the corresponding  horizontally plumbed
surfaces.  For each annulus $C_i\subset Y$, fix a subannulus
$C_i'\subset C_i$ with geodesic boundary, and let $Y'\subset Y$ be
the complement of the $C_i'$.  By the definition of the plumbing
construction, $Y'$ is canonically identified with a subsurface
$X'\subset X$.
\par
Suppose $(X_n, \bft_n,\bfx_n)$ converges to $(X, \bft, \bfx)$, which is
exhibited by a sequence of maps $g_n\colon X^\circ_n \to X_n$ defined on
an exhaustion~$X^\circ_n$ of~$X$.  These $g_n$ are eventually defined on~$X'$
and define maps $h_n\colon Y^\circ_n\to Y_n$, defined on an exhaustion~$Y^\circ_n$
of~$Y'$.  These~$h_n$ satisfy the required
properties to exhibit convergence of $Y_n$ to $Y$, except they need
to be extended over the annuli $C_i'$ in order to be defined on an exhaustion
of~$Y$.
\par
  Using the coordinates on the plumbing fixtures, we identify $C_i$
  and the corresponding annuli $C_{i,n}\subset Y_i$ with the planar
annuli
\bes C_i \= \{|x_i| < |u| < 1\} \quad \text{and} \quad C_{i,n} = \{|x_{i,n}| < |u|
< 1\}\,.
\ees
Under this identification, $C_i'$ is the round subannulus $\{\delta_1 <
  |u| < \delta_2\}$.  In these coordinates, the restrictions of $h_n$ to $C_i
  \setminus C_i'$ must converge to the identity map, as in the limit they fix a
  point in each component of $\bdry C_i$ corresponding to a marked
  zero $z_i^\pm$, and preserve the form $du/u$.  In particular, the
  restriction of $h_n$ to $\bdry C_i'$ converges  $C^1$-uniformly to
  the identity.  We extend $h_n$ to a smooth map on $C_i'$ by mapping a radial
  segment joining $\delta_1 e^{i \theta}$ and $\delta_2 e^{i\theta}$
  to the geodesic joining the images of these points.  The pullbacks
  $h_n^* du/u$ then converge uniformly to $du/u$.  The extended maps
  $h_n$ are then quasiconformal on an exhaustion of $Y$ and satisfy
  the hypotheses of convergence in the quasiconformal topology on
  forms from Section~\ref{sec:Space1}.   It follows that $Y_n\to Y$ in
  $\ODehns$, so $\OPl$ is continuous.

  If $Y_n\to Y$ in $\ODehns$, convergence of $X_n$ to $X$ is proved
  similarly, by transporting maps exhibiting convergence of $Y_n$ to
  $Y$ to an exhaustion of $X'$, and then extending across the
  complementary punctured disks by a straight-line extension.
  Convergence of the $\bft_n$ is obvious as they are relative periods
  of a convergent sequence of forms.  It follows that $\OPl^{-1}$ is
  continuous, so $\OPl$ is a local homeomorphism as desired.

  The equivariance of the map with respect to the group
  $K_\Lambda=\Tw/\sTw$ follows from the construction in
  Section~\ref{sec:defPlmap}, since $K_\Lambda$ acts on the markings
  and the simple rescaling parameters only, and they both have been
  transported from the family over the model domain to the plumbed
  family.
\end{proof}

%%%%%%%%%%%%%%%%%%%%%%%%%%%%%%%%
\subsection{\changed{The complex structure on the Dehn space}}
\label{sec:DehnComplex}
%%%%%%%%%%%%%%%%%%%%%%%%%%%%%%%%
We can now establish Theorem~\ref{thm:DehnIsOrbi}, constructing the
complex structure of the Dehn space.
\par
For any degeneration $\Lambda'\rightsquigarrow \Lambda$, the
quotient group $\sTw[\Lambda]/\sTw[\Lambda']$ acts properly
discontinuously on $\ODehns[\Lambda']$, and the quotient $\ODstratumso$
can be regarded as an open subset of $\ODehns$.  We denote by
$\pi_\Lambda^{\Lambda'}\colon \ODehns[\Lambda'] \to\ODehns$ the
resulting local homeomorphism.

We define a \emph{plumbing chart} for $\ODehns$ to be any map $\psi$ obtained
as the composition $\psi=\pi_\Lambda^{\Lambda'}\circ\OPl$, where $\OPl$ is
a plumbing map of a small coordinate neighborhood $U\subset\calW_\epsilon
\times \Delta^H$  (which, recall, depends on numerous choices) with target $\ODehns[\Lambda']$, where the domain of $\OPl$ is sufficiently
restricted so that this composition is injective.  Since
$\pi_\Lambda^{\Lambda'}$ and $\OPl$ are both local homeomorphisms, any
plumbing chart is a homeomorphism onto its image.
\par
\begin{lm}
  \label{lem:transitions_are_holomorphic}
For any two plumbing charts $\psi_i\colon U_i\to \ODehns$, the
  transition function $\phi=\psi_2 \circ \psi_1^{-1}$ is biholomorphic
  on $V= \psi_1^{-1}\psi_2(U_2)$.
\end{lm}
\par
\begin{proof}
Recall that the generic stratum $\ODstratums[\Lambda][\emptyset] =
\Oteich/\sTw$ already has the structure of a complex manifold.
\par
Let $D_i\subset U_i$ be the divisor of coordinate hyperplanes.  By construction, each
plumbing chart sends $U_i\setminus D_i$ to
$\ODstratums[\Lambda][\emptyset]$, since no nodes are plumbed. Moreover, each plumbing chart is holomorphic on $U_i\setminus D_i$, since it was defined there by the universal property of
the generic stratum.  By the Riemann Extension Theorem (see
\cite[Theorem~I.C.3]{gunning}), the restriction of $\phi$ to
$V\setminus D_1$ then extends to a holomorphic function on $V$.  Since
$\phi$ is continuous on $V$, it must agree with this extension, so
$\phi$ is holomorphic.
\end{proof}
\par
\begin{proof}[Proof of Theorem~\ref{thm:DehnIsOrbi}]
We give $\ODehns$ an atlas of charts by taking the plumbing charts 
  defined above (with all possible choices made in the construction) together with any atlas of holomorphic charts for the generic
  stratum.  Since the charts of the form $\pi_\Lambda^{\Lambda'}\circ\OPl$
  cover the stratum $\ODstratums[\Lambda][\Lambda']$, altogether the
  collection of all these charts covers all of $\ODehns$.  Since all the transition
  functions are biholomorphic by the lemma above, this gives $\ODehns$ the structure of a
  complex manifold.
\par
The complex structure on $\ODehn$ stems from that on
  $\ODehns[\Lambda]$ and the $K_\Lambda$-equivariance of the plumbing
  map.
\end{proof}

\subsection{Horizontal extension of perturbed period coordinates}
\label{sec:extendperturbed}
%%%%%%%%%%%%%%%%%%%%%%%%%%%%%%%%

We finish this section by extending the perturbed period coordinates
discussed in Section~\ref{sec:modif} to the case with horizontal
nodes.  These can be regarded as a generalization of the classical
period coordinates, giving explicit local coordinates at the boundary
of $\LMS$.  These coordinates will not be used in the current paper,
but the construction is essential for further
applications, e.g.\ in \cite{CoMoZaEU} or \cite{bdg}.
\par
More precisely, suppose that we have chosen the local gluing data
to plumb each of the $H$ horizontal nodes by a plumbing fixture $\WW = \WW(x)$
for $x \in \Delta_\ve$, see~\eqref{eq:horizplumb}.
Let $\calY \to U = \calW \times \Delta_\ve^H$ be the family that results
from plumbing the horizontal nodes of $\calX \to \calW$ as in
Section~\ref{sec:plhoriz}. Our goal now is to extend the perturbed period
map~$\PPer$ from Section~\ref{sec:perturbed} to a local diffeomorphism whose
domain is~$U$.
\par
Suppose the $j$-th horizontal node $q_j$ lies on the level $i=i(j)$ subsurface
$\Sigma_{(i)}\subset\Sigma$. Let $\beta_j$ be a path that stays in $\Sigma_{(i)}$,
which represents a homology class in $\Sigma$ relative to~$Z_i$ \changed{defined in~\eqref{eq:PiZi}}
 (or a path ending at the points
in the image of $\sigma^{+}$ if needed) and that crosses once the seam
of~$q_j$, and crosses no other seams. Such a path exists, since each
component of $X$ contains at least one point of the set~$Z_i$.
Let $\alpha_j$ be the loop around~$q_j$. We define the {\em perturbed period map}
\begin{equation}\label{eq:PPM}
\ePP \= \PPer \times \Phor\colon U \to \CC^{L^{\bullet}(\Lambda)} \times \oplus_i
\calR_i' \times \CC^H\,,
\end{equation}
where
\begin{equation}\label{eq:horextPPM}
\Phor\colon  \begin{cases} \begin{array}{ccl}
U &\to&  \CC^H, \quad \\
([X,\bfeta,\bft],(x_j)) &\mapsto& \left(
\bfe \Biggl(\frac{\int_{\beta_j}  \eta_{(i)} + \xi_{(i)}}{
\int_{\alpha_j}  \eta_{(i)} +  \xi_{(i)}}\Biggr)_{j=1,\dots, H}\right)\,,
     \end{array}
\end{cases}
\end{equation}
and for integration we use the $f$-images of the
\index[family]{d030@$x_j$! Smoothing parameters for the horizontal nodes}paths in the corresponding
fiber of  the family $\calY \to U$. Note that we integrate the form in the fibers
of $\calY$ which is the family obtained after the horizontal plumbing map. In particular,
the horizontal node~$q_j$ has been smoothed out in the fibers of $\calY \to U$ above
the locus $x_j\neq 0$, using the standard plumbing fixture. The exponentiation makes this map
well-defined, despite the fact that~$f$ is only well-defined up to composition by elements in the
twist group $\Tw$. Indeed, any two images of $\beta_j$ differ by a power of the Dehn
twist about~$\alpha_j$.
\par
\begin{prop} \label{prop:extPP}
The perturbed period map $\ePP$ is a local diffeomorphism in
a neighborhood of $\calW \times \bfzero$.
\end{prop}
\par
\begin{proof} Using Proposition~\ref{prop:pertper} the claim follows from the
fact that the components of the map $\Phor$ are non-constant,
holomorphic (since the $\alpha_j$-periods tend to a non-zero residue
and the imaginary part of the $\beta_j$-period divided by the $\alpha_j$-period
goes to $+\infty$), and independent of each other by the construction of plumbing
fixtures disjointly and independently.
\end{proof}
\par
\begin{exa}
 We describe the perturbed period coordinates in the case of a curve with two irreducible components $X_{1}$ and $X_{2}$ that meet at two horizontal nodes.  Take a model differential such that its restriction $\eta_{1}$ to the component $X_{1}$ is in $\omoduli[1,3](2,-1,-1)$ and the restriction $\eta_{2}$ to $X_{2}$ is in $\omoduli[1,4](1,1,-1,-1)$. In this case, the GRC space is the subset of  the product of the $H^{1}$ such that the sum of the residues of the $\eta_i$ at each node is zero. Of course one of the two equations is redundant because of the residue theorem, and hence the GRC space is a hyperplane (i.e.~of codimension one only). Moreover, since the twisted differential has only one level, there is no modifying differential, hence the perturbed period map $\PPer$ is the usual period map.
 \par
 Now let us denote by $(x_{1}, x_{2})$ the coordinates in $\Delta_\ve^2$. Moreover for $i=1,2$, let $\beta_i$ be an immersed arc crossing exactly once the seam of the node $q_i$ from the double zero of $\eta_{1}$ to one of the zeros of $\eta_{2}$.
 Then the map $\Phor|_{{\bf0}\times \Delta_\ve^2}$ is given by
 \begin{equation}\label{eq:exhorper}
  (x_{1},x_{2})\mapsto (k_{1}x_{1} , k_{2}x_{2})\,,
 \end{equation}
where the $k_i$ are non-zero constants.
\par
To see this, we decompose the path $\beta_i$ into three paths as
follows.  The first path $\beta_i^{1}$ joins the double zero of
$\eta_{1}$ to the marked point in $X_{1}$ used to put the plumbing
fixture. The second path $\beta_i^{2}$ is the (image in the plumbed
surface of the) path in the plumbing fixture of
Equation~\eqref{eq:horplumbfix} joining the two marked points. The
last path $\beta_i^{3}$ joins the point of $X_{2}$ used for the
plumbing fixture and the endpoint of $\beta_i$. In this setting, the
period of $\beta_i$ is the sum of the periods of the three arcs
$\beta_i^j$. Note that the periods of $\beta_i^{1}$ and $\beta_i^{3}$
are constants $c_i^{1}$ and $c_i^{3}$ above $\Delta^{2}$. An easy
computation using Equation~\eqref{eq:horplumbform} shows that the
period of $\beta_i^{2}$ is equal to $r_i\log(x_i)$, where $r_i$ is the
residue of $\eta$ at~$q_i$. Hence the $\beta_i$-period is of the type
$c_i^{1} +r_i\log(x_i) + c_i^{3}$. This gives
Equation~\eqref{eq:exhorper}, where the $k_i$ are the exponentials of
$(c_i^{1} + c_i^{3})/r_i$.
 \end{exa}

%%%%%%%%%%%%%%%%%%%%%%%%%%%%%%%%%%%%%%%%%%%%%%%%%%%%%%%%%%%%%%%
\section{Families of \msds} \label{sec:famnew}
%%%%%%%%%%%%%%%%%%%%%%%%%%%%%%%%%%%%%%%%%%%%%%%%%%%%%%%%%%%%%%%

In this section we define families of \msds, generalizing the definition of
a single \msd in Section~\ref{sec:AugTeich}. Eventually we will give an
algebraic description of the stack of multi-scale differentials as a blowup.
The starting point will be a flat family of pointed stable curves
$(\pi\colon\calX\to B, \bfz)$ over an arbitrary base~$B$, possibly reducible and
non-reduced.  We will first define a germ of \msds at a point $p\in B$.  Roughly
speaking, this will consist of four pieces of data: the structure of an enhanced
level graph on the dual graph $\Gamma_p$ of the fiber $X_p$, a \emph{rescaling ensemble},
which is a germ of a morphism  $R_p\colon B_p \to \Tnorm[\Gamma_p]$
\index[family]{d020@$R_p\colon B_p \to \Tnorm[\Gamma_p]$! Rescaling ensemble}
to the normalization of the level rotation torus closure (recall the definition and details of this from Section~\ref{sec:LRTC}), a {\em collection of rescaled
differentials} $\omega_{(i)}$, and finally prong-matchings at all nodes of the
family, such that for every non-semipersistent node (as defined below) the prong-matching
is naturally induced by the family. These data satisfy some restrictions analogous
to those of a single twisted differential, and there is an equivalence relation given by
the action of the level rotation torus, analogous to the definition of a single \msd.
\par
We will show in Proposition~\ref{prop:mds1dim} that in favorable circumstances,
for example
for a family over a smooth base curve~$B$ with no persistent nodes, giving a \msd
simply amounts to giving a family of stable differentials of type~$\mu$, that does not
vanish identically on any fiber, \changed{but which may vanish on some
irreducible components of fibers}.  

%%%%%%%%%%%%%%%%%%%%%%%
\subsection{Germs of families of \msds} \label{sec:germfam}
%%%%%%%%%%%%%%%%%%%%%%%

We will define all the notions locally first, so until
Section~\ref{sec:presheafmsd}
we will work with a family over the germ $B_p$ of an analytic space~$B$ at a point
$p\in B$.
\par
Recall e.g.\ from \cite[Proposition~X.2.1]{acgh2} that for each node~$q_e$
of~$X_p$ there is a function $f_e\in \calO_{B, p}$, which we call a  {\em smoothing
  parameter},
\index[family]{d050@$f_e$! Smoothing parameters}
so that the family has the local form $u_e v_e = f_e$ in some coordinates in 
a neighborhood of~$q_e$.  The parameter $f_e$ is only defined up to multiplication
by a unit in $\calO_{B, p}$. We will write $[f_e]\in \calO_{B, p} / \calO_{B, p}^*$
for its equivalence class.
\par
Given an enhanced level graph $\Gamma_p$, suppose we have a morphism
$R\colon B\to \Tnorm[\Gamma_p]$.  This morphism determines for each vertical edge~$e$
a function $f_e\in \calO_B$ and for each level~$i$ a function $s_i\in \calO_B$,
such that if an edge $e$ joins levels $j<i$, then
\begin{equation}
  \label{eq:2}
  f_e^{\kappa_e} \= s_j \dots s_{i-1}\,.
\end{equation}
\par
\begin{df} \label{def:RescEns}
A \emph{rescaling ensemble} is a morphism $R\colon B\to \Tnorm[\Gamma_p]$ such
that the parameters $f_e\in \calO_B$ for each vertical edge~$e$ determined by~$R$
lie in the equivalence class~$[f_e]$ determined by the family $\pi\colon\calX\to B$.
\end{df}
\par
The $s_i$ will be called the \emph{rescaling parameters} determined by $R$,
in parallel with the notion defined in Section~\ref{sec:universalMD}.
The rescaling ensemble $R$ can be thought of as a choice of these parameters
that satisfies \eqref{eq:2} for each edge~$e$ of $\Gamma_p$, together with the
choices of appropriate roots of these which define a lift to $\Tnorm[\Gamma_p]$,
see Proposition~\ref{prop:charLRTC} for the precise statement.
\par
\begin{df}  \label{def:collRD}
A \emph{collection of rescaled differentials of type~$\mu$} at $p\in B$
is a collection of germs of sections $\omega_{(i)}$ of $\omega_{\calX/B}$
defined on open subsets $U_i$ of~$\calX$, indexed by the levels~$i$ of the
enhanced level graph~$\Gamma_p$.  Each $U_i$ is required to be a neighborhood
of the subcurve~$X_{p,( \leq i)}$ with the points of its intersection with
$X_{p, (>i)} \cup \calZ^\infty$ removed. For each level~$i$ and each edge~$e$
of $\Gamma_p$ whose lower vertex is at level~$i$ or below, we define
$r_{e, (i)}\in \calO_{B, p}$ to be the period of $\omega_{(i)}$ along the oriented
vanishing cycle~$\gamma_e$ for the node~$q_e$.  We require the collection
to satisfy the following constraints:
\begin{enumerate}
\item For any levels $j<i$ the differentials satisfy $\omega_{(i)}
  \= s_j \cdots s_{i-1} \omega_{(j)}$ on $U_i\cap U_j$ for some $s_k \in \calO_{B, p}$
  with $s_k(p)=0$ (where $k=j,\ldots,i-1$).
\item For any edge $e$ joining levels $j<i$ of $\Gamma_p$, there are germs of
functions $u_e, v_e$ defined on a neighborhood of the corresponding
node in $\calX$, and a germ of a function $f_e$ on~$B$, such that the family locally has the
form $u_e v_e = f_e$, and in these coordinates
\begin{equation}\label{eq:5}
  \omega_{(i)} \= (u_e^{\kappa_e} + f_e^{\kappa_e}r_{e, (i)})\frac{du_e}{u_e}
  \quad\text{and}  \quad
\omega_{(j)} \= - ( v_e^{-\kappa_e} + r_{e, (j)}) \frac{dv_e}{v_e}\,,
\end{equation}
where $\kappa_e$ is the enhancement of $\Gamma_p$.  The irreducible components
of $\calX|_{V(f_e)}$ where $\omega_{(i)}$ is zero or $\infty$ are called
respectively {\em vertical zeros} and {\em vertical poles}.
\item The $\omega_{(i)}$ have order~$m_k$ along the sections $\calZ_k$
that meet the level-$i$ subcurve of~$X_p$; these are called {\em horizontal
zeros and poles}. Moreover, $\omega_{(i)}$ is holomorphic and non-zero away
from its horizontal and vertical zeros and poles.
\item (Global Residue Condition)  Let $\Sigma$ be the topological surface
obtained by smoothing each node of $X_p$, and regard the vanishing
cycles~$\gamma_e$ as oriented curves on~$\Sigma$. Then each relation
\bes
\qquad
\sum_{e\colon \lbot \leq i} \alpha_e \gamma_e = 0 \quad \text{in $H_1(\Sigma
\smallsetminus P_s,\QQ)$
for some $\alpha_e \in \QQ$ implies}\! \sum_{e \colon \lbot \leq i}
\alpha_e r_{e,(i)} = 0.
\ees
\end{enumerate}
If the rescaling and smoothing parameters for the collection $\omega_{(i)}$
agree with those of the rescaling ensemble~$R$, we call them {\em compatible}.
We denote the collection by $\bfomega = (\omega_{(i)})_{i \in L^\bullet(\Gamma_p)}$
or by $\bfomega_p$.
\end{df}
\par
Some remarks to unravel the meaning of this definition are in order.
Condition (2) is often automatic from Theorem~\ref{thm:NF}. However, for the case of semipersistent nodes defined below that theorem does not
apply, and the condition needs to be imposed.
\par
Condition (3) ensures that each $\omega_{(i)}$ is not identically zero on a
neighborhood of~$X_{p,(i)}$.  Condition (1) ensures that $\omega_{(i)}$
vanishes on~$X_{p,(<i)}$. Moreover, $\omega_{(i)}$ vanishes
on a neighborhood in~$\calX$ of~$\changed{X_{(j)}}$ for some level $j<i$, if some~$s_k$ with
$j\le k\le i-1$ vanishes in a neighborhood of~$p$ on~$B$.
\par
Conditions (4) and (1) together imply the usual global residue
condition in each nearby fiber $X_q$ (using the enhanced structure of
$\Gamma_{X_q}$ which we define below by undegenerating from $\Gamma_p$). Note
that~$r_{e,(i)}$ agrees with $2\pi \sqrt{-1} $ times the residue of $\omega_{(i)}$ at~$q_e$
over the locus where the node~$q_e$ persists. By condition (1), given two levels $j<i$
and any edge $e$ such that $\lbot \leq j < i$ we have
\begin{equation*}
  r_{e,(i)} \= s_j \cdots s_{i-1} r_{e, (j)}\,.
\end{equation*}
In particular, if $s_j=0$ for some $\lbot \leq j$,  then $r_{e, (i)}=0$. Consequently,
the relations reflect the global residue condition as stated
in Proposition~\ref{prop:kdiffGRC}, as we now explore in an example.
\par
\begin{exa} ({\em Definition~\ref{def:collRD} extends to fiberwise GRC.})
\label{ex:GRCinfdef}
We consider the level graph given by Figure~\ref{fig:rescond}, where  $\kappa_{e_i}=2$ for
every~$i$. Consider a collection of rescaled differentials
with $\omega_{(-1)} \= t^2 \omega_{(-2)}$  (while~$\omega_{(0)}$ will not matter for us) over $B= \Spec \CC[t]/(t^{3})$, and let $r_{e,(i)}\colon B \to \CC$ be as in Definition~\ref{def:collRD}. The usual GRC from Section~\ref{sec:deftwd}
states that the residues at the point $t=0$ satisfy $r_{e_1,(-1)}(0) =
r_{e_2,(-1)}(0) = 0$.
 \begin{figure}[htb]
   \centering
\begin{tikzpicture}[scale=1]
\coordinate (a1) at (-1,0);\fill (a1) circle (2pt);
\coordinate (a2) at (1,0);\fill (a2) circle (2pt);
\coordinate (a3) at (0,-1);\fill (a3) circle (2pt);
\coordinate (a4) at (0,-2);\fill (a4) circle (2pt);
\draw (a1)  -- (a3) coordinate[pos=.4] (b1);
\draw (a2)  -- (a3) coordinate[pos=.4] (b2);
\draw (a1)  .. controls ++(-90:1) and ++(120:.5) .. (a4) coordinate[pos=.5] (b3);
\draw (a2)  .. controls ++(-90:1) and ++(60:.5) .. (a4) coordinate[pos=.5] (b4);

\node[right] at (b1) {$e_{1}$};
\node[left] at (b2) {$e_{2}$};
\node[left] at (b3) {$e_{3}$};
\node[right] at (b4) {$e_{4}$};
\end{tikzpicture}
\caption{An enhanced level graph illustrating versions of the GRC.}
\label{fig:rescond}
\end{figure}
\par
Since the vanishing cycles corresponding to the edges~$e_{1}$ and~$e_{3}$ are
homologous, Definition~\ref{def:collRD}~(4) states that
\bes
0 \= r_{e_{1},(-1)} + r_{e_{3},(-1)} \= r_{e_{1},(-1)} + t^2r_{e_{3},(-2)} \,.
\ees
This condition reproduces the GRC when setting $t=0$, and imposes a stronger
constraint on the higher order terms of the expansion in~$t$ of the residues.
\end{exa}
\par
\medskip
In preparation for the notion of prong-matchings in families, we define a subtle variant
of the usual notion of persistent nodes that becomes relevant over a non-reduced base,
and thus in particular for first order deformations.
\par
\begin{df} \label{def:semiper} Given a germ of a family $\pi \colon \calX \to B_p$
at~$p$, we say that a node $e$ is \emph{persistent} if $f_e=0$. If the dual
graph~$\Gamma_p$ has been provided with an enhanced level graph structure, we say
that a node $e$ is \emph{semipersistent} if $f_e^{\kappa_e}=0$.
\end{df}
\par
We start with a discussion of prong-matchings in families, generalizing the definitions
in Section~\ref{sec:prmatch}.
Suppose first that $q_e$ is a persistent node joining levels $j<i$ of $\Gamma_p$.  In
local analytic coordinates, the family is of the form $u_ev_e=0$.  We write $Q_e$ for
the nodal subspace cut out by $u_e=v_e=0$, so that $Q_e$ can be thought of as the image
of a nodal section $B\to \calX$.  We write $\calN_e^+$ and $\calN_e^-$ for the normal
bundles to $Q_e$ in each branch of~$\calX$ along~$Q_e$. These are line bundles on~$Q_e$,
because~$Q_e$ is a Cartier divisor in each branch, and by pullback via the nodal
section they can be regarded as line bundles on $B$.
\par
We also have the rescaled differentials $\omega_{(i)}$ and $\omega_{(j)}$ defined
near $Q_e$ in its respective branches, and choose local coordinates $u_e$ and $v_e$
so that these differentials are
in their normal form \eqref{eq:5} (with $f_e=0$).  The \emph{prongs} of $Q_e$ are
then the $\kappa_e$ sections of the dual line bundles $(\calN_e^\pm)^*$ given by
$v_+ = \theta_+ \frac{\partial }{\partial u_e}$ and
$v_- = \theta_- \frac{\partial }{\partial v_e}$, where $\theta_\pm$ range over all
possible $\kappa_e$-th roots of
unity.  A \emph{prong-matching} at $Q_e$ is a section $\sigma_e$ of
$\calP_e=\calN_e^+ \otimes \calN_e^-$ such that $\sigma_e(v_+ \otimes v_-)^{\kappa_e} = 1$
for any two prongs~$v_+$ and~$v_-$.  Intuitively each prong $v_+$ is matched to the
unique prong $v_-$ such that $\sigma_e(v_+ \otimes v_-)=1$.
\par
Prong-matchings can be defined similarly for a non-persistent node.  In this case,
the function~$f_e$ defines a subspace $B_e\subset B$ over which this node persists.
The entire discussion of the previous paragraph can be carried out over~$B_e$, and
one defines a prong-matching as an appropriate section of $\calP_e$, which is now
a line bundle over $B_e$.
\par
For a non-semipersistent node $e$, there is a natural \emph{induced prong-matching}
$\sigma_e$ over~$B_e$ which is defined by the choice of the rescaled
differentials~$\omega_{(i)}$ and the rescaling ensemble.  This prong-matching~$\sigma_e$
is defined explicitly in local coordinates by writing it as $\sigma_e=d u_e
\otimes dv_e$, where $u_e$ and $v_e$ are as in \eqref{eq:5}.  Any two possible
choices of~$u_e$ and $v_e$ are of the form $\alpha u_e$ and $\alpha^{-1} v_e$
for some unit $\alpha \in \calO_{B, p}^*$, so the induced prong-matching does
not depend on this choice.
\par
We can now package everything into the local version of our main objects.
\par
\begin{df} \label{def:germMSD}
Given a family of pointed stable curves $(\pi\colon\calX\to B, \bfz)$ and $B_{p}$
a germ of~$B$ at~$p$, the {\em germ of a family of \msds} of type $\mu$ over $B_{p}$
is an equivalence class of the following set of data:
\begin{enumerate}
\item the structure of an enhanced level graph on the dual graph $\Gamma_p$ of the
fiber $X_p$,
\item a rescaling ensemble $R\colon B \to \Tnorm[\Gamma_p]$, compatible with
\item a collection of rescaled differentials $\bfomega = (\omega_{(i)})_{i \in
L^\bullet(\Gamma_p)}$ of type~$\mu$, and
\item a collection of prong-matchings $\bfsigma = (\sigma_e)_{e \in E(\Gamma)^v}$, which
are sections of $\calP_e$ over~$B_e$.  For the non-semipersistent nodes, these
are required to agree with the induced prong-matchings defined above.
\end{enumerate}
The $\calO_{B,p}$-valued points of the  level rotation torus $T_{\Gamma_p}(\calO_B)$
act on all of the above data, and we consider
the data $(\omega_{(i)}, R,\sigma_e)$ to be {\em equivalent} to $(\rho \cdot \omega_{(i)},
\rho^{-1}\cdot R,\rho \cdot \sigma_e )$ for any $\rho\in T_{\Gamma_p}(\calO_B)$.
Here the torus action is defined by $\rho\cdot \omega_{(i)} = s_i \omega_{(i)}$
and $\rho\cdot \sigma_e = f_e \sigma_e$, and  $\rho^{-1}\cdot (\ )$ denotes
post-composition with the multiplication by~$\rho^{-1}$.
\par
  Replacing $T_{\Gamma_p}(\calO_{B, p})$ with the extended level rotation
  torus $\Textd[\Gamma_p](\calO_{B, p})$, the analogous object is called a
{\em germ of a family of projectivized \msds}.
\end{df}
\par
\begin{rem} \label{rem:ptwdismsd}
  Note that over a reduced point, this definition of a family of \msds agrees with Definition~\ref{df:msd}.
\end{rem}
\par
A \emph{morphism between two germs of \msds}
\be \label{eq:defmorphismMSD}
(\pi'\colon\calX'\to B', \bfz', \Gamma_{p'}, \bfomega', \bfsigma')
\,\to\, (\pi\colon\calX\to B, \bfz, \Gamma_{p}, \bfomega, \bfsigma)
\ee
is a pair of germs of morphisms $\phi\colon B'\to B$ and $\tilde{\phi}\colon
\calX'\to \calX$ that jointly define a morphism of families of pointed
stable curves (see {\cite[p.~281]{acgh2}}), such that the induced isomorphism
of dual graphs $ \Gamma_{p'}\to \Gamma_p$ is also an isomorphism of enhanced
level graphs and such that $\tilde{\phi}^* (\bfomega,\bfsigma)$ is equivalent
to $(\bfomega', \bfsigma')$.
\par
For later use in the context of marked differential, and as an auxiliary object
in the next subsection, we define a {\em simple rescaling ensemble} to be
a (germ of a) morphism $R^s \colon B \to \ol{T}_{\Gamma_p}^{s}$
\index[family]{d030@$R^s \colon B \to \ol{T}_{\Gamma_p}^{s}$! Simple rescaling ensemble}
to the simple level rotation torus closure, such that the composition with
$\ol{T}_{\Gamma_p}^{s} \to \Tnorm[\Gamma_p]$ is a rescaling ensemble in the above
sense. Concretely, the map~$R^s$ is given by \emph{simple rescaling parameters}
$t_i \colon B \to \CC$ (in terminology consistent with that in
Section~\ref{sec:universalMD}). Given a simple rescaling ensemble~$R^s$,
the associated (non-simple) rescaling ensemble is obtained by taking
\begin{equation}\label{eq:fesifromti}
s_i \= t_i^{a_i} \quad \text{and} \quad f_e \= \prod_{i=\lbot}^{\ltop-1} t_i^{m_{i,e}}\,,
\end{equation}
similarly to~\eqref{eq:TsT}.

%%%%%%%%%%%%%%%%%%%%%%%
\subsection{Restriction of germs to nearby points}
\label{sec:presheafmsd}
%%%%%%%%%%%%%%%%%%%%%%%

Now we allow $B$ to be any complex analytic space containing a point~$p$. Before
giving the global definition of families, we need to  define restrictions
of germs. For this purpose consider a germ of \msds at~$p$ given by the
data $(\Gamma_p, \bfomega =(\omega_{(i)}), \bfsigma = (\sigma_e))$
with $\bfomega$ compatible with~$R_p$. Let $U\subset B$ be a
neighborhood of~$p$ over which~$R_p$ and all~$\sigma_e$ are defined. For every
$q\in U$ we wish to define the germ of \msds at $q$ induced by this germ  at $p$.
\par
First we explain how this datum  defines an undegeneration of enhanced level
graphs $\degen\colon\Gamma_{q}\rightsquigarrow \Gamma_p$ as in Section~\ref{sec:order}. With notation of that section, this undegeneration is given by $J=\{j_{-1},\dots,j_{-M}\}$, where $j_{k}\in J$ if and only if $s_{j_{k}}(q) = 0$. Moreover recall that $j_0=0$ and $j_{-M-1} = -N-1$ (though they are not part of $J$).
More explicitly, the map of dual graphs $\delta\colon\Gamma_p\to\Gamma_{q}$ is obtained by
contracting every vertical~edge $e$ such that $f_e({q})\neq 0$.  (Whether
horizontal edges are contracted or not is determined by the fiber~$X_{q}$.)
If~$e$ is contracted and joins levels $j<i$, then since $f_e^{\kappa_e} = s_j
\cdots s_{i-1}$, the rescaling parameter $s_k({q})\neq 0$ for each $j\leq k < i$.
We then define the order on~$\Gamma_{q}$ so that the $k$'th level of $\Gamma_{q}$ corresponds to a maximal interval $(j_{k-1}, j_k]$ in $L^\bullet(\Gamma_p)$
such that $s_{j_k}({q})=0$ for $j_{k}\in J$ and $s_i({q})\neq 0$ for every smaller~$i$ in this
interval.
% We denote by $I \subset L^\bullet(\Gamma_p)$ the set of right-endpoints $i_\ell$.
The map $\delta$ is then compatible with
the enhancements of these dual graphs and defines a degeneration
$\degen$, as desired.
\par
Second, we define the restriction of $R_p$ to a rescaling
ensemble at~$q$. The undegeneration $\degen$ induces a corresponding
homomorphism $\degen_*\colon T_{\Gamma_{q}}\to T_{\Gamma_p}$ defined
in Section~\ref{sec:levrottori} and a homomorphism $\degenbar_*\colon
\Tnorm[\Gamma_q]\to\Tnorm[\Gamma_p]$, which is equivariant with respect to
the action of each torus on the normalization of its closure.
\par
\begin{prop} \label{prop:restRE}
Given a rescaling ensemble $R_p\colon B\to \Tnorm[\Gamma_p]$, there exists
a neighborhood $V\subset B$ of $q$, a rescaling ensemble $R_q\colon V
\to \Tnorm[\Gamma_q]$, and $\tau\in T_{\Gamma_p}$ such that
\begin{equation*}
   \degenbar_* \circ R_q \= \tau\cdot R_p\,
\end{equation*}
as germs at $q$. Any two such $\tau$ differ by composition with an element
of $T_{\Gamma_q}$.
\end{prop}
\par
\begin{proof}
We take the fiber product of~$R_p$ with the finite map $\ol{p} \colon
\ol{T}_{\Gamma_p}^{s}\to \ol{T}_{\Gamma_p}^{s}/K_{\Gamma_p} =\Tnorm[\Gamma_p]$ to obtain some simple
rescaling ensemble $R_p^s\colon B^s \to \ol{T}_{\Gamma_p}^{s}$, defined on some
ramified cover~$B^s$ of~$B$. We solve the equation for the corresponding simple
objects and then descend. Let $q'\in B^s$ denote some preimage of $q\in B$.
Two consecutive levels $i,i+1\in L^\bullet(\Gamma_p)$ have the same image
in $L^\bullet(\Gamma_q) = L^\bullet(\Gamma_{q'})$ if and only if $t_i(q')\ne 0$. As
in Lemma~\ref{lm:monoLRT}, the image of the monomorphism $\degenbar^s_* \colon
\ol{T}_{\Gamma_{q'}}^{s} \to \ol{T}_{\Gamma_p}^{s}$ of simple level rotation tori is
cut out by the equations $t_i=1$ for all levels~$i$ such that the images of level~$i$ and~$i+1$
are the same in $L^\bullet(\Gamma_q)$. We define $\tau^s\in T_{\Gamma_p}^{s}$ by
\bes
(\tau^s)_i \=\begin{cases}
1/ t_i(q'), & \mbox{if $\degen(i)=\degen(i+1)$} \\
1, & \mbox{otherwise}\,.             \end{cases}
\ees
This ensures that $\tau^s \cdot R_p^s$ is in the image of $\degen^s_*$,
and so there exists a simple rescaling ensemble~$R_q^s$ defined on a neighborhood of $q'$ in $B'$ such that
$\degenbar_*^s \circ R^s_q \= \tau^s\cdot R^s_p$.
\par
The multiplication map $\tau^s$ is $K_{\Gamma_p}$-equivariant, since the torus
$T_{\Gamma_p}^{s}$ is commutative. Since $R_p^s$ and $\degenbar_*^s$  are
$K_{\Gamma_p}$-equivariant by construction as fiber product, the map~$R_q^s$
descends to the required map $R_q$, and we let~$\tau = \ol{p}(\tau^s)$.
\par
To show the uniqueness of $\tau$ up to the action of $T_{\Gamma_q}$,  observe
that if for some other~$\tau'$ the composition $\tau'\cdot R_p$ were to also
lie in the image of $\degenbar_*$, then the values of $\tau'\cdot R_p$
and $\tau\cdot R_p$ must all be equal on all the edges of~$\Gamma_p$
that are contracted in $\Gamma_q$, and on all levels $i\in L^\bullet(\Gamma_p)$
such that levels~$i$ and~$i+1$ have the same image in $L^\bullet(\Gamma_q)$.
Thus $\tau'\cdot\tau^{-1}\in T_{\Gamma_p}$ must act trivially on all such edges
and levels. But this is precisely to say that $\tau'\cdot\tau^{-1}$ lies in
the image of $T_{\Gamma_q}$ embedded into $T_{\Gamma_p}$.
\end{proof}
\par
Third, we define the collection of rescaled differentials at~$q$. For
each~$k \in L^\bullet(\Gamma_q)$ let $(j_{k-1}, j_k]$ be its preimage in
$L^\bullet(\Gamma_p)$. We act on $(\bfomega,\bfsigma)$ by the~$\tau$
from the preceding proposition.  The rescaling ensemble $\degenbar_* \circ R_q$
has $s_i=1$ for each $i \in(j_{k-1}, j_k)$ and moreover, for any edge~$e$
of $\Gamma_p$ joining two levels in this interval, we have $f_e=1$.
By Condition~(1) in Definition~\ref{def:collRD} the restrictions of
$(\tau\cdot\bfomega)_{(i)}$ to a neighborhood of the fiber over~$q$
agree for $i \in (j_{k-1},j_k]$ on their overlap. So we define $\bfomega_q$ to be
the collection of differentials over~$q$ obtained by this gluing for
all $k \in L^\bullet(\Gamma_q)$.
\par
The last datum to define is a prong-matching for each edge of $\Gamma_q$.
An edge~$e$ of~$\Gamma_p$ persists in~$\Gamma_q$ exactly when $q \in B_e$,
the subscheme defined by~$f_e = 0$. The prong-matching $\sigma_e$
is a section of~$\calP_e$ and as such restricts to a germ of a section
over the neighborhood of $q$ intersected with $B_e$.
\par

%%%%%%%%%%%%%%%%%%%%%%%
\subsection{The global situation.} \label{sec:globalfam}
%%%%%%%%%%%%%%%%%%%%%%%

We finally obtain global objects by patching together germs using the restriction
procedure of the previous subsection. Essentially, we mimic the definition of
sheafification of a presheaf.
\par
\begin{df} \label{def:MSD}
Given a family of pointed stable curves $(\pi\colon\calX\to B, \bfz)$ \changed{over an analytic space~$B$}, a {\em
family of \msds} of type $\mu$ over~$B$ is a collection of germs of \msds of
type~$\mu$ for every point $p \in B$ such that if the germs at~$p$ and at~$p'$ are
both defined at~$q$, their restrictions to~$q$ are equivalent germs.
\end{df}
\par
\changed{Morphisms (and hence isomorphisms) of such families are defined
in the obvious way using the germwise notion of morphisms~\eqref{eq:defmorphismMSD}.}
We usually refer to a \msd by $(\bfomega_p, \bfsigma)_{p \in B}$ or just by $\bfomega$,
suppressing~$\Gamma_p$ and~$R_p$ to simplify notation. 
\par
Given a family of \msds over~$B$ and a map $\varphi \colon B' \to B$,
we can pull back the family to a family of \msds over~$B'$ by pulling
back each germ.  For this purpose we note that rescaling ensembles and
prong-matchings have obvious pullbacks by pre-composing the maps
with~$\varphi$, and that the collection of rescaled differentials can be
pulled back as sections of the relative dualizing sheaf.  The notion
of a family of \msds can be regarded as a moduli functor
$\MSfun\colon (\text{Analytic\ spaces}) \to (\text{Sets})$
\index[teich]{h010@$\MSfun$!Functor of \msds} that associates to an
analytic space~$B$ the set of isomorphism classes of families of \msds
of type~$\mu$ over $B$.  Similarly, there is a projectivized analogue
$\PP\MSfun$.  The notion of families of \msds defines in an obvious
way a {\em groupoid $\MSgrp$} \index[teich]{h020@$\MSgrp$!Grupoid of
  \msds} that retains the information of isomorphisms (Section~X.12 of
\cite{acgh2} provides a textbook introduction, highlighting the
difference between~$\MSgrp$ and~$\MSfun$).  In Section~\ref{sec:moduli
  space} we will see that this is a Deligne-Mumford stack.
\par
\medskip
Much of the data of \msds is determined automatically in good circumstances.
The reader should keep in mind the following situation that will be a special
case of the considerations in Section~\ref{sec:triangsystem}.
\par
\begin{prop} \label{prop:mds1dim}
If~$B$ is a smooth reduced curve, then giving a \msd of type~$\mu$ on a family $\calX \to B$
without persistent nodes simply amounts to specifying a family~$\omega$ of stable
differentials of type~$\mu$ in the generic fiber, which does not identically vanish on any fiber.
\end{prop}
\par
\begin{proof}
Since $B$ is smooth and one-dimensional, Proposition~\ref{prop:normal-rescalable}
below implies that the family $(\calX\to B, \omega)$ is adjustable and
hence orderly (see Definitions~\ref{def:rescalable} and~\ref{def:orderly}
below). The claim then follows from Proposition~\ref{prop:orderly-msd}.
\end{proof}
\par
In contrast to this we observe:
\par
\begin{exa}\label{ex:rescond} ({\em Lower level differentials are not
determined by~$\omega_{(0)}$.}) If~$\calO_B$ admits a zero divisor~$s$, say $s\cdot y = 0$,
then differentials on the lower level components of a collection of rescaled
differentials with given~$\omega_{(0)}$ may be not uniquely determined. In fact,
if $\omega_{(0)} = s \omega_{(-1)}$, then we also have $\omega_{(0)} = s (\omega_{(-1)}
+ y \xi)$ for any differential~$\xi$.
\end{exa}

%%%%%%%%%%%%%%%%%%%%%%%
\subsection{Adjustable and orderly families} \label{sec:triangsystem}
%%%%%%%%%%%%%%%%%%%%%%%

In this section we analyze the ingredients of multi-scale differentials and
when their existence is automatic. The study here will be needed for the
description of the moduli space of \msds as a blowup of the normalization of
the IVC in Section~\ref{subsec:blowup}.
\par
For families of pointed stable differentials $(\pi\colon\calX\to B, \omega,
\bfz)$ considered in this section, we make a {\em standing assumption}
that $\omega$ does not vanish identically on any fiber of~$\pi$.
\par
\begin{df}\label{def:rescalable}
A family of pointed stable differentials $(\pi\colon\calX\to B, \omega, \bfz)$
is called \emph{adjustable of type $\mu$}, if for every $p\in B$ and for
every irreducible component~$X$ of the fiber $X_p$ over $p$, there exists
a family of differentials~$\changed{\eta_X}$ defined over a neighborhood of~$X$
minus the
horizontal poles and minus the intersection with the other components of~$X_p$,
such that
\begin{itemize}
\item there exists a non-zero regular function $h\in \calO_{B,p} \setminus \{0\}$
such that $\omega = h\changed{\cdot \eta_X}$,
\item \changed{for every~$X$} the differential \changed{$\eta_X$ does} not
vanish identically on~$X$,
\item and $\eta_X|_X$ has zero or pole order $m_j$ prescribed by $\mu$ at every
marked point $z_j \in X$, and has no other zeros or poles in the smooth locus
of~$X$.
\end{itemize}
\par
Such a function~$h$ is called an \emph{adjusting parameter}
\index[family]{d070@$h$, $h_{v}$, $h_i$! Adjusting parameter}
for $(\calX, \omega)$ at the component~$X$, and~\changed{$\eta_X$} is called an
\emph{adjusted differential}.
\end{df}
Later we will show that under some mild assumptions an adjustable family
naturally yields the data of a family of multi-scale differentials (see
Proposition~\ref{prop:orderly-msd}).
\par
The adjusting parameter~$h$ is not unique, since multiplying~$\eta_X$ by a
unit in $\calO_{B,p}$ and multiplying $h$ by the inverse of such a unit gives
another adjusting parameter. The following example
shows that the existence of adjusting parameters is a non-trivial condition.
\par
\begin{exa}({\em Adjusting parameters may not exist.})  Recall from
  \cite[Example~3.2]{strata} that there exist pointed stable
  differentials whose associated twisted differentials are not
  unique. Consider such a pointed stable differential $(X,\omega)$ and
  two distinct associated twisted differentials $(X,\eta_{1})$ and
  $(X,\eta_{2})$. Take two families of generically smooth stable
  differentials $(\pi_{i}\colon \calX_{i} \to B_{i},\omega_{i})$ above
  smooth curves $B_{i}$ for $i = 1, 2$, such that the adjusted
  differentials induce the twisted differentials $\eta_{i}$ at the
  points $p_{i}$. Construct a nodal curve $B$ by taking the union of
  $B_{1}$ and $B_{2}$ glued at $p_{1}$ and $p_2$. Since the fiber over~$p_1$ of the family of stable differentials $(\calX_1,\omega_1)$ coincides with the fiber of $(\calX_2,\omega_2)$ over $p_2$, we can glue $\calX_1$ and $\calX_2$ to form a family of pointed
stable differentials $(\pi\colon \calX \to B,\omega)$.  This family \changed{of
pointed stable differentials}
  is not adjustable since the adjusted differentials of the two
  branches do not coincide \changed{on the fiber of $\calX$} over the node of the base curve $B$.
\end{exa}
\par
Next we show that if the base is sufficiently nice, adjusting parameters
do exist.
\par
\begin{prop}
  \label{prop:normal-rescalable}
  \changed{Given any family $(\calX, \omega)$ satisfying the standing assumption, if the base $B$ is normal and if $\omega$ does not vanish identically on any irreducible component of $\calX$, then $(\calX, \omega)$ is adjustable.}
%If the base $B$ is normal, then any family $(\calX, \omega)$ satisfying the standing assumption is adjustable. 
Moreover, any two adjusting parameters
for a given point~$p \in B$ and a given irreducible component~$X$ of $X_p$ differ by
multiplication by a unit in $\calO_{B,p}$.
\end{prop}
\par
We first recall some terminology. Denote $\calZ = \sum_{j=1}^n m_j \calZ_j$ a divisor on~$\calX$, where $\calZ_j\subset\calX$ is the image of the section of the~$j$-th zero or pole $z_j$ of~$\omega$. We call an effective Cartier divisor $V\subset\calX$ a {\em vertical} divisor if the image $\pi(V)\subset B$ is a divisor. Note that any section~$B\to\calX$ is not vertical because it projects to all of~$B$. In particular, the divisors~$\calZ_j$ and $\calZ$ are not vertical. A vertical divisor is called a
vertical zero divisor of~$\omega$ if it is contained in the zero locus of~$\omega$ (and being vertical ensures that it is not contained in the support of~$\calZ$).
\par
\begin{proof}
Suppose~$\omega$ vanishes identically on an irreducible component~$X$
of the fiber $X_p$ for some $p\in B$. Then~$X$ is contained in the
vertical zero divisor of~$\omega$. More precisely, let $W\subset
\calX$ be a small neighborhood of \changed{a smooth point of $X\setminus (X_p\setminus X)$}, and let $U = \pi(W) \subset B$ be the corresponding neighborhood of $p$ in $B$. Then $W\cong U \times \Delta$ where $\Delta$ is a disk. Let $V\subset W$ be the vertical zero divisor of~$\omega$ in~$W$, i.e.
$V$ is the zero divisor of~$\omega$ regarded as a holomorphic section of the twisted dualizing line bundle $\omega_{\calX/U}(-\calZ)$ restricted to $W$, so that
in particular $V$ contains the generic point of $X$. Since $V$ is the zero locus of a holomorphic section of a line bundle, it is an effective Cartier divisor (possibly reducible and non-reduced), and we denote by $h\in \calO_W \cong \calO_{U \times \Delta}$ the local defining equation of $V$.
\par
We claim that $h$ does not depend on the second factor $\Delta$, i.e.~$h$ can be regarded as a function defined on the base $U$.
If this were not the case, then for a generic point $b\in U$ we would be able to solve the equation $h(b, x)=0$, but then~$V$ would map onto~$U$, which contradicts that $V$ is a vertical divisor. We thus conclude that $h\in \calO_U$.
\par
Let $W'\subset W$ be the smooth locus of $W$. Then $(h^{-1} \omega)|_{W'}$ is holomorphic and can have horizontal zeros only, as the vertical zero divisor $V$ is canceled out by $h^{-1}$. Since $U\subset B$ is normal, it implies that $W \cong U \times \Delta$ is normal and the singular locus $W\setminus W'$ has codimension two or higher. By Hartogs' theorem, $(h^{-1} \omega)|_{W'}$ extends to~$W$ holomorphically and can still have horizontal zeros only, as the zero locus of $h^{-1} \omega$ must be of codimension one (if not empty). It implies that
the zero locus of $h^{-1} \omega$ in $W$ does not contain the generic point of $X$, and hence $(h^{-1} \omega) |_{X\cap W}$ is holomorphic and not identically zero. Since $X\cap W$ contains the generic point of~$X$, it follows that $h^{-1}\omega$ does not vanish identically on~$X$. Thus $h$ is the desired adjusting parameter for~$X$ and $\eta = h^{-1}\omega$ is the corresponding adjusted differential.
\par
Suppose that $h_1$ is another adjusting parameter for~$X$.  Note that $h_1^{-1}\omega = (h/h_1) \eta$.
If $h/h_1$ has a zero or pole at $p$, then $h_1^{-1}\omega$ would have a zero or pole along the entire~$X$, which contradicts the definition of adjusting parameter. We thus conclude that any two adjusting parameters for~$X$ differ by multiplication by a unit in $\calO_{B,p}$.
\end{proof}
\par
In the algebraic situation, we show that adjusting parameters exist
\'etale locally, which will be used in the proof of Theorem~\ref{thm:LMS-blowup}
as a step towards the algebraicity of the moduli space of multi-scale differentials.
\par
\begin{prop}\label{prop:alg-rescalable}
Under the assumptions of Proposition~\ref{prop:normal-rescalable}, if moreover
the family $(\calX \to B, \omega)$ is algebraic with~$B$ irreducible and smooth
generic fiber, then there exists  an \'etale
base change $B' \to B$ and a preimage~$p'$ of~$p$ such that the adjusting
parameters for the pullback family $(\calX' \to B', \omega')$ at~$p'$ can
be chosen in the algebraic local ring $\calO^{\alg}_{B',p'}$.
\end{prop}
\par
\begin{proof} The existence of a function in the local ring after an \'etale base
change of $R = \calO^{\alg}_{B,p}$ is by definition equivalent to finding such
a function in the (strict) Henselization $R^h$ of~$R$. Let $h_{\an} \in \calO^{\an}_{B,p}$
be an analytic adjusting parameter for an irreducible component $X$ of $X_p$
provided by Proposition~\ref{prop:normal-rescalable}. We view $h_{\an}$ as an
element of the local ring completion $\widehat{R}$. The proof consists of two steps.
First we show that there exists an algebraic function~$h \in R$ such that
$h_{\an}$ divides~$h$, as elements of~$\widehat{R}$. Secondly,
we show that for any factorization $h = h_1 \cdot h_2$ in~$\widehat{R}$ there exist
$h_1', h_2' \in R^h$ with $h = h_1' \cdot h_2'$ and such that $h_i/h_i'$ is a unit
in~$\widehat{R}$. Applying this to $h_1 = h_{\an}$ gives the result by taking $h'_1$
as the desired (algebraic) adjusting parameter.
\par
For the first claim, note that the vertical divisor~$V$ (as in the proof of
Proposition~\ref{prop:normal-rescalable})
is contained in the locus of singular fibers~$S \subset \calX$, which is
the $\pi$-preimage of a Cartier divisor $D \subset B$. Let~$m$ be the vanishing
order of~$V$ at a generic point near the component~$X$ under consideration.
Then $h_{\an}$ divides the defining equation~$h$ of $mS$,
where~$h$ can be regarded as a function in~$R$ defining the Cartier
divisor~$mD$.
\par
For the second claim, consider the factorization as decomposition of the
associated Cartier divisor $D = D_1 \cup D_2$, where $D = V(h)$ and
$D_i = V(h_i)$, in $\Spec(\widehat{R})$. Note that a prime ideal of~$R^h$
remains prime after lifting in $\widehat{R}$, as a consequence of Artin approximation
(see e.g.\ \cite[Section~3A, Proposition~1.9]{HandAlg2}). It follows that the
decomposition of~$D$ induces a decomposition $D = D_1' \cup D_2'$
in $\Spec(R^h)$ into (a priori Weil) divisors with $D_i' = V(I_i')$ and
$I_i' \otimes_{R^h} \widehat{R} = \langle h_i \rangle$. Since $R^h \to
\widehat{R}$ is a faithfully flat morphism, this implies that $I_i'$
is locally free of rank one, too, and their generators $h_i'$ are the
elements we wanted to construct.
\end{proof}
\par
Suppose for the rest of this section that~$(\calX, \omega)$ is an adjustable
family of differentials over~\changed{an analytic space} $B$. Given $p\in B$, let~$V_p$
be the quotient of the set $\calO_{B,p} \setminus \{0\}$ by
the multiplicative group of units $\calO_{B,p}^\times$. The divisibility relation
induces a partial order on $V_p$, and we write $h_2\preccurlyeq h_1$
if $h_1\mid h_2$.  For each fiber $X_p$ the structure of the
family near~$p$ can be encoded by decorating the dual graph $\Gamma_p$.
We assign to each edge corresponding to a non-persistent node the germ
$f_e\in V_p$, where $uv=f_e$ is a model for the family near the node
represented by~$e$.  We assign to each vertex the function $h_v\in V_p$,
where $h_v$ is an adjusting parameter for the family at the component
represented by~$v$. We emphasize that each of the functions~$f_e$ and~$h_v$ is
only defined up to multiplication by some unit in the local ring.
\par
The vertices of $\Gamma_p$ have the usual partial order as defined in
Section~\ref{sec:deftwd}. This partial order can be understood also in terms
of the divisibility relation on the set of~$h_v$.
Suppose an edge~$e$ connects two vertices~$v$ and~$v'$.  Then the edge~$e$ is
horizontal if $h_{v}\asymp h_{v'}$, and vertical otherwise, with $v \preccurlyeq v'$
if and only if $h_v \preccurlyeq h_{v'}$.  In this case, we in fact have
\begin{equation}
  \label{eq:h_ratio}
  h_v/h_{v'} \= f_e^{\kappa_e}
\end{equation}
 as shown in the proof of Theorem~\ref{thm:NF}.
\par
In general, the divisibility relation among the $h_v$ may not be a full order,
because for two vertices~$v$ and~$v'$ that are {\em not} connected by an edge,
it can happen that $h_v$ and~$h_v'$ do not divide each other (see e.g.\
Example~\ref{ex:cherry}).
We will be especially interested in families for which it is a full order.
\par
\begin{df}
  \label{def:orderly}
  An adjustable family $(\pi\colon\calX\to B, \omega)$ is called
  \emph{orderly} if for every point $p\in B$,
the divisibility relation induces a full order on the set of adjusting
parameters $\left(h_v\right)_{v\in \Gamma_p}$.
\end{df}
\par
After these preparations we will now show that all the ingredients in the
definition of a family of \msds can be read off from an orderly family,
except possibly missing a compatible rescaling ensemble, whose existence
can be further guaranteed when the base of the family is normal.
\par
\begin{prop}
  \label{prop:orderly-msd}
An orderly family $(\pi\colon\calX\to B, \omega,\bfz)$ over a normal
base~$B$ determines an enhanced level graph, a collection of rescaled differentials
of type $\mu$, a collection of prong-matchings for every $p\in B$ and a
compatible rescaling ensemble as described in Definition~\ref{def:germMSD}.
Namely, such $(\calX, \omega,\bfz)$ determines a unique family of multi-scale differentials
of type~$\mu$.
\end{prop}
\par
\begin{proof}
The divisibility order of the family $(\calX, \omega,\bfz)$ gives the dual graph
$\Gamma_p$ the structure of a level graph, which we normalize so that the level
set is~$\uN$. For each level~$i$, we denote by $h_i$ the adjusting parameter
for some arbitrarily chosen vertex of level~$i$.
\par
Define the germs of holomorphic functions~$s_i \in \calO_{B,p}$ by
$s_0=h_0=1$ and
\begin{equation} \label{def:ti_from_orderly}
  s_i \,:=\, h_{i}/h_{i+1}
\end{equation}
for all $i<0$.  For each~$i$, define the germ of a family of differentials
\begin{equation}
  \label{eq:rescaled_def}
  \omega_{(i)}:= \omega/(s_0 \dots s_i) = \omega/h_{i}
\end{equation}
which is generically holomorphic and non-zero on each level~$i$ component of $X_p$,
vanishes identically on all lower lever components, and has poles along each higher
level component.  For an edge $e$ of $\Gamma_p$ joining levels $j < i$, the pole
order of $\omega_{(j)}$ (minus one) at the corresponding node determines the
enhancement $\kappa_e$. Moreover, the local normal form expressions of  $\omega_{(i)}$
and $\omega_{(j)}$ as in \eqref{eq:5} follow from Theorem~\ref{thm:NF}.
The coordinates $u_e, v_e$ putting the differential in the normal form can also be used to define the
prong-matching $\sigma_e = du_e\otimes dv_e$ at $e$. We thus conclude that
the $\omega_{(i)}$ give a collection of rescaled differentials of type~$\mu$ at~$p$
with the~$s_i$ as rescaling parameters as in Definition~\ref{def:collRD}.
\par
We will show the existence of a compatible rescaling ensemble $R\colon B\to \Tnorm[\Gamma_p]$ in three steps.  First,
as a consequence of Theorem~\ref{thm:NF}, a map $R'\colon B \to \overline{T}'_{\Gamma_p}$ can always be found by using the tuples $s_i$ and $f_e$ as above such that they satisfy \eqref{eq:2}, where $\overline{T}'_{\Gamma_p}$ denotes the entire torus cut out by Equation \eqref{eq:rrho}. Next, we want the image of $R'$ to lie in the desired connected component $\overline{T}_{\Gamma_p}$ of  $\overline{T}'_{\Gamma_p}$, and this can be done as follows.  The torus $\overline{T}'_{\Gamma_p}$ has a map to $(\CC^*)^{N}$ by projection, which is an isogeny since these are two tori of the same dimension. Choose in each connected component of $\overline{T}'_{\Gamma_p}$ an element in the kernel of this projection. Note that modifying the tuples $s_i$ and $f_e$ by the chosen kernel elements does not change the rescaling parameters $s_i$, but it changes the $f_e$, and thus by choosing the suitable kernel elements the whole collection can be made to lie in the connected component $\overline{T}_{\Gamma_p}$. Finally, if the base $B$ is normal, the map $R'\colon B\to \overline{T}_{\Gamma_p}$ automatically factors through the normalization of the level rotation torus closure, by the universal property of normalization, and thus gives the rescaling ensemble $R\colon B\to\Tnorm[\Gamma_p]$.
\end{proof}

%%%%%%%%%%%%%%%%%%%%%%%%%%%%%%%%
\section{Real oriented blowups}
\label{sec:rob}
%%%%%%%%%%%%%%%%%%%%%%%%%%%%%%%%

The goal of this section is to define a canonical real oriented blowup
for a family of \msds; see Section~\ref{sec:weld} for a discussion of
welding in terms of the real oriented blowup for the case of one Riemann surface. This construction will be used in
Section~\ref{sec:SL2R}, where we will show that the action of $\SLtwoR$
extends naturally to the real oriented blowup of the moduli space of
\msds along its boundary. This blowup is also used to define families of
marked \msds in Section~\ref{sec:markfamily}.  In fact there are two
real oriented blowups, associated with a rescaling ensemble~$R$ and with
a simple rescaling ensemble~$R^s$.
\par
In Section~\ref{sec:markedauxds} we develop the parallel case of families of marked model differentials.

%%%%%%%%%%%%%%%%%%%%%%%%%%%%%%%%
\subsection{The real blowup construction}
\label{sec:robconstruction}
%%%%%%%%%%%%%%%%%%%%%%%%%%%%%%%%

We start with the local version, which only depends on the rescaling ensemble.
\par
\begin{prop} \label{prop:ORBL}
Let $\pi\colon\calX\to B_p$ be a germ of a family of curves with
a rescaling ensemble~$R$. Then there exists the (local) \emph{\lw real blowup},
which is a map $\widehat{\pi}\colon \widehat{\calX}\to \widehat{B}_p$
of topological spaces with the following properties:
\index[family]{b050@$\widehat{\pi}\colon \widehat{\calX}\to \widehat{B}$!Level-wise real blowup of a family of \msds}
\begin{itemize}
\item[(i)] There are surjective \changed{continuous} maps
$\varphi_\calX \colon \widehat{\calX}\to \calX$ and $\varphi_B \colon
\widehat{B}_p\to B_p$, \changed{which are diffeomorphisms away from the nodal locus},
such that $\pi \circ \varphi_\calX = \varphi_B \circ
\widehat{\pi}$.
\item[(ii)] All fibers of $\widehat{\pi}$ are vertically welded
  surfaces (in the sense of Section~\ref{sec:weld}).
\item[(iii)] The fiber of $\varphi_B$ over the point~$p \in B_p$ is
a disjoint union of tori isomorphic to~$(S^1)^{L(\Gamma_p)}$.
\end{itemize}
\par
Moreover, the \lw real blowup is functorial under pullbacks
via maps $B_p' \to B_p$ of the base.
\end{prop}
\par
Note that the \lw real blowup does not modify the neighborhoods
of horizontal nodes. Hence in general the fibers of $\widehat{\pi}$ remain
nodal.
\par
The fibers of $\varphi_B$ are connected if $\pi$ has no vertical
persistent nodes, but may not be connected in general. The
prong-matching singles out a specific connected component
in each fiber of $\varphi_B$. We now perform the above construction
globally.
\par
\begin{thm} \label{thm:ORBL}
A family of \msds~$(\bfomega,\bfsigma)$ on $\pi\colon\calX\to B$ singles
out a connected component $\ol{B_p}$ of the local \lw real blowup
$\widehat{B_p}$ for each germ~$B_p$. We denote by $\overline{\pi}_p \colon
\overline{\calX} \to \overline{B_p}$ the restriction of the local
\lw real blowup to $\overline{B_p}$.
\par
These germs glue to a global surjective differentiable map
$\overline{\pi}\colon \overline{\calX} \to \overline{B}$,
the {\em (global) \lw real blowup}.
\index[family]{b060@$\ol{\pi}\colon \ol{\calX}\to \ol{B}$! Global \lw
real blowup}
Moreover, the global \lw real blowup is functorial under pullbacks
via maps $B' \to B$ of the base.
\par
If $B$ is a manifold, then $\overline{B}$ is a manifold with corners.
\end{thm}
\par
Our construction is closely related to a number of real oriented blowup
constructions that appear in the literature, e.g.\  the Kato-Nakayama blowup
of a log structure \cite{KatoNakayama}, see also  \cite{Kato} and
\cite{ACGHOSS}. The distinguishing feature here is that the blowup is
determined by the level structure of \msds.
\par
\begin{proof}[Proof of Proposition~\ref{prop:ORBL}]
The rescaling ensemble $R$ gives a collection of rescaling and smoothing
parameters $(s_i, f_e)_{i\in L(\Gamma_p), e\in E(\Gamma_p)^v}$
which are germs of functions on $B_p$. 
We introduce for each of the variables an $S^1$-valued partner variable,
denoted by the corresponding capital letter. Concretely, we define
$\widehat{B}_p \subset B_p \times (S^1)^{\vertedge[\Gamma_p]} \times
(S^1)^{L(\Gamma_p)}$ by the equations
\begin{equation}\label{eq:defhatB}
F_e|f_e| \= f_e\,, \quad S_i|s_i| \= s_i\,, \quad
\text{and} \quad F_e^{\kappa_e} \= S_j\dots S_{i-1}\,,
\end{equation}
where $\kappa_e$ is the enhancement at the edge~$e$ joining
levels $j<i$ of $\Gamma_p$. Note that these equations still make
sense if some $s_i$ (or $f_e$) is identically zero, in which case~$S_i\in S^1$ (resp.~$F_e$) is an independent variable, not related to $s_i$ or $f_e$. The map $\varphi_B:\widehat{B}_p \to B_p$ is then given by
the projection onto the first factor.
\par
Next we define the family $\widehat{\pi} \colon \widehat{\calX} \to \widehat{B}$
as follows. Near a smooth point in the fiber~$X_p$  we simply pull back a
neighborhood via $\varphi_B$. In the neighborhood $Y$ of a vertical node~$q_{e}$ given by the equation $u_{e}^{+}u_{e}^{-}=f_{e}$, we define
\index[family]{b015@$u_{e}^{+}u_{e}^{-}=f_{e}$!Local equation of a family near the node  $q_{e}$}
$\widehat{Y} \subset \varphi_B^*(Y) \times (S^1)^{2}$ by
\begin{equation}\label{eq:defhatX}
U_e^\pm |u_e^\pm| \= u_e^\pm \quad \text{and} \quad U_e^+ U_e^- \= F_e\,.
\end{equation}
The fibers of $\widehat{\pi}$ are not yet smooth in a neighborhood
of the preimages of the vertical nodes (as can be seen by computing the
Jacobian matrix of the defining equations), but we are in the
setting of~\cite[Section~X]{acgh2}, see in particular p.~154 \changed{and
Proposition~9.16.} There it is shown that \changed{the topological spaces
underlying~$\widehat{\pi}$ can be provided with the structure of a  fibration
of real-analytic manifolds which agrees with the initial analytic structure
outside the blown-up subspaces.}
%\bes
%(u_e^\pm, U_e^\pm) \mapsto (\changed{|u_e^\pm|-|u_e^\mp|,|u_e^\pm u_e^\mp|} ,U_e^\pm)\,=:\,
%(r,s,U_e^\pm)
%\ees
%is a map from a real-analytic manifold to a real-analytic manifold with corners
%(stemming from the boundary of the base $r=0$) that admits an inverse which is
%however merely continuous. The pullback of the analytic structure on the
%target provides the fibers of~$\widehat{\pi}$ with a smooth real analytic
%structure away form the horizontal nodes that can be checked to agree with
%the one of a welding.
\par
The functoriality of this construction is obvious.
\end{proof}
\par
\begin{proof}[Proof of Theorem~\ref{thm:ORBL}]
  In view of Remark~\ref{rem:ptwdismsd}, Corollary~\ref{cor:compLRT} and the
  definition of $\hat B$ given in Equation~\eqref{eq:defhatB} imply
the first claim.
\par
Suppose the germs at~$p$ and at~$p'$ are both defined at~$q$
and differ there by the action of $(r_i,\rho_e) \in T_{\Gamma_q}$.
Then multiplying $S_i$ by $r_i/|r_i|$ and $F_e$ by $\rho_e/|\rho_e|$
provides the identification of the additional parameters of the
\lw real blowup.
\end{proof}
\par
In the special case that all nodes of~$\pi$ are persistent, the base
$\ol{B}$ of the \lw\ real blowup is isomorphic to $B \times (S^1)^{L(\Gamma_p)}$, with parameters $\bfS = (S_i)$.
Denoting by $\theta(\bfS)$ the argument of the $S_i$ divided by~$2\pi$, the fiber
of $\overline{\pi}$ over a point $(p,\bfS) \in \ol{B}$ is simply the surface~$\calX_p$
welded according to the
\prma $\theta(\bfS) \cdot \bfsigma$, where this map  is  defined in~\eqref{eq:ConPMD2}.
This also justifies the use of overlines for both constructions.
\par

%%%%%%%%%%%%%%%%%%%%%%%%%%%
\subsection{Families of marked \msds}
\label{sec:markfamily}
%%%%%%%%%%%%%%%%%%%%%%%%%%%

We aim to define a marked version of families of \msds. The general strategy
is that we only mark families of vertically welded surfaces. We get rid
of persistent vertical nodes by welding and we remove non-persistent vertical
nodes using the \lw real blowup. The following construction of marking
appears also in \cite[Section~5]{HubKoch}, for curves without a differential.
\par
Let $\bfs \subset \Sigma$ be a collection of~$n$ points on a
topological surface~$\Sigma$. Let $(\pi\colon \calY \to B,\bfz)$ be a pointed
family of vertically welded surfaces.
We define the \emph{presheaf of markings}
${\rm Mark}(\calY/B)$ by associating with an open set~$U \subseteq B$
the set of \diffeogens $\Sigma \times U \to \pi^{-1}(U)$ respecting
the marked sections~$\bfs$ and~$\bfz$, up to isotopies over~$U$.
A {\em marking $f$ of the family~$\pi\colon\calY\to B$} is a global section
of the sheaf associated with ${\rm Mark}(\calY/B)$, i.e.\ a compatible
collection of $f_{U_i} \in {\rm Mark}(\calY/B)(U_i)$ for sets~$U_i$ that
cover~$B$.
\par
For any fixed subgroup~$G$ of the mapping class group $\Mod$ we similarly define
the {\em presheaf of $G$-markings ${\rm Mark}(\calY/B; G)$} by enlarging the
equivalence relation (from merely isotopies) to include pre-composition of the
diffeomorphisms by an element in~$G$. A {\em $G$-marking $f$} of~$\pi$ is
a global section of the sheaf associated with ${\rm Mark}(\calY/B; G)$.
\par
\smallskip
We can now define the marked version of families of \msds. It starts with germs and
glues them by sheafification, as in the unmarked case. Let $\Lambda$
be an enhanced multicurve and $\Gamma = \Gamma(\Lambda)$ the underlying
enhanced level graph.
\par
\begin{df} \label{def:markedmsds}
Given a family of pointed stable curves $(\pi\colon\calX\to B, \bfz)$
and $B_{p}$ a germ of~$B$ at~$p$, the {\em germ of a $\Lambda$-marked
family of \msds} of type $\mu$ over $B_{p}$ is an equivalence class
of the following set of data:
\begin{itemize}
\item[(i)] a germ $(\pi\colon\calX\to B, \bfz, \Gamma_{p}, \bfomega',
\bfsigma')$ of a family of \msds, and
\item[(ii)] a $\Tw[\eL]$-marking~$f$ of the \lw\ real
blowup~$\overline{\pi}\colon \overline{\calX} \to \overline{B}_p$.
\end{itemize}
The level rotation torus $T_{\Gamma_p}(\calO_B)$ acts on all of the above
data (see~\eqref{eq:defFD} for the action on the marking), and we consider
two germs~\emph{equivalent} if they differ by the action of an element
in $T_{\Gamma_p}(\calO_B)$.
\par
A morphism between two germs  $(\calX', \bfomega', \bfsigma', f')$ and
$(\calX, \bfomega, \bfsigma, f)$ is a morphism $(\phi,\widetilde{\phi})$
of the underlying \msds such that the induced map $\overline{g}
\colon \overline{\calX'} \to \overline{\calX}$ commutes with the
$\Tw[\eL]$-marking, up to an isotopy respecting the marked points.
\end{df}
\par
If $B$ is a (reduced) point, then $\overline{B}$ is the $\arg$-image of the
level rotation torus,  and a marked \msd is a family of markings of the family
of welded surfaces over $\overline{B}$.
\par
Given a map $\psi \colon B' \to B$, the functoriality
of the \lw\ real blowup allows to define the pullback of markings
along $\psi $ by pulling back the germ of the family as in
Section~\ref{sec:famnew} and by restricting the markings along the induced
map $\overline{\psi} \colon \overline{B'} \to \overline{B}$.
In this way we define a \emph{family of $\Lambda$-marked \msds} by
sheafification, just as in Definition~\eqref{def:MSD}.
We have thus defined a moduli functor $\MSfmark$ of $\Lambda$-marked \msds.
\index[teich]{h030@$\MSfmark$!Functor of marked \msds}
The notion of a family of $\Lambda$-marked \msds and this functor has an obvious
projectivized version, denoted by $\PP\MSfmark$.
\par

%%%%%%%%%%%%%%%%%%%%%%%%%%%
\subsection{Families of simple marked \msds}
\label{sec:simplemarkfamily}
%%%%%%%%%%%%%%%%%%%%%%%%%%%

There is a similar definition of the notion of a family of simple marked
\msds of type~$\mu$. We aim to require a $\sTw[\eL]$-marking rather
than merely a $\Tw[\eL]$-marking but this requires to change the blowup.
Suppose $\pi\colon\calX\to B$ is a germ of a family of curves with
a simple rescaling ensemble~$R^s$. Then a version of Proposition~\ref{prop:ORBL}
holds verbatim, but the torus in its point~(iii) is the real torus
associated with the simple level rotation torus rather than associated
with the level rotation torus. To prove this version of the proposition we introduce
for each of the variables of $T_\Lambda$ an $S^1$-valued partner variable,
denoted by the corresponding capital letter. Concretely, we define
$\widehat{B} \subset B \times (S^1)^{L(\Gamma_p)}$  by the equations
\begin{equation}\label{eq:defhatBs}
T_i\cdot |t_i| \= t_i \,.
\end{equation}
(This torus is a finite cover of the one given in~\eqref{eq:defhatB}.
The map is given by letting $S_i = T_i^{a_i}$ and $F_e\= T_{j}^{m_{e,j}}\dots
T_{i-1}^{m_{e,i-1}}$, if~$e$ is an edge connecting levels $j<i$ and where
the~$m_{e,i}$ were defined in \eqref{eq:aidef}.
The remaining construction of the smooth family over this blowup is
the same as above.)  This procedure should properly be called simple \lw\ real blowup, but the
context (the rescaling ensemble) will always make it clear which version we use.
\par
The zoo of definitions given so far culminates in the following, the
moduli functor that will turn out to be indeed represented by a smooth space.
\par
\begin{df} \label{def:markedsimplemsds}
Given a family of pointed stable curves $(\pi\colon\calX\to B, \bfz)$
and $B_{p}$ a germ of~$B$ at~$p$, the {\em germ of a family of simple
  $\Lambda$-marked
\msds} of type $\mu$ over $B_{p}$ is an equivalence class of the
following set of data:
\begin{enumerate}
\item the structure of an enhanced level graph on the dual graph $\Gamma_p$
of the fiber $X_p$,
\item a simple rescaling ensemble $R^s \colon B \to \ol{T}_{\Gamma_p}^{s}$,
compatible with
\item a collection of rescaled differentials $\bfomega = (\omega_{(i)})_{i \in
L^\bullet(\Gamma_p)}$ of type~$\mu$, and
\item a collection of prong-matchings $\bfsigma = (\sigma_e)_{e \in E(\Gamma)^v}$.
For the non-semipersistent nodes, these are required to agree with the
induced prong-matchings defined before Definition~\ref{def:germMSD}.
\item a $\sTw[\eL]$-marking~$f$ of the simple \lw\ real
blowup~$\overline{\pi}\colon \overline{\calX} \to \overline{B}$.
\end{enumerate}
The simple level rotation torus $T^s_{\Gamma_p}(\calO_B)$ acts on all of the
above data and we consider two germs~\emph{equivalent} if the differ by
the action of an element in $T^s_{\Gamma_p}(\calO_B)$.
\end{df}
\par
With the obvious definition of morphisms, pullbacks and sheafification,
this defines the notion of a \emph{family of simple marked \msds}.
This defines a functor that we denote by $\sMSfmark$ and its projectivized
variant by $\PP\sMSfmark$.
\par
\medskip
The group $K_\Gamma$ acts on germs (and on families where the level graph
is an undegeneration of~$\Gamma$) by post-composing the marking~$f$
with the given element in~$K_\Gamma$. The quotient functor
is exactly the functor of (non-simple) marked \msds, since in the presence
of just the resulting quotient~$\Tw[\eL]$-marking a simple rescaling
ensemble up to the equivalence relation generated by $T^s_{\Gamma_p}(\calO_B)$
is the same as a rescaling ensemble up to the  equivalence relation
generated by $T_{\Gamma_p}(\calO_B)$.
\par
The following proposition will be used in Sections~\ref{sec:UnivDehn}
and~\ref{sec:moduli space} to prove the universal property of the Dehn space
and of the moduli space of \msds.
\par
\begin{prop} \label{prop:markliftedfam}
For any family of \msds~$(\pi\colon \calX \to B, \bfz, \bfomega,\bfsigma)$
and any $p \in B$, for any \changed{enhanced} multicurve~$\eL$ such that $\Gamma(\eL)$ is a degeneration of~$\Gamma_p$, there exists a neighborhood~$U$ of~$p$ such that $\pi|_U$ can
be provided with a $\Tw[\eL]$-marking.
\par
If the family admits a simple rescaling ensemble~$R^s$, then there exists
a neighborhood~$U$ of~$p$ such that $\pi|_U$ can
be provided with a $\sTw[\eL]$-marking.
\end{prop}
\par
\begin{proof}
We need to provide the \lw real blowup~$\overline{\pi}|_U$ with a
$\Tw[\eL]$-marking~$f$. For this purpose we take~$U$ to be simply connected,
provide some fiber of $\overline{\pi}$ with a marking and transport the
marking along local smooth trivializations of $\overline{\pi}$. We only
need to make sure that the monodromy in this process is contained
in $\Tw[\eL]$. By the choice of~$U$, and since by Theorem~\ref{thm:ORBL}
the fibers of $\ol{U} \to U$ are ($\arg$-images of) level rotation tori,
the monodromy is generated by level rotation. From the definition of level
rotation tori at the beginning of Section~\ref{sec:levrottori},
it is now obvious that the  monodromy is $\Tw[\eL]$.
\par
The second statement follows in the same way using the simple version
of the real blowup.
\end{proof}

%%%%%%%%%%%%%%%%%%%%%%%%%%%%%%
\subsection{Families of marked model differentials}
\label{sec:markedauxds}
%%%%%%%%%%%%%%%%%%%%%%%%%%%%%%

To highlight similarities and differences, and for further use, we now define (the easier) families of marked and simple marked model differentials. This will be used to verify the relevant universal properties of the model domain in Section~\ref{sec:UnivDehn}.

Recall that families of model differentials
are constrained to be equisingular, but as a trade-off they carry for each
level an additional parameter~$t_i$ that is allowed to be zero, thus mimicking
degenerations. While for families of \msds we needed to start with
a germwise definition to be able to control degeneration, here we can give
the global definition right away.
\par
While \msds are based on a collection of rescaled differentials, the
simpler notion of a model differential is based on the simple notion of
twisted differentials. We adapt the definition from Section~\ref{sec:deftwd}
to families.
\par
\begin{df}\label{df:famtwisteddiffs}
A \emph{family of twisted differentials $\bfeta$ of type~$\mu$} on
an equisingular family~$\pi \colon \calX \to B$ of pointed stable
curves compatible with~$\Gamma$ is a collection of families of
meromorphic differentials~$\eta_{(i)}$ on the subcurves~$\calX_{(i)}$
at level~$i$, which satisfies the obvious analogues of the conditions in Section~\ref{sec:deftwd}, interpreting the residues as regular
functions on the base~$B$.
\end{df}
\par
\begin{df} \label{def:aux_germ}
Let $(\pi\colon\calX\to B, \bfz)$ be an equisingular family of pointed stable
curves. A {\em family of $\Lambda$-marked simple  \auxds of type~$\mu$ over $B$}
is an equivalence class of the following set of data:
\par
\begin{enumerate}
\item the structure of an enhanced level graph on the dual graph~$\Gamma$
of any fiber of~$\pi$,
\item a \changed{map $R^{s} \colon B \to \ol{T}_{\Gamma}^s$, called simple rescaling ensemble,}
\item a collection $\bfeta = (\eta_{(i)})_{i \in L^\bullet(\Gamma)}$ of families of
twisted differentials of type~$\mu$ compatible with~$\Gamma$,
\item a collection $\bfsigma = (\sigma_e)_{e \in E(\Gamma)^v}$ of \prmas
for~$\bfeta$,
\item a $\sTw[\eL]$-marking~$f$ of the \lw\ (simple) real oriented
blowup~$\overline{\pi}\colon \overline{\calX} \to \overline{B}$
defined using~\eqref{eq:defhatBs}.
\end{enumerate}
The simple level rotation torus $T^s_{\Gamma}(\calO_B)$ acts on the above
data by \changed{its natural action on~$(\bfeta,\bfsigma)$ and by
  postcomposing~$R^s$
with the multiplication map by the inverse.
Two elements in the same $T^s_{\Gamma}(\calO_B)$-orbit} are defined to be
\emph{equivalent}.
\par
Replacing $T_{\Gamma}^s(\calO_{B})$  with the extended level rotation
torus $\Tsextd[\Gamma](\calO_{B})$, the analogous object is called
a {\em family of marked simple projectivized \auxds of type~$\mu$ over~$B$}.
\end{df}
\par
The notion of a morphism is derived from morphisms of pointed stable
curves as in~\eqref{eq:defmorphismMSD}. We denote the functor of model
differentials by $\sMDfmark$  and its projectivized version by $\PP\sMDfmark$.
\index[teich]{h040@$\sMDfmark$!Functor of  \auxds}
\par
\begin{rem} \label{rem:namemodeldiff}
Since a simple rescaling ensemble is simply a collection  of
functions $\bft = (t_i)_{i \in L(\Gamma)}$  in $\calO_{B}$,  we will denote
a family of simple marked \auxds interchangeably by the representatives
$(\bfeta, R^{s},\bfsigma, f)$ or by $(\bfeta, \bft,\bfsigma, f)$
of the $T^s_{\Gamma}(\calO_{B,p})$-orbits. \changed{Obviously here we do not require
any compatibility of the rescaling ensemble with the local equations at the
nodes, as the family is equisingular.}
\end{rem}
\par
Similarly we can define (non-simple) marked \auxds by taking the
$K_\Gamma$-quotient or by using non-simple rescaling examples; we can also define the unmarked versions. Since these other versions will not be needed for what follows, we do not give the details.

%%%%%%%%%%%%%%%%%%%%%%%%%%%%%%%%%%%%
\section{The universal property of the Dehn space}
\label{sec:UnivDehn}
%%%%%%%%%%%%%%%%%%%%%%%%%%%%%%%%%%%%%%%%%%%%

The purpose of this section is to show the following two results.
\par
\begin{thm} \label{thm:univnonsimpleMMS}
\changed{For any enhanced level graph~$\Lambda$}
the Dehn space $\ODehn$ is the fine moduli space, in the category
of complex analytic spaces, for the functor~$\MSfmark$
of marked \msds.
\end{thm}
\par
To obtain this, we first prove the simple marked version of this
statement, and then descend by the $K_\Lambda$-action.
\par
\begin{prop}  \label{prop:univMMS}
\changed{For any enhanced level graph~$\Lambda$}
the simple Dehn space $\ODehns$ is the fine moduli space, in the category
of complex analytic spaces, for the functor~$\sMSfmark$
of simple marked \msds.
\end{prop}
\par
\changed{Since the complex structure on $\ODehns$ stems from the model domain,
we will first establish the universal property of the model domain. There, the
candidate for the universal family has been given in
Section~\ref{sec:universalMD}, over an open cover of the simple model domain.
In Proposition~\ref{prop:sMDuniv} we will deduce the universal property of the simple model domain
from the obvious universal properties of the \Teichmuller space and the associated
vector bundle. Here and for the functor~$\sMSfmark$ we use the fact that a functor that is a sheaf, and which has an open cover by representable subfunctors,
is representable. The sheaf property for model differentials is obvious, while for multi-scale differentials it has been established in
Section~\ref{sec:presheafmsd}. It thus suffices to fix one of the open
patches~$\calW_\epsilon \times \Delta^H$ on which the map $\OPl$ was defined
in Section~\ref{sec:Dehn}, 
and then to establish the universal property for a given family $\pi\colon\calY \to B$
of stable curves, endowed with a family of simple marked \msds $(\bfomega, \bfsigma, f)$,
such that all fibers are isomorphic to a fiber in the plumbed family over $\calW_\epsilon \times \Delta^H$.
}
%\par
%OLD, TO BE ERASED: Given a family $\pi\colon\calY \to B$ of stable curves with a family
%of simple marked \msds $(\bfomega, \bfsigma, f)$, we want to construct
%functorially a map $m\colon B \to \ODehns$ such that the pullback of the
%universal family agrees with the given family.  Since the complex structure
%on $\ODehns$ stems from the model domain, we will first establish the
%universal property of the model domain. The map~$m$ will be constructed
%by using the universal property of the model domain to map there, and then by
%plumbing using the map $\OPl$ defined in Section~\ref{sec:Dehn}.
\par
To be able to use the universal property
of the model domain, we will need to define an {\em unplumbing} construction that
takes \msds on~$\calY$ to model differentials on an equisingular family
$\calX \to B$. Like the plumbing, the unplumbing construction will
depend on several choices, and we will need to carefully
arrange the choices consistently on $\calY$ and on the
universal family.
\par

%%%%%%%%%%%%%%%%%%%%%%%
\subsection{The universal family over the model domain.}
\label{sec:uniMD}
%%%%%%%%%%%%%%%%%%%%%%%

We first exhibit the functor that the space $\ptwT$ represents.
The following definition extends to families the pointwise definition
that appeared already in Section~\ref{sec:prmatch}, using the notion of
markings in families that is now at our disposal (compare also to the definitions
in Section~\ref{sec:markedauxds}).
\par
\begin{df}\label{def:famPTWD}
An {\em equisingular  family of \ptwds of type $(\mu,\eL)$ over an
analytic space $B$} is
\begin{enumerate}[(i)]
\item a family $(\eta_v)_{v \in V(\Gamma)}$ of twisted differentials of type~$\mu$,
compatible with $\Gamma(\eL)$ as in Definition~\ref{df:famtwisteddiffs},
\item a family of \prmas~$\sigma$, and
\item a family of markings  $f\in {\rm Mark}(\ol{\calX}_{\sigma}/B)$
of the welded family.
\end{enumerate}
\vspace{-1.5em}
\end{df}
\par
This definition is much simpler than Definition~\ref{def:markedmsds} or
Definition~\ref{def:aux_germ} and does not require a blowup of the base,
since there is no equivalence relation by the action of a level rotation
torus, which has a non-trivial fundamental group.
\par
\smallskip
We can now state the universal property of the space $\ptwT$. The proof
is rather obvious and mainly serves to recall notation.
\par
\begin{prop} \label{prop:univPTWT}
Let $\eL \subset \Sigma$ be a fixed enhanced multicurve.
The \Teichmuller space of \ptwds
$\ptwT$ is the fine moduli space for the functor that associates
to an analytic space~$B$ the set of equisingular families
of \ptwds of type $(\mu,\eL)$ over~$B$.
\end{prop}
\par
\begin{proof}
An equisingular family of pointed stable curves defines, by normalization,
a collection of families of pointed smooth curves with additional marked
sections corresponding to the branches of the nodes. Conversely, such a
collection of families of smooth pointed curves, together with a pairing of a subset
of the marked sections, defines an equisingular family. From this observation
it is obvious that the boundary stratum $\Bteich$ of the classical
augmented \Teichmuller space comes with a universal family
$(\pi\colon \unifam \to \Bteich, \bfz, (f_v)_{v \in V(\Lambda)})$
of pointed stable curves equisingular of type $\Gamma(\Lambda)$,
constructed by gluing families of smooth curves $\pi\colon \unifam_{v} \to \Bteich$
along the nodes given by the marked sections $q_e^\pm$
corresponding to the edges $e$ of $\Gamma(\Lambda)$. Here
$f_v \in {\rm Mark}(\unifam_{v}/\Bteich)$ is a \Teichmuller marking
by the surface $\Sigma_v$ (corresponding to the component $v \in V(\Gamma)$
of $\Sigma \setminus \Lambda$, with the boundary curves contracted to points).
The universal property follows from the universal properties for the
\Teichmuller spaces of the pieces $(\calX_v, \bfz_v, \bfq_e^{\pm}, f_v)$.
\par
Recall from Section~\ref{sec:TeichTwds} that there is a closed subspace
$\BPteich\subset \Bteich$ defined to be the quotient of $\OBnoGRC$ under the
action of~ $(\CC^*)^{V(\eL)}$.
The family $\pi$ can be restricted to $\BPteich$, pulled back to $\OBnoGRC$, and
then restricted to $\OBteich$.
Since the total space of a vector bundle represents the functor of sections
of the bundle, $\OBteich$ comes with a universal family $(\pi\colon \unifam \to
\changed{\OBteich}, \bfz,(f_v)_{v \in V(\Lambda)}, (\eta_v)_{v \in V(\Lambda)})$, where
$\bfeta = (\eta_v)_{v \in V(\Lambda)}$ is a twisted differential of type $(\mu,\eL)$,
and the remaining data are as above.
\par
Now we construct the family of markings in the welded surfaces  $\ol{\pi}\colon
\barunifam_{\sigma} \to \ol{\ptwT} \cong \ptwT$. Then we mark the welded surfaces
by~$\Sigma$ in such a way that fiberwise after  pinching~$\Lambda$ we obtain
the collection $(f_v)_{v \in V(\Lambda)}$. The remaining data are the pullbacks of
the ones defined above. Since $\ptwT\to \OBteich$ is an (infinite) covering
map (see Section~\ref{sec:ptwds}), the universal property follows from the
universal properties of covering spaces.
\end{proof}
\par
The family of model differentials over the model domain was already
constructed in Section~\ref{sec:modeldomain}, and its universal property
follows from the construction and from the universal property of the
augmented \Teichmuller space of flat surfaces.
\par
\begin{prop} \label{prop:sMDuniv}
The simple model domain~$\Omega\barMDs$ is the fine moduli space for the
functor of simple marked \auxds $\sMDfmark$, and
$\barMDs$ is the fine moduli space for~$\PP\sMDfmark$.
\end{prop}
\par
The model domain~$\Omega\barMD$ (considered as a quotient stack)
is thus isomorphic to the functor of \auxds $\MDfmark$.
\par
\begin{proof}
\changed{
Recall the construction of Section~\ref{sec:universalMD}. There we worked with a sufficiently small open set $\calV \subset \msT$, and denoted $\calW\subset \Omega\barMDs$ its preimage. We then constructed a family of curves~$\pi: \calX \to \calW$, with a family of
twisted differentials~$\bfeta$ and a collection of functions~$\bft$ on~$\calW$,
together with a collection of prong-matchings~$\bfsigma$ and
$\Lambda$-markings induced by the markings of the family over $\ptwT$, as in the
proof of Proposition~\ref{prop:markliftedfam}. By Remark~\ref{rem:namemodeldiff},
this is indeed a family of model differentials in the sense of
Definition~\ref{def:aux_germ}.
All this data depended on the choice of a section~$\mathscr{S}\colon \calV \to \calW$. Any other choice of a section can be obtained from~$\mathscr{S}$ by multiplying it by an
element of $T^s_{\Gamma}(\calV)$, and thus leads to an equivalent
family of model differentials, in the sense of Definition~\ref{def:aux_germ}. In particular, all these families constructed locally over each~$\calW$ glue to a family over all of~$\Omega\barMDs$.}
\par
\changed{To verify the universal property, we take a test family
$(\pi'\colon\calX'\to B', \bfz', \bfeta', \bft',\bfsigma', f')$
of $\Lambda$-marked simple \auxds of type~$\mu$. Just as in the previous
Proposition~\ref{prop:univPTWT}, the obvious universal property of the space $\msT = \ptwT[\Lambda] /
\CC^{L(\Lambda)}$, for families as
in Definition~\ref{def:famPTWD}, considered up to $\CC^{L(\Lambda)}$-action, yields the map $m_0: B' \to  \msT$. Working locally, we assume that the image of $m_0$ is contained in~$\calV$. Thus to verify the universal property we need to define a moduli map $m:B'\to\calW$ that lifts the map $m_0$, and show that~$m$ is unique up to equivalence. Then by existence and uniqueness of such local moduli maps, they must agree on various neighborhoods~$\calW$, and define a moduli map for a test family over an arbitrary base~$B'$ (whose image~$m_0(B')$ is not contained in a single~$\calV$).}

\changed{To construct the map $m$, note that since the family of differentials $\bfeta'$ is equivalent to the pullback of the universal family of differentials $\bfeta$ over~$\calX$, there exists an element of $T^s_{\Gamma}(\calV)$ acting by which maps $m_0^*\bfeta$ to $\bfeta'$. We act by this element to identify these two families of differentials, and then define $m$ to be the unique lift of $m_0$ such that $\bft' =
m \circ \bft$ (explicitly, recalling from Section~\ref{sec:universalMD} that we have a local trivialization $\calW \to \calV \times\cx^{L(\Lambda)}\times \cx^*$, so that defining~$m$ amounts to defining local functions $\calB'\to \cx^{L(\Lambda)}$, this means we simply take the ratios $\bft'/\bft$).}
\end{proof}
\par
%\begin{proof}
%OLD VERSION Recall that as discussed in Section~\ref{sec:modeldomain}, the family
%over $\Omega\MDs$ is simply the $\svTw$-quotient of the family over $\ptwT$.
%We showed in Proposition~\ref{prop:univPTWT} that the latter is the
%universal family of marked \ptwds. The family over the other strata
%of $\Omega\barMDs$ is constructed by covering the space by charts,
%and considering the scale comparison. The universal property thus
%immediately follows from the universal property of $\ptwT$.
%\end{proof}
%
%
%%%%%%%%%%%%%%%%%%%%
\subsection{The unplumbing construction}\label{sec:unplumb}
%%%%%%%%%%%%%%%%%%%%

The unplumbing construction associates with a family of \msds  a family
of \auxds over the same base. The rough idea is to pinch off neighborhoods degenerating to
nodes, in order to create equisingular families,
and then to record the degeneracy of the nodes as the parameters $\bft$
of \auxds. Technically, we cannot pinch off curves
without modifying the differential, due to the presence of non-trivial periods
over what we want to be the vanishing (pinching) cycles.
This forces us to subtract beforehand some perturbation differentials,
whose role is inverse to that of the modifying differentials. \changed{For the following construction we fix once and for all a base point~$p_h$ in a disc
and base points~$p^\pm$ in the upper and lower end of a plumbing fixture.}
\par
\begin{prop} \label{prop:unplumbconst} \changed{Suppose we are given} 
a family of simple marked \msds with all data $(\calY \to B, (\omega_{(i)})_{i \in
L^\bullet(\Gamma_p)}, R^s, \bfsigma, f)$, \changed{and denote by $\Lambda$ the enhanced multicurve of $X_p$. Suppose we fix moreover a maximal multicurve $\Lambda_{\max} \supseteq \Lambda$, sections $\sigma_e^\pm$ on the top and bottom end of each vertical node and on either side of a horizontal node, and a section~$\sigma_h$ for each half-edge corresponding to a marked zero.
If these sections are sufficiently close to the node (resp.\ the zero)
then} there is an unplumbing construction that produces, \changed{after
possibly restricting~$B$ to a neighborhood of~$p$,} a family $(\calX \to B, (\eta_{(i)})_{i \in L^\bullet(\Gamma_p)}, R^s, \bfsigma', f)$ of~$\Lambda$-marked simple \auxds with the following properties:
\begin{itemize}
\item[(i)]  The construction is the identity over the locus~$B^\Lambda$
of all $q \in B$ such that $\Gamma_q = \Gamma(\Lambda)$.
\item[(ii)] If $B$ is an open neighborhood of~$p$ in the simple Dehn space,
then the map $u \colon B \to \Omega\barMDs$ induced by the unplumbing of the
universal family of model differentials, restricted over $B$, is a local biholomorphism.
\end{itemize}
\changed{The unplumbing procedure is uniquely determined by the additional data
of~$\Lambda_{\max}$, $\sigma_e^\pm$ and $\sigma_h$.}
%\item[(ii)] The construction depends only on a finite number of
%choices of topological data and on a choice of a section near each vertical node.
\end{prop}
\par
\begin{proof}
The unplumbing construction is \lw, similarly to how plumbing was defined
in Section~\ref{sec:Dehn}.
For simplicity of the exposition, we only treat in detail the case where $\eG$
has two levels, and no horizontal nodes. (\changed{The case of more levels is done
by induction, and the horizontal nodes are straightforward as we simply
invert classical plumbing there.}) We may thus \changed{restrict~$B$
to the open set where the enhanced multicurve is an undegeneration of~$\Lambda$
and write $\omega_{(0)} = t^{a_1} \cdot \omega_{(-1)}$, and $B^\Lambda\subset B$ is
then precisely the vanishing locus of~$t$. (The reader should keep in mind that
the subsequent construction is the identity if $B^\Lambda = B$, i.e., if the family
is equisingular to start with.)}
\par
For the definition of a  perturbation differential
%we start by choosing some maximal multicurve~$\Lambda_{\max}\supseteq \Lambda$. We
denote by~$V$ the image of~$\Lambda$ in $H_1(\Sigma \setminus P_\bfs; \ratls)$ and, as in Proposition~\ref{prop:constrModif}, we let~$V'$ be the subspace generated by curves in~$\Lambda_{\max}$ and loops around points in~$P_{\bfs}$. Let  $\rho\colon B \to \Hom_\QQ(V,\CC)$ be the periods of~$\omega_{(0)}$ along~$\Lambda$ and let $\widetilde{\rho}$ be the extension of~$\rho$ by zero on a subset~$S$ of $\Lambda_{\max}$ generating $V'/V$. A {\em perturbation differential} is a
meromorphic section~$\xi$ of the relative dualizing sheaf $\pi_* \omega_{\calY/B}$
such that the periods of~$\xi$ are~$\widetilde{\rho}$. A perturbation
differential exists, and is uniquely determined by the choice of the
topological datum~$\Lambda_{\max}$ and the subset~$S$. Since $\widetilde{\rho}$
is divisible by~$s=t^{a_1}$, the perturbation differential vanishes identically on the fibers over~$B^\Lambda$.
\par
\changed{We first use Theorem~\ref{thm:NF} near each vertical node to put the differential into normal form $\omega_0 = (u^\kappa + r)du/u$ on the image
of a plumbing fixture. This theorem also implies that the chart is uniquely
determined by requiring that~$p^-$ in the plumbing fixture maps to~$\sigma_e^-$. Cutting off at the seam $|u|=|v|$ of the plumbing fixture and glueing a pole (with residue~$-r$) in the $v$-coordinates gives the lower level surfaces $(X_{-1},\omega_{(-1)})$.}
\par
\changed{Second we consider the other piece after cutting, the upper level
surface. We apply Theorem~\ref{thm:deformed_standard_coordinates} to obtain
a coordinate~$u'$ that puts $\omega_{(0)} - \zeta = du_e/u_e$ in standard form.
With the requirement that the point~$p^+$ in the $u_e$-disc maps to~$\sigma_e^+$ this theorem also gives the uniqueness of the coordinates~$u'$. We now glue in 
the family of discs $(\Delta \times B,u_e^{\kappa_e} \tfrac{du_e}{u_e})$. We proceed the same way with all the marked zeros (that may have been blurred by adding~$\zeta$), using~$p_h$ and~$\sigma_h$ to specify the glueing uniquely.
Third, we split horizontal nodes the same way using Theorem~\ref{thm:NFhoriz}
and rely on the nearby sections for uniqueness. We thus obtain an equisingular
family $\pi\colon \calX \to B$.}
\par
For such an equisingular family, the space of prong-matchings is an unramified
cover of the base. Thus, to obtain the prong-matching $\bfsigma'$
we simply extend the prong-matching~$\bfsigma|_p$ in a locally constant way.
The \lw real blowup of $(\calY,R^s)$ and the \lw real blowup of~$\calX$
as defined in Section~\ref{sec:rob} are almost-diffeomorphic.
(The \diffeogen
is given by the identity on the upper and lower surface, blurred near the
marked zeros, and both the degenerate plumbing fixtures in~$\calX$
and the plumbing fixtures of~$\calY$ are replaced by the welded fixtures
as defined in~\eqref{eq:defhatX}.) We can thus transport the marking~$f$ via
this almost-diffeomorphism. The rescaling ensemble~$R^s$ is the same on both sides of
the construction. Finally we verify that the equivalence relations are the
same on both sides, since in both cases they stem from the $\Tsimp$-action for the
differentials and prong-matchings, and from $\sTw[\eL]$ for the markings.
\par
Property~(i) is obvious. Finally, to prove (ii)  it suffices to show that the tangent map to~$u$ is surjective at any point of~$B^\Lambda$. We argue similarly to the alternative proof of Proposition~\ref{prop:plumbinj} \changed{(as indicated in the last paragraph of that proof)}, using the fact that perturbed periods
give local coordinates on the model domain, as shown in Proposition~\ref{prop:pertper}.
Indeed, by construction, on the restriction of the family $\calY|_{B^\Lambda}\to
B^\Lambda$ over $B^\Lambda$, the map~$u$ is the identity. Working in perturbed periods
coordinates on the
model domain, it thus suffices to show that the directions corresponding to
changing the parameters~$\bft$ of the model differential are in the range of
the tangent map to~$u$. This is indeed obvious since those parameters are given
by~$R^s$, which is part of the datum of the unplumbed model differential.
\end{proof}
\par

%%%%%%%%%%%%%%%%%%%%
\subsection{\changed{Log period coordinates}}\label{sec:eLogPer}
%%%%%%%%%%%%%%%%%%%%

Consider a family $\calY\to B$ of multi-scale differentials around a
point $p\in B$. In this section we define a collection of functions on~$B$,
possibly after restricting to a smaller neighborhood of~$p$. We will show
in Proposition~\ref{prop:elogPerCoords} that these functions, which we call
extended log period coordinates, are indeed local coordinates on Dehn space.
The construction of extended log periods combines the ideas from
\cite{benirschke} and of exponentiation used to define extended perturbed periods
in Section~\ref{sec:extendperturbed}, to construct coordinates that are intrinsic
to the family besides topological choices, i.e., neither depending on local
parameters near horizontal nodes (as in \cite{benirschke}) nor on modification
differentials and nearby sections (as perturbed period coordinates do). While we
are motivated by the construction from \cite{benirschke},  we give direct
prooofs, not relying on any results from~\cite{benirschke}, which appeared after,
and was based on, the current manuscript.
\par
Recall from Section~\ref{sec:perturbed} that the construction of perturbed
period coordinates starts by choosing  a basis
$\gamma^{(i)}_1,\dots,\gamma^{(i)}_{n(i)}$ of the GRC subspace $\calR_i^{\rm grc}
\subseteq H_1(\Sigma^c_{(i)}\! \setminus\! P_i, Z_i; \CC)$
for each level~$i$. (We often drop the upper level index.)
Since in~$\Sigma^c_{(i)}$ nodes have been pinched and poles have been removed,
the paths representing homology classes~$\gamma_j$ do not intersect horizontal
nodes.
\par
For each such path~$\gamma_j$ we define a~\emph{downward extension} to
be an element $\wt{\gamma}_j \in H_1(\Sigma \setminus P, Z; \CC)$ such that $\wt{\gamma}_j$ can
be represented by a path whose image (after node contraction and cutting off
lower levels) at levels~$i$ and above, that is in
$\oplus_{I \geq i} H_1(\Sigma^c_{(I)}\! \setminus\! P_I, Z_I; \CC)$,
is homologous to~$\gamma_j$ and if moreover the intersection number of~$\wt{\gamma}_j$ with
the vanishing cycle around any horizontal node is zero.
\par
\begin{lm} \label{le:downext}
Every~$\gamma_j$ has a downward extension~$\wt{\gamma}_j$. 
\end{lm}
\par
\begin{proof}
Every flat surface with a pole (at the node connecting to the level above)
has at least one zero. If all these zeros are at nodes, continue further downward.
\end{proof}
\par
To motivate the definition of the log period we consider in the case
that~$B$ is a reduced analytic space the formula due to Benirschke
\cite{benirschke}. We may cut off~$\wt{\gamma}_j$ at the highest level
below~$i$ that is persistent, replace~$\wt{\gamma}_j$
%on the locus~$\{t_{-1}\cdots t_{i} \neq 0\}$
by any nearby representative~$\ol{\gamma}_j$ that does not cross the nodes and let
\bes
%\Psi({\wt{\gamma}_j}) \= \frac{1}{t^{a_{-1}}_{-1}\dots t^{a_i}_i}
%\left[\int_{\wt\gamma_j}\omega - \sum_{e\in E(\Gamma)}\langle \wt\gamma_j,
%\lambda_e\rangle r_e \ln (s_e)\right]\,.
\Psi(\wt{\gamma}_j)\coloneqq \int_{\ol{\gamma}_j}\omega_{(i)}
- \sum_{e\in E(\Gamma)}\langle \wt\gamma_j, \lambda_e\rangle r_e \ln (s_e)\,.
\ees
Here $\lambda_e$ is the vanishing cycle corresponding to the node~$e$,
the `residue' $r_e \coloneqq \tfrac{1}{2\pi i}\int_{\lambda_e}\omega_{(i)}$ is
the integral over it, and $\ln (s_e)=\sum_{\ell(e-)}^{\ell(e+)-1}a_i\ln (t_i)$
for some branch of the logarithm. Note that the $\wt{\gamma}_j$-period
is multi-valued, but so is the logarithm, and the linear combination above is
chosen precisely so that these two ambiguities cancel. 
\par
For general~$B$ we cannot work with nearby paths. Instead we incorporate
manually the effect of subtracting the log terms, namely removing the residue
contributions to the integral. For this purpose we cover the
neighborhoods of the vertical nodes of~$\calY$ with disjoint plumbing
fixtures~$\VV_k$ with sections~$p_k^\pm$ at the upper and lower end of the fixtures.
We label them so that $\wt{\gamma}_j$ crosses precisely the fixtures
for $k=1,\ldots,r$ (with $r=0$ possible, e.g., if~$\gamma_j$ is closed).
We decompose~$\wt{\gamma}_j$ into a union of path segments $\wt{\gamma}_j^{I}$
at level $I \leq i$ starting and ending at the points $p_k^\pm$ and
into the segments joining~$p^+_k$ to~$p_k^-$ within~$\VV_k$. Finally,
using Theorem~\ref{thm:NF}
we may choose normal form coordinates on the plumbing fixtures so that
the differentials are $\omega_{(\ell^+_k)} = (u^{\kappa_k} +r)du/u$.
\par
\begin{df} \label{def:logperiod}
The \emph{log period} of the downward extension~$\wt{\gamma}_j$ is defined to be
\be \label{eq:logperiod}
\Psi({\wt{\gamma}_j}) \,\coloneqq\, %\prod_{m=I}^{i-1} t_m^{a_m}
%\frac{1}{t^{a__{-1}\dots t_i}\, \Bigl(
\left(\sum_{I\leq i} \int_{\wt{\gamma}_j^{I}} \omega_{(i)}\right)
\,\,+\,\, \sum_{k=1}^r \pm  \left(\prod_{m=\ell^+_k}^{i-1} t_m^{a_m}\right)
\int_{p_k^-}^{p_k^+} u^\kappa \, \frac{du}u \qquad \in \calO_B
%\Bigr)
\ee
using the positive sign if and only if~$\wt{\gamma}_j$ crosses the plumbing
fixture in the upward direction.
\end{df}
\par
One checks that for a family over a reduced base~$B$, when both definitions make sense, they agree for the right choices of the
branches of log, while a different choice of log would give functions that
differ from~$\Psi({\wt{\gamma}_j})$ by a $\ZZ$-linear combination of
$2\pi i r_e$.
\par
Homotopy invariance of path integrals and coordinate independence of
residues shows:
\begin{lm}
The log period~$\Psi({\wt{\gamma}_j})$ is independent of the choices
of the path decomposition made in its construction.
\end{lm}
\par
Next, we consider the extension of these functions that records the data of
horizontal nodes. Recall that we chose in Section~\ref{sec:extendperturbed}
for each horizontal
node~$q_j$ a path~$\beta_j$ crossing the seam corresponding to~$q_j$ and
no other seams. Let~$\alpha_j$ be the vanishing cycle corresponding to~$q_j$,
i.e., the path along the seam. We may assume that a basis of the span of
the set of vanishing cycles of horizontal nodes is among the~$\gamma_j$'s.
In particular, since~$\alpha_j$ is closed, we may assume that~$\alpha_j$ is
its own downward extension. Just as in Lemma~\ref{le:downext},
one shows that each~$\beta_j$ has a downward extension~$\wt{\beta}_j$, where
we now require that the intersection number of~$\wt{\beta}_j$ with the vanishing
cycle around any horizontal node except for~$\alpha_j$ is zero.
\par
Suppose that~$q_j$ it at level~$i$. Using Theorem~\ref{thm:NFhoriz}, choose local
coordinates $uv = f$ near that node so that the differential~$\omega_{(i)}
= r_j du/u$ is in standard form, where $r_j = \tfrac{1}{2\pi i} \int_{\alpha_j}
\omega_{(i)}$. We decompose~$\wt{\beta}_j$ at two sections~$\sigma^u$ and~$\sigma^v$
near the horizontal node (with $u(\sigma^u) \neq 0$ and $v(\sigma^v) \neq 0$) into
the path segment~$\wt{\beta}^u_j$ ending at~$\sigma^u$, the path
segment~$\wt{\beta}^v_j$ starting at~$\sigma^v$ and the path segment
from~$\sigma^u$ to~$\sigma^v$.
\par
\begin{df} \label{def:Elogperiod} We define
\bes
\Psi^E({\wt{\beta}_j}) \,\coloneqq\,
\bfe \Bigl(\frac{\Psi(\wt{\beta}^u_j) + \Psi(\wt{\beta}^v_j)}{r_j}\Bigr)
\cdot \frac{f}{ u(\sigma^u) v(\sigma^v)}
\ees
to be the \emph{exponentiated log period} of the downward
extension~$\wt{\beta}_j$.
\end{df}
\par
If~$B$ is reduced we may take, away from the locus where the horizontal
node persists, a nearby representative of~$\wt{\beta_j}$,  and observe that
\bes
\Psi^E({\wt{\gamma}_j}) \= \bfe \Bigl( \frac{\Psi(\wt{\beta_j})}{r_j} \Bigr),
\ees
since the additional summand in the exponential here is
\bes
\frac{1}{r_j}\int_{\sigma^u}^{\sigma^v} \omega_{(i)} \=  \log(u(\sigma^v))- \log(u(\sigma^u))
\=  \log( f/v(\sigma^v)u(\sigma^u),
\ees
since $u(\sigma^v) = f/v(\sigma^v)$.

%%%%%%%%%%%%%%%%%%%%
\subsection{\changed{The universal family and its extended log periods}
}\label{sec:UniveLogPer}
%%%%%%%%%%%%%%%%%%%%

Heading for the proof of Proposition~\ref{prop:univMMS}
we may, as explained in the paragraph after its statement, 
restrict attention to the open subfunctor such that all fibers are
isomorphic to those over a subset $W = \OPl(\calW_\epsilon \times \Delta^H)
\subset\ODehns$ over which a family of curves and a family of differential
forms~$\bfomega$ has been constructed in Section~\ref{sec:Dehn}, see in
particular Figure~\ref{cap:plumbing} and Equation~\eqref{eq:therescaledform}.
We denote this family $\pi^{\rm uni}: \calY^{\rm uni} \to W$ in the sequel.
\par
We provide the family $\calY^{\rm uni}$ with the remaining data to make it a family
of multi-scale differentials, our candidate for the universal family.
First, we endow it with the induced prong-matchings, explicitly given
in~\eqref{eq:sigmaeJ}, and declare the holomorphic functions~$\bft$ to be the
rescaling ensemble. Finally, we define the (`universal') marking. 
Since plumbing is fiberwise an almost-homeomorphism between the welded original
surface and the plumbed surface (the map $f^{L(\Lambda)}$ constructed in
Section~\ref{sec:mpa}), we can apply the real blowup construction in families,
given by Theorem~\ref{thm:ORBL}, both to the original family of model differentials
and to the plumbed family of multi-scale differentials, obtaining an
\diffeogen of the level-wise real blowups of these families. We then define
the marking to be the postcomposition of the marking on the universal family of
model differentials with this almost-homeomorphism.
\par
\begin{prop} \label{prop:elogPerCoords}
The rescaling parameters~$t_i$ for $i=-1,\ldots,-L$ together with the log
periods~$\Psi(\wt{\gamma}_j^{(i)})$ for $i=0,\ldots,L$ and $j=1,\ldots,
n(i)-1 + \delta_{i,0}$ and the exponentiated log periods $\Psi^E(\wt{\beta}_j)$
for $j=1,\ldots,H$ of the universal family $\pi^{\rm uni}$ form local coordinates
near any point of~$W$.
\end{prop}
\par
We call these coordinates the \emph{extended log periods}
$\mathrm{ELP}\coloneqq (\bft, \{\Psi(\wt{\gamma}_j^{(i)})\}_{i,j}, \{\Psi^E(\wt{\beta}_j)\}_j)$.
\par
\begin{proof} It suffices to show that the coordinate change from
extended perturbed periods to extended log periods is invertible at 
any point~$p$. The perturbed periods of $\gamma_j$ are simply the summand
$I=i$ of the first sum in $\Psi(\wt{\gamma}_j)$ if we choose to break
at the points $p^+_k = \sigma_e^+$ for the indices~$k$ corresponding to
the nodes~$e$ where $\gamma_j$ starts or ends (if any). The remaining
terms in the first sum of~\eqref{eq:logperiod} are~$O(t_{i-1})$ and so are all
terms in the second sum except for the contributions of up to two plumbing
fixtures that touch level~$i$. For those we decompose the integral as
$\int_{0}^{p_k+}$, which is a constant by the choice of~$\sigma_e^+$,
and $\int_{0}^{p_k^-}$, which is again~$O(t_{i-1})$. In short, there exist
constants~$C_{i,j}$ such that
\bes
\Psi(\wt{\gamma}^{(i)}_j) \= \PPer_i(\gamma^{(i)}_j) + C_{i,j} + O(t_{i-1})\,.
\ees
The same decomposition for the $\beta_j$-curves gives
\be
\Psi^E(\wt{\beta}_j) \= \Phor_j \cdot \bfe(C_{j}/r_j) \cdot (1+ O(t_{i-1}))
\ee
for some constants~$C_j$. Consequently, the determinant of the Jacobi
matrix~$J$ for the coordinate change between the extended perturbed periods and extended log periods is
\bes
\det(J) \= \prod_{j=1}^H \bigl(\bfe(C_{j}/r_j) \cdot (1+ O(t_{i-1})\bigr),
\ees
which is clearly non-zero.
\end{proof}
\par

%%%%%%%%%%%%%%%%%%%%
\subsection{Consistent unplumbing and the proof of Proposition~\ref{prop:univMMS}}
\label{sec:consist}
%%%%%%%%%%%%%%%%%%%%

\changed{Given a family of simple multi-scale differentials $\calY \to B$,
we assume the base $B\ni p$ to be sufficiently small, and that the fiber over~$p$
is represented by some point of~$W$. We will now define the moduli
map~$m\colon B\to W$, after possibly further shrinking~$B$. Once we prove the uniqueness
of such a moduli map locally, the existence of the moduli map over an arbitrary
base~$B$ follows by gluing such locally defined moduli maps.}
\par
\changed{We can moreover assume that the pinched multicurve $\Lambda_p$ is equal
to $\Lambda$, by proving the proposition inductively on multicurves, under
inclusion. As in the unplumbing construction, we give the details for the case when~$\calY_p$ has just two levels, as the general case is analogous, but would require more cumbersome notation.}
\par
\changed{We recall that the complex structure on the Dehn space is defined by
using plumbing, and we will thus apply the unplumbing construction to two
different families.} First, we apply unplumbing to~$\pi \colon \calY\to B$, to obtain a family of model differentials $\pi_\calX \colon \calX \to B$. Second, we apply unplumbing \changed{to $\pi^{\rm uni} \colon \calY^{\rm uni} \to W$ to obtain an (unplumbed) family of model differentials $\pi^{\rm unp} \colon \calX^{\rm unp} \to W$. To ensure that the comparison of~$\pi_\calX$ and~$\pi^{\rm unp}$ yields the correct moduli map~$m$ for~$\calY$, we need} to make consistent choices for these two unplumbing constructions, as follows. We choose a maximal
multicurve~$\Lambda_{\rm max}$ on~$\calY_p$, as required for Proposition~\ref{prop:unplumbconst}, and choose the same maximal multicurve on~$\calY$ and on~$\calY^{\rm uni}$, which is possible since all the fibers of $\pi^{\rm uni}$ are marked surfaces. \changed{Next, we choose both on~$\calY$ and $\calY^{\rm uni}$ the sections~$\sigma_h$ near the marked zeros to be the same fixed (small) $\omega_{(0)}$-period away from the corresponding zero. Finally we specify the choices for~$\sigma_e^\pm$. For~$\sigma_e^-$ observe that a differential on lowest level always has a zero~$z$ and choose~$\sigma_e^-$ near the lower end of the node corresponding to~$e$ a fixed $\omega_{(-1)}$-period away from~$z$. (If~$\calY_p$ has more than two levels and a middle level has no zeros except at nodes going down choose a downward extension~$\wt{\gamma}$ of a path to such a zero and place~$\sigma_e^-$ so that~$\Psi(\wt{\gamma})$ is constant.) Choose~$\sigma_e^+$ to be a fixed~$\omega_{(0)}$-period away from~$\sigma_h$, if the component at the upper end of~$e$ has a marked zero. If not, choose an arbitrary nearby section~$\sigma_e^+$ near a first upper end of a node, and use this in place of~$\sigma_h$ as reference point for all other~$\sigma_e^+$. We will see that this arbitrary choice will not matter.}
\par
Let $u_B \colon B \to \barMDs$ be the moduli map obtained by applying the universal
property for the simple model domain to the unplumbed family $\pi_\calX \colon
\calX\to B$ and let~$u\colon W \to \barMDs$ be the moduli map for the
unplumbing~$\pi^{\rm unp}$ of the (candidate) universal family, as in
Proposition~\ref{prop:unplumbconst}. The latter map is locally invertible by statement~(iii) of that proposition. We claim that (after possibly shrinking~$B$ to fit domains) the
composition
\bes
m \= u^{-1} \circ u_B \colon B \to W
\ees
will then be the moduli map for~$\pi$, i.e.,~that the family of multi-scale
differentials $\pi \colon \calY\to B$ is the pullback of the \changed{candidate
universal family~$\pi^{\rm uni}$ under the map~$m$. By definition there is an
isomorphism of families of model differentials $h_\calX \colon m^* \calX^{\rm unp}
\to \calX$, and thus to show the pullback property we need to modify~$h_\calX$
to an isomorphism of families of \msds $h \colon m^* \calY^{\rm uni}
\to \calY$.}
\par
This isomorphism~$h$ is constructed level by level, and for clarity of exposition we will again only deal with the two-level situation without horizontal nodes, as in the proof of Proposition~\ref{prop:unplumbconst}. \changed{Suppose each top level component has a marked zero. Then since the location of the points~$\sigma_e^\pm$ and~$\sigma_h$ is specified by periods, $h_\calX$ identifies these points on~$m^* \calY^{\rm uni}$ and on~$\calY$. We proceed as in the plumbing construction, inverting the steps in the unplumbing construction. That is, we use the point~$\sigma_e^\pm$ to glue the plumbing fixture ot the lower level so that~$p^\pm$ is mapped to~$\sigma_e^\pm$. Using the modification differential~$\xi$ determined by~$\Lambda_{\max}$, the differential on the plumbing fixture glues both to~$\eta_{(-1)}$ and to~$\eta_{(0)} + \xi$. Proceeding similarly (see Sections~\ref{sec:plhomeo} and~\ref{sec:plhoriz}) near the zeros and near the horizontal nodes,  we get back~$\calY$ from~$\calX$ and $\calY^{\rm uni}$ from~$\calX^{\rm unp}$. We define~$h$ to be~$h_\calX$ on the upper and lower level subsurfaces and the identity on the plumbing fixtures. These maps respect the differentials $\omega_{(i)}$, since~$\xi$ is determined by the topological choice of~$\Lambda_{\max}$, and they glue since $h_\calX$ identifies the marked points as remarked above.}
\par
\changed{If a top level component of~$\calX$ does not have a marked zero,
$h_\calX((\sigma_e^+)_{\rm uni})$ will in general differ from $\sigma_e^+$ on~$\calX$ at the `first' upper and of the node where the choice was arbitrary. However we can perform a relative period deformation ('Schiffer variation', see e.g.\ \cite[Section~2]{McMtwists} for various viewpoints of the construction) to translate the neighborhood of zero corresponding to~$e$, so that $h_\calX((\sigma_e^+)_{\rm uni})$ and~$\sigma_e^+$ agree. If we apply the \emph{same} Schiffer variation to all upper ends of nodes of this component of~$\calX$, our choice of~$\sigma_e^+$ implies that $h_\calX((\sigma_e^+)_{\rm uni}) = \sigma_e^+$ for all those nodes. We now proceed as in the previous paragraph.} 
%
%
%
%
%In that setting, the lower level subsurfaces~$\calX_{(-1)}$ and~$\calY_{(-1)}$ with their differentials are simply the same by construction, and this also holds for the universal families. This defines the map~$h$ on the families of lower subsurfaces.
%\par
%On the upper level subsurfaces a perturbation differential has been added. The consistent choice of~$\Lambda_{\rm max}$ implies that the $m$-pullback of the perturbation differential on~$\calY^{\rm uni}_{(0)}$ agrees with that on~$\calY_{(0)}$. Similarly, since all choices in both unplumbing constructions are the same, the local modifications near the marked zeros agree under $m$-pullback. We thus define the map~$h$ on the upper level surfaces. The plumbing fixtures (i.e.\ the functions~$f_e$) are compatible, since these functions can be read off from the rescaling ensemble. Lastly, it remains to
%check that the way the plumbing fixtures are glued in is compatible, so that
%the piecewise defined isomorphisms~$h$ glue to a global isomorphism. This
%follows from the fact that the sections (which determine the normal form uniquely)
%were chosen using the same relative period of~$\omega_{(-1)}$. This completes
%the proof of the proposition for points~$p$ where $\Lambda_p=\Lambda$.
\par
\changed{To show the uniqueness of the moduli map~$m$ we use the uniqueness
of the moduli map for model differentials, but we have to show that~$m$ does
not depend on the choices made for the unplumbing constructions. Suppose
a different collection of choices gives another unplumbed family
$\pi'\colon \calX' \to B$ with corresponding moduli map $u'_B \colon B
\to \barMDs$, and gives similarly~$u'$ and $m' = (u')^{-1} \circ u'_B$, with
an isomorphism   $h' \colon (m')^* \calY^{\rm uni} \to \calY$ of families of \msds.
To show that $m = m'$ we use the extended log period coordinates.
These are independent of any choices except for the paths
$\wt{\gamma}_j$ and $\wt{\beta}_j$, which we choose at the central fiber
over~$p$ once and for all, and then propagate to a neighborhood of~$p$ in~$B$. Note that extended log periods are specifically
designed to not depend on any monodromy arising in that way. Fixing the same $\wt{\gamma}_j$ and $\wt{\beta}_j$ in the fiber over~$p$, the
extended log period coordinates can then be read off from~$\calY$, and
the isomorphisms~$h$ and~$h'$ imply that  $\mathrm{ELP} \circ m =
\mathrm{ELP} \circ m'$. Proposition~\ref{prop:elogPerCoords} now
gives the desired conclusion}.
\par
\changed{
\begin{rem}
If~$B$ is a reduced complex analytic space, we can alternatively prove uniqueness
of the moduli map without appealing to~$\mathrm{ELP}$, by stratifying~$B$
according to the pinched multicurve~$\Lambda$ and using that over each
stratum~$B_\Lambda$ a family of marked multi-scale differentials is just a family
of marked prong-matched twisted differentials. These come with a fine moduli
space, which proves the uniqueness of $m|_{B_\Lambda}$ for each~$\Lambda$, and thus of~$m$. If
$B$ is non-reduced, say a thick point, this argument fails to give uniqueness
outside the central fiber, necessitating the above use of extended log periods.
\end{rem}}
\par
\changed{
\begin{rem}
One could attempt to bypass the preceding proof of the uniqueness of the moduli map that uses consistent unplumbing, by using the log periods and constructing an isomorphism by integration. Indeed, using the extended log period coordinates one can define the moduli map $m\colon B\to W$ from any base~$B$ (possibly after shrinking~$B$) to a (sufficiently small)~$W$. One can then pull back the universal family over~$W$ to obtain two families of multi-scale differentials over~$B$, with the same central fiber and with the same extended log period coordinates, and would need to prove that they are isomorphic. This is rather straightforward by integration if all curves in the family are smooth. However, to make this work in full detail requires overcoming various complications: one has to choose the rescaling ensembles for the two families compatibly; then one should use the standard form near the nodes to argue that the two families can be made to locally agree near the (non-persistent) nodes. Finally the equality of extended log periods would then be used to show that the isomorphism of the families near the nodes can be chosen to extend to the entire families. Since this requires a similar amount of work as our consistent unplumbing constructions, we do not provide further details.
\end{rem}
}

%%%%%%%%%%%%%%%%%%%%
\subsection{The proof of
Theorem~\ref{thm:univnonsimpleMMS}}\label{sec:Dehnfine}
%%%%%%%%%%%%%%%%%%%%

As in the proof of Proposition~\ref{prop:univMMS}, we start with the local version and
then glue the moduli maps as above. Suppose we are given a family $(\calY \to B,
(\omega_{(i)})_{i \in L^\bullet(\Gamma_p)}, R, \bfsigma, f)$ of marked \msds, defined
on a neighborhood of~$p$. We proceed similarly to the proof of
Proposition~\ref{prop:restRE} and let $B^s \to B$ be the fiber product
of~$R \colon B \to \Tnorm[\Gamma_p]$ with the finite quotient map $\ol{p} \colon
\ol{T}_{\Gamma_p}^{s}\to \ol{T}_{\Gamma_p}^{s}/K_p =\Tnorm[\Gamma_p]$. The
pullback family $\calY^s \to B^s$ comes with a map $R^s \colon B^s \to
\ol{T}_{\Gamma_p}^{s}$ and is thus a family of simple marked \msds. By
Proposition~\ref{prop:univMMS} we obtain a moduli map $m^s \colon B^s \to
\ODehns$. Composing~$m^s$ with the quotient map $\ODehns \to \ODehn$ we get a map $B^s \to \ODehn$ that is clearly $K_{\Gamma_p}$-invariant by construction. It thus
descends to the required moduli map $m \colon B \to \ODehn$.
\par

\section{The moduli space of \msds}
\label{sec:moduli space}
%%%%%%%%%%%%%%%%%%%%%%%%%%%%%%%%%%%%%

We now have all the tools that are necessary to prove the main theorems
announced in the introduction. Denote by $\PP\obarmoduli[g,n]^{\rm{ninc}}(\mu)$
the normalization of the incidence variety compactification, where the incidence
variety is considered as a substack of $\pobarmoduli[g,n]$. In this section we
will show that the stack~$\PP\MSgrp$ of \msds can be obtained from
$\PP\obarmoduli[g,n]^{\rm{ninc}}(\mu)$ as the normalization of a certain
explicit (complex algebraic, not real oriented) blowup, called the orderly blowup.
We will then be able
to conclude the proof of Theorem~\ref{intro:main} and Theorem~\ref{intro:funct},
in particular proving algebraicity of $\LMS$. We will conclude this Section with
an example that illustrates the orderly blowup and the necessity of the
subsequent normalization.
\par

%%%%%%%%%%%%%%%%%%%%%%%%
\subsection{A blowup description}
\label{subsec:blowup}
%%%%%%%%%%%%%%%%%%%%%%%%

The incidence variety compactification in general can have bad singularities.
For instance, it can fail to be normal, as it can have multiple local
irreducible components along the locus of pointed stable differentials that
admit more than one compatible enhanced structure on the dual graph (see
e.g.~\cite[Example 3.2]{strata}), and its normalization may still be quite
singular, e.g.\ not even $\QQ$-factorial, as shown in the following example.
\par
\begin{exa}({\em The IVC may be not $\QQ$-factorial.})
\label{ex:q-factorial}
Consider a level graph with three levels such that the top level has one vertex $X_0$,
the level~$-1$ has two vertices~$X_1$ and~$X'_1$, and the bottom level has
one vertex $X_2$, where $X_0$ is connected to each of~$X_1$ and~$X'_1$ by one edge,
and~$X_2$ is connected to each of~$X_1$ and $X'_1$ by one edge.  In other words, the
graph looks like a {\em rhombus}.  Since the level graph has three levels and no
horizontal edges, the corresponding stratum has codimension two in the moduli space
of \msds $\PP\LMS$. On the other hand, since $X_1$ and $X'_1$ are disjoint, when
considering the incidence variety compactification $\PP\obarmoduli[g,n]^{\rm{inc}}(\mu)$
we lose the information of relative sizes of rescaled differentials $\lambda \eta_1$
and $\lambda'\eta'_1$ on $X_1$ and $X'_1$, where $\lambda, \lambda' \in \CC^*$, and
hence the corresponding locus has codimension three in
$\PP\obarmoduli[g,n]^{\rm{inc}}(\mu)$. Namely, the map $\PP\LMS \to
\PP\obarmoduli[g,n]^{\rm{inc}}(\mu)$ locally around these loci looks like a
$\PP^1$-fibration, where $\PP^1 = [\lambda, \lambda']$ (in the degenerate case
$\lambda = 0$, $X_1$ goes lower than $X'_1$ and the graph has four levels, and
vice versa for $\lambda'=0$). One can check that locally outside of these loci the
map does not have positive dimensional fibers. We thus obtain locally a {\em small
  contraction} (which means no divisors get contracted), and consequently the target
space $\PP\obarmoduli[g,n]^{\rm{inc}}(\mu)$ (as well as its normalization) is
not $\QQ$-factorial (see e.g.~\cite[Corollary 2.63]{kollarmori}).
\end{exa}
\par
\medskip
Given an adjustable but not necessarily orderly family $(\calX \to B,\omega)$ (as
defined in Section~\ref{sec:triangsystem}), we first describe a canonical way to
blow up the base~$B$ so that the pullback family under this base change becomes
orderly. Let $X_v$ and $X_{v'}$ be two irreducible components of the fiber $X_p$
over some $p\in B$. The family fails to be orderly if neither of the adjusting
parameters $h$ and $h'$ for $X_v$ and $X_{v'}$, respectively, divides the other one,
as elements in~$\calO_{B,p}$. Therefore, we perform the following blowup construction.
\par
Let $U\subset B$ be a (sufficiently small) neighborhood of $p$ such that there
exist adjusting parameters $\{h_1,\dots, h_n\}$ for the family $\calX|_U$.
The \emph{disorderly ideal} $\calD_U\subset \calO_{U,p}$ for $\calX|_U$ at~$p$
is the product of all ideals of the form $(h_{i_1}, \dots, h_{i_k})$, where
$\{ i_1,\dots, i_k\}$ ranges over all subsets of components of~$X_p$ on
which~$\omega$ vanishes identically.
\par
We denote by $\wt U$ the blowup of~$U$ along $\calD_U$, and call it the
\emph{orderly blowup}. If $U'$ is an open subset of~$U$ such that $\calX|_{U'}$ becomes less degenerate, namely, some $h_i$ becomes a unit in $U'$, or the ratio of some $h_i$ and $h_j$ becomes a unit, then $\calD_U|_{U'}$ possibly differs from $\calD_{U'}$ by some repeated factors of ideals.
Note that blowing up the principal ideal of a non-zero-divisor (i.e. the underlying subscheme is an effective Cartier divisor) is simply the identity map, and moreover, blowing up a product of ideals is the same as successively blowing up (the total transform of) each ideal (see e.g.~\cite[{Tag 01OF}]{stacks-project}). This implies that  for any two open subsets $U_1$ and $U_2$, we can glue~$\wt{U_1}$ and $\wt{U_2}$ along their common restriction $\wt{U_1\cap U_2}$. In other words, this local blowup construction chart by chart leads to a well-defined global space, which we denote by~$\wt B$, and there exists a blowdown morphism $\wt B\to B$ locally given by $\wt U\to U$.
\par
\begin{exa}
We illustrate the behavior of disorderly ideals by the following example. Suppose the special fiber $X_p$ consists of four irreducible components $X_0, X_1, X_2, X'_2$ such that $X_0$ is on top level which connects to $X_1$ on level $-1$, and $X_1$ connects to $X_2$ and $X'_2$ on lower levels that we cannot order.  Let $h_1, h_2, h'_2$ be the adjusting parameters for $X_1, X_2, X'_2$ respectively, and assume that they are not zero divisors.
 Then the partial order implies that $h_1$ divides both $h_2$ and~$h'_2$,
and hence
$$\calD_U = (h_1)(h_2)(h'_2)(h_1, h_2)(h_1, h'_2)(h_2, h'_2)(h_1, h_2, h'_2) = (h_1)^4(h_2)(h'_2)(h_2, h'_2)$$
for a sufficiently small neighborhood~$U$ of $p$. Suppose $q\in U$ is a nearby point such that the fiber $X_q$ is less degenerate in the sense that the nodes connecting $X_2, X'_2$ to~$X_1$ are smoothed, i.e.~suppose $X_q$ has only one lower level component with adjusting parameter~$h_1$ and both $h_2, h'_2$ become $h_1$ multiplied by some units in a neighborhood $U'\subset U$ of $q$. Then
$\calD_{U'} = (h_1)$, which differs from $\calD_{U}|_{U'} = (h_1)^7$ by a power of $(h_1)$. In particular, the ideals $(h_1)$ and $(h_1)^7$ define different
subschemes in $U'$. However, since both ideals are principal, blowup along each of them is thus the identity map, so the resulting spaces are isomorphic to each other.
\end{exa}
\par
We need the following lemmas about the properties of disorderly ideals.
\par
\begin{lm}
\label{lm:principal-product}
Let $R$ be a local ring and $I, J\subset R$ be two ideals such that the product ideal $IJ$ is a principal ideal generated by a non-zero-divisor. Then both $I$ and $J$ are principal ideals generated by non-zero-divisors.
\end{lm}
\par
\begin{proof}
Suppose $IJ = (a)$ for some non-zero-divisor $a$.  Then there exist $b_i \in I$ and $c_i \in J$ such that $b_1c_1 + \dots + b_n c_n = a$, which implies that
$b_1 (c_1/a) + \dots + b_n (c_n / a) = 1$ as a relation in the ring of fractions. Since the (unique) maximal ideal of $R$ consists exactly of all non-unit elements, it follows that
some $b_i (c_i / a)$ must be a unit, hence $I = (b_i)$.
\end{proof}
\par
\begin{lm}
\label{lm:principal}
Let $R$ be a local ring and let $h_1, \dots, h_n\in R$ be some elements that are non-zero-divisors. Let $D = \prod (h_{i_1}, \dots, h_{i_k})$ be the product of ideals where $\{ i_1,\dots, i_k\}$ ranges over all subsets of $\{1, \dots, n\}$. Then $D$ is a principal ideal $(h)$ with $h$ being a non-zero-divisor if and only if $h_1,\dots, h_n$ are fully ordered by the divisibility relation.
\end{lm}
\par
\begin{proof}
If $h_1, \dots, h_n$ are fully ordered by divisibility, it is clear that $D = (h)$ where $h$ is given by certain products of powers of the $h_i$, and by assumption each $h_i$ is a non-zero-divisor. Conversely if $D = (h)$ is principal with $h$ being a non-zero-divisor, then the same holds for each factor $(h_{i_1}, \dots, h_{i_k})$ by Lemma~\ref{lm:principal-product}. Suppose $(h_1, \dots, h_n) = (b)$ such that $h_i = b t_i$ for $t_i$ in $R$ and $b$ being a non-zero-divisor.  Then there exist $u_i$ in $R$ such that $u_1t_1 + \dots + u_nt_n = 1$.
If all of $t_1, \dots, t_n$ are not units, then the ideal $(t_1,\dots, t_n)$ is contained in the (unique) maximal ideal of the local ring $R$, which is absurd because it also contains $1$.
Hence we may assume that $t_1$ is a unit in $R$, which implies that~$h_1$ divides $h_2, \dots, h_n$. Carrying out the same analysis for the ideal $(h_2, \dots, h_n)$ and repeating the process thus implies the desired claim.
\end{proof}
\par
The orderly blowup construction possesses some functorial property.
\par
\begin{prop}
\label{prop:blowup}
Given an adjustable family of differentials $(\pi\colon\calX\to B,\omega,\bfz)$, the pullback family $\wt\pi\colon\wt\calX\to\wt B$ over the orderly blowup $\wt B\to B$ \index[family]{b060@$\wt\pi\colon\wt\calX\to\wt B$! Orderly blowup of an adjustable family of  differentials} is orderly. Moreover, any dominant map $\pi\colon B' \to B$, such that the pullback family $\calX'\to B'$ is orderly, factors through $\wt B$.
\end{prop}
\par
\begin{proof}
It suffices to check the claim locally over each~$U$, with the disorderly ideal $\calD_U$ in the preceding setup.
The first statement then follows from Lemma~\ref{lm:principal}. More precisely, on the orderly blowup, the pullback of $\calD_U$ becomes a principal ideal, and hence at every point of~$\wt U$ the pullback family of differentials has adjusting parameters (given by the pullback of the functions $h_i$) that are fully ordered by divisibility, which implies that the family is orderly over~$\wt U$.
\par
The second statement follows from the universal property of blowup (see e.g.~\cite[{Tag 01OF}]{stacks-project}). Let $U' = \pi^{-1}(U)$. Since $\pi$ is dominant, the pullback of any adjusting parameter $\pi^{*}h_i$ is a non-zero-divisor, and moreover $\pi^{*} \eta_{(i)} = \pi^{*}\omega / \pi^{*}h_i $ holds for the adjusted differential $\eta$ on any irreducible component $X_i$ of any fiber $X_p$ over a point  $p\in U$. Hence these $\pi^{*}h_i$ can be used as adjusting parameters for the pullback family over~$U'$. Since the pullback family is orderly, these adjusting parameters $\pi^{*}h_i$ in $U'$ are fully ordered by divisibility, and consequently the corresponding disorderly ideal $\pi^{*}\calD_U$ in $U'$ is principal (and generated by a non-zero-divisor). Since the blowup of $\calD_U$ is the final object that turns $\calD_U$ into a principal ideal (generated by a non-zero-divisor), it implies that $\pi\colon U' \to U$ factors through $\wt U$.
\end{proof}
\par
We remark that there is some flexibility in choosing the local disorderly ideals. For instance, we can alternatively take $D = \prod (h_{i_1}, \dots, h_{i_k})$ to be the product of ideals ranging over all subsets of cardinality at least two.  This ideal differs from the original definition of the disorderly ideal by a product of principal ideals, and hence the blowup with center~$D$ gives the same space as the orderly blowup. We can also take the product $D = \prod (h_i, h_j)$ over all pairs of $h_i$ and $h_j$ that do not satisfy the divisibility relation. Then after blowing up the adjusting parameters are pairwise orderly, hence are orderly altogether.
\par
We warn the reader that the orderly blowup of a normal base may fail to be
normal, as illustrated by the following example.
\par
\begin{exa} ({\em A non-normal orderly blowup})\label{ex:orbnotnormal}
Let $x$ and $y$ be the standard coordinates of $B = \CC^2$. Then $x^2$
and $y^3$ do not divide each other in the local ring of the origin.  The orderly
blowup $\wt{B}$ for the ideal $(x^2, y^3)$ can be described by
\begin{eqnarray}
\label{eq:non-normal}
  \left\lbrace(x,y)\times [u, v] \in \CC^{2}\times \PP^{1} : x^{2}v-y^{3}u=0 \right\rbrace\,.
  \end{eqnarray}
Then we see that $\wt{B}$ is singular along
the entire exceptional curve over $x = y = 0$. It implies that $\wt{B}$ is not normal, since a normal algebraic surface can have only isolated singularities.
\end{exa}
\par
We are now ready to apply these considerations to the IVC.
\par
\begin{lm}
The incidence variety compactification $\PP\obarmoduli[g,n]^{\rm{inc}}(\mu)$ can
be considered as a closed substack of $\PP\obarmoduli[g,n]$.
\end{lm}
\par
\begin{proof} We can restrict to the neighborhood of a stable curve with
dual graph~$\Gamma$. It suffices to realize that conditions of the existence
of a twisted differential compatible with an enhanced level structure are
closed conditions that can be read off from a family of pointed stable curves.
This is clear both for the existence of differentials (i.e.\ sections of
a line bundle determined by the family and the marked points) and for the global
residue condition (vanishing condition for sums of residues associated with these
differentials).
\end{proof}
\par
\begin{thm}
\label{thm:LMS-blowup}
The moduli stack of \msds $\MSgrp$ is equivalent (as an analytic stack) to
the normalization $\wt{\Omega\calM}^{\rm n}_{g,n}(\mu)$ of the orderly blowup
of the normalization~$\obarmoduli[g,n]^{\rm{ninc}}(\mu)$ of the incidence
variety compactification. This orderly blowup $\wt{\Omega\calM}_{g,n}(\mu)$, and thus also $\MSgrp$,
is the analytification of an algebraic stack.
\end{thm}
\par
See \cite[Chapter~XII]{acgh2} for a general introduction to (algebraic) stacks
and \cite{ToenTheses} for analytic stacks and analytification.
\par
\begin{proof} First, we make sure that the operations of normalization
(e.g.\ \cite[Example~8.3]{acgh2}) and orderly blowup make sense in this
context, considering $\obarmoduli[g,n]^{\rm{ninc}}(\mu)$
as an analytic stack. Proposition~\ref{prop:normal-rescalable} ensures that
over $\obarmoduli[g,n]^{\rm{ninc}}(\mu)$ the family of one-forms that are the top
level forms of the twisted differential is adjustable.
In a local quotient groupoid presentation $[U/G]$ we see that $G$-pullbacks
of adjusting parameters are again adjusting parameters. Consequently, the
disorderly ideal is~$G$-invariant and in view of
Lemma~\ref{lm:disorderlycompatible} below we can consider the orderly blowup
$\wt{\Omega\calM}^{\rm n}_{g,n}(\mu)$ as an analytic stack.
\par
Proposition~\ref{prop:orderly-msd} then ensures that the resulting orderly
family over $\wt{\Omega\calM}^{\rm n}_{g,n}(\mu)$ gives a family of
multi-scale differentials of type $\mu$. This family induces a map
of stacks $\wt{\Omega\calM}^{\rm n}_{g,n}(\mu) \to \MSgrp$.
\par
Conversely, a family in the stack $\MSgrp$ is orderly by definition, hence
by Proposition~\ref{prop:blowup} we obtain a map of stacks $\MSgrp \to
\wt{\Omega\calM}_{g,n}(\mu)$. Since $\MSgrp$ is normal, this map factors through
$\wt{\Omega\calM}^{\rm n}_{g,n}(\mu)$, which gives the desired inverse map.
\par
To show that the orderly blowup of~$\obarmoduli[g,n]^{\rm{ninc}}(\mu)$
is an algebraic stack, use that the normalization $\obarmoduli[g,n]^{\rm{ninc}}(\mu)$
is an algebraic stack and pass from a local quotient groupoid presentation $[U/G]$
to a presentation  $[U'/G']$ on a sufficiently large \'etale cover $U' \to U$ such
that adjusting parameters exist (algebraically) on~$U'$. The existence of such a~$U'$
is guaranteed by Proposition~\ref{prop:alg-rescalable}. That the blowup is a stack is justified
as in the analytic case. Since we could have used $[U'/G']$ also as a presentation
in the analytic case, we have justified the last statement of the theorem.
\end{proof}
\par
In the preceding proof we used the following general fact.
\par
\begin{lm} \label{lm:disorderlycompatible}
Let $X$ be an algebraic or analytic variety and let $\calI$ be
a coherent ideal sheaf. Let $f\colon \tilde{X} \to X$ be a finite \'etale
Galois cover with group~$G$ and let $\calJ = f^{-1} \calI$ be the
pullback (that agrees with the inverse image here). Then there
is a natural $G$-action on $\Bl_\calJ\tilde{X}$ and an isomorphism
$\Bl_\calI X \cong (\Bl_\calJ\tilde{X}) /G$ (in the algebraic or analytic
category).
\end{lm}
\par
\begin{proof}
The existence of the $G$-action on  $\Bl_\calJ\tilde{X}$ follows directly
from the universal property for~$\calJ$. The universal property of~$\calI$
(in the version \cite[Corollary II.7.16]{hartshorne}) also gives a map
$\Bl_\calJ\tilde{X} \to \Bl_\calI X$. This map is equivariant with respect
to the $G$-action on the domain and the trivial action on the target. It thus
descends to a map $\Bl_\calJ\tilde{X} /G \to \Bl_\calI X$.
\par
Conversely, take the fiber product $F = \tilde{X} \times_X  \Bl_\calI X$.
Since the pullback of $\calI$ to~$F$ (via $\Bl_\calI X \to X$) is principal,
the pullback of~$\calJ$ to~$F$ is principal. The universal property
for~$\calJ$ gives a map~$F \to  \Bl_\calJ\tilde{X}$, which is obviously $G$-equivariant.
Since $F/G = \Bl_\calI X$ by definition, as fiber product, the map descends
to the desired map $\Bl_\calI X \to \Bl_\calJ\tilde{X} /G$ on the $G$-quotients.
\end{proof}
\par
It is well-known that the blowup of a projective scheme along a globally defined ideal sheaf (or equivalently a globally defined subscheme) remains to be projective. Nevertheless, we remark that in general gluing local blowups can lead to a non-projective global space (recall the famous Hironaka's example, see e.g.~\cite[Appendix B.3, Example 3.4.1]{hartshorne}). The preceding subsection thus does not answer the question on the projectivity of the coarse space associated
with $\PP\MSgrp$.
\par

%%%%%%%%%%%%%%%%%%%%%%%%
\subsection{The universal property}
%%%%%%%%%%%%%%%%%%%%%%%%

Recall that for a complex orbifold with local orbifold charts $(U,G)$ there
is an underlying complex space with charts being the (in general singular)
complex spaces $U/G$.
\par
\begin{thm} \label{thm:coarseMS}
The complex space associated with the moduli space of \msds $\LMS$ is the
coarse moduli space for the functor~$\MSfun$ of \msds of type~$\mu$.
\end{thm}
\par
\begin{proof}
Given a family of \msds $(\pi\colon \calX \to B, \bfomega) \in \MSfun(B)$, we want to provide
the family locally near any point~$p$ with a marking,  and then construct
the moduli map $m\colon B \to \LMS$ as the
composition of the moduli map from Theorem~\ref{thm:univnonsimpleMMS}
and the natural quotient map.
\par
For this purpose, we choose for any point $p \in B$ an enhanced
multicurve~$\Lambda_p$ on~$\Sigma$ with $\Gamma(\Lambda_p)
= \Gamma_p$. For a sufficiently small neighborhood $U_p$ of~$p$ we apply
Proposition~\ref{prop:markliftedfam} to provide the family with a marking.
The moduli map given by Theorem~\ref{thm:univnonsimpleMMS}
composed with the projection then gives a map $U_p \to \ODehn \to
\LMS$. These maps glue, since any two choices
of marking differ by the action of an element in the mapping class
group. This argument, together with the universal property of $\ODehn$, also
implies the bijection on complex points and the maximality required as
properties of a coarse moduli space.
\end{proof}
\par
\begin{proof}[Proofs of Theorem~\ref{intro:main} and Theorem~\ref{intro:funct}
completed] For Theorem~\ref{intro:main}, the density~(1) and the description of the boundary divisor~(2) have
been taken care of in Theorem~\ref{thm:LMScomplex}. Compactness~(3)
is the content of Theorem~\ref{thm:MSDcompact}, and the coarse moduli space property~(4)
has been addressed in Theorem~\ref{thm:coarseMS}.
\par
The forgetful map~(5) is obvious, e.g.\ it follows from
Theorem~\ref{thm:LMS-blowup} and its proof.
\par
For Theorem~\ref{intro:funct}, the property of being a proper Deligne-Mumford stack carries over
from $\barmoduli[g,n]$ all the way up through $\PP\obarmoduli[g,n]$,
$\PP\obarmoduli[g,n]^{\rm{ninc}}(\mu)$, and $\wt{\PP\Omega\calM}^{\rm n}_{g,n}(\mu)$.
The isomorphism in the statement of Theorem~\ref{intro:funct} is then obvious
since our compactification does not alter the interior $\PP\omoduli[g,n](\mu)$,
and on a smooth curve a \msd is simply an abelian differential of type~$\mu$.
\par
It remains to show the algebraicity and properness of the stack
$\PP\LMS$. The algebraicity is a consequence of
Theorem~\ref{thm:LMS-blowup} which implies that the coarse moduli space
provided by Theorem~\ref{thm:coarseMS} is algebraic. This space is analytically
locally covered by quotient stacks $[U_X/G_X]$ near~$X$, where the group
$G_X$ is an extension of an automorphisms group of pointed curves by
the group $K_\Gamma$, by the definition of morphisms of marked simple
\msds. Here~$\Gamma$ is the level graph compatible with the point~$X$.
We need to show that the coarse space is also  covered
by quotient stacks $[U'_X/G_X]$ \'etale locally. This is a consequence of Artin's
approximation theorem \cite[Corollary~2.6]{artin} that provide an
\'etale neighborhood in the coarse space isomorphic to a neighborhood
of the origin in $\AA^n/G_X$. Intersecting the preimage in $\AA^n$
with its $G_X$-images to get a $G_X$-invariant neighborhood gives
the $U'_X$ we need. Finally, properness follows from  the usual composition
rules for proper map, using the fact that the map
for a stack to its coarse space is proper.
\end{proof}
\par
\begin{rem}
As a consequence of the orderly blowup description, we see
that the isotropy group of a point in the stack $\MSgrp$ has no
contribution from the group $K_\Gamma$ defined in Section~\ref{sec:coverGH}.
Consequently, whenever this group is non-trivial, the stacks
$\LMS$ and $\MSgrp$ are not locally isomorphic.
\par
Conversely, whenever $K_\Gamma$ is trivial, the functors of marked
and simple marked multi-scale differentials agree, as we have seen
in Section~\ref{sec:simplemarkfamily}. By definition
$\LMS$ and $\MSgrp$ are then isomorphic in a neighborhood of such a stratum.
This happens of course in the interior $\omoduli$, but also at the
generic point of any boundary divisor, by definition of $K_\Gamma$.
\end{rem}

%%%%%%%%%%%%%%%%%%%%%%%%
\subsection{Some moduli spaces in genus zero and cherry divisors}
%%%%%%%%%%%%%%%%%%%%%%%%

To illustrate the necessity of both the orderly blowup and the subsequent normalization
in the passage from IVC to~$\MSgrp$ we consider the following class of divisors.
\par
A {\em cherry divisor} is a boundary divisor of~$\LMS$ such that the generic \msd has one top-level component  and two components at the second level, each connected to the top level by one node. Note that the forgetful map from the moduli space of \msds to the incidence variety compactification (and hence to the Deligne-Mumford compactification) contracts any cherry divisor, as we saw in Example~\ref{ex:q-factorial}.
\par
\begin{exa} \label{ex:cherry}
  ({\em The cherry requires both the orderly blowup and the normalization})
We consider the incidence variety compactification of
$\PP\omoduli[0,5](2,1,0,0,-5)$, with two marked regular points on the surface. Note that in this
case the IVC is simply $\barmoduli[0,5]$, and in particular it is smooth.
On the right in Figure~\ref{fig:cherry} we schematically depict the local
structure of $\PP\obarmoduliinc[0,5]{2,1,0,0,-5}$ near the point that is
the image of the cherry divisor in the moduli space of \msds. We will study the cherry where the marked points meet the zeros of orders~$1$ and $2$, respectively. This point is the intersection of two boundary divisors of the IVC, the first one parameterizing the differentials where the zero of order $1$ meets a marked point, and the second parameterizing the differentials where the zero of order~$2$ meets the other marked point. We introduce local coordinates $x,y$ on the IVC near this cherry point, such that the first divisor is the locus $\lbrace x=0 \rbrace$ and the second one is $\lbrace y=0 \rbrace$. Note that the number of prongs is respectively equal to $\kappa_{1}=2$ and $\kappa_{2}=3$ along these two divisors.
  \par
    \begin{figure}[htb]
   \centering
\begin{tikzpicture}[scale=1]

%%%%%%%%%%% The the moduli space of \msds %%%%%%%%%%%%%%%%%
\begin{scope}
 \draw (-2,0) .. controls ++(-30:1) and ++(210:1) ..(2,0)coordinate[pos=.5] (b3);
  \draw (-1.8,.1) .. controls ++(-90:.5) and ++(70:.5) ..(-2.1,-2.1) coordinate (b1);
  \draw (1.8,.1) .. controls ++(-90:.5) and ++(110:.5) ..(2.1,-2.1) coordinate (b2);
%%%%%%%%%%%%%%%%%%%%%%%%%%%%%%%%%%%%%%
\coordinate (a1) at (0,.8);\fill (a1) circle (2pt);
\coordinate (a2) at (.3,.3);\fill (a2) circle (2pt);
\coordinate (a3) at (-.3,.3);\fill (a3) circle (2pt);
\draw (a1) -- (a2);
\draw (a1) -- (a3);
\draw (a1) --++ (90:.2) node[above] {\tiny{$-5$}} ;
\draw (a2) --++ (-125:.2)node[below] {\tiny{$0$}};
\draw (a2) --++ (-55:.2) node[below] {\tiny{$2$}};
\draw (a3) --++ (-125:.2) node[below] {\tiny{$1$}};
\draw (a3) --++ (-55:.2)node[below] {\tiny{$0$}};
%%%%%%%%%%%%%%%%%%%%%%%%%%%%%%%%%%%
\coordinate (a1) at (-2.4,1.2);\fill (a1) circle (2pt);
\coordinate (a2) at (-2.2,.7);\fill (a2) circle (2pt);
\coordinate (a3) at (-2.6,0.2);\fill (a3) circle (2pt);
\draw (a1) -- (a2);
\draw (a1) -- (a3);
\draw (a1) --++ (90:.2) node[above] {\tiny{$-5$}} ;
\draw (a2) --++ (-125:.2)node[below] {\tiny{$0$}};
\draw (a2) --++ (-55:.2) node[below] {\tiny{$2$}};
\draw (a3) --++ (-125:.2) node[below] {\tiny{$1$}};
\draw (a3) --++ (-55:.2)node[below] {\tiny{$0$}};
%%%%%%%%%%%%%%%%%%%%%%%%%%%%%%%%%
\coordinate (a1) at (2.4,1.2);\fill (a1) circle (2pt);
\coordinate (a3) at (2.2,.7);\fill (a3) circle (2pt);
\coordinate (a2) at (2.6,0.2);\fill (a2) circle (2pt);
\draw (a1) -- (a2);
\draw (a1) -- (a3);
\draw (a1) --++ (90:.2) node[above] {\tiny{$-5$}} ;
\draw (a2) --++ (-125:.2)node[below] {\tiny{$0$}};
\draw (a2) --++ (-55:.2) node[below] {\tiny{$2$}};
\draw (a3) --++ (-125:.2) node[below] {\tiny{$1$}};
\draw (a3) --++ (-55:.2)node[below] {\tiny{$0$}};
%%%%%%%%%%%%%%%%%%%%%%%%%%%%%%%%
\coordinate (a1) at (-2.8,-1.2);\fill (a1) circle (2pt);
\coordinate (a2) at (-3,-1.7);\fill (a2) circle (2pt);
\draw (a1) -- (a2);
\draw (a1) --++ (90:.2) node[above] {\tiny{$-5$}} ;
\draw (a1) --++ (0:.2)node[right] {\tiny{$0$}};
\draw (a1) --++ (180:.2) node[left] {\tiny{$2$}};
\draw (a2) --++ (-125:.2) node[below] {\tiny{$1$}};
\draw (a2) --++ (-55:.2)node[below] {\tiny{$0$}};
%%%%%%%%%%%%%%%%%%%%%%%%%%%%%%%%%
\coordinate (a1) at (2.8,-1.2);\fill (a1) circle (2pt);
\coordinate (a2) at (3,-1.7);\fill (a2) circle (2pt);
\draw (a1) -- (a2);
\draw (a1) --++ (90:.2) node[above] {\tiny{$-5$}} ;
\draw (a1) --++ (0:.2)node[right] {\tiny{$0$}};
\draw (a1) --++ (180:.2) node[left] {\tiny{$1$}};
\draw (a2) --++ (-125:.2) node[below] {\tiny{$2$}};
\draw (a2) --++ (-55:.2)node[below] {\tiny{$0$}};
%%%%%%%%%%%%%%%%%%%%%%%%%%%%
\node[below] at (b1) {$\lbrace x=0 \rbrace$};
\node[below] at (b2) {$\lbrace y=0 \rbrace$};
\node[below] at (b3) {$\lbrace x=0 \rbrace \cap \lbrace y=0 \rbrace$};
  \end{scope}

\draw[->] (3.7,-1) -- (5.1,-1);
%%%%%%%%%%% deligne mumford compactification %%%%%%%%%%%%%%%%%
\begin{scope}[xshift=7cm]
  \draw (.3,.4) .. controls ++(-90:.2) and ++(80:.7) ..(-1,-2.1)coordinate (b1);
  \draw (-.3,.4) .. controls ++(-90:.2) and ++(100:.7) ..(1,-2.1)coordinate (b2);
%%%$$%%%%%%%%%%%%%%%%%%%%%%%%%%%%%
\coordinate (a1) at (1,0.5);\fill (a1) circle (2pt);
\coordinate (a2) at (1.4,0);\fill (a2) circle (2pt);
\coordinate (a3) at (.6,0);\fill (a3) circle (2pt);
\draw (a1) -- (a2);
\draw (a1) -- (a3);
\draw (a1) --++ (90:.2) node[above] {\tiny{$-5$}} ;
\draw (a2) --++ (-125:.2)node[below] {\tiny{$0$}};
\draw (a2) --++ (-55:.2) node[below] {\tiny{$2$}};
\draw (a3) --++ (-125:.2) node[below] {\tiny{$1$}};
\draw (a3) --++ (-55:.2)node[below] {\tiny{$0$}};
%%%%%%%%%%%%%%%%%%%%%%%%%%%%%%%%%%%%%%%%
\coordinate (a1) at (-1.2,-1.2);\fill (a1) circle (2pt);
\coordinate (a2) at (-1.4,-1.7);\fill (a2) circle (2pt);
\draw (a1) -- (a2);
\draw (a1) --++ (90:.2) node[above] {\tiny{$-5$}} ;
\draw (a1) --++ (0:.2)node[right] {\tiny{$0$}};
\draw (a1) --++ (180:.2) node[left] {\tiny{$2$}};
\draw (a2) --++ (-125:.2) node[below] {\tiny{$1$}};
\draw (a2) --++ (-55:.2)node[below] {\tiny{$0$}};
%%%%%%%%%%%%%%%%%%%%%%%%%%%%%%%%%%%%%%%%%
\coordinate (a1) at (1.2,-1.2);\fill (a1) circle (2pt);
\coordinate (a2) at (1.4,-1.7);\fill (a2) circle (2pt);
\draw (a1) -- (a2);
\draw (a1) --++ (90:.2) node[above] {\tiny{$-5$}} ;
\draw (a1) --++ (0:.2)node[right] {\tiny{$0$}};
\draw (a1) --++ (180:.2) node[left] {\tiny{$1$}};
\draw (a2) --++ (-125:.2) node[below] {\tiny{$2$}};
\draw (a2) --++ (-55:.2)node[below] {\tiny{$0$}};

\node[below] at (b1) {$\lbrace x=0 \rbrace$};
\node[below] at (b2) {$\lbrace y=0 \rbrace$};
  \end{scope}
 \end{tikzpicture}
\caption{The orderly blowup of the incidence variety compactification of $\PP\omoduli[0,5](2,1,0,0,-5)$ at a cherry point.}
\label{fig:cherry}
\end{figure}
  \par
Let us perform the orderly blowup in the neighborhood of the cherry point. We have
to blow up the ideal $(x^{2},y^{3})$ discussed in Example~\ref{ex:orbnotnormal} (see~\eqref{eq:non-normal} for the description).
We recall that the total space of this blowup is not normal, and that the exceptional locus of this orderly blowup is a $\PP^{1}$ which is parameterized by the ratio of the differentials on the two lower level components. This exceptional locus meets the strict transforms of the two divisors $\lbrace x=0\rbrace$ and $\lbrace y=0\rbrace$ in two distinct points. The complete picture of this orderly blowup is represented in Figure~\ref{fig:cherry}. Hence in this case the moduli space of \msds is obtained by normalizing the orderly blowup of the IVC, and this normalization is not the identity map. We note moreover that in this case all prong-matchings are equivalent, and thus this difficulty is not due to the choice of a prong-matching.

\smallskip
\par
We now illustrate the fact that the orderly blowup does not see the prong-matchings in general. Consider  the stratum $\PP\omoduli[0,5](1,1,0,0,-4)$. We will study the cherry where the marked points meet respectively the simple zeros, so that the number of prongs is $\kappa_1=\kappa_2=2$, and there are two non-equivalent prong-matchings on the generic cherry curve.
\par
The orderly blowup is given by the equation
  \[  \left\lbrace(x,y)\times [u,v] \in \CC^{2}\times \PP^{1} : x^{2}v-y^{2}u=0 \right\rbrace\,.\]
Note that this space has two locally irreducible branches meeting along the exceptional divisor. In the moduli space of \msds, the limits from these two branches will give non-equivalent prong-matchings for the limiting twisted differential. But  in the orderly blowup,  both branches converge to the same limit. Hence it is not possible to distinguish the prong-matchings from the orderly blowup. However, the normalization of the orderly blowup precisely separates these two branches corresponding to the two non-equivalent prong-matchings.
\end{exa}
%%%%%%%%%%%%%%%%%%%%%%%%%%%%%%%%%%%%
\section{Extending the $\GLtwoRplus$-action to the boundary} \label{sec:SL2R}
%%%%%%%%%%%%%%%%%%%%%%%%%%%%%%%%%%%%%

The goal of this section is to modify the boundary of the moduli space of
\msds in such a way that the $\GLtwoRplus$-action on the open stratum extends
to this boundary, and such that the quotient of this compactification
by rescaling by positive real numbers is compact. The reason we need to
consider rescaling by $\RR_{>0}$ instead of by $\CC^{*}$ is essentially due to
the fact that $\GLtwoRplus$ does not act meaningfully on $\omoduli[g,n](\mu) /
\CC^{*}$ but it does act on $\omoduli[g,n](\mu) / \RR_{>0}$,
as ${\rm SO}_2(\RR)$ is not  contained in the center of of $\GLtwoRplus$
but $\RR_{>0}$ is.
The concept of \lw real blowup provides the setup for this purpose. A related
bordification, also a manifold with corners,
is also studied in an ongoing project of Smillie and Wu with the goal of
understanding the $\GLtwoRplus$-action near the boundary. While the constructions
have certain similarities, they apparently differ e.g.\ in the treatment
of horizontal nodes.
\par
\begin{thm} \label{thm:extendSL2}
The $\GLtwoRplus$-action on the moduli space $\omoduli[g,n](\mu)$ extends to a
continuous $\GLtwoRplus$-action on the \lw real blowup $\RBLMS$
\index[teich]{f060@$\RBLMS$! Level-wise real blowup of  $\LMS$}of the moduli
space of multi-scale differentials $\LMS$ along its total boundary divisor.
\end{thm}
\par
In comparison to Section~\ref{sec:rob}, note that $\RBLMS$
agrees with $\overline\LMS$ (where the long upper bar refers to the \lw real
blowup) because the generic fiber is smooth, and there are no persistent nodes.
\par
The basic objects parameterized by $\RBLMS$ are \emph{real} \msds,
replacing \msds. The definition is very similar to Definition~\ref{def:one_msd},
simply replacing the equivalence relation to be by real scaling.
\par
\begin{df}
  \label{def:one_rsd}
A {\em real \msd of type $\mu$} on a stable curve~$X$ is
\begin{itemize}
  \item[(i)] a full order $\cleq$ on the dual graph~$\Gamma$ of~$X$,
  \item[(ii)] a differential $\omega_{(i)}$ on each level $X_{(i)}$, such that
the collection of these differentials satisfies the properties of a twisted
differential of type~$\mu$ compatible with $\cleq$, and
\item[(iii)]  a \prma $\sigma=(\sigma_e)$ where $e$ runs through all vertical edges
of~$\Gamma$.
\end{itemize}
Two real \msds are considered equivalent if they differ by rescaling at each
level (except the top level) by multiplication by a positive real number.
\end{df}
\par
To properly work with families of such differentials, we have to leave the category
of complex spaces. Recall that {\em manifolds with corners} are topological spaces
locally modeled on $[0,\infty)^k \times \RR^{n-k}$. These spaces form a category
(with a notion of smooth maps, see \cite{joyce} for a recent
account with definitions and caveats, but we will not detail here). Since
 $\LMS$ already has non-trivial orbifold structures, we in fact have to work with
{\em orbifolds with corners}, where the local orbifold charts are manifolds with
corners and where the local group actions are smooth maps preserving the boundary.
\par
\begin{thm} \label{thm:points_in_RBLMS}
The \lw real blowup $\RBLMS$ is an orbifold with corners. Its points
correspond bijectively to isomorphism classes of real \msds.
The orbifold structure is exclusively due to automorphisms of
flat surfaces, as for $\omodulin$.
\end{thm}
\par
The last statement says that $\RBLMS$ resolves the quotient singularites
by the groups $K_\Lambda = \Tw/\sTw$ of $\LMS$:
\par
Given Theorem~\ref{thm:points_in_RBLMS}, we define
the action of $A \in \GLtwoRplus$ on $\RBLMS$ by
\bes
  A \cdot \left((X_{(i)}), (\omega_{(i)}), \sigma \right)
  \,\= \, \left((A \cdot (X_{(i)}, \omega_{(i)})), A\cdot \sigma\right)\,,
\ees
where $i\in L^\bullet(\Gamma)$. The
first argument is the usual $\GLtwoRplus$-action on the components
of the stable curve. For the second argument we use the action of~$A$ on
the set of directions and note that a matching of horizontal directions
(for $\bfomega$) gives a matching of directions of slope $A \cdot \left(
\begin{smallmatrix} 1 \\ 0 \end{smallmatrix}\right)$ (for $A \cdot \bfomega$) that can be
reconverted into a matching of horizontal directions.
\par
The notion of families of real \msds (over bases being orbifolds with
corners) can now be phrased as in Section~\ref{sec:famnew}, with the equivalence relation
changed from  $\Tprong[\Gamma](\calO_{B,p})$-action to \lw $\RR_{>0}$-multiplication,
as above. We will now essentially construct a universal object.
\par
All the statements in the above theorems are local, since continuity of the
$\GLtwoRplus$-action
can also be probed by a neighborhood of the identity element. We thus pick a
point~$p \in \LMS$ and work in a neighborhood~$U$ that will be shrunk for convenience,
e.g.\ to apply Proposition~\ref{prop:markliftedfam} and to find an enhanced
multicurve~$\eL$ with $\Gamma(\eL) = \Gamma_p$ and provide the restriction of the
universal family over (the orbifold chart of)~$U$ with a $\eL$-marking. We may thus
view $U \subset \ODehn$.
\par
We provide the pullback to the \lw real blowup~$\widehat{U}$ of the universal
family over~$U$ with real \msds. (Note that here as in
Definition~\ref{def:one_rsd}  above, real \msds live on stable
complex curves.) The smooth (differentiable) family $\widehat{X} \to \widehat{U}$
constructed in Theorem~\ref{thm:ORBL} is tacitly used for the marking, but we do not
treat the issue whether differentials can be pulled back there. Let $t_i$ be
the rescaling parameters of the \msd $\bfomega = (\omega_{(i)})_{i \in L^\bullet(\Lambda)}$
on~$U$, and let~$T_i$ and~$F_e$ be the $S^1$-valued functions used in the blowup
construction (Section~\ref{sec:robconstruction}).
\par
\begin{proof}[Proof of Theorem~\ref{thm:points_in_RBLMS}]
The second statement is an immediate consequence of Theorem~\ref{thm:coarseMS}
about the points in $\LMS$ and of Proposition~\ref{prop:ORBL}~(ii).
\par
For the first statement we may use charts of the \lw real blowup of $\ODehns$ as
orbifold charts. There, the boundary is a normal crossing divisor with one
component $D_i = \{t_i =0\}$ for each level (but the top level). The real blowup
of a normal crossing divisor is then known to be a manifold with corners
(see e.g.\ \cite{acgh2}, Section~X.9, in particular page~150),
as in the last statement of Theorem~\ref{thm:ORBL}.
\par
The fibers of the \lw real blowup are a torus consisting of $S^1$ for
each level below zero, see Proposition~\ref{prop:ORBL}. The group
$K_\Lambda = \Tw/\sTw$ acts freely on this torus. This proves the last
statement.
\end{proof}
\par
\begin{proof}[Proof of Theorem~\ref{thm:extendSL2}]
It remains to justify the continuity of the $\GLtwoRplus$-action. Consider a
sequence~$\{\widehat{p}_n\}$ converging to $\widehat{p}$ in $\widehat{U} \subset
\widehat{\ODehn}$. By definition of the topology, this is equivalent (with notations
as in Section~\ref{sec:rob})
to the convergence of the image points $\varphi_U(\widehat{p}_n)$ to $
\varphi_U(\widehat{p})$ in $ U \subset \ODehn$ and to the convergence of
$F_e(\widehat{p}_n) $ to $F_e(\widehat{p})$ and $T_i(\widehat{p}_n)
$ to $ T_i(\widehat{p})$. In turn, the convergence in $\ODehn$ is manifested
by diffeomorphism~$g_n$ satisfying Definition~\ref{def:augmented_topology_def} with
the compatibility with markings relaxed up to elements in $\Tw[\Lambda]$. We aim
to justify the convergence of the sequence of image points $\varphi_U(A \cdot
\widehat{p}_n) $ to $ \varphi_U(A \cdot \widehat{p}_n) \in U \subset \LMS$.
Since $T_i(A \cdot \widehat{p}_n) $ converges to $ T_i(A \cdot \widehat{p})$,
and similarly for~$F_e$, by continuity of the $\GLtwoRplus$-action on~$S^1$,
this will conclude the
proof.
\par
For our aim, we use the maps $A \cdot g_n \cdot A^{-1}$, where $A \cdot (\,)$ denotes
the induced $\GLtwoRplus$-action on pointed flat surfaces. This map is well-defined
away from the seams and we observe that it can be extended to a
differentiable map across each seam using the action of $\GLtwoRplus$
on the seam identified with~$S^1$.
\end{proof}
\par

\section*{Index of notation}

In this section we summarize the notations thematically. In each theme we mainly give the notations in chronological order, omitting the introduction.

\indexprologue{\phantom{bla}\ }
\printindex[surf]

\indexprologue{\phantom{bla}\ }

\printindex[graph]
\indexprologue{Most of the following spaces have a projectivized variant which is indicated with the symbol $\PP$.}

\printindex[teich]
\indexprologue{\phantom{bla}\ }
\printindex[family]
\indexprologue{\phantom{bla}\ }

\printindex[plumb]
\indexprologue{The groups below usually have  ``extended'' analogues
which we denote by a superscript~$\bullet$.}
\printindex[twist]
\indexprologue{\phantom{bla}\ }
\printindex[other]

\par
\printbibliography
\end{document} 

%%% Local Variables:
%%% mode: latex
%%% TeX-master: t
%%% End: